\documentclass[10pt]{amsart} 

\usepackage{amsfonts}
\usepackage{tocvsec2}
\usepackage{combelow}
\usepackage{amsmath}
\usepackage{amssymb} 
\usepackage{amsthm}
\usepackage{comment}
\usepackage{tikz-cd}
\usepackage[style=alphabetic, backend=biber, eprint=true, maxbibnames=99, doi=false, url = false, isbn=false]{biblatex}

\usetikzlibrary{calc}
\tikzset{curve/.style={settings={#1},to path={(\tikztostart)
    .. controls ($(\tikztostart)!\pv{pos}!(\tikztotarget)!\pv{height}!270:(\tikztotarget)$)
    and ($(\tikztostart)!1-\pv{pos}!(\tikztotarget)!\pv{height}!270:(\tikztotarget)$)
    .. (\tikztotarget)\tikztonodes}},
    settings/.code={\tikzset{quiver/.cd,#1}
        \def\pv##1{\pgfkeysvalueof{/tikz/quiver/##1}}},
    quiver/.cd,pos/.initial=0.35,height/.initial=0}
\definecolor{citation}{rgb}{0,.40,.80}
\usepackage[hidelinks]{hyperref}
\hypersetup{
   colorlinks=true,
   citecolor=black,
    linkcolor=citation,
    }
\usepackage{edmaths}
\usepackage{geometry}
\title{CoHA of Cyclic Quivers and an Integral Form of Affine Yangians}
\author{Shivang Jindal}
\date{\today}

\begin{document}

\maketitle

\begin{abstract}
We calculate the deformed and non-deformed cohomological Hall algebra (CoHA) of the preprojective algebra for the case of cyclic quivers by studying the Kontsevich-Soibelman CoHA and using tools from cohomological Donaldson-Thomas theory. We show that for the cyclic quiver of length $K$, this algebra is the universal enveloping algebra of the positive half of a certain extension of matrix differential operators on $\C^{*}$, while its deformation gives a positive half of an explicit integral form of Guay's Affine Yangian $\ddot{\mathcal{Y}}_{\hbar_1,\hbar_2}(\mathfrak{gl}(K))$. By the main theorems of \cite{botta2023okounkovs} and \cite{schiffmann2023cohomological}, we also determine the Maulik-Okounkov Yangian for the case of cyclic quivers. Furthermore, we explain the construction of factorization coproduct, provide evidence for the strong rationality conjecture, calculate the spherical subalgebra of the non-deformed CoHA for any quiver without loops, recover results about the CoHA of compactly supported semistable sheaves on the minimal resolution of the Kleinian singularity $\C^2/\mathbb{Z}_{K}$ and identify a commutative algebra inside the additive shuffle algebra associated to the cyclic quiver. We end by conjecturally relating the obtained integral form with the algebra defined by Gaiotto-Rapčák-Zhou, in the context of twisted M-theory. 
\end{abstract}

\tableofcontents

\section{Introduction}

Given a quiver $Q$ and potential $W \in \C Q/[\C Q,\C Q]$, where $\C Q$ denotes the path algebra, in \cite{kontsevich2011cohomological} Kontsevich and Soibelman defined a cohomological Hall algebra structure on the critical cohomology \[ \mathcal{A}_{Q,W}:= \bigoplus_{\bd \in \N^{Q_0}} \HH(\mathfrak{M}_{\bd}(Q),\pPhi_{\Tr(W)} \Q_{\mathfrak{M}_{\bd}(Q)}^{\vir}) \] as a mathematical definition of the algebra of BPS states. Here $\pPhi_{\Tr(W)}$ is the vanishing cycle functor associated to the function $\Tr(W)_{\bd}$ on the moduli space of $\bd$ dimensional representations $\mathfrak{M}_{\bd}(Q)$ of the quiver $Q$. The main result of this thesis calculates the algebra and its deformations $\mathcal{A}^{T}_{Q,W}$ for the quiver $\tilde{Q^{K}}$ and potential $\tilde{W^{K}}$ as defined in Example \ref{canonicalcubicpotential}. This choice of quiver with potential, often called the tripled quiver with canonical potential, provides a presentation of the moduli space of representations of the preprojective algebra of the cyclic quiver $Q^{K}$, equipped with an endomorphism, as a global critical locus. This gives rise to new integral forms of affine Yangians and Lie algebras, which are of independent interest. We start by explaining our main motivations, coming from enumerative geometry and the local geometry of the moduli space of objects in 2-Calabi-Yau categories.

\subsection{Noncommutative Donaldson-Thomas theory}

Given any projective 3-Calabi-Yau (CY) variety $X$, indivisible Chern class $\alpha \in \HH^{\textrm{even}}(X,\mathbb{Z})$ and a generic stability condition $\zeta$, the Donaldson-Thomas invariants $\DT^{\zeta}_{\alpha}$ defined in \cite{thomas2001holomorphic} are virtual counts of semistable sheaves on $X$ of given Chern class $\alpha$. They are defined by taking the degree of the virtual fundamental class of the moduli space of sheaves $\mathcal{M}^{\zeta}_{\alpha}(X)$. Later, given any scheme $Y$, Behrend defined a constructible function $\nu_Y \colon Y \rightarrow \C$ and showed that  $\DT^{\zeta}_{\alpha}$ is precisely the weighted Euler characteristic $\chi(\mathcal{M}^{\zeta}_{\alpha}(X), \nu_{\mathcal{M}^{\zeta}_{\alpha}(X)})$ \cite{behrend2005donaldsonthomas}[Theorem 4.18]. When $Y$ can be written as the critical locus of a function $f: \tilde{Y} \rightarrow \C$ where $\tilde{Y}$ is smooth, then there exists the vanishing cycle sheaf $\pPhi_{f} \Q^{\vir}$ which satisfies (see \cite{kiem2016categorification}[2.1] for details) \[\sum (-1)^{i} \dim \HH^{i}(\tilde{Y},\pPhi_f \Q^{\vir}) = \chi(Y,v_{Y}). \] Thus, the cohomology \[\HH(\tilde{Y},\pPhi_f \Q^{\vir})\] could be understood as a categorification of the DT invariant. 

For any 3CY variety $X$, the moduli space $\mathfrak{M}^{\zeta}_{\alpha}(X)$ can rarely be written as a global critical locus. However, by the result of \cite{Toda_2018} analytically locally these moduli spaces can be written as the moduli space of representations of a Jacobi algebra $\Jac(Q,W)$(Definition \ref{jacobialgebra})\footnote{To be more precise, Toda's theorem gives an analytic potential and so we need to consider an analytic Jacobi algebra $\Jac \{Q, W \}$ as defined in \cite{davison2023refined}[Definition 2.7].}. Furthermore, the moduli stack of representations of a Jacobi algebra can be written as the critical locus of a smooth function $\Tr(W)$ on a smooth moduli space $\mathfrak{M}_{\bd}(Q)$. The study of the DT invariants of a non-compact Calabi–Yau threefold via virtual counts of moduli of $\Jac(Q,W)$ modules, where $\Jac(Q,W)$ is the Jacobi algebra derived equivalent to $X$, was initiated by Szendr\H{o}i in the case of the conifold \cite{Szendr_i_2008}. For any quiver $Q$, we can define the tripled quiver $\tilde{Q}$ with the canonical cubic potential $\tilde{{W}}$ (Example \ref{canonicalcubicpotential}). When $X= \C^{3}$ the moduli stack $\Coh_{n}(\C^{3})$ of torsion sheaves of length $n$ can be described as a moduli stack of $n$ dimensional representations of the Jacobi algebra of the tripled Jordan quiver $\tilde{Q_{\Jor}}$ with canonical potential $\tilde{W}_{\Jor}: = x[y,z]$.

The example concerning the main result of this paper captures the geometry of Kleinian surfaces. Let $S_{K} \rightarrow \C^2/\mathbb{Z}_{K+1}$ be the minimal resolution of the quotient singularity where the action of $\Z_{K+1}$ is given by $(x,y) \rightarrow (\omega x, \omega^{-1}y)$ and $\omega^{K+1}=1$ is a primitive root of unity. We then have an equivalence of derived categories
\begin{align} \label{derivedequivalence}
\Db(\Coh(S_{K} \times \C)) \simeq \Db(\Jac(\QTC{K},\WTC{K})) 
\end{align} 
where $Q^K$ is the cyclic quiver with $K+1$ vertices. The Jacobi algebra $\Jac(\QTC{K},\WTC{K})$ is the $3$CY completion of the preprojective algebra $\mathrm{\Pi}_{Q^K}$ \cite{keller2021introduction}.

\subsection{Cohomological Hall algebras of preprojective algebras}

There is a deep relationship between cohomological Donaldson-Thomas Theory of objects in the 3CY completion $\tilde{\mathcal{C}}$ of a 2CY category $\mathcal{C}$ and the Borel-Moore homology of objects in $\mathcal{C}$ which goes under the name of dimensional reduction (\cite{Kinjo_2022}, \cite{davisondeformed}, \cite{davison2016critical}[Appendix]). In particular, the dimensional reduction theorem implies that for any quiver $Q$, there is an isomorphism\footnote{Note that the quotient stack $\mathfrak{M}_{\bd}(\mathrm{\Pi}_Q)$ is not smooth, so here by $\Q^{\vir}$ we mean the shifted complex $\Q[-2(\bd,\bd)]$, where $(\hyphen, \hyphen)$ denotes the Euler form, as defined in Definition \ref{eulerform}.}   
\begin{equation} \label{dimensionreductionexample}
 \HH(\mathfrak{M}_{\bd}(\tilde{Q}),\pPhi_{\Tr(\tilde{W})}\Q^{\vir}_{\mathfrak{M}_{\bd}(\tilde{Q})}) \simeq \HH^{\textrm{BM}}(\mathfrak{M}_{\bd}(\mathrm{\Pi}_Q),\Q^{\vir}) 
\end{equation}
The mixed Hodge structure on the right, i.e. on the BM homology of the stack of representations of the preprojective algebra $\mathrm{\Pi}_Q$, plays an important role in multiple branches of mathematics. In particular, by the Ext quiver construction, the stack of representations of the preprojective algebra \'{e}tale locally models the geometry of the stack of objects in arbitrary 2CY categories possessing good moduli spaces \cite{davison2023purity}. Examples of such moduli spaces include Higgs bundles on smooth projective curves, moduli of local systems on Riemann surfaces, and coherent sheaves on K3 surfaces. The first two examples also make studying this space a valuable tool in the study of non-abelian Hodge theory, while on the other hand, there is an algebra structure on \[ \mathcal{A}_{\mathrm{\Pi}_Q} := \bigoplus_{\bd \in \N^{Q_0}} \HBM(\mathfrak{M}_{\bd}(\mathrm{\Pi}_Q),\Q^{\vir}) \]
which has been defined in the context of equivariant cohomology for the case of Jordan quiver \cite{schiffmann2012cherednik} in the context of AGT conjecture and later defined for arbitrary quivers in \cite{Yang_2018}. This algebra acts on the cohomology of Nakajima quiver varieties by raising operators and, in some sense, is the biggest possible algebra acting on the cohomology of Nakajima quiver varieties, making it an important object in geometric representation theory.
It is shown in \cite{Ren:2015zua}[Appendix A] and \cite{Yang_2019} that up to a sign twist, the isomorphism in (\ref{dimensionreductionexample}) in fact preserves the algebra structure and thus is an isomorphism of algebras. Thus studying $\mathcal{A}_{Q,W}$ completely determines $\mathcal{A}_{\mathrm{\Pi}_Q}$.

\section{Results}

\subsection{Integral matrix $\mathcal{W}_{1+\infty}$ Lie algebras} \label{matrixwLiealgebra} 
Let $\hbar$ be a formal variable and let $D_{\hbar}(\C^{*})$ be the algebra of $\hbar \hyphen$differential operators on $\C^{*}$. It is defined as a unital associative $\C[\hbar]$ linear algebra generated by $z^{\pm 1}, D$ subject to the relations:
\begin{equation}
D z = z(D + \hbar), \ z^{+ 1} z^{-1} = z^{-1} z^{+1} = 1.
\end{equation} 
We will view $D_{\hbar}(\C^{*})$ as a $\C[\hbar]$ linear Lie algebra with the commutator Lie bracket $[\cdot, \cdot]$, coming from the associative algebra structure on $D_{\hbar}(\C^{*})$. A central extension of the Lie algebra $D_{\hbar}(\C^{*})$, after specialization at $\hbar =1$, is called $\mathcal{W}_{1+\infty}$ Lie algebra. The $\mathcal{W}_{1+\infty}$ Lie algebra was first introduced in \cite{18f811f0-ae01-3567-abe3-dc5d50ca3cde}. Its connection with the vertex algebras $\mathcal{W}(\mathfrak{gl}_N)$ for central charge $N$, where $N$ is any natural number, has been developed in \cite{Frenkel_1995} \cite{kac1995representation}. Their matrix 
generalization was considered in \cite{Awata_1995} (see also \cite{strominger2021w}, \cite{eberhardt2019matrixextended} and \cite{Creutzig_2019} for a more physical perspective). 

Let $D_{\hbar}(\C^{*}) \otimes \mathfrak{gl}_K$ be the $\C[\hbar]$ linear associative algebra of $K \times K$ matrices with values in $D_{\hbar}(\C^{*})$, i.e. for any $f(z,D) \otimes X$ and $ g(z,D) \otimes Y $ in $D_{\hbar}(\C^{*}) \otimes \mathfrak{gl}_K $ where $f(z,D),g(z,D) \in D_{\hbar}(\C^{*})$ and $X,Y \in \mathfrak{gl}_K$, the multiplication is given by \[ (f(z,D) \otimes X) (g(z,D) \otimes Y) = f(z,D)g(z,D) \otimes XY. \]

We view $D_{\hbar}(\C^{*}) \otimes \mathfrak{gl}_K$ as a $\C[\hbar]$ linear Lie algebra with the commutator Lie bracket coming from the associative algebra structure on $D_{\hbar}(\C^{*}) \otimes \mathfrak{gl}_K$, i.e. Lie bracket is given by 
\[ [ f(z,D) \otimes X, g(z,D) \otimes Y] = (f(z,D) g(z,D)) \otimes XY - (g(z,D)f(z,D)) \otimes YX  \] for any $f(z,D) \otimes X,g(z,D)  \otimes Y $ in  $D_{\hbar}(\C^{*}) \otimes \mathfrak{gl}_K$. The elements \[T_{k,a}(X)= z^kD^a \otimes X \textrm{ where } k \in \mathbb{Z}, a \geq 0, X \in \mathfrak{gl}_{K} \] form a spanning subset of $D_{\hbar}(\C^{*}) \otimes \mathfrak{gl}_K$. We slightly extend this Lie algebra by defining an integral form.
% Note that this Lie algebra deforms the $\O(\TT^{*}\C^{*}) \otimes \mathfrak{gl}_K$. 

\begin{defn} \label{integralsubalgebra}
Let \[\mathcal{W}_K \subset (D_{\hbar}(\C^{*}) \otimes \mathfrak{gl}_{K}) \otimes_{\C[\hbar]} \C[\hbar^{\pm 1}] \] be the $\C[\hbar]$ linear subspace spanned by $T_{k,a}(X), X \in \mathfrak{gl}_{K}$ and \[t_{k,a} := T_{k,a}(1)/\hbar \] where $k \in \mathbb{Z}, a \geq 0, X \in \mathfrak{gl}_K$ and by $1$, we mean the identity matrix $\Id \in \mathfrak{gl}_K$. 

\end{defn}
The subspace $\mathcal{W}_K \subset (D_{\hbar}(\C^{*}) \otimes \mathfrak{gl}_{K}) \otimes_{\C[\hbar]} \C[\hbar^{\pm 1}]$ in fact forms a Lie subalgebra (Proposition \ref{integralsubLiealgebra}). 
We remark that the Lie algebra $\mathcal{W}_{K}$ is also considered in \cite{gaiotto2023deformed}. We consider the classical limit as $\hbar \rightarrow 0$ of $\mathcal{W}_K$. 

Let $\O(\mathrm{T}^{*}\C^{*})$ be the ring of functions on $\TT^{*}\C^{*}$. We denote by $\mathfrak{po}(\TT^{*}\C^{*})$, the Lie algebra on $\O(\TT^{*}\C^{*})$, where the Lie bracket is given by the standard Poisson bracket $\{ \hyphen, \hyphen \}$. Let $\O(\TT^{*}\C^{*}) \otimes \mathfrak{sl}_K$ be the Lie algebra with the Lie bracket given by \[ [f \otimes X, g \otimes Y ] = fg \otimes [X,Y]\] for any $X,Y \in \mathfrak{sl}_K$ and $f,g \in \O(\TT^{*}\C^{*})$. 

The Lie algebra $\mathfrak{po}(\TT^{*}\C^{*})$ acts on $\O(\TT^{*}\C^{*}) \otimes \mathfrak{sl}_{K}$ via the Poisson bracket with the first tensor product, i.e. \[f \cdot( g \otimes X):= \{ f,g \} \otimes X, \] for any $f \in \mathfrak{po}(\TT^{*}\C^{*})$ and $g \otimes X \in \O(\TT^{*}\C^{*}) \otimes \mathfrak{sl}_{K}$. Let \[\mathfrak{po}(\TT^{*}\C^{*}) \ltimes (\O(\TT^{*}\C^{*}) \otimes \mathfrak{sl}_{K}) \] be the Lie algebra given by the semidirect product of $\mathfrak{po}(\TT^{*}\C^{*})$ with $\O(\TT^{*}\C^{*}) \otimes \mathfrak{sl}_{K}$. The classical limit $\hbar \rightarrow 0$ of $\mathcal{W}_{K}$ is exactly the Lie algebra $\mathfrak{po}(\TT^{*}\C^{*}) \ltimes (\O(\TT^{*}\C^{*}) \otimes \mathfrak{sl}_{K})$, i.e. There is an isomorphism of Lie algebras \[ (\mathcal{W}_{K})/(\hbar=0) \simeq \mathfrak{po}(\TT^{*}\C^{*}) \ltimes (\O(\TT^{*}\C^{*}) \otimes \mathfrak{sl}_{K}). \]

% When we consider $\hbar \rightarrow 0$, we obtain a Lie algebra $L_K$ generated by  $T_{k,a}(X)$ where $k \geq 1, a \geq 0, X \in \mathfrak{sl}_{K}$ or $k=0, a \geq 0, X \in \mathfrak{n}_{K}$ where $\mathfrak{n}_{K} \subset \mathfrak{gl}_{K}$ is the Lie subalgebra generated by upper diagonal matrices and $t_{k,a}$ where $k \geq 1, a \geq 0$ with the relations 
%   \begin{align*}\label{Relations_1}
%    [t_{m,a},T_{n,b}] &= (na-mb)T_{m+n,a+b-1}(X) \\ [t_{m,a},t_{n,b}] &= (na-mb)t_{m+n,a+b-1} \\ [T_{m,a}(X),T_{n,b}(Y)] &= T_{m+n,a+b}([X,Y])
%    \end{align*} 

In Proposition \ref{generatorrelationswkliealgebra}, we present this Lie algebra via generators and relations. Our first main result describes the cohomological Hall algebra $\mathcal{A}_{\QTC{K},\WTC{K}}$ as a universal enveloping algebra of the positive half of the Lie algebra $\mathfrak{po}(\TT^{*}\C^{*}) \ltimes (\O(\TT^{*}\C^{*}) \otimes \mathfrak{sl}_{K})$.
\begin{mainthm}[Theorem \ref{Theorem1}] \label{th}
Let $K \geq 1$. There is an isomorphism of Lie algebras \[ \abps{K} \simeq (\mathfrak{po}(\TT^{*}\C^{*}) \ltimes (\O(\TT^{*}\C^{*}) \otimes \mathfrak{sl}_{K+1}) )^{+} \] between the the affinized BPS Lie algebra $\abps{K}$ for the cyclic quiver (Section \ref{affinizedBPSliealgebra}) and the positive half $(\mathfrak{po}(\TT^{*}\C^{*}) \ltimes (\O(\TT^{*}\C^{*}) \otimes \mathfrak{sl}_{K}) )^{+}  \subset \mathfrak{po}(\TT^{*}\C^{*}) \ltimes (\O(\TT^{*}\C^{*}) \otimes \mathfrak{sl}_{K}) $, defined in Definition \ref{positivehalfofclassicallimit}. This gives an isomorphism of algebras \[ \mathcal{A}^{\chi}_{\QTC{K},\WTC{K}} \simeq 
  \bU((\mathfrak{po}(\TT^{*}\C^{*}) \ltimes(\O(\TT^{*}\C^{*}) \otimes \mathfrak{sl}_{K+1}) )^{+}), \]
where $\chi$ is a sign twist, defined in Section \ref{signtwist}. 
\end{mainthm}

It is also possible to recover the positive half of $\mathcal{W}_K$ itself, without taking the classical limit. To do so, we consider the deformed cohomological Hall algebra. For any torus $T$ acting on the quiver $\QTC{K}$, leaving the potential $\WTC{K}$ invariant, there is a cohomological Hall algebra (Section \ref{hallalgebra}) structure on the equivariant vanishing cycle cohomology 
\[ \mathcal{A}^{T}_{\QTC{K},\WTC{K}} :=  \bigoplus_{\bd \in \N^{K+1}} \HH_{T}(\mathfrak{M}_{\bd}(\QTC{K}),\pPhi_{\Tr(\WTC{K})} \Q_{\mathfrak{M}_{\bd}(\QTC{K})}^{\vir}).\]

Our next main result, calculates this algebra for $T= \C^{*}$, whose action on $\mathfrak{M}_{\bd}(\QTC{K})$, leaves the added loops $\omega_i$ invariant for all $i \in Q_0$. We set up a bit more notation. For any commutative ring $R$ and a $R$ linear Lie algebra $\mathfrak{g}$, its tensor algebra is defined to be
\[ T_R(\mathfrak{g}) := \bigoplus_{m \geq 0} \underbrace{\mathfrak{g} \otimes_R \dots \otimes_R \mathfrak{g}}_{\text{m times }}, \] and correspondingly, the universal enveloping algebra over $R$ is defined to be the quotient 
\[ \bU_{R}(\mathfrak{g}) := T_R(\mathfrak{g})/ \langle ab-ba - [a,b] \mid a,b \in \mathfrak{g} \rangle. \] %while the symmetric algebra over $R$ is defined to be the quotient 
%\[ \Sym_{R}(\mathfrak{g}) := T_{R}(\mathfrak{g})/\langle ab-ba = 0 \mid a,b \in \mathfrak{g} \rangle. \]

%We consider $\C^{*}$ action on the moduli space $\mathfrak{M}_{\bd}(\QTC{K})$ induced by the action of $\C^{*}$ on $\QTC{K}$ by acting on orignal arrows $a$ with weights $1$, opposite arrows $a^{*}$ with weights $-1$ and leaving the added loops invariant. Then there is a cohomological Hall algebra structure on equivariant vanishing cycle cohomology 

\begin{mainthm}[Theorem \ref{Theorem2}]
For $K \geq 1$, let $\C^{*}$ act on representations of the tripled quiver $\QTC{K}$ by acting on all the original arrows $a$ with weight $1$, all opposite arrows $a^*$ with weight $-1$ and with weight $0$ on added loops $\omega_i, i \in [0,K]$ (Section \ref{bpsLiealgebracyclic}). Then there is an isomorphism of $\C[\hbar]$ linear Lie algebras \[\widehat{\mathfrak{g}}^{\BPS, \C^{*}}_{\QTC{K},\WTC{K}} \simeq (\mathcal{W}_{K+1})^{+} \] between the deformed affinized BPS Lie algebra, defined in the Section \ref{affinizedBPSliealgebra} and the positive half of an integral form of differential operators on $\C^{*}$ valued in matrices, defined in Defintion \ref{positivehalfofdeformedmatrixdifferentialoperators}. This gives an isomorphism of $\C[\hbar]$ linear algebras \[ \mathcal{A}^{\C^{*},\chi}_{\QTC{K},\WTC{K}} \simeq 
  \bU_{\C[h]}( (\mathcal{W}_{K+1})^{+})  \] where $\chi$ is a sign twist, defined in Section \ref{signtwist}. \end{mainthm} Next, we cansider the action of a larger torus $\T = \C^{*} \times \C^{*}$, which acts on the added loops $\omega_i$ for all $i \in Q_0$ non-trivially. This gives rise to quantum groups. 
\subsection{Integral form of Yangians}
Given a semisimple Lie algebra $\mathfrak{g}$, in \cite{doi:10.1142/9789812798336_0013} Drinfeld defined the Yangian $\mathcal{Y}_{\hbar}(\mathfrak{g})$ as a deformation of the universal enveloping algebra $\bU(\mathfrak{g}[x])$ of the polynomial current Lie algebra $\mathfrak{g}[x]$. Yangians form a family of quantum groups related to rational solutions of classical Yang-Baxter equations. For any quiver $Q$, in \cite{maulik2018quantum}, the authors defined the Yangian $\textbf{Y}^{\textrm{MO}}_{Q}$ by geometrically defining $R$ matrices using stable envelopes on the singular cohomology of Nakajima quiver varieties and applying the $FRT$ formalism to construct quantum groups. One of Maulik and Okounkov's motivations was to realise the Quantum cohomology ring as a maximal commutative subalgebra of the Yangian. From the calculations in \cite{MR2525779} for the quiver variety for cyclic quivers, it was clear that one needs slightly larger algebra than the usual affine Yangians for the case of cyclic quivers. They conjectured a strong relation between the quantum cohomology of the Nakajima quiver varieties and the Bethe subalgebras of $\textbf{Y}^{\textrm{MO}}_{Q}$. However, the algebras $\textbf{Y}^{\textrm{MO}}_{Q}$ are not known for cases beyond ADE and Jordan quivers. The computation for the Jordan quiver occupies much of the latter half of \cite{maulik2018quantum}, where it is shown to be isomorphic to the affine Yangian $\ddot{\mathcal{Y}}_{\hbar_1,\hbar_2}(\mathfrak{gl}(1))$ of $\widehat{\mathfrak{gl}(1)}$ and used in their proof of AGT conjecture. For the ADE quivers $Q$, a central quotient of $\mathbf{Y}^{\mathrm{MO}}_{Q}$ is shown to be isomorphic to the Yangian $\mathcal{Y}_{\hbar}(\mathfrak{g}_{Q})$, as defined by Drinfeld \cite{mcbreenthesis}. 

For any quiver with potential $Q,W$, the cohomological Hall algebras are also known to satisfy a similar property, i.e. there is a Lie algebra $\mathfrak{g}^{\BPS,T}_{Q,W}$ such that $\mathcal{A}^{T}_{Q,W} \simeq \Sym_{\HH_{T}(\pt)}(\mathfrak{g}^{\BPS,T}_{Q,W}[u])$ as graded vector spaces \cite{davison2020cohomological}, allowing us to think of $\mathcal{A}^{T}_{Q,W}$ as generalized Yangians. For particular choices of torus, it is shown that for ADE quivers, there is an isomorphism of algebras $\mathcal{A}^{\T}_{\tilde{Q},\tilde{W}} \simeq Y^{+}_{\hbar}(\mathfrak{g}_{Q})$(\cite{Yang_2018}+ \cite{schiffmann2017cohomological}) and for Jordan quiver, there is an isomorphism of algebras $\mathcal{A}^{\T}_{\tilde{Q_{\Jor}},\tilde{W_{\Jor}}} \simeq \ddot{\mathcal{Y}}^{+}_{\hbar_1,\hbar_2}(\mathfrak{gl}(1))$ by (\cite{davison2022affine}+ \cite{rapcak2023cohomological}) and \cite{schiffmann2012cherednik} where $+$ denotes the positive half. Furthermore, for all tripled quivers with canonical potential $\tilde{Q},\tilde{W}$, there is an isomorphism of algebras 
\[ \textbf{Y}^{\textrm{MO}, +}_{Q} \simeq \mathcal{A}^{T}_{\tilde{Q},\tilde{W}}\] as proved in \cite{botta2023okounkovs} and isomorphism \[ \textbf{Y}^{\textrm{MO}, -}_{Q} \simeq \mathcal{A}^{T,\tilde{\mathcal{N}}}_{\tilde{Q},\tilde{W}}\] is proved in \cite{schiffmann2023cohomological}, where $\mathcal{A}^{T,\tilde{\mathcal{N}}}_{\tilde{Q},\tilde{W}}$ is a nilpotent cohomological Hall algebra. Here, $T$ is an appropriately chosen torus. Strictly speaking, Maulik-Okounkov Yangians are only defined with respect to $ T$-equivariant cohomology, where $T$ is some torus that scales the natural symplectic form on the Nakajima quiver varieties associated to the quiver $Q$, non-trivially. 

We now consider the cyclic quiver $Q^{K}$ and choose $\T = \C^{*} \times \C^{*}$ to be a 2 torus. The spherical subalgebra of the localized $K$-theoretic version of the cohomological Hall algebra \cite{pădurariu2019ktheoretic} of the preprojective algebra, defined by considering $K$ theory instead of cohomology, was considered in \cite{neguţ2015quantum} for the case of cyclic quivers. It is shown to be isomorphic to the positive half of the quantum toroidal algebra $\ddot{\bU}_{q,t}(\mathfrak{gl}(K))$. Its rational counterpart, the affine Yangian, has been defined in \cite{GUAY2007436} to be a $\C[\hbar_1,\hbar_2]$ algebra $\ddot{\mathcal{Y}}_{\hbar_1,\hbar_2}(\mathfrak{gl}(K))$ for $K >2$ as a deformation of a central extension of $\bU(\mathfrak{sl}_{K}[u^{\pm},v])$. The case when $K=2$ was first defined by Kodera in \cite{koderayangian} and also considered by Bershtein-Tsymbaliuk in \cite{Bershtein_2019}. A specialization of this algebra at $\hbar_1=\hbar_2$ has been shown to be isomorphic to a spherical subalgebra of $\mathcal{A}^{\C^{*}}_{\QTC{K},\WTC{K}}$ where $\C^* \subset \T$ acts by scaling the arrows with the same weight in \cite{yang2018pbw}. However $\mathcal{A}^{\T}_{\tilde{Q},\tilde{W}}$ is not spherically generated and this makes computing $\mathcal{A}^{\T}_{\tilde{Q},\tilde{W}}$ essentially a new problem.  In Proposition \ref{Isomorphismloc} implies that $\mathcal{A}^{\T}_{\QTC{K},\WTC{K}} \otimes_{\HH_{\T}(\pt)} \mathrm{Frac}(\HH_{\T}(\pt))\simeq \ddot{\mathcal{Y}}_{\hbar_1,\hbar_2}(\mathfrak{gl}(K)) \otimes_{\HH_{\T}(\pt)} \mathrm{Frac}(\HH_{\T}(\pt))$. Motivated by this, we slightly enlarge the algebra by including the non-spherically generated part. Note that our only addition is to add elements $\mathfrak{K}^{(r)}_{\pm}$ to the existing definition of Affine Yangian. We use the conventions in \cite{Bershtein_2019}. We define 
\begin{defn}[CoHA Affine Yangian] \label{cohayangiandefn}
Assume $n \geq 2$. Let $\mathcal{Y}^{(n), \CoHA}_{\hbar_1,\hbar_2}$ be the $\mathbb{C}[\hbar_1,\hbar_2]$ algebra generated by $X^{\pm}_{i,r},H_{i,r}$ for any $i \in \mathbb{Z}/n\mathbb{Z}, r \in \Z_{\geq 0}$ and $\mathfrak{K}_{+}^{(r)}, \mathfrak{K}_{-}^{(r)}$ for $r \in \Z_{\geq 0}$, with the relations
\begin{subequations}
\begin{align}
    [X^{+}_{i,r},X^{-}_{j,s}] \tag{R0}\label{R0} &= \delta_{i,j} H_{i,r+s} \\  [H_{i,r},H_{j,s}] &= 0 \tag{R1}\label{R1}
\end{align} 
\end{subequations} for any $n \geq 2$. When $n >2$, 

\begin{subequations}
\begin{align}
     [X_{i,r+1}^{\pm},X^{\pm}_{j,s}] - [X^{\pm}_{i,r},X^{\pm}_{j,s+1}] &= -m_{ij}(\hbar_1+\hbar_2/2)[X^{\pm}_{i,r},X^{\pm}_{j,s}] \pm a_{ij} \hbar_2/2 \{X^{\pm}_{i,r},X^{\pm}_{j,s} \}  \tag{R2} \label{R2}\\ 
    [H_{i,r+1},X^{\pm}_{j,s}] - [H_{i,r},X^{\pm}_{j,s+1}] &= -m_{ij} (\hbar_1+\hbar_2/2)[H_{i,r},X^{\pm}_{j,s}] \pm a_{ij} \hbar_2/2 \{H_{i,r},X^{\pm}_{j,s} \} \tag{R3} \label{R3} \\
    [X^{\pm}_{i,r},X^{\pm}_{j,s}] &= 0 \ \textrm{ for all } \ |i-j|> 1  \tag{R4} \label{R4} \\
    \sum_{\sigma \in S_2} [X^{\pm}_{i,r_{\sigma(1)}},[X^{\pm}_{i,r_{\sigma(2)}},X^{\pm}_{i\pm1,s}]] &= 0 \ ; \  \ [H_{i,0},X^{\pm}_{j,s}] = \pm a_{ij} X^{\pm}_{j,s} \tag{R5} \label{R5}
\end{align} \end{subequations}\vspace{-7mm}\begin{align} \tag{Int1} \label{Int1}
(\hbar_1)(\hbar_1+\hbar_2)\mathfrak{K}^{(r)}_{\pm} = \mathrm{T}^r \left(\sum_{i \in \mathbb{Z}/n\mathbb{Z}} [X^{\pm}_{i,0},[X^{\pm}_{i+1,1}, [X^{\pm}_{i+2,0},[ \cdots, X^{\pm}_{i+n-1,0}]]]] \right) \nonumber \\ -\hbar_2 \mathrm{T}^{r}\left(\sum_{i \in \mathbb{Z}/n\mathbb{Z}} X^{\pm}_{i,0} \left([X^{\pm}_{i+1,0},[X^{\pm}_{i+2,0},[\cdots, X^{\pm}_{i+n-1,0}]]] \right))\right) \nonumber
\end{align}for any $i,j \in \mathbb{Z}/n \mathbb{Z}$ and $r,s \in \mathbb{Z}_{\geq 0}$ where $m_{ij}=-\delta_{i+1,j}+\delta_{i,j+1}$, $a_{ij}= 2 \delta_{i,j}-\delta_{i,j+1}-\delta_{i,j-1}$. Here $\mathrm{T}$ is an operator defined so that $\mathrm{T}(X^{\pm}_{i,r}) := \pm (X^{\pm}_{i,r+1}), \mathrm{T}(H_{i,r})=0$ and $\mathrm{T}$ is a derivation, i.e. $\mathrm{T}(ab) := \mathrm{T}(a)b+a \mathrm{T}(b)$ for any $a,b$. By the curly brackets, we mean $\{a,b \}:= ab+ba$ and finally, $S_2$ is the symmetric group in $2$ elements.  

When $n=2$, we have the following relations:
\begin{equation}\tag{R2.1}\label{R2.1}
  [X^\pm_{i,r+1},X^\pm_{i,s}]-[X^\pm_{i,r},X^\pm_{i,s+1}]=\pm h_2\{X^\pm_{i,r},X^\pm_{i,s}\},
\end{equation}
\begin{equation} \tag{R2.2} \label{R2.2}
\begin{split}
  & [X^\pm_{i,r+2},X^\pm_{j,s}]-2[X^\pm_{i,r+1},X^\pm_{j,s+1}]+[X^\pm_{i,r},X^\pm_{j,s+2}]=\\
  & \hbar_1(\hbar_1+\hbar_2)[X^\pm_{i,r},X^\pm_{j,s}]\mp \hbar_2(\{X^\pm_{i,r+1},X^\pm_{j,s}\}-\{X^\pm_{i,r},X^\pm_{j,s+1}\})\ \mathrm{for}\ j\ne i,
\end{split}
\end{equation}
\begin{equation}\tag{R3.1}\label{R3.1}
  [H_{i,r+1},X^\pm_{i,s}]-[H^\pm_{i,r},X^\pm_{i,s+1}]=\pm h_2\{H_{i,r},X^\pm_{i,s}\},
\end{equation}
\begin{equation}\tag{R3.2}\label{R3.2}
\begin{split}
  & [H_{i,r+2},X^\pm_{j,s}]-2[H_{i,r+1},X^\pm_{j,s+1}]+[H_{i,r},X^\pm_{j,s+2}]=\\
  & \hbar_1(\hbar_1+\hbar_2)[H_{i,r},X^\pm_{j,s}]\mp \hbar_2(\{H_{i,r+1},X^\pm_{j,s}\}-\{X_{i,r},X^\pm_{j,s+1}\})\ \mathrm{for}\ j\ne i,
\end{split}
\end{equation}
\begin{equation}\tag{R4.1}\label{R4.1}
  [H_{i,0},X^\pm_{j,s}]=\pm a_{i,j} X^\pm_{j,s},\ \
  [H_{i,1},X^\pm_{i+1,s}]=\mp (2H^\pm_{i+1,s+1}+\hbar_2\{H_{i,0},X^\pm_{i+1,s}\}),
\end{equation}
\begin{equation}\tag{5.1}\label{R5.1}
  \underset{r_1,r_2,r_3}\Sym [X^\pm_{i,r_1},[X^\pm_{i,r_2},[X^\pm_{i,r_3},X^\pm_{i+1,s}]]]=0.
\end{equation}
\begin{align} \tag{Int2} \label{Int2}
(\hbar_1)(\hbar_1+\hbar_2)\mathfrak{K}^{(r)}_{\pm} = \mathrm{T}^r \left([X^{\pm}_{0,0},X^{\pm}_{1,1}]+ [X^{\pm}_{1,0},X^{+}_{0,1}] \right) \nonumber \\ -\hbar_2 \mathrm{T}^{r}\left(X^{\pm}_{0,0}X^{+}_{1,0}+ X^{+}_{1,0}X^{+}_{0,0}\right)  \nonumber
\end{align}
\end{defn}

Clearly $\mathcal{Y}^{(n), \CoHA}_{\hbar_1,\hbar_2} \otimes_{\C[\hbar_1,\hbar_2]} \C[\hbar_{1}^{\pm 1},(\hbar_1+\hbar_{2})^{\pm 1}] \simeq \ddot{\mathcal{Y}}_{\hbar_1,\hbar_2}(\mathfrak{gl}(n)) \otimes_{\C[\hbar_1,\hbar_2]} \C[\hbar_{1}^{\pm 1},(\hbar_1+\hbar_{2})^{\pm 1}]$, since we have only added elements $\mathfrak{K}^{(r)}_{\pm}$ to the originial definition of $\ddot{\mathcal{Y}}_{\hbar_1,\hbar_2}(\mathfrak{gl}(K))$ which by relations (\ref{Int1}) and (\ref{Int2}), after multiplication by $(\hbar_1)(\hbar_1+\hbar_2)$ belong to $ \ddot{\mathcal{Y}}_{\hbar_1,\hbar_2}(\mathfrak{gl}(n))$. We then have

\begin{mainthm}[Theorem \ref{equivcoha}]
For $K \geq 1$, let $\T = \C^{*} \times \C^{*}$ be the 2-torus which acts on representations of the tripled quiver $\QTC{K}$ by scaling all the original arrows $a$ with weight $(1,0)$, all opposite arrows $a^*$ with weight $(0,1)$ and added loops $\omega_i, i \in [0,K]$ with weight $(-1,-1)$ (Section \ref{bpsLiealgebracyclic}). Then we have an isomorphism of $\C[\hbar_1,\hbar_2]$ algebras
\[\mathcal{A}^{\T,\chi}_{\QTC{K},\WTC{K}} \simeq \mathcal{Y}^{(K+1),+ , \CoHA}_{-\hbar_2,\hbar_1+\hbar_2}\] where $\mathcal{Y}^{(K+1),+ , \CoHA}_{-\hbar_2,\hbar_1+\hbar_2} \subset \mathcal{Y}^{(K+1), \CoHA}_{-\hbar_2,\hbar_1+\hbar_2}$ is the subalgebra generated by $X^{+}_{i,r}$ and $\mathfrak{K}^{(r)}_{+}$ for $i \in [0,K], r \geq 0$ and $\chi$ is a sign twist, defined in Section \ref{signtwist}. 
\end{mainthm}
Let $\overline{\mathcal{Y}}^{n, \CoHA}_{\hbar_1,\hbar_2}$ be the subalgebra of $ \mathcal{Y}^{n,\CoHA}_{\hbar_1,\hbar_2}$ which is generated by generators $\langle \mathfrak{K}^{(r)}_{+}, X^{\pm}_{i,r}, H_{i,r} \mid i \in \mathbb{Z}/n\mathbb{Z}, r \geq 0 \rangle$ (So we only exclude the generators $\mathfrak{K}^{(r)}_{-}$ for $r \geq 0$). Let $\overline{\mathcal{Y}}^{n, \CoHA,\mathfrak{e}}_{\hbar_1,\hbar_2} = \overline{\mathcal{Y}}^{n, \CoHA}_{\hbar_1,\hbar_2}[\langle k_{i,r} \rangle]$ be the extended algebra, where we have added central elements $\langle k_{i,r} \mid i \in \mathbb{Z}/n\mathbb{Z}, r \geq 0 \rangle$ to the list of generators. Then we have
\begin{maincoro}[Theorem \ref{moyangianconjecture}] \label{corollarymo}
For $K \geq 1$, the Maulik-Okounkov Yangian $\mathbf{Y}^{\mathrm{MO},\mathbb{T}}_{Q^{K}}$ is isomorphic to an extended integral form of Guay's affine Yangian. i.e. There is an isomorphism of $\C[\hbar_1,\hbar_2]$ linear algebras  \[ \mathbf{Y}^{\mathrm{MO},\mathbb{T}}_{Q^K} \simeq \overline{\mathcal{Y}}^{(K+1), \CoHA,\mathfrak{e}}_{-\hbar_2,\hbar_1+\hbar_2}. \]
\end{maincoro}

%We expect (See Section \ref{moyangianconjecture}), that the subalgebra of $\mathcal{Y}^{(K+1),\CoHA}_{\hbar_1,\hbar_2}$ generated by $X^{\pm}_{i,r},H_{i,r}$ and $\mathfrak{K}^{(r)}_{+}, i \in [0,K], r \geq 0$(excluding $\mathfrak{K}^{(r)}_{-}, r \geq 0$) is isomorphic to a central quotient of Maulik-Okounkov Yangian $\textbf{Y}^{\textrm{MO}}_{Q^K}$.   
%The above theorem is proven by studying the non-deformed algebras, which boils down to studying affinized BPS Lie algebra $\abps{}$, as defined in \cite{davison2022affine}(Definition \ref{factorizationcoproduct}), by giving $\mathcal{A}_{\tilde{Q},\tilde{W}}$ a locally constant factorization algebra structure over $\mathbb{A}^{1}$.  \textcolor{red}{change}

\subsubsection*{Proof Strategy for Theorems A, B, C} 
We remark that although by taking the classical limit, Theorem C $\implies$ Theorem B $ \implies$ Theorem A, to prove them, we first prove Theorem A. There exists a cocommutative coproduct on $\mathcal{A}^{\chi}_{\QTC{K},\WTC{K}}$ which is compatible with the cohomological Hall algebra structure on $\mathcal{A}^{\chi}_{\QTC{K},\WTC{K}}$ (Section \ref{affinizedBPSliealgebra}). By the Milnor-Moore theorem, $\mathcal{A}^{\chi}_{\QTC{K},\WTC{K}} \simeq \bU(\mathcal{P})$ where $\mathcal{P}$ is the Lie algebra of the primitive elements with respect to this coproduct. We refer to this Lie algebra as the affinized BPS Lie algebra $\abps{K}$. There exists a filtration, referred to as the perverse filtration $\mathfrak{P}$ on $\mathcal{A}^{\chi}_{\QTC{K},\WTC{K}}$, given by a decomposition type theorem. The associated graded algebra with respect to this filtration is isomorphic to $\Sym(\mathfrak{g}^{\BPS}_{\QTC{K},\WTC{K}}[u])$, where $\mathfrak{g}^{\BPS}_{\QTC{K},\WTC{K}}$ is the non-affinized BPS Lie algebra (Section \ref{perversefiltrationandbps}). This implies that, as a vector space $\abps{K} \simeq \mathfrak{g}^{\BPS}_{\QTC{K},\WTC{K}}[u]$, but not as a Lie algebra (Section \ref{bpsLiealgebracyclic}). With the help of an action of the three-dimensional Heisenberg Lie algebra (Section \ref{HeisLiealgebra}), we show that a much smaller subspace of $\abps{K}$, and the relations among them, determine the affinized BPS Lie algebra. Finally, to compute the relations, we construct a geometric action of $\mathcal{A}^{\chi}_{\QTC{K},\WTC{K}}$ on the cohomology of the equivariant Hilbert scheme $\Hilb^{\Z_{K+1}}(\C^2)$ and the cohomology of Hilbert scheme of points on the minimal resolution $\Hilb(S_K)$ (Chapter \ref{chapter: quivervarieties}). By identifying the induced action of $\mathfrak{g}^{\BPS}_{\QTC{K},\WTC{K}}$ with the action of the affine Lie Algebra and the infinite-dimensional Heisenberg Lie algebra due to Nakajima, we prove Theorem A in Section \ref{sectionaffinized}. 

To calculate $\mathcal{A}^{\T,\chi}_{\QTC{K},\WTC{K}}$, we use the embedding into the cohomological Hall algebra $\mathcal{A}^{\T,\chi}_{\QTC{K}}$ of the same quiver but without potential. The algebra structure on $\mathcal{A}^{\T,\chi}_{\QTC{K}}$ can be explicitly described in terms of a shuffle product (Section \ref{injectiontoshuffle}). The calculation of the undeformed version gives a minimal set of generators of $\mathcal{A}^{\T,\chi}_{\QTC{K},\WTC{K}}$ (Proposition \ref{imageofcoha}). We find the image of these generators in $\mathcal{A}^{\T,\chi}_{\QTC{K}}$ (Proposition \ref{imageimaginary}). This allows us to construct a homomorphism of algebras from the positive half of Guay's affine Yangian to the algebra $\mathcal{A}^{\T,\chi}_{\QTC{K},\WTC{K}}$ (Section \ref{mapyangian}). Finally, with the help of the PBW theorem for $\mathcal{A}^{\T,\chi}_{\QTC{K},\WTC{K}}$, we prove that this morphism is an the isomorphism, proving Theorem C (Proposition \ref{Isomorphismloc} and Theorem \ref{equivcoha}). Finally, we take the classical limit $\hbar_1+ \hbar_2 \rightarrow 0$ to calculate the $\C^*$ deformed algebra $\mathcal{A}^{\C^{*},\chi}_{\QTC{K},\WTC{K}}$, proving Theorem B.
\subsection{Spherical subalgebra }

The spherical subalgebra, i.e, the subalgebra generated by the smallest possible dimension vectors of the deformed version of the cohomological Hall algebra, has been studied extensively in \cite{yang2017cohomological}. An important tool used in \cite{yang2017cohomological} is embedding into the cohomological Hall algebra of quiver without potential. However, no such injection exists in the non-deformed case (Section \ref{notinjection}). As an application of general tools from \cite{davison2022affine}, developed in Section \ref{HeisLiealgebra}, we can precisely calculate the spherical part for any quivers without loops. 

\begin{mainthm}[Theorem \ref{sphericalsubalgebra}] \label{thm:mainA}
Let $Q$ be any quiver without loops. Then there is an isomorphism of algebras \[\bU(\mathfrak{n}^{+}_{Q}[D]) \simeq \mathcal{SA}^{\chi}_{\tilde{Q},\tilde{W}}\] where $\mathfrak{n}_{Q}^{+}$ is the positive half of the Kac-Moody Lie algebra (Definition \ref{negativehalfLiealgebra}) associated with $Q$, $\mathfrak{n}_{Q}^{+}[D]$ is the current Lie algebra (Definition \ref{currentliealgebra}), $\mathcal{S}\mathcal{A}^{\chi}_{\tilde{Q},\tilde{W}} \subset \mathcal{A}^{\chi}_{\tilde{Q},\tilde{W}}$ is the subalgebra generated by $\mathcal{A}^{\chi}_{\tilde{Q},\tilde{W}, \delta_i}$ for $i \in Q_0$ and $\chi$ is a sign twist, defined in Section \ref{signtwist}.  
\end{mainthm}

We remark that for quivers without loops, $\mathcal{SA}_{\tilde{Q},\tilde{W}}^{\chi} = \mathcal{A}_{\tilde{Q},\tilde{W}}^{\chi}$ if and only if the quiver is an orientation of a finite type ADE Dynkin quiver \cite{davison2022affine}[Proposition 5.7].

\subsection{Strong rationality conjecture}
Given $f\colon  X \rightarrow Y$ a three-fold flopping contraction with exactly one exceptional fibre, which is isomorphic to a rational curve $C$, one can associate numerous invariants. One is the Gopakumar-Vafa invariant $n_{C,r}$. In \cite{davison2023refined}, the author gives a cohomological categorification $\BPS_{r,d}$ of these numbers in terms of BPS cohomology of the contraction algebra associated with the flopping curve $C$, where $d$ is the Euler characteristic of the coherent sheaves on $X$. The invariance of BPS invariants under the choice of Euler characteristic for $3CY$ varieties has been conjectured in the work of Pandharipande and Thomas \cite{Pandharipande_2009}. By analogy, the conjecture \cite{davison2023refined}[Conjecture 5.9] is the statement that this invariance lifts to an isomorphism of BPS cohomology 
\[ \BPS_{r,d} \simeq \BPS_{r,d+1}.\]

Here, the isomorphism is constructed by defining a Hecke correspondence from the cohomological Hall algebra. In particular, it implies the strong rationality conjectures in  \cite{toda2022gopakumarvafa}\cite{Pandharipande_2009} for flopping curves. We can consider a similar Hecke correspondence for the case when $X = S_K \times \C$ where $S_K \rightarrow \C^2/\Z_{K+1}$ is the minimal resolution of a Kleinian singularity. Then the proof of Theorem C implies that the Hecke correspondence indeed gives an isomorphism between BPS cohomology, as we vary $d$.  

\begin{maincoro}[Corollary \ref{strongrationalityconjectureproof}]
    The cohomological lift of the strong rationality conjecture is true for $S_K \times \C$. 
\end{maincoro}

\subsection{Cohomological Hall algebra of Kleinian  surfaces}
Given any smooth quasi-projective surface $S$, in \cite{kapranov2022cohomological}, the authors defined a cohomological Hall algebra structure on the Borel-Moore homology of the stack of compactly supported sheaves on $S$. Furthermore, this structure was upgraded to the categorical level in \cite{porta2022twodimensional}. Given any finite subgroup $G \subset \SL_2(\C)$, let $S_G \rightarrow \C^2/G$ be the minimal resolution. Then in \cite{diaconescu2020mckay}, the authors study the cohomological Hall algebra $\mathcal{A}^{\omega}_{\mu}(S_G)$ of $\omega$-semistable properly supported sheaves on $S_G$ with fixed slope $\mu$. They show that there is an isomorphism of algebras 
\[ \mathcal{A}^{\omega}_{\mu}(S_G) \simeq \mathcal{A}^{\zeta}_{0}(\mathrm{\Pi}_{Q_G})  \] where $Q_G$ is the affine ADE quiver associated to the group $G$ by the McKay correspondence. Here $\mathcal{A}^{\zeta}_{0}(\mathrm{\Pi}_Q)$ is the cohomological Hall algebra of slope $0$, $\zeta$-semistable representations of the preprojective algebra $Q$, where the stability condition $\zeta$ is explictly determined in terms of $\omega$ and $\mu$. Now assume $G= \mathbb{Z}_{K+1}$. The main result and the properties of the affinized BPS Lie algebra allow us to explicitly determine $\mathcal{A}^{\zeta'}_{\mu'}(\mathrm{\Pi}_{Q^{K}})$  as a subalgebra of $\mathcal{A}_{\mathrm{\Pi}_{Q^{K}}}$,  for any choice of $\zeta'$ and $\mu'$. In particular, as a corollary, we recover the calculation of the cohomological Hall algebra of $0$-dimensional sheaves $\mathcal{A}(S_K)$ as in \cite{mellit2023coherent} and of $\mathcal{A}^{\omega}_{\mu}(S_K)$ as in \cite{diaconescu2020mckay}. 
\begin{maincoro}

\begin{enumerate}
\item Let $W(S_K)$ be the Lie algebra defined in the Corollary (\ref{W_sLiealgebra}). We then have an isomorphism of algebras (Corollary (\ref{W_sLiealgebra}))
\[\mathcal{A}(S_K) \simeq \bU(W(S_K)).\]
\item Let $W^{\omega}_{\mu}(S_K)$ be the Lie algebra as defined in the Definition \ref{Womegamuliealgebra}. We then have an isomorphism of algebras (Corollary \ref{womegamuLiealgebra})
\[ \mathcal{A}^{\omega}_{\mu}(S_K) \simeq \bU(W^{\omega}_{\mu}(S_K)).\]
\end{enumerate}
\end{maincoro}

\subsection{Shuffle algebra realization and commutative subalgebras}
It is an interesting problem to construct the realization of quantum groups in terms of a shuffle algebra, which is an algebra structure on the symmetric polynomial ring with a shuffle product. This idea dates back to the work of Enriques in \cite{enriquez2000quantization}. In \cite{neguţ2022shuffle}, the localized K-theoretic Hall algebras are realized in terms of polynomials in the shuffle algebra satisfying Feigin-Odesskii type wheel conditions\footnote{In some references, the shuffle algebra is defined to be the subalgebra of symmetric Laurent polynomials satisfying the wheel conditions. See \cite{neguţ2022shuffle}[Definition 2.9].}. However, determining the explicit integral form is far from being known. In \cite{neguţ2022integral}, the author described the seminilpotent K-theoretic Hall algebra of the preprojective algebra of the Jordan quiver as a shuffle algebra with polynomials satisfying certain divisibility conditions, giving a shuffle realization of
an integral form of $\ddot{\bU}_{q,t}(\mathfrak{gl}(1))$. In \cite{zhao2019feiginodesskii}, the author gives some necessary but not sufficient divisibility conditions in the shuffle algebra associated to the surface to be in the image of the K-theoretic Hall algebra of any smooth projective surface. In \cite{Tsymbaliuk_2017}, the author describes an integral form of the Yangian of $\mathfrak{sl}_{n+1}$ by divisibility conditions. 

In Section \ref{imageimaginary}, we describe the image of $\mathcal{A}^{\T,\chi}_{\tilde{Q^{K}},\tilde{W^K}}$ in the shuffle algebra $\mathcal{A}^{\T,\chi}_{\tilde{Q^{K}}}$ as a subalgebra generated by certain specific polynomials. We remark that the additive shuffle algebra (without wheel conditions) considered in \cite{Bershtein_2019}[Section 5.2] is isomorphic to the algebra $\mathcal{A}^{\T, \chi}_{\tilde{Q^{K}}}$ after localization (Remark \ref{ontwoshufflealgebras}). 

We now state an interesting corollary of this embedding. By computing the image of a natural trivial Lie subalgebra of $\mathfrak{g}^{\BPS,\T}_{\QTC{K},\WTC{K}}$, we define elements $L_n, n \geq 1$ in the shuffle algebra (Corollary \ref{elementsLn}) such that 
\begin{maincoro}[Corollary \ref{elementsLn}]
For all $n \geq 1$, let $A^{\T}_K \subset \mathcal{A}^{\T,\chi}_{\QTC{K}}$ be the subalgebra generated by $L_n$. Then $A^{\T}_K \simeq \C[L_1,L_2,\cdots]$ is a polynomial subalgebra. 
\end{maincoro}
 For the quantum toroidal algebra $\ddot{\bU}_{q,t}(\mathfrak{gl}(1))$, a commutative polynomial subalgebra is defined in \cite{Feigin_2011}[Theorem 2.4] and its action on the $K$ theory of $\Hilb^n(\C^2)$ has been studied in relation to the theory of Macdonald polynomials. Its rational counterpart, has been defined in \cite{TTsymbaliuk_2017}[Theorem 7.7], and it can be seen as the image of the commutative Lie algebra $\mathfrak{g}^{\BPS,\T}_{\tilde{Q_{\Jor}},\tilde{W}_{\Jor}}$ in the shuffle algebra from \cite{davison2022affine}[Section 3.13]. %We thus expect that the operators $L_n$ to be additive degeneration of operators $F_n$ defined in \cite{Feigin_2015}[Theorem 2.11] in the setting of quantum toroidal algebra $\ddot{\bU}_{q,t}(\mathfrak{gl}(K))$ for $\mathfrak{gl}(K), K \geq 2$. 
 
\subsection{Algebras in physics literature}
In \cite{gaiotto2023deformed}, for every $K \neq 2$, the authors construct certain double deformed current algebras $A^K$ as the algebra of gauge-invariant local observables on M2 branes, in the twisted M-theory background. They further construct a $\C[\hbar_1,\hbar_2]$ linear algebra $Y^{K}$ by gluing two copies of $A^{K}$. They show that $Y^K$ is also an integral form of the affine Yangian $\ddot{\mathcal{Y}}_{\hbar_1,\hbar_2}(\mathfrak{gl}(K))$ and after quotienting by a central element, this algebra deforms $\mathcal{W}_K$. We expect that (see Section \ref{speculations}) 
%By proposition \ref{loopcohayangian}, this is also satisfied by the quotient of $\mathcal{Y}^{(n),\CoHA}_{\hbar_1,\hbar_2}$. We thus expect that in the same way as $\mathcal{Y}^{(n), \CoHA}_{\hbar_1,\hbar_2}$ and for $K=1$, infact $Y^{1} \simeq \ddot{Y}_{\hbar_1,\hbar_2}(\mathfrak{gl}(1))$. Furthermore 
\begin{Conjecture}
For any $K > 2$, there is an isomorphism of algebras
\[ Y^{K} \simeq \mathcal{Y}^{(K),\CoHA}_{\hbar_1,\hbar_2}. \]

\end{Conjecture}
This would also explain the expectation in \cite{costello2017holography}[Section 1.8] concerning the embedding of the cohomological Hall algebra in Costello's deformed double current algebra $\O_{\hbar_1}(\mathfrak{M}_{\hbar_2}(\bullet, K))$, since it is shown in \cite{gaiotto2023deformed}[Theorem 2] that Costello's algebra is isomorphic to a particular integral form of the algebra $A^{K}$ which contains the positive half of $Y^{K}$ \cite{gaiotto2023deformed}[Definition 7.01].

%where second isomorphism should hold after quotienting $Y^{\text{MO}}_{Q^{K-1}}$ by some central elements. 
\subsection{Acknowledgements}
I want to thank my parents and family for their support and my advisor Ben Davison for developing and teaching me the technology to study cohomological Hall algebras. I thank Andrei Negu\cb{t}, Oleksandr Tsymbaliuk for answering my emails and Francesco Sala, Yehou Zhou, Eric Vasserot and Olivier Schiffmann for explaining their work. I would especially like to thank Misha Bershtein for answering my questions about Quantum Groups. I would also like to thank Olivier Schiffmann for introducing me to this area of research. I thank Lucas Buzaglo, Lucien Hennecart, Sebastian Schlegel Mejia, Tommasso Maria Botta, Woonam Lim, Sarunas Kaubrys, and Sonja Klisch for helping me out with annoying precision. I thank Dinakar Muthiah and Alexander Shapiro for helpful comments on my PhD thesis on which this paper is based on. 
\subsection{Notations} \label{notations}
\begin{itemize}
    \item All schemes/stacks are considered over $\C$. 
%    \item For any smooth stack $\mathfrak{X}$, we set $\Q_{\mathfrak{X}}^{\vir}:= \Q_{\mathfrak{X}}[\dim(\mathfrak{X})] \in \Per(\mathfrak{X})$ and if there is action of an extra(depending on the context) torus $T$ on $\mathfrak{X}$ then we set $\Q^{\vir}_{\mathfrak{X}/T} := \Q_{\mathfrak{X}/T}[\dim(\mathfrak{X})]$. The shifted perverse $t$ structure on $\Db_{\text{const}}(\mathfrak{X}/T)$ is defined by setting ${}^{\mathfrak{p'}}\tau^{\leq i} : = {}^{\mathfrak{p'}}\tau^{\leq i- \dim(T)}$ and ${}^{\mathfrak{p'}}\tau^{\geq i} := {}^{\mathfrak{p'}}\tau^{\geq i-\dim(T)}  $. Let $\Per^{'}(\mathfrak{X}/T)$ to be heart with respect to this shifted perverse structure. This is defined so that for any $\mathcal{F} \in \Per^{'}(\mathfrak{X}/T)$, $(\mathfrak{X} \rightarrow \mathfrak{X}/T)^{*}(\mathcal{F}) \in \Per(\mathfrak{X})$.
    \item For any algebraic group $G$, $\text{BG} := \pt/G$ is the stack theoretic quotient, while $\HH_{G}:= \HH(\text{BG},\Q)$.

\end{itemize}

\section{Quivers and Moduli Spaces}

In this chapter, we will introduce our main object of study. We first set up some notations.

A quiver is a directed graph $Q = (Q_0,Q_1,s,t)$, where $Q_0,Q_1$ are finite sets and $s,t: Q_1 \rightarrow Q_0$. The set $Q_0$ is the set of vertices, and $Q_1$ is the set of edges. For every edge $\alpha \in Q_1$, $s(\alpha)$ is the source of the edge, and $t(\alpha)$ is the target of the edge. 

\begin{defn}[Path algebra]
    Let $\C Q$ be the path algebra of $Q$, i.e., the algebra over $\C$ having the paths in $Q$ as a basis, with multiplication given by concatenation of paths. For each vertex $i \in Q_0$, there is a lazy path of length $0$ starting and ending at $i$, and we denote by $e_i$ the resulting element of $\C Q$.
\end{defn}

Given a quiver $Q$, the doubled quiver $\overline{Q}$ is the quiver obtained by adding a new arrow $a^{*}$ for each arrow $a \in Q_1$, with $s(a^{*})=t(a)$ and $t(a^{*})=s(a)$. We then define 

%setting $\overline{Q}_0 = Q_0$ and $\overline{Q}_1 = Q_1 \cup Q_1^{*}$, where $Q_1^{*} = \{ a^{*} |a \in Q_1 \}$

\begin{defn}[Preprojective algebra]
The preprojective algebra is defined as the quotient \[ \mathrm{\mathrm{\Pi}}_Q := \C \overline{Q}/ \langle \sum [a,a^{*}] \rangle. \]
    
\end{defn}
From now on, by a representation of a quiver, we mean a left $\ CQ$ module. 
 
\begin{defn}[Dimension]
    A left $\C Q$ module $\rho$ is said to be of dimension $\bd \in \N^{Q_0}$ defined by \[ \bd_i := \dim(e_i \cdot \rho).\]
\end{defn}

\begin{convention}\label{indices}
For quiver $Q$, and for any $0 \leq i \leq |Q_0|-1$ we set $\delta_i= (\underbrace{0,\cdots,0}_{ i \ 0 s},1,0,\cdots,0) \in \N^{Q_0}$ and for any $m \leq |Q_0|-1$, we set \[[i,m) = \sum_{j=i}^{j=i+m \mod(|Q_0|)} \delta_j \in \N^{Q_0}\] is a cyclic string of $0$s and $1$s with $m+1$ consecutive $1$s starting from position $i$ and remaining terms $0$. Note that $[i,0)= \delta_i$ and $[i+1,|Q_0|-2) = \delta-\delta_i$ where $\delta=(\underbrace{1,\cdots,1}_{|Q_0| \textrm{ times }})$.   
\end{convention}

Given two representations $\rho,\rho^{\prime}$ of dimension $\bd, \be$ respectively, we consider the Euler form \[ (\rho, \rho^{\prime}) := \dim(\Hom(\rho, \rho^{\prime})) - \dim(\Ext^{1}(\rho,\rho^{\prime})),\] which in this case, only depends on the dimensions $\bd,\be$. 

\begin{defn}[Euler form] \label{eulerform}
    Given a pair of dimension vectors $\bd,\be \in \N^{Q_0}$, we define the Euler form \[ \chi_{Q}(\bd,\be) :=  \sum_{i \in Q_0} \bd_i \be_i - \sum_{\alpha \in Q_1} \bd_{s(\alpha)} \be_{t(\alpha)}.\]
\end{defn}

\begin{defn}[Serre Subcategory]
A Serre subcategory $\mathcal{S} \subset \mathbb{C} Q-\mathrm{mod}$ such that for every short exact sequence of representations \[ 0 \rightarrow \rho_1 \rightarrow \rho_2 \rightarrow \rho_3 \rightarrow 0,\] $\rho_2$ is in $\mathcal{S}$ if and only if $\rho_1,\rho_3$ are in $\mathcal{S}$. 
\end{defn}

\begin{defn}[Potential]
A potential\footnote{Also called superpotential in physics terminology} on a quiver $Q$ is an element $W \in \C Q/[\C Q,\C Q]$. A potential is given by a linear combination of cyclic words in $Q$, where two cyclic words are considered to be the same if one can be cyclically permuted to be the other. If $W$ is a single cyclic word and $a \in Q_1$, then we define \[ \partial W/\partial a = \sum _{W = cac'} c'c\]
    And we extend this definition linearly to general $W$. 
\end{defn}

\begin{defn}[Jacobi algebra] \label{jacobialgebra}
    The Jacobi algebra associated to $(Q,W)$ is defined to be \[ \Jac(Q,W) := \C Q/\langle \partial W/\partial a  | a \in Q_1 \rangle. \]
   
\end{defn}

\begin{example}[Canonical Cubic Potential] \label{canonicalcubicpotential}
Given a quiver $Q$, we consider the quiver $\Tilde{Q}$, given by adding loops $\omega_i$ to each vertex $i \in  \overline{Q}_0$. Then the tripled quiver $\Tilde{Q}$ carries the canonical cubic potential \[ \tilde{W} = \left(\sum_{i \in Q_0} \omega_i \right) \left( \sum_{a \in Q_1} [a,a^{*}] \right). \] We have an isomorphism \[ \mathrm{\mathrm{\Pi}}_Q[\omega] \simeq \Jac(\tilde{Q},\tilde{W})\] for the polynomial ring $\mathrm{\mathrm{\Pi}}_Q[\omega]$ where $\omega \mapsto \sum \omega_i$. For $Q$ not of finite type, the representation category of $\Jac(\tilde{Q},\tilde{W})$ is the 3CY completion of the 2CY category of representations of $\mathrm{\mathrm{\Pi}}_Q$ \cite{keller2021introduction}[Section 5].
\end{example}

\begin{defn}[Nilpotent Representation] \label{nilpotetnserre}
We define the notion of nilpotency as considered in \cite{bozecschiffmann}, which is based on the work of Lusztig  \cite{lusztig}. Let $\mathcal{N}_{\overline{Q}} \subset \C \overline{Q}$-mod be the subcategory of semi-nilpotent representations of the quiver $\overline{Q}$, i.e representations $\rho$ for which there exist some $l$ such that $\mathbb{C}\overline{Q}_{\geq l} \cdot \rho = 0$. Furthermore, 
    
    \begin{itemize}
        \item We say that a representation of $\mathrm{\Pi}_Q$ is nilpotent if underlining $\mathbb{C}\overline{Q}$ module is nilpotent. We shall denote that subcategory of representations by $\mathcal{N}$.
        \item Let $\tilde{Q}$ be the tripled quiver introduced in Example \ref{canonicalcubicpotential}. We define $\tilde{\mathcal{N}} \subset \C \tilde{Q}$-mod to be the subcategory for which underlining $\overline{Q}$ module is in $\mathcal{N}_{\overline{Q}}$. 
    \end{itemize} 
\end{defn}

We will also consider representations under stability conditions.

    \begin{defn}[Semistability] \label{semistability}
    Given a quiver $Q$, a stability condition is a tuple $\zeta \in \Q^{Q_0}$. For a non-zero representation $\rho$ of $Q$, we define the slope \[\mu_{\zeta}(\rho) := \frac{\chi_{\zeta}(\dim(\rho)):= \zeta \cdot \dim(\rho)}{|\dim(\rho)|}.\] 
    Note that it only depends on $\dim(\rho)$. A representation $\rho$ is semistable if for all non-zero, proper $\rho^{\prime} \subset \rho$, we have $\mu_{\zeta}(\rho^{\prime}) \leq \mu_{\zeta}(\rho)$ and stable if the inequalities are strict. We say that $\zeta$ is generic for dimension $\bd$ if for all $\bd^{\prime} < \bd$, we have $\mu_{\zeta}(\bd^{\prime}) \neq \mu_{\zeta}(\bd)$.
    \end{defn}
    For $\mu \in (-\infty,\infty)$ a slope, we denote by \[\mathrm{\Lambda}^{\zeta}_{\mu} \subset \N^{Q_0} \] the submonoid of dimension vectors $\bd$ such that $\bd=0$ or $\mu_{\zeta}(\rho)=\mu$. 
     \subsection{Torus Actions} Given any quiver $Q$ and a lattice $N = \mathbb{Z}^{r}$, let the map $\bw: Q_1 \rightarrow N$ be a weighting function on the edges. Fix $T = \Hom_{\textrm{Group}}(N, \C^{*})$. For every $a \in Q_1$, we can define a homomorphism $T \rightarrow \C^{*}$ given by $t \rightarrow t(\textbf{w}(a))$. This gives a morphism of stacks $\lambda_{a} \colon \textrm{BT} \rightarrow \BC$. Taking pullback, defines a cohomology class $\bt(a) \in \HH^{2}(\textrm{BT},\Q)$ given by $\bt(a) = \lambda_a^{*}(u)$ where $u$ is the generator of $\HH^{2}(\BC,Q) \simeq \Q[u]$. Given a path $p = a_1 a_2 \cdots a_n $ on the quiver $Q$, we define its weight to be $\textbf{w} = \sum_{i=1}^{i=n} \textbf{w}(a_i)$. A linear combination of paths is said to be homogeneous if the weights of each path in the linear combination are the same. 
\section{Moduli Spaces of quiver representations}
Given a quiver $Q$ and dimension $\bd$, define 
\begin{align*}
\Rep_{\bd}(Q) &:= \prod_{a \in Q_1} \Hom(\C^{\bd(s(a))},\C^{\bd(t(a))}) \\ 
\GL_{\bd} &:= \prod_{i \in Q_0} \GL_{\bd(i)} 
\end{align*} 
where the group $\GL_{\bd}$ acts on $\Rep_{\bd}(Q)$ by conjugation. Let 
\[\mathfrak{M}_{\bd}(Q) = \Rep_{\bd}(Q)/\GL_{\bd}(Q) \] be the moduli stack of $\bd$ dimensional representations of $Q$. Let $\mathcal{M}_{\bd}(Q) = \Spec(\C[\Rep_{\bd}(Q)]^{\GL_{\bd}})$ be the coarse moduli space parametrizing semisimple points. We then have the affinization map given by taking the direct sum of successive quotients in the Jordan-Holder filtration of the representation
\[ \JH_{\bd} \colon \mathfrak{M}_{\bd}(Q) \rightarrow \mathcal{M}_{\bd}(Q).\]

There is an action of $T$ on $\Rep_{\bd}(Q)$ by \[ (t, (\rho_{a})_{a \in Q_1}) \rightarrow (t(\textbf{w}(a)) \cdot \rho_a)_{a \in Q_1}\] and then we define \[ \mathfrak{M}_{\bd}^{T}(Q) := \Rep_{\bd}(Q)/(\GL_{\bd}(Q) \times T). \] Similarly, we define the stack \[ \mathcal{M}^{T}_{\bd}(Q) := \mathcal{M}_{\bd}(Q)/T\] where $T$ acts on $\mathcal{M}_{\bd}(Q)$ since the action of $\Rep_{\bd}(Q)$ commutes with the action of the gauge group $\GL_{\bd}$. We have an induced morphism
\[ \JH^{T}_{\bd} \colon \mathfrak{M}^{T}_{\bd}(Q) \rightarrow \mathcal{M}^{T}_{\bd}(Q). \]

We can also consider the versions of these objects with stability conditions. Given a King's stability condition $\zeta$, we define $\Rep_{\bd}^{\zeta-\sss}(Q) \subset \Rep_{\bd}(Q)$ to be the open subvariety of $\zeta$ semistable representations. We refer \cite{reineke2008modulirepresentationsquivers} for an excellent introduction to quiver moduli spaces.

We define the quotient stack $\mathfrak{M}^{T,\zeta-\sss}_{\bd}(Q) : = \Rep_{\bd}^{\zeta-\sss}(Q)/(\GL_{\bd}(Q) \times T)$. When there is no $T$, we have the coarse moduli space  $\mathcal{M}^{\zeta-\sss}_{\bd}(Q)$ given by GIT Quotient. Since the points of this variety are in bijection with polystable representations, it is easy to see that $T$ acts on $\mathcal{M}^{\zeta-\sss}_{\bd}(Q)$. We define $\mathcal{M}^{T,\zeta-\sss}_{\bd}(Q)$ to be the quotient stack $\mathcal{M}^{\zeta-\sss}_{\bd}(Q)/T$. We then have the induced morphism \[ \JH^{T,\zeta}_{\bd} \colon \mathfrak{M}^{T,\zeta-\sss}_{\bd}(Q) \rightarrow \mathcal{M}^{T,\zeta-\sss}_{\bd}(Q). \] We denote by $\JH^{T,\zeta}_{\mu} = \oplus_{\bd \in \mathrm{\Lambda}^{\zeta}_{\mu}} \JH^{T,\zeta}_{\bd}: \mathfrak{M}^{T,\zeta-\sss}_{\mu}(Q) \rightarrow \mathcal{M}^{T,\zeta-\sss}_{\mu}(Q)$, where $\mathfrak{M}^{T,\zeta-\sss}_{\mu} := \oplus_{\bd \in \mathrm{\Lambda}^{\zeta}_{\mu}} \mathfrak{M}^{T,\zeta-\sss}_{\bd}(Q)$ and $\mathcal{M}^{T,\zeta-\sss}_{\mu} := \oplus_{\bd \in \mathrm{\Lambda}^{\zeta}_{\mu}} \mathcal{M}^{T,\zeta-\sss}_{\bd}(Q)$. Finally for two dimension vectors $\bd^{\prime}$ and $\bd^{\prime \prime}$, we define $\Rep^{T,\zeta-\sss}_{\bd^{\prime},\bd^{\prime \prime}} \subset \Rep^{T,\zeta-\sss}_{\bd^{\prime}+\bd^{\prime \prime}}$ to be the subvariety of representations such that the flag $\C^{\bd^{\prime}_{i}} \subset \C^{\bd^{\prime}_{i}+ \bd^{\prime \prime}_{i}}$ for all $i \in Q_0$ is preserved. Let $\GL_{\bd^{\prime},\bd^{\prime \prime}} \times T \subset \GL_{\bd^{\prime}+\bd^{\prime \prime}} \times T$ be the subgroup preserving this flag. We then define the quotient stack of extensions \[\mathrm{Ext}^{T,\zeta-\sss}_{\bd^{\prime},\bd^{\prime \prime}}(Q): = \Rep^{T,\zeta-\sss}_{\bd^{\prime},\bd^{\prime \prime}}(Q)/(\GL_{\bd^{\prime},\bd^{\prime \prime}} \times T).\]

For any algebra $A$ given by the quotient of $\mathbb{C}Q$ by a two sided ideal which is preserved by $T$, we define quotient stack $\mathfrak{M}^{T,\zeta-\sss}_{\bd}(A)$ as the $\GL_{\bd}$ quotient of $\zeta$ semistable $\bd$ dimensional representations of $A$. We similarly have stacks $\Ext^{T,\zeta-\sss}_{\bd^{\prime},\bd^{\prime \prime}}(Q)$ and coarse moduli spaces $\mathcal{M}_{\bd}^{T,\zeta-\sss}(A)$. We will be focusing especially on the case when $A=\mathrm{\Pi}_Q$, which is the quotient of the path algebra of the doubled quiver $\overline{Q}$. This stack can be more conceptually understood in the following way. 

\subsection{Representations of preprojective algebra} \label{moduliofnilpotent}
Let $\overline{Q}$ be the doubled quiver. Then we can identify \[ \mathrm{Rep}_{\bd}(\overline{Q}) \simeq T^{*} \mathrm{Rep}_{\bd}(Q)\] by the pairing between $\Hom(\mathbb{C}^{\bd(s(a))},\mathbb{C}^{\bd(t(a))})$ and $\Hom(\mathbb{C}^{\bd(t(a))},\mathbb{C}^{\bd(s(a))})$ for any arrow $a \in Q_1$. The symplectic manifold $T^*\mathrm{Rep}_{d}(Q)$ carries the action of Gauge group $\mathrm{GL}_{\bd}$ with moment map

\[\mu_{Q,\bd}: \Rep_{\bd}(\overline{Q}) \rightarrow \mathfrak{gl}_{\bd}:= \prod_{i \in Q_0} \mathfrak{gl}_{\bd_i}\] given by \[(\rho(a),\rho(a^{*}))_{a \in Q_1} \mapsto \sum_{a \in Q_1} [\rho(a),\rho(a^{*})] .\] Here we have identified $\mathfrak{gl}_{\bd}$ with $\mathfrak{gl}_{\bd}^{\vee}$ by the trace pairing. Then a substack of $\mathfrak{M}_{\bd}(\overline{Q})$, the quotient stack $\mu^{-1}_{Q,\bd}(0)/\mathrm{GL}_{\bd}$ can be identified with the stack of representations of the preprojective algebra $\mathfrak{M}_{\bd}(\mathrm{\Pi}_Q)$. 

Let $\mathfrak{M}_{\bd}^{\zeta-\sss}(\mathrm{\Pi}_Q)= \mu^{-1,\zeta-\sss}(0)/\mathrm{GL}_{\bd}$ be the open substack of $\zeta$ semisstable stable quiver representations. We may also consider GIT quotients $\mathcal{M}_{\bd}^{\zeta-\sss}(\mathrm{\Pi}_Q)$. Similar to quiver representations, we have semisimplification morphism \[ \JH: \mathfrak{M}_{\bd}(\mathrm{\Pi_Q}) \rightarrow \mathcal{M}_{\bd}(\mathrm{\Pi_Q}).\] When the torus $T$ preserves the moment map relations, we can construct moduli spaces $\mathfrak{M}_{\bd}^{T,\zeta-\sss}(\mathrm{\Pi_Q})$ and $\mathcal{M}_{\bd}^{T,\zeta-\sss}(\mathrm{\Pi_Q})$ in exactly similar way as before. 

%These stacks are often highly singular, as seen in the following example. 
%\begin{example}
%Let $Q$ be the Jordan quiver. Then the $\mu^{-1}_{Q_{\Jor,\bd}}(0)$ is space of commuting matrices and thus the stack of preprojective algebra representations is \[ \mathfrak{M}_{d}(\mathrm{\Pi}_{Q_{\Jor}}) = \{ A,B \in \End(\mathbb{C}^n) \mid [A,B]=0 \}/\GL_{d}.\] This can be equivalently described as the stack $\Coh_{d}(\mathbb{A}^2)$ of length $d$ torsion sheaves on $\mathbb{A}^2$ and the coarse moduli space $\mathcal{M}_{d}(\mathrm{\Pi_{Q_{\Jor}}})$ can be identified with symmetric space of $d$ points on plane $\Sym^n(\mathbb{C}^2)$. Finally, the Jordan holder map is identified with the support morphism $\Supp: \Coh_{d}(\mathbb{A}^2) \rightarrow \Sym_n(\mathbb{C}^2)$. 
%\end{example}

Earlier, we defined the notion of nilpotency $\mathcal{N}$ for preprojective algebra representation. Following Lusztig, we can define Lusztig nilpotent stack $\mathfrak{M}^{\mathcal{N}}_{\bd}(\mathrm{\Pi}_Q)$ as the closed substack $\mu_{\bd}^{-1,\mathcal{N}}(0)/GL_{\bd}$ of nilpotent representations of $\mathrm{\Pi}_Q$ where $\mu_{\bd}^{-1,\mathcal{N}}(0) := (\mu^{-1}_{\bd}(0) \cap \Rep_{\bd}^{\mathcal{N}}(\overline{Q})$.

%\begin{example} \label{nilpotentcommuting}
%Let $Q$ be the Jordan quiver. Then the Lusztig nilpotent stack $\mathfrak{M}^{\mathcal{N}}_{d}(\mathrm{\Pi}_{Q^{\Jor}})$ parametrizes commuting nilpotent matrices of dimension $d$. 
%\end{example}

\section{Vanishing Cycles} \label{vanishingcycles}
Vanishing cycles provide a coefficient system of counting points in the moduli space of objects in 3CY categories, which are locally a critical locus. In this section, we will introduce vanishing cycles and the properties we need in later sections. For more details, refer to \cite{MR1074006}, \cite{davison2020cohomological} and \cite{achar}. Let $f: X \rightarrow \mathbb{C}$ be a function on a smooth complex variety $X$. Then we consider the following commutative diagram: 
\[\begin{tikzcd}[ampersand replacement=\&]
	{X_0} \& {X_{\leq 0}} \& X \& {X_{>0}} \\
	{\{z=0 \}} \& {\{z \in \mathbb{C} \mid \mathfrak{Re}(z) \leq 0 \}} \& {\mathbb{A}^{1}} \& {\{z \in \mathbb{C} \mid \mathfrak{Re}(z)> 0 \}}
	\arrow["{i_0}", from=1-1, to=1-2]
	\arrow["i", curve={height=-20pt}, dashed, from=1-1, to=1-3]
	\arrow[from=1-1, to=2-1]
	\arrow["{i_{\leq 0}}", hook, from=1-2, to=1-3]
	\arrow[from=1-2, to=2-2]
	\arrow["f", from=1-3, to=2-3]
	\arrow["r"', from=1-4, to=1-3]
	\arrow[from=1-4, to=2-4]
	\arrow[hook, from=2-1, to=2-2]
	\arrow[hook, from=2-2, to=2-3]
	\arrow[hook', from=2-4, to=2-3]
\end{tikzcd}\]

\begin{defn}[Nearby cycle functor]
 The perverse\footnote{We have shifted the usual functor by [-1].} nearby cycle functor 

 \[ {}^{\mathfrak{p}}\mathrm{\Psi}_{f}: \Dbc(X) \rightarrow \Dbc(X)\] is defined to be 
 \[\mathrm{\Psi}_{f}  := i_{*}i^{*}r_{*}r^{*}[-1] \]
 
 \end{defn}

By adjunction, there is a natural transformation $\mathrm{id} \rightarrow r_{*}r^{*}$, precomposing with $i_{*}i^{*}$ yields a natural transformation

\begin{equation} \label{exacttriangle}i_{*}i^{*}[-1] \rightarrow {}^{\mathfrak{p}}\mathrm{\Psi}_f.
\end{equation}

\begin{defn}[Vanishing cycle functor]

    The vansishing cycle functor \[ {}^{\mathfrak{p}}\mathrm{\Phi}_{f}: \Dbc(X) \rightarrow \Dbc(X)\] is defined to be the cone of above morphism
    \[ {}^{\mathfrak{p}}\mathrm{\Phi}_{f}:= \mathrm{Cone}(i_{*}i^{*}[-1] \rightarrow {}^{\mathfrak{p}}\mathrm{\Psi}_f) \]
\end{defn}

These functors behave nicely with the perverse structure, in particular, the \textbf{vanishing cycle sheaf} $: \pPhi_{f}(\Q^{\vir}_{X})$ is a perverse sheaf on $X$.

Like perverse sheaves, vanishing cycles can be generalized to Artin stacks (See \cite{MR2480756} and \cite{tubach}. In this thesis, we will mainly be dealing with the case of smooth quotient stacks $X/G$, where one can understand its cohomology as the equivariant cohomology $H_{G}(X,\pPhi_f)$. 
\begin{defn}[Critical cohomology]
Given a function $f: \mathfrak{X} \rightarrow \mathbb{A}^1$, we define critical cohomology with slight abuse of notation: \[ \HH(\fX, \pPhi_f) : = \HH(\fX, \pPhi_f \Q^{\vir}_{\mathfrak{X}}). \]
\end{defn}

The vanishing cycle functor enjoys many nice properties, which we shall explain now. We refer \cite{davison2020cohomological} for more details. We assume that all the stacks considered are quotient stacks.

\subsection{Thom-Sebastiani} \label{ThomSebastini} There is an analogue of the K\"{u}nneth isomorphism for the critical cohomology. Let $f$ and $g$ be two functions on smooth stacks $\mathfrak{X}$ and $\mathfrak{Y}$ respectively. The Thom-Sebastiani theorem states that there is a natural equivalence \[ \TS: \pPhi_{f \boxplus g}(\mathcal{F} \boxtimes \mathcal{G})|_{f^{-1}(0) \times g^{-1}(0)} \simeq \pPhi_{f} \mathcal{F} \boxtimes \pPhi_{g} \mathcal{F} \] and when $\Supp(\pPhi_{f \boxplus g}(\mathcal{F} \boxtimes \mathcal{G})) \subset f^{-1}(0) \times g^{-1}(0)$, we get an isomorphism on cohomology 
\[ \TS \colon \HH(\mathfrak{X}, \pPhi_{f} \mathcal{F}) \otimes  \HH(\mathfrak{Y}, \pPhi_{f} \mathcal{G}) \simeq \HH(\mathfrak{X} \times \mathfrak{Y},\pPhi_{f \boxplus g}(\mathcal{F} \boxtimes \mathcal{G})).  \]

\subsection{Pullback} \label{Pullback} Given a morphism $p \colon \mathfrak{X} \rightarrow \mathfrak{Y}$ between smooth stacks, there is a natural morphism \[ p^{*} \pPhi_{f}  \rightarrow \pPhi_{fp}  p^{*} \] which is an isomorphism when the map $p$ is smooth. Applying this to $\Q_{\mathfrak{Y}}$ and pushing forward to the point, we get a pullback map
\[ p^{*} \colon \HH(\mathfrak{Y},\pPhi_{f} \Q_{\mathfrak{M}}) \rightarrow \HH(\mathfrak{X},\pPhi_{fp} \Q_{\mathfrak{X}}).\]

\subsection{Verdier Duality} \label{VerdierDual} There is a natural isomorphism $\D \pPhi_{f} \simeq \pPhi_f \D$ where $\D$ is the Verdier duality functor. 

\subsection{Pushforward} \label{Pushforward} Given any morphism $p \colon \mathfrak{X} \rightarrow \mathfrak{Y}$ and a function $f \colon \mathfrak{Y} \rightarrow \C$ there is a natural map \[ \pPhi_{f}p_{*}  \rightarrow p_{*}\pPhi_{fp}\] which is an isomorphism when $p$ is proper. Applying this to $\mathcal{F} = \Q_{\mathfrak{X}}$ and composing gives a  morphism $\pPhi_{f} \Q_{\mathfrak{Y}} \rightarrow p_{*}(\pPhi_{fp} \Q_{\mathfrak{X}})$. Applying $\D$ and pushing forward to the point yields a map for $p$ proper:\[ p_{*} \colon  \HH(\mathfrak{X}, \pPhi_{fp}  \Q_{\mathfrak{X}}) \rightarrow \HH(\mathfrak{Y}, \pPhi_{f}  \Q_{\mathfrak{Y}})[2(\dim(\mathfrak{Y})-\dim(\mathfrak{X}))]. \] 

\subsection{Dimension Reduction} \label{dimensionreduction}

Let $X=X' \times \A^n$ be a smooth variety, let $f$ be a function on $X$ of weight one for the scaling action on $\mathbb{A}^n$ and trivial action on $X$ and let $\pi \colon X \rightarrow X'$ be the projection map so that we can write \[f = \sum_{1 \leq i \leq n} x_i \pi^{*}(f_i)\] where $f_i$ are functions on $X'$ and $x_i$ are coordinates on $\A^{n}$. Let \[Z':= \{ x \in X^{\prime} \mid f_i(x)=0 \ \forall \  x \in X^{\prime}\} \]and let $Z := \pi^{-1}(Z')$. Let $i_{Z'}: Z' \hookrightarrow X'$, and $i: Z \rightarrow X$ be the inclusions. Then the natural map \[\pi_{!}\pPhi_{f}\pi^{*} \rightarrow \pi_{!} \pi^{*}(i_{Z^{\prime}})_{*}(i_{Z^{\prime}})^{*}\] is an isomorphism \cite{davison2016critical}[Appendix A].

\begin{example}
We apply dimension reduction isomorphism to the sheaf $\mathcal{F} = \mathbb{Q}_{X^{\prime}}$. This gives an isomorphism of complexes

\[ \pi_{!} \pPhi_{f} \mathbb{Q}_X \simeq \pi_{!} \pi^{*}(i_{Z^{\prime}})_{*}\mathbb{Q}_{Z^{\prime}} \simeq \pi_{!}i_{*}\mathbb{Q}_{Z} = \mathbb{Q}_{Z^{\prime}}[-2n].\]
Taking global sections gives an isomorphism of cohomology
\begin{equation}
\HH_c(X, \pPhi_{f} \mathbb{Q}_X) \simeq \HH_c(Z, \mathbb{Q}_{Z}) \simeq \HH_c(Z^{\prime}, \mathbb{Q}_{Z^{\prime}}[-2n]).
\end{equation}

Applying Verdier duality gives an isomorphism of cohomology:

\[ \HBM(Z^{\prime},\mathbb{Q}_{Z^{\prime}}[-2n]) \simeq \HH(X, \pPhi_f(\mathbb{Q}^{\vir}_X)[\dim(X)]).\]

This isomorphism also extends to quotient stacks. Suppose that an algebraic group $G$ acts on $X^{\prime}$ and $\mathbb{A}^n$ such that the function $f$ is also $G$ invariant. Then we get an isomorphism of sheaves \[ \pi_{!} \pPhi_{f} \mathbb{Q}_{X/G} \simeq  \mathbb{Q}_{Z^{\prime}/G}[-2n].\] Taking Verdier dual gives an isomorphism of sheaves 
\begin{equation} \label{dimensionreduction11} \mathbb{D}\mathbb{Q}_{Z^{\prime}/G}[2n-(\dim(X/G))] \simeq \pi_{*} \pPhi_{f} \mathbb{Q}^{\vir}_{X/G}.\end{equation}

%Suppose that $i_S: S \rightarrow Z^{\prime}$ is some $G$ invariant subspace, then after taking pullback along $i_S$ and taking global sections gives isomorphism of cohomology

%\[ \HBM(Z'/G,\Q) \simeq \HH(X/G, \pPhi_{f}\Q)[2 \dim(X'/G)].\]

\end{example}

\section{Kontsevich-Soibelman Cohomological Hall algebra} \label{hallalgebra}We recall the construction of the cohomological Hall algebra of a quiver with potential from \cite{kontsevich2011cohomological}. Let $Q$ be any quiver with potential $W$ and let $T$ be any torus defined by weighting $\textbf{w}: Q_1 \rightarrow \mathbb{Z}^r$ such that the weight of potential $W$ is 0. Then given any stability condition $\zeta$ and slope $\mu$, we consider the graded vector space \[ \mathcal{A}^{T,\zeta}_{Q,W,\mu} := \bigoplus_{\bd \in \mathrm{\Lambda}^{\zeta}_{\mu}} \mathcal{A}^{T,\zeta}_{Q,W,\bd} \] where \[ \mathcal{A}^{T,\zeta}_{Q,W,\bd}  := \HH(\mathfrak{M}^{T,\zeta-\sss}_{\bd}(Q), \pPhi_{\Tr^{\zeta}_{\bd}(W)}\Q^{\vir}), \] and $\Tr(W)^{\zeta}_{\bd} \colon \mathfrak{M}^{T,\zeta-\sss}_{\bd}(Q) \rightarrow \C$ is the trace function\footnote{We shall ignore the subscripts from $\Tr(W)^{\zeta}_{\bd}$ when context is clear.}, the grading is given by dimension and cohomology and $\Q^{\vir}$ is the constant $T$ equivariant perverse sheaf defined\footnote{We choose this twist since we want to think of $\mathcal{A}^{T}_{Q,W}$ as version of $\mathcal{A}_{Q,W}$ with equivariant parameters. For $\mathcal{A}_{Q,W}$, the sheaf $\Q^{\vir} = \Q[-(\bd,\bd)]$ is perverse, while for $\mathcal{A}^{T}_{Q,W}$ we choose the sheaf so that restriction to the fibres of projection $\mathfrak{M}^{T}_{\bd}(Q) \rightarrow \pt/T$ are perverse.} by 
\[ \Q^{\vir}|_{\mathfrak{M}^{T}_{\bd}(Q)} = \Q_{\mathfrak{M}^{T}_{\bd}(Q)}[-(\bd,\bd)]. \] 
We then consider the convolution diagram \[\begin{tikzcd}[ampersand replacement=\&] \label{cohadiagram}
	\& {\mathfrak{M}^{T,\zeta-\sss}_{\bd_1,\bd_2}(Q)} \\
	{\mathfrak{M}^{T,\zeta-\sss}_{\bd_1}(Q) \times_{[\pt/T]}\mathfrak{M}^{T,\zeta-\sss}_{\bd_2}(Q)} \&\& {\mathfrak{M}^{T,\zeta-\sss}_{\bd_1+\bd_2}(Q)}
	\arrow["{\pi_1 \times \pi_3}"', from=1-2, to=2-1]
	\arrow["{\pi_2}", from=1-2, to=2-3]
\end{tikzcd}\]
By taking the inverse of the Thom-Sebastiani isomorphism (\ref{ThomSebastini}), we have the isomorphism of graded vector spaces:
\begin{align*}
\textrm{TS}^{-1} \colon 
\HH(\mathfrak{M}^{T,\zeta-\sss}_{\bd_1}(Q), \pPhi_{\Tr(W)} \Q^{\vir}) & \otimes \HH(\mathfrak{M}^{T,\zeta-\sss}_{\bd_2}(Q), \pPhi_{\Tr(W)} \Q^{\vir})\\& \simeq 
\HH(\mathfrak{M}^{T,\zeta-\sss}_{\bd_1}(Q) \times \mathfrak{M}^{T,\zeta-\sss}_{\bd_2}(Q), \pPhi_{\Tr(W)_{\bd_1} \boxplus \Tr(W)_{\bd_2}} \Q^{\vir}). 
\end{align*}

By the natural map $[\pt/T] \rightarrow \pt$, we have a forgetful morphism 
$j: \mathfrak{M}^{T,\zeta-\sss}_{\bd_1}(Q) \times_{[\pt/T]} \mathfrak{M}^{T,\zeta-\sss}_{\bd_2}(Q) \rightarrow  \mathfrak{M}^{T,\zeta-\sss}_{\bd_1}(Q) \times \mathfrak{M}^{T,\zeta-\sss}_{\bd_2}(Q)$, so by taking the pullback (\ref{Pullback}), we have the morphism of graded vector spaces 

\begin{align*}
j^{*} \colon \HH(\mathfrak{M}^{T,\zeta-\sss}_{\bd_1}(Q) \times \mathfrak{M}^{T,\zeta-\sss}_{\bd_2}(Q), & \pPhi_{\Tr(W)_{\bd_1} \boxplus \Tr(W)_{\bd_2}} \Q^{\vir}) \\ & \simeq \HH(\mathfrak{M}^{T,\zeta-\sss}_{\bd_1}(Q) \times_{[\pt/T]} \mathfrak{M}^{T,\zeta-\sss}_{\bd_2}(Q), \pPhi_{\Tr(W)_{\bd_1} \boxplus \Tr(W)_{\bd_2}} \Q^{\vir}).
\end{align*}

By taking pullback (\ref{Pullback}) along the morphism $\pi_1 \times \pi_3$ we have the morphism of graded vector spaces 
 
\begin{align*}
(\pi_1 \times \pi_3)^{*} \colon \HH(\mathfrak{M}^{T,\zeta-\sss}_{\bd_1}(Q)  & \times_{[\pt/T]} \mathfrak{M}^{T,\zeta-\sss}_{\bd_2}(Q),  \pPhi_{\Tr(W)_{\bd_1} \boxplus \Tr(W)_{\bd_2}} \Q^{\vir}) \rightarrow \\ & \HH(\mathfrak{M}^{T,\zeta-\sss}_{\bd_1,\bd_2}(Q), \pPhi_{\Tr(W)|_{\mathfrak{M}^{T,\zeta-\sss}_{\bd_1,\bd_2}(Q)}}\Q_{\mathfrak{M}^{T,\zeta-\sss}_{\bd_1,\bd_2}(Q)})[-(\bd_1,\bd_1)-(\bd_2,\bd_2)].
\end{align*}
since $\dim(\mathfrak{M}_{\bd}(Q)) = -(\bd,\bd)$. Then finally doing the pushforward (\ref{Pushforward}) along $\pi_2$, gives a morphism of graded vector spaces 
\begin{align*}(\pi_2)_{*} \colon \HH(\mathfrak{M}^{T,\zeta-\sss}_{\bd_1,\bd_2}(Q),  \pPhi_{\Tr(W)|_{\mathfrak{M}^{T,\zeta-\sss}_{\bd_1,\bd_2}(Q)}}\Q_{\mathfrak{M}^{T,\zeta-\sss}_{\bd_1,\bd_2}(Q)})[-(\bd_1,\bd_1)-& (\bd_2,\bd_2)] \rightarrow \\ &\HH(\mathfrak{M}^{T,\zeta-\sss}_{\bd_1+\bd_2}(Q),\pPhi_{\Tr(W)}\Q^{\vir}). 
\end{align*} Note that above respects the cohomological grading since  \[\dim(\mathfrak{M}_{\bd_1+\bd_2}^{T,\zeta-\sss})- 2 \dim(\mathfrak{M}^{T,\zeta-\sss}_{\bd_1,\bd_2}(Q)) = -(\bd_1,\bd_1)-(\bd_2,\bd_2).\]Composing the above morphisms together gives a morphism of cohomologically graded vector spaces  \[m_{\bd_1,\bd_2}^{T}:= (\pi_2)_{*} (\pi_1 \times \pi_3)^{*} j^{*}  \textrm{TS}^{-1} \colon  \mathcal{A}^{T,\zeta}_{Q,W,\bd_1} \otimes \mathcal{A}^{T,\zeta}_{Q,W,\bd_2}  \rightarrow \mathcal{A}^{T,\zeta}_{Q,W,\bd_1+ \bd_2}. \]
Taking the direct sum of maps $m_{\bd_1,\bd_2}$ over all $(\bd_1,\bd_2) \in \mathrm{\Lambda}^{\zeta}_{\mu} \times \mathrm{\Lambda}^{\zeta}_{\mu}$ defines a map of $\mathrm{\Lambda}^{\zeta}_{\mu} \times \mathbb{Z}$ graded vector spaces \[ m^{\zeta}_{\mu} \colon \mathcal{A}^{T,\zeta}_{Q,W\mu} \otimes \mathcal{A}^{T,\zeta}_{Q,W,\mu} \rightarrow \mathcal{A}^{T,\zeta}_{Q,W,\mu} \] where by $\mathrm{\Lambda}^{\zeta}_{\mu}$ we mean the dimension grading while $\Z$ is the cohomological grading. The map $m^{\zeta}_{\mu}$ in fact gives $\mathcal{A}^{T,\zeta}_{Q,W,\mu}$ a structure of $\mathrm{\Lambda}^{\zeta}_{\mu} \times \Z$ graded associative algebra structure \cite{kontsevich2011cohomological}[Section 2.3]. We will often denote the multiplication by $\star$. We note that when $\zeta=(0,\cdots,0)$ and $\mu=0$ then $\mathfrak{M}^{T,\zeta-\sss}_{\bd}(Q) = \mathfrak{M}^{T}_{\bd}(Q)$ and the direct sum is over all $\bd$ in $\N^{Q_0}$. We will denote the resulting algebra by $\mathcal{A}^{T}_{Q,W}$ and $\mathcal{A}_{Q,W,\mu}$ for the case without $T$. Note that by definition $\mathcal{A}^{T,\zeta}_{Q,W,\mu}$ is a $\HHT$ linear algebra. 

\subsection{Nilpotent 3d CoHA}
Following \cite{davison2020cohomological}, one can define a CoHA structure for any geometrically defined Serre subcategory of the quiver representations. In this thesis, we will only consider a case of this construction which is given by the Serre category $\tilde{\mathcal{N}}$  of nilpotent representations of the tripled quiver $\tilde{Q}$, defined in Section \ref{nilpotetnserre}. Let $\overline{\omega}: \mathfrak{M}_{\bd}^{\tilde{\mathcal{N}},\zeta-\sss}(\tilde{Q}) \rightarrow \mathfrak{M}_{\bd}^{\zeta-\sss}(\tilde{Q})$ denote the inclusion of the stacks, where $\mathfrak{M}^{\overline{N}}_{\bd}(\tilde{Q})$ is the stack of $\zeta$ semistable $\bd$ dimensional representations of quiver $\tilde{Q}$. We consider the graded vector space \[\mathcal{A}^{T,\zeta,\tilde{\mathcal{N}}}_{Q,W,\mu} := \bigoplus_{\bd \in \mathrm{\Lambda}^{\zeta}_{\mu}} \mathcal{A}^{T,\zeta,\tilde{\mathcal{N}}}_{Q,W,\bd} \]
where we define
\[\mathcal{A}^{T,\zeta,\tilde{\mathcal{N}}}_{Q,W,\bd} :=   \HH(\mathfrak{M}^{T,\zeta-\sss}_{\bd}(Q), \overline{\omega}_{*}(\overline{\omega})^{!}\pPhi_{\Tr(W)}\Q^{\vir}).\]

Then, in the same way as the last section, the CoHA convolution diagram constructs an associative algebra structure on $\mathcal{A}^{T,\zeta,\tilde{\mathcal{N}}}_{Q,W,\mu}$.

\subsection{Relative CoHA}
The algebra $\mathcal{A}^{T,\zeta}_{Q,W,\mu}$ admits a lift to an algebra object $\mathcal{RA}^{T,\zeta}_{Q,W,\mu}$ in the derived category $\Dbc(\Per^{\prime}(\mathcal{M}^{T,\zeta}_{\mu}(Q)))$ of perverse sheaves on the coarse moduli space. \footnote{The shifted perverse $t$ structure on $\Dbcc(\mathcal{M}^{T,\zeta-\sss}_{\bd}(Q))$ is defined by setting ${}^{\mathfrak{p}'}\!\tau^{\leq i} : = {}^{\mathfrak{p}}\!\tau^{\leq i- \dim(T)}$ and ${}^{\mathfrak{p'}}\!\tau^{\geq i} := {}^{\mathfrak{p}}\!\tau^{\geq i-\dim(T)}  $. Let $\Per^{'}(\mathfrak{X}/T)$ to be heart with respect to this shifted perverse structure. This is defined so that for any $\mathcal{F} \in \Per^{'}(\mathfrak{X}/T)$, $(\mathfrak{X} \rightarrow \mathfrak{X}/T)^{*}(\mathcal{F}) \in \Per(\mathfrak{X})$.}This lift is sometimes useful in studying this algebra as many of the products in algebra can then be interpreted as morphisms of semisimple Perverse sheaves, which are relatively easier to understand. Let \[\mathfrak{M}^{T,\zeta-\sss}_{\mu}(Q) := \coprod_{\bd \in \mathrm{\Lambda}^{\zeta}_{\mu}} \mathfrak{M}^{T, \zeta- \sss}_{\bd}(Q) \] and similarly let \[\mathcal{M}^{\zeta-\sss}_{\mu}(Q) := \coprod_{\bd \in \mathrm{\Lambda}^{\zeta}_{\mu}} \mathcal{M}^{\zeta- \sss}_{\bd}(Q).\] We denote $\mathfrak{M}^{T,\zeta-\sss}_{\mu}(Q)$ and $\mathcal{M}^{T,\zeta-\sss}_{\mu}(Q)$  by $\mathfrak{M}^{T}(Q)$ and $\mathcal{M}^{T,\zeta-\sss}_{\mu}(Q)$ respectively, in the case when $\zeta=0, \mu=0$. Finally let $\JH^{T}_{\mu}:= \coprod_{\bd \in \mathrm{\Lambda}^{\zeta}_{\mu}} \JH^{T}_{\bd}: \mathfrak{M}^{T,\zeta-\sss}_{\mu}(Q) \rightarrow \mathcal{M}^{T,\zeta-\sss}_{\mu}(Q)$ to be the semisimplification morphism. 

We consider the direct sum map \[ \mathcal{M}^{\zeta}_{\mu}(Q) \times \mathcal{M}^{\zeta}_{\mu}(Q) \xrightarrow{\oplus} \mathcal{M}^{\zeta}_{\mu}(Q) \]  It is proved in \cite{svenrein} that this morphism is a finite and commutative morphism of schemes. This morphism lifts to the $T$ quotients
\[ \mathcal{M}^{T,\zeta}_{\mu}(Q) \times_{[\pt/T]} \mathcal{M}^{T,\zeta}_{\mu}(Q) \xrightarrow{\oplus^T} \mathcal{M}^{T,\zeta}_{\mu}(Q) \] is also finite and commutative. We consider the monoidal product \[ \mathcal{F} \boxtimes_{\oplus^{T}} \mathcal{G} := \oplus^{T}_{*}(\mathcal{F} \boxtimes_{[\pt/T]} \mathcal{G})\] over $[\pt/T]$ on $\Dbc(\Per^{\prime}(\mathcal{M}^{T,\zeta}_{\mu}(Q))$, allowing us to consider algebra objects in $\Dbc(\Per^{\prime}(\mathcal{M}^{T,\zeta}_{\mu}(Q))$. Since the morphism $\oplus^{T}$ is finite, it implies that the monoidal product is bi-exact and symmetric. 

Let $\mathrm{\Lambda}^{\zeta}_{\mu} \times [\pt/T]$ be the stack consisting of the copy of $[\pt/T]$ for every dimension vector $\bd \in \mathrm{\Lambda}^{\zeta}_{\mu}$. Then the $\mathrm{\Lambda}^{\zeta}_{\mu}$ graded algebra structure on $\mathcal{A}^{T,\zeta-\sss}_{Q,W,\mu}$ can be understood as the algebra structure on the object $p_{*}\pPhi_{\Tr(W)}\mathbb{Q}^{\vir}_{\mathfrak{M}^{T,\zeta}_{\mu}(Q)} \in \Dbc(\mathrm{\Lambda}^{\zeta}_{\mu} \times [\pt/T])$ where $p: \mathfrak{M}^{\zeta}_{\mu}(Q) \rightarrow \mathrm{\Lambda}^{\zeta}_{\mu} \times [\pt/T]$ is the canonical morphism and $\mathrm{\Lambda}^{\zeta}_{\mu} \times [\pt/T]$ is a monoid over $[\pt/T]$ given by the sum of dimension vectors. We now consider something intermediate. 

Let \[ \mathcal{R}\mathcal{A}^{T,\zeta}_{Q,W,\mu} := (\JH^{T,\zeta}_{\mu})_{*}\pPhi_{\Tr(W)}(\mathbb{Q}^{\vir}_{\mathfrak{M}^{T,\zeta}_{\mu}(Q)}).\] Since, after doing semisimplification, non-trivial extensions are the same as the direct sum, the CoHA convolution diagram lifts to the commutative diagram

\[\begin{tikzcd}[ampersand replacement=\&]
	\& {\mathrm{Ext}^{T,\zeta-\sss}_{\bd_1,\bd_2}(Q)} \\
	{\mathfrak{M}^{T,\zeta-\sss}_{d_1}(Q) \times_{[\pt/T]}\mathfrak{M}^{T,\zeta-\sss}_{d_2}(Q)} \&\& {\mathfrak{M}^{T,\zeta-\sss}_{\bd_1+\bd_2}(Q)} \\
	{\mathcal{M}^{T,\zeta-\sss}_{d_1}(Q) \times_{[\pt/T]}\mathcal{M}^{T,\zeta-\sss}_{d_2}(Q)} \&\& {\mathcal{M}^{T,\zeta-\sss}_{\bd_1+\bd_2}(Q)}
	\arrow["{{\pi_1 \times \pi_3}}"', from=1-2, to=2-1]
	\arrow["{{\pi_2}}", from=1-2, to=2-3]
	\arrow["{\JH^{T}_{\bd_1} \times \JH^{T}_{\bd_2}}"', from=2-1, to=3-1]
	\arrow["{\JH^{T}_{\bd_1+\bd_2}}", from=2-3, to=3-3]
	\arrow["{\oplus^{T}}"', from=3-1, to=3-3]
\end{tikzcd}\]

Now, we can do the same game as in the definition of Kontsevich-Soibelman CoHA, but instead we push forward underlying sheaves to $\mathcal{M}^{T,\zeta-\sss}_{\mu}(Q)$. It is proved in \cite{davison2020cohomological} and \cite{davison2022bps} that this defines a morphism \[ \mathcal{R}m: \mathcal{R}\mathcal{A}^{T,\zeta}_{Q,W,\mu} \boxtimes_{\oplus^{T}} \mathcal{R}\mathcal{A}^{T,\zeta}_{Q,W,\mu} \rightarrow \mathcal{R}\mathcal{A}^{T,\zeta}_{Q,W,\mu}\] which makes $\mathcal{R}\mathcal{A}^{T,\zeta}_{Q,W,\mu}$ an algebra object inside the monoidal category $\Dbc(\Per^{\prime}(\mathcal{M}^{T,\zeta}_{\mu}(Q)))$. We will often refer to this as the relative CoHA. Taking cohomology of the moprhism $\mathcal{R}m$ above, we obtain the usual CoHA $\mathcal{A}^{T,\zeta}_{Q,W,\mu}$.

\subsection{Tautological bundles}\label{tautologicalbundle}

Given a family $\mathcal{F}$ of $\bd$ dimensional $\zeta$ semistable representations of $\C Q$ we can obtain a line bundle by taking $\Det(\mathcal{F})$. This gives a map $\Det \colon \mathfrak{M}_{\bd}^{T,\zeta-\sss}(Q) \rightarrow \BCu$ where $\BCu = \pt/\mathbb{C}^{*}_{u}$ and we use variable $u$ to emphasise that $u \in \HH^{2}(\BCu,\Q)$ is the tautological generator for this $\C^{*}$. Thus, we have a map of stacks
\begin{equation}(\Det \times \id) \colon \mathfrak{M}_{\bd}^{T,\zeta-\sss}(Q) \rightarrow \BCu \times \mathfrak{M}_{\bd}^{T,\zeta-\sss}(Q)  \end{equation} which induces a natural map in the constructible category \begin{equation}  \label{tautologicalmap} \pPhi_{0 \boxplus \Tr(W)_{\bd}} \Q_{\BCu 
 \times \mathfrak{M}_{\bd}^{T,\zeta-\sss}(Q)} \rightarrow  (\Det \times \id)_{*}\pPhi_{\Tr(W)_{\bd}}\Q_{\mathfrak{M}_{\bd}^{T,\zeta-\sss}(Q)}\end{equation} while by 
 Thom-Sebastiani $\pPhi_{0 \boxplus \Tr(W)_{\bd}} \Q_{\BCu 
 \times \mathfrak{M}_{\bd}^{T,\zeta-\sss}(Q)} \simeq \Q_{\BCu}  \boxtimes \pPhi_{\Tr(W)_{d}}\Q_{\mathfrak{M}_{\bd}^{T,\zeta-\sss}(Q)}$ and so after shifting by $[(\bd,\bd)]$ and taking global sections, we get 

 \[ (\Det \times \id)^{*} \colon \HH( \BCu,\Q) \otimes \HH(\mathfrak{M}^{T,\zeta-\sss}_{\bd}(Q),\pPhi_{\Tr(W)_{\bd}}) \rightarrow \HH(\mathfrak{M}^{T,\zeta-\sss}_{\bd}(Q),\pPhi_{\Tr(W)_{\bd}}) \]
 This gives an action of $\Q[u] = \HH(\BCu,\Q)$ on $\mathcal{A}^{T,\zeta}_{Q,W,\mu}$ which we denote by $\cdot$. We will see later that this action acts by derivation. 

 We may also pushforward the morphism \ref{tautologicalmap} by $\JH^{T,\zeta}_{d}$ to induce similar action of $\Q[u]$ on $\mathcal{RA}^{T,\zeta}_{Q,W\mu}$. 

\subsection{Perverse filtration and BPS Lie algebra}

\subsubsection{Sign Twist}\label{signtwist}

We slightly twist the multiplication on $\mathcal{A}^{T,\zeta}_{Q,W}$ so that the resulting algebra satisfies the PBW theorem. Given dimension vectors $\bd^{\prime}, \bd^{\prime\prime}$, we define 
\begin{align*}
\tau(\bd^{\prime},\bd^{\prime \prime}) = \chi_{Q}(\bd^{\prime},\bd^{\prime })\chi_{Q}(\bd^{\prime \prime},\bd^{\prime \prime}) + \chi_Q(\bd^{\prime},\bd^{\prime \prime}) 
\end{align*} 
Then a bilinear form $\psi$ on $(\mathbb{Z}/2\mathbb{Z})^{Q_0}$ is defined so that we have \[ \psi(\bd^{\prime},\bd^{\prime \prime}) + \psi(\bd^{\prime \prime},\bd^{\prime}) \equiv \tau(\bd^{\prime},\bd^{\prime \prime}) \mod 2. \] Then $\mathcal{A}^{T,\zeta,\psi}_{Q,W}$ is defined as an algebra on the same vector space, but with the multiplication $m^{\psi}_{\bd^{\prime},\bd^{\prime \prime}} = (-1)^{\psi(\bd^{\prime},\bd^{\prime \prime})} m_{\bd^{\prime},\bd^{\prime \prime}}$. Similarly we define $\mathcal{R}\mathcal{A}^{T,\zeta,\psi}_{Q,W,\mu}$ to be the algebra object with the same underlying complex of perverse sheaves on $\mathcal{R}\mathcal{A}^{T,\zeta,\psi}_{Q,W,\mu}$ but with the modified multiplication $\mathcal{R}m^{\psi}_{\bd^{\prime},\bd^{\prime \prime}} = (-1)^{\psi(\bd^{\prime},\bd^{\prime \prime})} \mathcal{R}m_{\bd^{\prime},\bd^{\prime \prime}}$

For a general situation of quiver $Q$ with potential $W$, we fix a choice of $\psi$. For the tripled quiver $\tilde{Q}$ and canonical potential $\tilde{W}$, we have 
\[ \tau_{\tilde{Q}}(\bd^{\prime},\bd^{\prime \prime}) \equiv \sum_{\alpha: i \rightarrow j} \bd^{\prime}_{i}\bd^{\prime \prime}_j + \bd^{\prime}_{j}\bd^{\prime \prime}_{i} \] and so we choose \[ \psi(\bd^{\prime},\bd^{\prime \prime}) = \chi_{Q}(\bd^{\prime},\bd^{\prime \prime}) = \sum_{i \in Q_0} \bd^{\prime}_{i} \bd^{\prime \prime}_{i} - \sum_{\alpha: i \rightarrow j} \bd^{\prime}_{i} \bd^{\prime \prime}_{j}.\] when the quiver $Q$ is fixed, we will denote this sign twist by $\chi$. 

%From now on, whenever we refer $\mathcal{A}^{T,\zeta}_{\tilde{Q},\tilde{W}}$ we would actually mean $\mathcal{A}^{T,\zeta,\chi}_{\tilde{Q},\tilde{W}}$ and by $\mathcal{A}^{T,\zeta}_{Q,W}$ we would mean $A^{T,\zeta,\psi}_{Q,W}$ for $\psi$ satisfying above assumptions.

%\textcolor{red}{maybe its better to do this section in twisted and later fix the twisting}
\subsection{Perverse filtration and integrality}  \label{perversefiltrationandbps}
We assume that quiver $Q$ is symmetric from now on\footnote{This is not essential and there are more general results involving $\mu$ generic stability conditions in \cite{davison2020cohomological}}. The morphism $\JH^{T,\zeta}_{\bd}$ behaves like a proper map \cite{davison2020cohomological} by which we mean that it satisfies an analog of the decompositon theorem, i.e. the direct image $(\JH^{T,\zeta}_{\bd})_{*} \pPhi_{\Tr(W)}$ splits, that is we have a non-canonical isomorphism in $\Db(\Per^{\prime}(\mathcal{M}^{T,\zeta-\sss}_{\bd}(Q)))$: 
\begin{equation} \label{splitting1}
(\JH^{T,\zeta}_{\bd})_{*} \pPhi_{\Tr(W)} \Q^{\vir} \simeq \bigoplus_{i \in \mathbb{Z}} {}^{\mathfrak{p'}} \! \mathcal{H}^{i}((\JH^{T,\zeta}_{\bd})_{*} \pPhi_{\Tr(W)})[-i]. 
\end{equation} Furthermore, it is proved in \cite{davison2020cohomological} that ${}^{\mathfrak{p'}}\!\mathcal{H}^{i}((\JH^{T,\zeta}_{\bd})_{*} \pPhi_{\Tr(W)})=0$ for all $i <1$, so the above splitting starts from $i \geq 1$. For any $i$, applying the natural transformation ${}^{\mathfrak{p'}}\!(\tau^{\leq i}) \rightarrow \id$, gives a split morphism \[ {}^{\mathfrak{p'}}\!\tau^{\leq i}((\JH^{T,\zeta}_{\bd})_{*} \pPhi_{\Tr(W)} \Q^{\vir}) \rightarrow ((\JH^{T,\zeta}_{\bd})_{*} \pPhi_{\Tr(W)} \Q^{\vir}). \] Thus pushing forward all the way to the point, gives an injection \[ \mathfrak{P}^{i} \mathcal{A}^{T,\zeta, \psi}_{Q,W,\bd} := \HH( \mathcal{M}_{\bd}^{T,\zeta-\sss}(Q), {}^{\mathfrak{p'}}\!\tau^{\leq i}((\JH^{T,\zeta}_{\bd})_{*} \pPhi_{\Tr(W)} \Q^{\vir}) ) \rightarrow \mathcal{A}^{T,\zeta, \psi}_{Q,W,\bd}, \]
giving an increasing filtration \[\mathfrak{P}^{\bullet}(\mathcal{A}^{T,\zeta, \psi}_{Q,W,\bd}) := 0 \subset \mathfrak{P}^{1}(\mathcal{A}^{T,\zeta,\psi}_{Q,W,\bd}) \subset \mathfrak{P}^{2}(\mathcal{A}^{T,\zeta,\psi}_{Q,W,\bd}) \cdots \subset \mathcal{A}^{T,\zeta, \psi}_{Q,W,\bd} \] of $\mathcal{A}^{T,\zeta,\psi}_{Q,W,\bd}$. This filtration respects the algebra structure on $\mathcal{A}^{T,\zeta,\psi}_{Q,W,\mu}$. It is shown in \cite{davison2022affine}[Theorem 3.3] that in fact the vector space $\oplus_{\bd \in \mathrm{\Lambda}^{\zeta}_{\mu}} \mathfrak{P}^{1}(\mathcal{A}^{T,\zeta,\psi}_{Q,W,\bd})$ is preserved under the Lie bracket induced by the associative algebra structure on $\mathcal{A}^{T,\zeta,\psi}_{Q,W,\mu}$. The BPS Lie algebra is defined to be the  $\mathrm{\Lambda}^{\zeta}_{\mu} \times \mathbb{Z}$ graded vector subspace\footnote{Since we define the BPS Lie algebra after fixing the twist $\psi$, we omit it from the notation.} \[ \mathfrak{g}^{\BPS, T,\zeta}_{Q,W,\mu} := \bigoplus_{\bd \in \mathrm{\Lambda}^{\zeta}_{\mu} \backslash {\mathbf{0}}} \mathfrak{P}^{1}(\mathcal{A}^{T,\zeta,\psi}_{Q,W,\bd}) \subset \mathcal{A}^{T,\zeta,\psi}_{Q,W,\mu}. \]

Furthermore, by the action of tautological bundle (Section \ref{tautologicalbundle}), we have a morphism of graded vector spaces $\HH(\BCu,\Q) \otimes \mathfrak{g}^{\BPS, T,\zeta}_{Q,W,\mu} \rightarrow \mathcal{A}^{T,\zeta,\psi}_{Q,W,\mu}$ which induces a morphism from the tensor algebra 
\begin{align}\label{tenosralgebra}
\textrm{T}_{\HH_{T}(\pt)}(\HH(\BCu,\Q) \otimes \mathfrak{g}^{\BPS,T,\zeta}_{Q,W,\mu}) \rightarrow \mathcal{A}^{T,\zeta,\psi}_{Q,W,\mu}.  
\end{align}

Let $\Sym_{\HHT}(\HH(\BCu,\Q) \otimes \mathfrak{g}^{\BPS, T, \zeta}_{Q,W,\mu}) \subset \textrm{T}(\HH(\BCu,\Q) \otimes \mathfrak{g}^{\BPS, T, \zeta}_{Q,W,\mu}) $ be the subspace of symmetric tensors, then restricting morphism (\ref{tenosralgebra}) yields a map of $\mathrm{\Lambda}^{\zeta}_{\mu} \times \mathbb{Z}$ graded $\HHT$ modules
\begin{align}\label{pbwmorphism}
\Sym_{\HHT}(\HH(\BCu,\Q) \otimes \mathfrak{g}^{\BPS,T,\zeta}_{Q,W,\mu}) \rightarrow \mathcal{A}^{T,\zeta,\psi}_{Q,W,\mu}. 
\end{align} We then have
\begin{thm}[PBW Isomorphism \cite{davison2022affine}, \cite{davison2020cohomological}] \label{PBWtheorem}
    The morphism (\ref{pbwmorphism}) is an isomorphism of $\mathrm{\Lambda}^{\zeta}_{\mu} \times \mathbb{Z}$ graded $\HHT$ modules. In fact, let $\Gr_{\mathfrak{P}}(\mathcal{A}^{T,\zeta,\psi}_{Q,W,\mu})$ be the associated graded algebra with respect to the Perverse filtration. Then there is an isomorphism of algebras \begin{equation} \label{PBW isomorphism}\Gr_{\mathfrak{P}}(\mathcal{A}^{T,\zeta,\psi}_{Q,W,\mu}) \simeq \mathbf{Sym}_{\HHT}( \HH(\BCu,\Q) \otimes \mathfrak{g}^{\BPS,T,\zeta}_{Q,W,\mu}) \end{equation}
    where for any $\HHT$ module $V$, $\mathbf{Sym}_{\HHT}(V) = T_{\HHT}(V)/\langle x \otimes y- y \otimes x=0 \mid x,y \in V \rangle $ is the quotient of the tensor algebra. 
\end{thm}

\begin{comment}
An analog of the PBW theorem exists for any Serre subcategory. We specify the case when the Serre subcategory is $\tilde{\mathcal{N}}$ and $\zeta=0,\mu=0$. In this case, it is proved in \cite{davison2020cohomological} that:

\begin{thm}\label{nilpotentpbw}
There exist a $\mathbb{N}^{Q_0} \times \mathbb{Z}$ graded nilpotent BPS Lie algebra $\mathfrak{g}^{\BPS,T,\tilde{\mathcal{N}}}_{Q,W}$ such that there is isomorphism of cohomologically graded vector spaces
\[\Sym_{\HHT}(\HH(\BCu,\Q) \otimes \mathfrak{g}^{\BPS,T,\tilde{\mathcal{N}}}_{\tilde{Q},\tilde{W}}) \rightarrow \mathcal{A}^{T,\tilde{\mathcal{N}}\psi}_{\tilde{Q},\tilde{W}}.  \]
\end{thm}
\end{comment}

\subsubsection{Relative Integrality Theorem}
There is a lift of Proposition \ref{PBWtheorem} to the Relative CoHA. Following \cite{davison2020cohomological}, we define the BPS sheaf $\mathcal{BPS}^{T,\zeta}_{Q,W,\mu}$ to be 
\begin{align} \label{BPSsheaf}
\mathcal{BPS}^{T,\zeta}_{Q,W,\mu} := {}^{\mathfrak{p'}}\!\tau^{\leq 1} ((\JH^{T,\zeta}_{\mu})_{*} (\pPhi_{\Tr(W)}))[1]. 
\end{align} So that
\begin{align} \label{bpsLiealgebrcohomology}
\mathfrak{g}^{\BPS, T,\zeta}_{Q,W,\mu} \simeq \HH(\mathfrak{M}^{T,\zeta-\sss}_{\mu}(Q),\mathcal{BPS}^{T,\zeta}_{Q,W,\mu}[-1]).
\end{align} 

By definition, we have a canonical morphism \[\mathcal{BPS}^{T,\zeta}_{Q,W,\mu}[-1] \rightarrow \mathcal{RA}^{T,\zeta}_{Q,W,\mu}.\]

Applying the action of $\mathbb{Q}[u]$ introduced in Section \ref{tautologicalbundle}, yields a morphism of complexes

\[ \HH(\BCu,\mathbb{Q}) \otimes \mathcal{BPS}^{T,\zeta}_{Q,W,\mu}[-1] \rightarrow \mathcal{RA}^{T,\zeta}_{Q,W,\mu}.\]

Since the category $\Dbc(\Per^{\prime}(\mathcal{M}^{T,\zeta}_{Q,W,\mu}(Q)))$ is symmetric monoidal over $[\pt/T]$, for any object $\mathcal{F} \in \Dbc(\Per^{\prime}(\mathcal{M}^{T,\zeta}_{Q,W,\mu}(Q)))$, we may define \[\Sym_{\oplus^{T}}(\mathcal{F}) := \bigoplus_{i \geq 0} \Sym^{i}_{\oplus^T}(\mathcal{F})\] where $\Sym^i_{\oplus^T}(\mathcal{F})$ is the $S_i$ invariant part of \[ \underbrace{\mathcal{F} \boxtimes_{\oplus^T} \cdots \boxtimes_{\oplus^T}\mathcal{F}}_{i \mathrm{ \ times }}.\] 
Thus, in exactly the same way as before, since $\mathcal{RA}^{T,\zeta}_{Q,W,\mu}$ has relative CoHA structure, the above morphism lifts to a morphism 

\begin{equation} \label{relativeintegrality}
\Sym_{\oplus^T}\left(\HH(\BCu,\mathbb{Q}) \otimes \mathcal{BPS}^{T,\zeta}_{Q,W,\mu}[-1] \right) \rightarrow \mathcal{RA}^{T,\zeta}_{Q,W,\mu} \end{equation}
of objects in $\Dbc(\Per^{\prime}(\mathcal{M}^{T,\zeta}_{Q,W,\mu}(Q)))$. We then have

\begin{thm}[\cite{davison2020cohomological}] \label{integralitytheorem}
The morphism \ref{relativeintegrality} is an isomorphism and for any $\bd \in \mathrm{\Lambda}^{\zeta}_{\mu} \backslash 0$, there is an isomorphism of perverse sheaves

\begin{align}
    \mathcal{BPS}^{T,\zeta}_{Q,W,\bd} \simeq \begin{cases}
        \pPhi_{\Tr^{\zeta}_{\bd}(W)}\mathcal{IC}_{\mathcal{M}^{T,\zeta-\sss}_{\bd}(Q)} \ &\mathrm{if} \ \mathcal{M}^{\zeta-\mathrm{st}}_{\bd}(Q) \neq 0 \\
        0 \ & \mathrm{otherwise}
    \end{cases}
\end{align}

where $\Tr^{\zeta}_{\bd}(W): \mathcal{M}^{T,\zeta-\sss}_{\bd}(Q) \rightarrow \mathbb{C}$ is the unique function through which the of morphism $\Tr^{\zeta}_{\bd}(W): \mathfrak{M}^{T,\zeta-\sss}_{\bd}(Q) \rightarrow \mathbb{C}$ factors. 
\end{thm}

We recall the Wall Crossing theorem of Toda, which relates the BPS sheaves with or without stability.

\begin{prop}[\cite{toda2022gopakumarvafa}] \label{todatheorem}
For any symmetric quiver $Q$, potential $W$, $\zeta \in \Q^{Q_0}$ a stability condition and $\bd \in \N^{Q_0}$ a dimension vector, there is a natural isomorphism \[ (q^{\zeta}_{\bd})_{*} \mathcal{BPS}^{\zeta}_{Q,W,\bd} \simeq \mathcal{BPS}_{Q,W,\bd} \] where $q^{\zeta}_{\bd}$ is the GIT affinization map \[ q^{\zeta}_{\bd} \colon \mathcal{M}_{\bd}^{\zeta-\sss}(Q) \rightarrow \mathcal{M}_{\bd}(Q). \] 
\end{prop}

Note that the Lie bracket on $\mathfrak{g}^{\BPS,\zeta}_{Q,W,\mu}$ lifts to a morphism of complexes in $\Db(\Per(\mathcal{M}^{\zeta-\sss}_{\mu}(Q)))$  
\begin{align*}[ \cdot , \cdot ]\colon \mathcal{BPS}^{\zeta}_{Q,W,\mu}[-1] \boxtimes_{\oplus} \mathcal{BPS}^{\zeta}_{Q,W,\mu}[-1] \rightarrow \mathcal{BPS}^{\zeta}_{Q,W,\mu}[-1]
\end{align*}which after pushforward along the map $q_{\bd}$ and applying Toda's theorem, gives 
\begin{prop}\label{todaLiealgebras}
    There is an isomorphism of Lie algebras \begin{align}
\mathfrak{g}^{\BPS, \zeta}_{Q,W,\mu}  \simeq \bigoplus_{\bd \in \mathrm{\Lambda}^{\zeta}_{\mu} \backslash {\mathbf{0}}} \mathfrak{g}^{\BPS}_{Q,W,\bd}. 
\end{align} 

\end{prop}

\section{Preprojective Cohomological Hall algebra}
Let $\overline{w}: \overline{Q} \rightarrow N$ be a torus weighting on the doubled quiver such that the preprojective algebra relation $\sum_{a \in Q_1} [a,a^{*}]$ is homogeneous with respect to this weighting. We then consider the $\mathrm{\Lambda}^{\zeta}_{\mu} \times \mathbb{Z}$ graded $\HHT$ module \footnote{There also exist versions with a stability condition, however, we don't need to consider that for our applications.} 
\[ \mathcal{A}^{T}_{\mathrm{\Pi}_Q} := \bigoplus_{\bd \in \mathbb{N}^{Q_0}} \HBM(\mathfrak{M}^{T}_{\bd}(\mathrm{\Pi}_Q),\Q^{\vir})\]
where we define\footnote{Note that the quotient stack $\mathfrak{M}_{\bd}(\mathrm{\Pi}_Q)$ is not smooth, so the defined shifted complex is not the same as the intersection complex. Note that $\mathbb{D}\mathbb{Q}^{\vir}_{\mathfrak{M}^{T}_{\bd}(\mathrm{\Pi}_Q)} =(\mathbb{D}\mathbb{Q}_{\mathfrak{M}^{T}_{\bd}(\mathrm{\Pi}_Q)})[2(\bd,\bd)]$.} $\Q^{\vir}|_{\mathfrak{M}^{T}_{\bd}(\mathrm{\Pi}_Q)}:=  \Q_{\mathfrak{M}^{T}_{\bd}(\mathrm{\Pi}_Q)}[-2(\bd,\bd)]$. In \cite{Yang_2018} (For any quiver) and \cite{schiffmann2012cherednik} (For the Jordan quiver), by doing virtual pullback and push forward, the authors define a $\mathbb{N}^{Q_0} \times \Z$ graded associative algebra structure  $\mathcal{A}^{T}_{\mathrm{\Pi}_Q}$ which we refer to as the preprojective cohomological Hall algebra. We can also twist the multiplication in $\mathcal{A}^{T}_{\mathrm{\Pi}_Q}$ as  we did for $\mathcal{A}_{\tilde{Q},\tilde{W}}$ in Section \ref{signtwist}. Let us denote the resulting algebra by $\mathcal{A}^{T,\chi}_{\mathrm{\Pi}_Q}$. In a similar way, in \cite{schiffmanncohagenerators}, authors define a $\mathbb{N}^{Q_0} \times \mathbb{Z}$ graded algebra structure on 
\[ \mathcal{A}^{T,\mathcal{N}}_{\mathrm{\Pi}_Q}:= \bigoplus_{\bd \in \mathbb{N}^{Q_0}}\HBM(\mathfrak{M}^{T,\mathcal{N}}_{\bd}(\mathrm{\Pi}_Q),\Q^{\vir}). \]

These algebras can be seen as a particular case of the Kontsevich-Soibelman cohomological Hall algebra. 

\subsection{Tripled quiver with canonical cubic potential} \label{dimensionreductiontripled}In Example \ref{canonicalcubicpotential}, we saw that for the tripled quiver $\tilde{Q}$ with canonical cubic potential $\tilde{W}$, we have an isomorphism of algebras $\Jac(\tilde{Q},\tilde{W}) \simeq \mathrm{\Pi}_Q[\omega]$; this in fact allows us to see the algebra $\mathcal{A}^{T}_{\mathrm{\Pi}_Q}$ as the dimension reduction of $\mathcal{A}^{T}_{\tilde{Q},\tilde{W},\mu}$ where the torus $T$ action on $\mathfrak{M}_{\bd}(\tilde{Q})$ is defined by uniquely extending the action on $\mathfrak{M}_{\bd}(\overline{Q})$ such that it leaves the potential $\tilde{W}$ invariant. Let \[ \pi: \mathfrak{M}^{T}_{\bd}(\tilde{Q}) \rightarrow \mathfrak{M}^{T}_{\bd}(\overline{Q})\] be natural forgetful morphism. We have decomposition 
\begin{equation}\mathfrak{M}^{T}_{\bd}(\tilde{Q}) = \mathfrak{M}_{\bd}(\overline{Q}) \times \mathfrak{M}_{\bd}(L) \end{equation} where $L$ is quiver with vertices $Q_0$ and loops $\omega_i$ on each vertex $i$. We may consider the $\mathbb{C}^{*}$ action, which scales the loops. Then applying dimension reduction (\ref{dimensionreduction}) gives an isomorphism of complexes
\begin{equation} \label{dimensionreductionpreprojective}\mathbb{D}\mathbb{Q}^{\vir}_{\mathfrak{M}^{T}_{\bd}(\mathrm{\Pi}_Q)} \simeq \pi_{*} \pPhi_{\Tr(\tilde{W})} \mathbb{Q}^{\vir}_{\mathfrak{M}^{T}_{\bd}(\tilde{Q})}.\end{equation} Taking cohomology, gives an isomorphism of $\mathbb{N}^{Q_0} \times \mathbb{Z}$ graded $\HHT$ modules
\begin{align}
\textrm{DR}_{\bd}: \HH(\mathfrak{M}_{\bd}^{T}(\tilde{Q}), \pPhi_{\Tr(W)}) \simeq \HH(\mathfrak{M}_{\bd}^{T}(\mathrm{\Pi}_Q),\Q^{\vir}) . 
\end{align} 
In \cite{Ren:2015zua}[Appendix A] and \cite{Yang_2019}, it shown that taking sum of maps $\tilde{\mathrm{DR}}_{\bd}:= \mathrm{DR}_{\bd} \cdot \binom{\bd}{2}$ over all $\bd \in \mathrm{\Lambda}^{\zeta}_{\mu}$, where $\binom{\bd}{2} := \prod_{i \in Q_0} \prod \binom{\bd_i}{2}$ gives an isomorphism of $\mathrm{\Lambda}^{\zeta}_{\mu} \times \Z$ graded algebras 
\[ \tilde{\mathrm{DR}} \colon \mathcal{A}^{T}_{\tilde{Q},\tilde{W}} \simeq \mathcal{A}^{T,\zeta}_{\mathrm{\Pi}_Q}. \] A similar analysis for the case of nilpotent representations, as done in \cite{davison2022affine}, one obtains an isomorphism of algebras
\[ \tilde{\mathrm{DR}} \colon \mathcal{A}^{T,\tilde{\mathcal{N}}}_{\tilde{Q},\tilde{W}} \simeq \mathcal{A}^{T,\mathcal{N}}_{\mathrm{\Pi}_Q}. \]

\subsection{Support lemma} \label{LessPerverse}

In this section, we restrict ourselves to the non-equivariant setting. We remark that these results also hold when any torus $T$. We saw that morphism \ref{dimensionreductionpreprojective} is an isomorphism of complexes. Combined with the Integrality Theorem \ref{integralitytheorem}, this gives a integrality theorem for the pushforward $\JH^{\overline{Q}}_{*} \mathbb{D}\mathbb{Q}^{\vir}_{\mathfrak{M}^{T}(\mathrm{\Pi}_Q)}$, where we treat $\mathbb{D}\mathbb{Q}^{\vir}_{\mathfrak{M}^{T}(\mathrm{\Pi}_Q)}$ as a complex on $\mathfrak{M}(\overline{Q})$. We have a commutative diagram
\[\begin{tikzcd}[ampersand replacement=\&]
	{\mathfrak{M}(\tilde{Q})} \& {\mathfrak{M}(\overline{Q})} \\
	{\mathcal{M}(\tilde{Q})} \& {\mathcal{M}(\overline{Q})}
	\arrow["\pi", from=1-1, to=1-2]
	\arrow["{\mathrm{JH}^{\tilde{Q}}}"', from=1-1, to=2-1]
	\arrow["{\JH^{\overline{Q}}}", from=1-2, to=2-2]
	\arrow["{\pi^{\prime}}"', from=2-1, to=2-2]
\end{tikzcd}\]

Thus, we have an isomorphism of complexes
\begin{align*}
    \JH^{\overline{Q}}_{*}\mathbb{D}\mathbb{Q}^{\vir}_{\mathfrak{M}(\mathrm{\Pi}_Q)} \simeq \pi^{\prime}_{*}\JH^{\tilde{Q}}_{*}\mathbb{Q}^{\vir}_{\mathfrak{M}_{\bd}(\overline{Q})} \simeq \Sym_{\oplus}\left(q_{*}\mathcal{BPS}_{\tilde{Q},\tilde{W}}[-1] \otimes \HH(\BCu,\mathbb{Q}) \right) 
\end{align*}

It turns out that the complex $q_{*}\mathcal{BPS}_{\tilde{Q},\tilde{W}}[-1]$ itself becomes a perverse sheaf on $\mathcal{M}(\overline{Q})$. This is explained by the support lemma. Let $l: \mathfrak{M}_{\bd}(\overline{Q}) \times \mathbb{A}^1 \rightarrow \mathfrak{M}_{\bd}(\tilde{Q})$ be an embedding of stacks induced by the closed inclusion $\mathrm{Rep}_{\bd}(\overline{Q}) \times \mathbb{A}^1 \rightarrow \Rep_{\bd}(\tilde{Q})$ defined by sending a representation $\rho$ of $\overline{Q}$ and a scalar $\lambda$ to a representation of $\tilde{Q}$ whose underlying $\overline{Q}$ representation is $\rho$ and the loops $\omega_i$ act by multiplication with the scalar $\lambda$. Then it is proved in \cite{davison2022integrality}[Lemma 4.1] that 

\begin{thm}[Support Lemma] \label{supportlemma}
    There exist a unique complex \[\mathcal{BPS}_{\mathrm{\Pi}_Q} \in \Dbc(\Per(\mathcal{M}(\overline{Q}))) \] such that \[ \mathcal{BPS}_{\tilde{Q},\tilde{W}} \simeq l_{*}(\mathcal{BPS}_{\mathrm{\Pi}_Q} \otimes \mathbb{Q}_{\mathbb{A}^1}[1]).\]
\end{thm}
Thus, we have an isomorphism of complexes
\begin{align} \label{splitting2}
\JH^{\overline{Q}}_{*}\mathbb{D}\mathbb{Q}^{\vir}_{\mathfrak{M}(\mathrm{\Pi}_Q)} \simeq \Sym\left(\mathcal{BPS}_{\mathrm{\Pi}_Q} \otimes \HH(\BCu,\mathbb{Q}) \right) 
\end{align}

\subsection{Kac polynomials and Graded Dimension of BPS Lie algebra }\label{Kacspolynomial}
Given any quiver $Q$, for any finite field $\mathbb{F}_{q}$ and any dimension vector $\bd \in \N^{Q_0}$. A representation is said to be \textit{indecomposable} if it cannot be non-trivially written as a direct sum of two representations of $Q$. A representation $\rho$ is said to be \textit{absolutely indecomposable} if it remains indecomposable after tensoring with the closure of $\mathbb{F}_q$, i.e., $\rho \otimes_{\mathbb{F}_q} \overline{\mathbb{F}}_q$ is indecomposable. 

Let $a_{Q,\bd}(q)$ be the number of isomorphism classes of absolutely indecomposable $\bd$ dimensional representations of $Q$ over $\mathbb{F}_q$. There is an intricate relationship between the Kac polynomials and the Lie algebra associated with the quiver, which can be considered as the generalization of Gabriel's Theorem, which classifies finite type quivers. We recall

\begin{defn}[Associated Kac-Moody Lie algebra] \label{negativehalfLiealgebra}
For any quiver $Q$, the real subquiver $Q'$ is defined to be the subquiver of $Q$ containing those vertices of $Q$ that do not support any edge loops, along with all arrows between these vertices. 
Let $\mathfrak{g}_{Q'}$ be the associated Kac-Moody Lie algebra. It is defined as a Lie algebra generated by symbols $e_i.f_i,h_i; i \in Q^{\prime}_0$ with the relations 
\begin{align*}
&[h_i,e_j]=a_{ij}e_j\\
&[h_i,f_j]=-a_{ij}f_j\\
&[e_i,f_j]= \delta_{ij} h_i \\
&[e_i, \hyphen ]^{a_{ij}+1}(e_j)=0; \\
&[f_i, \hyphen ]^{a_{ij}+1}(f_j)=0; 
\end{align*}
where  $a_{ij}$ is the number of edges joining $i$ and $j$ in the graph obtained from $Q^{\prime}$. We identify $\mathbb{Z}^{Q^{\prime}_0}$ along with the symmetrized Euler form $\chi_{Q}(\bd,\be)+\chi_{Q}(\be,\bd)$ with the root lattice $\mathrm{\Delta}_{\mathfrak{g}_{Q^{\prime}}}$ of the Lie algebra $\mathfrak{g}_{Q^{\prime}}$ along with the Cartan pairing via the morphism $\bd \mapsto \sum \bd_i \alpha_i$ where $\alpha_i$ are the simple roots of $\mathfrak{g}_{Q^{\prime}}$. Let $\mathfrak{n}^{+}_{Q^{\prime}}$ be the Lie subalgebra generated by $e_i$ and $\mathfrak{n}^{-}_{Q^{\prime}}$ denote the Lie subalgebra generated by $f_i$. 
\end{defn}

Then it is a Theorem of Kac that:

\begin{thm}[\cite{Kac1980}]
$a_{Q,\bd}(q)$ is a monic polynomial in $q$ of degree $1-(\bd,\bd)$ with coefficients in $\mathbb{Z}$. It doesn't depend on the orientation of $Q$. For a quiver without loops, $a_{Q,\bd}(q)$ is non-zero if and only if $\bd$ is a positive root of the associated Lie algebra $\mathfrak{g}_Q$. 
\end{thm}

\begin{example}[Finite type]
If $Q$ if finite type $ADE$ quiver then by theorem of Gabriel,
\begin{align}
a_{Q^{ADE},\bd}(q) = \begin{cases}
    1, \mathrm{when \ } (\bd,\bd)=1 (\bd \ \mathrm{ \ is \ a \ positive \ root}) \\
    0, \mathrm{otherwise}
\end{cases} 
\end{align}
\end{example}
\begin{comment}
\begin{example}[Jordan Quiver] \label{kacpolynomialJordan}
If $Q$ is the Jordan quiver, then the path algebra is just the polynomial ring $\mathbb{F}_q[x]$. By classification of modules over PID, absolutely indecomposable representation of dimension $n$ is isomorphic to $\overline{\mathbb{F}}_q[x]/(x-a)^n$ for some $n$. Thus \[ a_{Q_{\Jor},d}(q) = q, \forall d \in \mathbb{Z}_{>0}.\]
\end{example}
\end{comment}
More generally, for an arbitrary quiver $Q$, there exists a formula to compute the Kac Polynomial due to Hua \cite{huakac}. Kac conjectured that the coefficients of $a_{Q,\bd}(Q)$ are always positive. For indivisible dimensional vectors, this was proved in 2004 by Crawley-Boevey and Van den Bergh (\cite{crawleyvanden}). Then, using DT invariants, Mozgovoy proved this for a quiver without loops in \cite{mozgovoy}. Finally, this conjecture was settled by Hausel, Letellier,
and Rodriguez-Villegas in \cite{hauselKac} for arbitrary quivers. There is also a proof of this conjecture using BPS Lie algebras. In fact, in \cite{davison2023purity}, Davison proved this conjecture by realizing the coefficients of Kac polynomials as the graded dimensions of BPS Lie algebras $\mathfrak{g}^{\BPS}_{\tilde{Q},\tilde{W}}$. We have 

\begin{thm}[\cite{davison2022integrality}]
For any quiver $Q$, 
\begin{equation} \label{kacpolynomialandrelationwithbpsliealgebra}
 \sum_{i } (-1)^{i} \dim(\mathfrak{g}^{\BPS,i}_{\tilde{Q},\tilde{W},\bd})q^{i/2} = a_{Q,\bd}(q^{-1}).
\end{equation} 
\end{thm}

\begin{comment}
\begin{example}[BPS Lie algebra for the Jordan quiver]\label{bpsliealgebrajordan}
Let $Q_{\Jor}$ be the Jordan quiver. Then by Equation \ref{kacpolynomialandrelationwithbpsliealgebra} and calculation of Kac polynomial for Jordan quiver in Example \ref{kacpolynomialJordan}, it follows that for all $n>0$, 
\[\mathfrak{g}^{\BPS}_{\tilde{Q_{\Jor}},\tilde{W_{\Jor},n}} = \mathbb{Q}[2].\]
Since all of these spaces are $1$ dimensional, for all $n$, choose any non-zero element $\alpha_n$ in $\mathfrak{g}^{\BPS}_{\tilde{Q_{\Jor}},\tilde{W_{\Jor},n}}$. Since the BPS Lie algebra is graded, it follows that the commutator \[ [\alpha_n,\alpha_m]=0\] since the commutator must be of cohomological degree $-4$, but there isn't any element of cohomological degree $-4$ in dimension $m+n$. Thus \[\mathfrak{g}^{\BPS}_{\tilde{Q_{\Jor}},\tilde{W_{\Jor}}} \simeq \bigoplus_{n \geq 1} \mathbb{Q}[2].\]
Although the Lie bracket is trivial, one should think of this as the positive half of the infinite-dimensional Heisenberg Lie algebra, as we shall see that this Lie algebra acts on the cohomology of the Hilbert scheme of points on a plane and its action lifts to the action of the full infinite-dimensional Heisenberg Lie algebra. 
\end{example}
\end{comment}

\begin{example}
We can also calculate the size of the Nilpotent cohomological Hall algebra. Let $i: \mathcal{M}^{\mathcal{N_{\overline{Q}}}}_{\bd} (\overline{Q}) \rightarrow \mathcal{M}_{\bd}(\overline{Q})$ be the inclusion of the coarse moduli space of nilpotent $\overline{Q}$ representations, clearly $\mathcal{M}^{\mathcal{N_{\overline{Q}}}}_{\bd} (\overline{Q}) $ is projective. We have a Cartesian diagram

    \[\begin{tikzcd}[ampersand replacement=\&]
	{\mathfrak{M}^{\mathcal{N}_{\overline{Q}}}_{\bd} (\overline{Q}) } \& {\mathfrak{M}_{\bd}(\overline{Q})} \\
	{\mathcal{M}^{\mathcal{N}_{\overline{Q}}}_{\bd} (\tilde{Q}) } \& {\mathcal{M}_{\bd}(\tilde{Q})}
	\arrow["{i^{\prime}}"', from=1-1, to=1-2]
	\arrow["{\mathrm{JH}}", from=1-1, to=2-1]
	\arrow["\lrcorner"{anchor=center, pos=0.125}, draw=none, from=1-1, to=2-2]
	\arrow["{\mathrm{JH}}"', from=1-2, to=2-2]
	\arrow["i", from=2-1, to=2-2]
\end{tikzcd}\]
and so we have 
\begin{align*}
\JH_{*}\mathbb{D}\mathbb{Q}^{\vir}_{\mathfrak{M}^{\mathcal{N}}(\mathrm{\Pi}_Q)} \simeq i^{!}\JH^{\overline{Q}}_{*}\mathbb{D}\mathbb{Q}^{\vir}_{\mathfrak{M}(\mathrm{\Pi}_Q)} \simeq \ \Sym\left(i^{!}\mathcal{BPS}_{\mathrm{\Pi}_Q} \otimes \HH(\BCu,\mathbb{Q}) \right) 
\end{align*}
and thus there is an isomorphism of graded vector spaces
\[ \mathcal{A}^{\tilde{\mathcal{N}}}_{\tilde{Q},\tilde{W}} \simeq \Sym\left(\HH(\mathcal{M}_{\bd}^{\mathcal{N}_{\overline{Q}}}(\overline{Q}), i^{!}\mathcal{BPS}_{\mathrm{\Pi}_Q}) \otimes \HH(\BCu,\mathbb{Q}) \right) \]
But since Verdier duality functor $\mathbb{D}$ commutes with the vanishing cycles and intersection cohomology sheaf, it follows that
\begin{align*}
    \HH(\mathcal{M}_{\bd}^{\mathcal{N}_{\overline{Q}}}(\overline{Q}), i^{!}\mathcal{BPS}_{\mathrm{\Pi}_Q}) \simeq (\HH_{c}(\mathcal{M}_{\bd}^{\mathcal{N}_{\overline{Q}}}(\overline{Q}), i^{*}\mathcal{BPS}_{\mathrm{\Pi}_Q}))^{\vee}
\end{align*}
Note that $(\mathcal{M}_{\bd}^{\mathcal{N}_{\overline{Q}}}(\overline{Q})$ is homotopic to $(\mathcal{M}_{\bd}(\overline{Q})$ via torus $T$ action which scales the arrows, for which the BPS sheaf is equivariant. Thus we have
\begin{align*}
\HH_{c}(\mathcal{M}_{\bd}^{\mathcal{N}_{\overline{Q}}}(\overline{Q}), i^{*}\mathcal{BPS}_{\mathrm{\Pi}_Q}) &= \HH(\mathcal{M}_{\bd}^{\mathcal{N}_{\overline{Q}}}(\overline{Q}), i^{*}\mathcal{BPS}_{\mathrm{\Pi}_Q}) \simeq \HH(\mathcal{M}_{\bd}(\overline{Q}), i^{*}\mathcal{BPS}_{\mathrm{\Pi}_Q}) \simeq \mathfrak{g}^{\BPS}_{\tilde{Q},\tilde{W}}.
\end{align*}
Thus we have isomorphism of graded vector spaaces \[ \HH(\mathcal{M}_{\bd}^{\mathcal{N}_{\overline{Q}}}(\overline{Q}), i^{!}\mathcal{BPS}_{\mathrm{\Pi}_Q})  \simeq (\mathfrak{g}^{\BPS}_{\tilde{Q},\tilde{W}})^{\vee} \] and so we recover \begin{equation} \label{nilpotentkacpolynomialandrelationwithbpsliealgebra}
 \sum_{i } (-1)^{i} \dim(\HH^{i}(\mathcal{M}_{\bd}^{\mathcal{N}_{\overline{Q}}}(\overline{Q}), i^{!}\mathcal{BPS}_{\mathrm{\Pi}_Q})  )q^{i/2} = a_{Q,\bd}(q).
\end{equation} 
\end{example}

\begin{comment}
\begin{example}
In Example \ref{nilpotentcommuting}, we saw that $\mathfrak{M}^{\mathcal{N}}_{d}(\mathrm{\Pi}_{Q^{\Jor}})$ parametrizes commuting nilpotent matrices. Now, by the above two examples, and dimension reduction in Section \ref{dimensionreductiontripled}, we have an isomorphism of vector spaces

\[ \bigoplus_{d \geq 0} \HBM(\mathfrak{M}^{\mathcal{N}}_{d}(\mathrm{\Pi}_{Q^{\Jor}}),\mathbb{Q}) \simeq \Sym\left(\bigoplus_{n \geq 1} \mathbb{Q}[u][-2] \right)\]
Note that for $n$, smallest cohomological degree in $\mathbb{Q}[u][-2]$ is $2$. We conclude that the lowest cohomological degree of $\HBM(\mathfrak{M}^{\mathcal{N}}_{d}(\mathrm{\Pi}_{Q^{\Jor}}),\mathbb{Q}))$ is $2$ and we have \[ \dim(\mathrm{H}^{\mathrm{BM},2}(\mathfrak{M}^{\mathcal{N}}_{d}(\mathrm{\Pi}_{Q^{\Jor}}),\mathbb{Q}))=1.\]
Thus, the top dimensional homology degree of the space $\mathcal{N}_2(\mathbb{C}^n)$ of commuting nilpotent matrices is $2n^2-2$ and of dimension $1$, implying that $\mathcal{N}_2(\mathbb{C}^n)$ is irreducible and of dimension $n^2-1$, recovering the theorem of Branovsky \cite{baranovsky}. 
\end{example}

\end{comment}

The above relations allow us to conclude non-trivial statements about the BPS Lie algebra. For any quiver $Q$, since $\Tr(\tilde{W})=0$ for $\bd = \delta_i$, we have an isomorphism of vector spaces $\mathcal{A}^{\chi}_{\tilde{Q},\tilde{W},\delta_i} \simeq \C[x_{i,1}]$. We define $\alpha_i \subset \mathcal{A}_{\tilde{Q},\tilde{W}}^{\chi}$ to be the inverse of $1 \in \C[x_{i,1}]$. Clearly $\alpha_i \in \mathfrak{g}^{\BPS,0}_{\tilde{Q},\tilde{W},\delta_i}$ when the vertex $i$ doesn't support any edge loops. We claim that for any $i,j$, $a_{Q,\delta_j+(1+a_{ij})\delta_i}(q)=0$, this is because Kac's polynomials are orientation independent. So if we choose the orientation such that all the arrows are from vertex $i$ to $j$, then there is no indecomposable representation of dimension $\delta_j + (1+a_{ij}) \delta_i$. We thus have a map of Lie algebras 
\begin{align}\label{mapfromnegativekac}
\mathfrak{n}_{Q'}^{+} \rightarrow \mathfrak{g}^{\BPS,0}_{\tilde{Q},\tilde{W}}
\end{align} given by $e_i \mapsto \alpha_i$ where $\mathfrak{n}_{Q}^{+}$. The equality mentioned above suffices to show that this morphism is an injection. This is done by comparing the action on the Nakajima quiver varieties (See Section \ref{cohaactionnakajimainfinity}) with Nakajima's action, which is known to be faithful. We have

\begin{prop}[\cite{davison2022bps}, Theorem 6.6] \label{zerocohomologybps}
The map (\ref{mapfromnegativekac}) gives an isomorphism of Lie algebras\footnote{In \cite{davison2022bps}, one writes the above isomorphism for the negative half. But to our perspective, it is more natural to identify it to the positive half as the cohomological Hall algebra acts by \textit{creation} operators on the cohomology of Nakajima quiver varieties.} \[ \mathfrak{n}^{+}_{Q'} \simeq \mathfrak{g}^{\BPS,0}_{\tilde{Q},\tilde{W}}. \]

\end{prop}

\subsection{Flat deformation}
In the case when the natural mixed Hodge structure on $\mathcal{A}_{Q,W}$ is pure, the deformed cohomological Hall algebra $\mathcal{A}^{T,\psi}_{Q,W}$ forms a flat deformation of $\mathcal{A}^{\psi}_{Q,W}$. 

\begin{prop} \label{flatdeformationofcoha}
Let $\mathbf{w} \colon Q_1 \rightarrow \Z^{s}$ be a $W$ invariant weighting function and let $v\colon \mathbb{Z}^{s} \rightarrow \mathbb{Z}^{s'}$ be a surjective morphism inducing an inclusion of tori $T' \rightarrow T$ where $T'$ is the torus associated to weighting $v \mathbf{w}$. Choose a splitting of $v: \Z^{s} \rightarrow \Z^{s'}$ inducing an isomorphism $
\HH_{T}(\pt)  \simeq \HH_{T'}(\pt) \otimes \HH_{T''}(\pt). $
Assume the mixed Hodge structure on $\mathcal{A}_{Q,W}$ is pure. Then we have an isomorphism of $\N^{Q_0} \times \Z$ graded $\HHT$ modules
\begin{align*}
\mathcal{A}^{T}_{Q,W} \simeq & \mathcal{A}^{T'}_{Q,W} \otimes \HH_{T''}(\pt) \\
\mathfrak{g}^{\BPS,T}_{Q,W} \simeq & \mathfrak{g}^{\BPS, T'}_{Q,W} \otimes \HH_{T''}(\pt),
\end{align*} while there is an isomorphism of $\N^{Q_0} \times \Z$ graded $\HH_{T^{\prime}}(\pt)$ algebras \[ \mathcal{A}^{T,\psi}_{Q,W} \otimes_{\HH_{T}(\pt)} \HH_{T'}(\pt) \simeq \mathcal{A}^{T',\psi}_{Q,W} \] and $\N^{Q_0} \times \Z$ graded $\HH_{T^{\prime}}(\pt)$ Lie algebras
\[ \mathfrak{g}^{\BPS, T}_{Q,W} \otimes_{\HH_{T}(\pt)} \HH_{T'}(\pt) \simeq \mathfrak{g}^{\BPS,T'}_{Q,W}. \]

\end{prop}

\begin{proof}
This statement is proved for $\tilde{Q},\tilde{W}$ in \cite{davison2022integrality}[Theorem 9.6]. The proof works as it is in this case, as it only relies on the purity of $\mathcal{A}_{Q,W}$. The statement $\mathfrak{g}^{\BPS, T}_{Q,W} \simeq  \mathfrak{g}^{\BPS, T'}_{Q,W} \otimes \HH_{T''}(\pt)$ as vector spaces follows from the PBW theorem (Proposition \ref{PBWtheorem}) while the isomorphism of algebras $\mathcal{A}^{T,\psi}_{Q,W} \otimes_{\HH_{T}(\pt)} \HH_{T'}(\pt) \simeq \mathcal{A}^{T',\psi}_{Q,W}$ implies in particular that there is map of $\N^{Q_0} \times \Z$ graded Lie algebras
$\mathfrak{g}^{\BPS, T}_{Q,W} \otimes_{\HH_{T}(\pt)} \HH_{T'}(\pt) \simeq \mathfrak{g}^{\BPS, T'}_{Q,W}$, which then becomes an isomorphism by a size argument. 
\end{proof}

We will, in particular, be applying the above proposition in the setting of the tripled quiver $\tilde{Q}$ with tripled canonical potential $\tilde{W}$, which is shown to be of pure mixed Hodge structure in \cite{davison2022integrality}[Theorem A]. In particular, this implies that 
\begin{thm} \label{free}
The cohomological Hall algebera $\mathcal{A}^{T}_{\tilde{Q},\tilde{W}}$ is free module over $\HHT$. 
\end{thm}

A similar theorem holds for the case of nilpotent CoHA. It is proved in \cite{davison2022integrality} and \cite{schiffmann2017cohomological} that

\begin{thm} \label{nilpotentfree}
There is an isomorphism of $\mathbb{N}^{Q_0} \times \mathbb{Z}$ graded $\HHT$ modules 
$\mathcal{A}^{T,\tilde{\mathcal{N}}}_{\tilde{Q},\tilde{W}} \simeq \mathcal{A}^{\tilde{\mathcal{N}}}_{\tilde{Q},\tilde{W}} \otimes \HHT $ i.e.  The nilpotent cohomological Hall algebra $\mathcal{A}^{T,\tilde{\mathcal{N}}}_{\tilde{Q},\tilde{W}}$ is free module over $\HHT$.
\end{thm}

\section{BPS Lie algebra of Cyclic quivers} \label{chapterbpscyclic}

Let $Q^{K}= \tilde{A_{K}}$ be the cyclic quiver of length $K+1$, i.e. we have $K+1$ vertices $0,1,\cdots,K$ with arrows $i \rightarrow i+1$ for $i=0,\cdots,K$ and an arrow $K \rightarrow 0$. We then consider the tripled cyclic quiver $\tilde{Q^{K}}$ with the canonical potential $\tilde{W^{K}}$ as in Example \ref{canonicalcubicpotential} (See Figure \ref{fig:triplecyclic quivers}). Let $T$ be any torus action which leaves the potential $\tilde{W_{K}}$ invariant. We consider two distinct choices for the torus $T$. 
\begin{figure}[h!]
    \centering
    \[\begin{tikzcd}[ampersand replacement=\&]
	\& 0 \\
	3 \&\& 1 \\
	\& 2
	\arrow["{\omega_0}"{description}, from=1-2, to=1-2, loop, in=55, out=125, distance=10mm]
	\arrow["{a_{30}^{*}}"{description}, curve={height=-6pt}, from=1-2, to=2-1]
	\arrow["{a_{01}}"{description}, curve={height=-6pt}, from=1-2, to=2-3]
	\arrow["{a_{30}}"{description}, curve={height=-6pt}, from=2-1, to=1-2]
	\arrow["{\omega_3}"{description}, from=2-1, to=2-1, loop, in=145, out=215, distance=10mm]
	\arrow["{a_{23}^{*}}"{description}, curve={height=-6pt}, from=2-1, to=3-2]
	\arrow["{a_{01}^{*}}"{description}, curve={height=-6pt}, from=2-3, to=1-2]
	\arrow["{\omega_1}"{description}, from=2-3, to=2-3, loop, in=325, out=35, distance=10mm]
	\arrow["{a_{12}}"{description}, curve={height=-6pt}, from=2-3, to=3-2]
	\arrow["{a_{23}}"{description}, curve={height=-6pt}, from=3-2, to=2-1]
	\arrow["{a_{12}^{*}}"{description}, curve={height=-6pt}, from=3-2, to=2-3]
	\arrow["{\omega_2}"{description}, from=3-2, to=3-2, loop, in=235, out=305, distance=10mm]
\end{tikzcd}\]
    \caption{$\tilde{Q^{3}},\tilde{W^{3}}= \sum_{i \in \mathbb{Z}/4\mathbb{Z}} \omega_{i}(a_{i,i+1}^{*}a_{i,i+1} - a_{i,i-1}a_{i,i-1}^{*})$}
    \label{fig:triplecyclic quivers}
\end{figure}
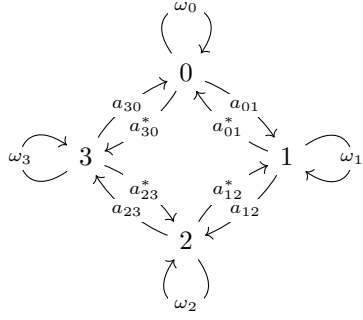

\begin{itemize}
    \item Action of $\C^{*}$ given by weighting function $\textbf{w}: \QTC{K} \rightarrow \mathbb{Z}$, such that $\textbf{w}(a)=1$, $\textbf{w}(a^{*}) = -1$ and $\textbf{w}(\omega_i)=0$ for all vertices $i \in Q^K_0$.  
    \item Action of $\T = \C^{*} \times \C^{*}$ given by weighting function $\textbf{w}: \QTC{K} \rightarrow \mathbb{Z}^2$ where $\textbf{w}(a)=(1,0)$, $\textbf{w}(a^{*})=(0,1)$ and $\textbf{w}(\omega_i)=(-1,-1)$ for all vertices $i \in Q^K_0$. 
\end{itemize} 
We denote by $\mathcal{A}^{\chi}_{\QTC{K},\WTC{K}}$, the cohomological Hall algebra of the tripled cyclic quiver with canonical potential and its $\C^{*}$ and $\T$ deformed versions by  $\mathcal{A}^{\C^{*},\chi}_{\QTC{K},\WTC{K}}$ and $\mathcal{A}^{\T^{*},\chi}_{\QTC{K},\WTC{K}}$ respectively. 

\subsubsection{(Deformed) BPS Lie algebra of Cyclic quiver} \label{bpsLiealgebra}

In this section, we shall calculate the BPS Lie algebra for tripled cyclic quivers with a canonical cubic potential. Following the work of Davison, Hennecart, and Schelegel Mejia, these Lie algebras are now known for arbitrary tripled quivers with canonical potential (See Section \ref{generalizedBPS}).

In the case of cyclic quivers, we can calculate these Lie algebras more directly, as we explain next. We first calculate the Kac polynomial for cyclic quivers, which tells us about the graded dimensions of BPS Lie algebras. One can check that (See \cite{mythesis}[Proposition 5.0.1]) 

\begin{prop} \label{kacpolynomialcyclicquiver}
The Kac polynomial of the cyclic quiver $Q^{K}$ is given by
\begin{equation}  a_{Q,\bd}(q) = \begin{cases}
    & q+ K \textrm{ when } \bd \in \mathbb{Z}_{> 0} \cdot \delta \textrm{ (positive imaginary root)} \\ & 1 \textrm{ when } \bd \in \mathrm{\Delta}^{+}_{\textrm{re}} \textrm{ (positive real root)}
\end{cases}
\end{equation}
where $\mathrm{\Delta}^{+}_{\textrm{re}}$ and $\delta= (1,\cdots,1)$ are respectively, the real positive roots and the primitive imaginary root of the affine Lie algebra $\tilde{\mathfrak{sl}}_{K+1}$. 
\end{prop}

\begin{prop}\label{bpsLiealgebracyclic}
There is an isomorphism of Lie algebras 
    \begin{align*}  
    \mathfrak{g}^{\BPS}_{\QTC{K},\WTC{K}} & \simeq \mathfrak{n}_{Q^{K}}^{+} \oplus s \Q[s]  \\ 
    \mathfrak{g}^{\BPS,\C^*}_{\QTC{K},\WTC{K}} & \simeq (\mathfrak{n}_{Q^{K}}^{+} \oplus s \Q[s]) \otimes \HH_{\C^{*}}(\pt) \\
 \mathfrak{g}^{\BPS, \T}_{\QTC{K},\WTC{K}} & \simeq (\mathfrak{n}_{Q^{K}}^{+} \oplus s \Q[s]) \otimes \HH_{\T}(\pt)
    \end{align*}
    where $\mathfrak{n}^{+}_{Q^{K}}\oplus s\Q[s]$ is a trivial central extension of $\fn_{Q^{K}}$ and $\deg(s^n) = n \cdot \delta$.
    
\end{prop}

\begin{proof}
From Proposition \ref{zerocohomologybps}, we have an isomorphism of Lie algebras \[ \mathfrak{n}_{Q^K}^{+} \simeq \HH^{0}(\mathfrak{g}^{\BPS}_{\QTC{K},\WTC{K}}).\]
In particular this also shows that $\HH^{0}(\mathfrak{g}^{\BPS}_{\QTC{K},\WTC{K}})$ is generated as a Lie algebra by $\alpha_i$, where we recall that $\alpha_{i} \in \mathfrak{g}^{\BPS}_{\QTC{K},\WTC{K},\delta_i}$. The Kac polynomial of the cyclic quiver (\ref{kacpolynomialcyclicquiver}) and it's relation with BPS Lie algebra (\ref{kacpolynomialandrelationwithbpsliealgebra}) also tells us that for every $n \in \Z_{>0}$, there exist $\gamma_{n \cdot \delta} \in \HH^{-2}(\mathfrak{g}^{\BPS}_{\QTC{K},\WTC{K},n \cdot \delta})$. By equation (\ref{kacpolynomialcyclicquiver}), there is no element of cohomological degree $-4$ in dimension $\mathbb{Z}_{> 0} \delta$. Thus it follows that $[\gamma_{n \cdot \delta}, \gamma_{m \cdot \delta}]=0$. Again, by equation \ref{kacpolynomialcyclicquiver}, there is no $-2$ cohomological degree term in dimension $\delta+\delta_i$ for any $i$, thus it follows that $[\gamma_{n \cdot \delta}, \alpha_i]=0$. But then $\alpha_i$ generates $\HH^{0}(\mathfrak{g}^{\BPS}_{\QTC{K},\WTC{K}})$ and all of $\HH^{-2}(\mathfrak{g}^{\BPS}_{\QTC{K},\WTC{K}})$ is central and so we are done. This argument does not work for deformed versions. However in \cite{davison2024bps}[Corollary 11.9], it is shown that for any quiver $Q$ and for any torus action such that the canonical cubic potential $\tilde{W}$ remains invariant,  $\mathfrak{g}^{\BPS,T}_{\tilde{Q},\tilde{W}}$ is a trivial deformation of $\mathfrak{g}^{\BPS}_{\tilde{Q},\tilde{W}}$, i.e. there is an isomorphism of Lie algebras 
\[ \mathfrak{g}^{\BPS,T}_{\tilde{Q},\tilde{W}} \simeq \mathfrak{g}^{\BPS}_{\tilde{Q},\tilde{W}} \otimes \HH_{T}(\pt).\] Hence, the conclusion follows.
\end{proof}
The BPS Lie algebras $\mathfrak{g}^{\BPS}_{\tilde{Q},\tilde{W}}$ for tripled quiver $\tilde{Q}$ and potential $\tilde{W}$ are now completely known in terms of generators and relations. This also gives a candidate for a double of the BPS Lie algebra. We explain the construction in general.
\subsection{Generalized Kac-Moody Lie algebra} \label{generalizedBPS}

In this section, we will recall the construction of the Generalized Kac-Moody Lie algebra for a monoid with a weight function defined in \cite{davison2024bps}, based on the earlier construction of Borcherds-Bozec Kac-Moody Lie algebras in \cite{bozec}. 

Let $(M,+)$ be a monoid. Let $(\hyphen,\hyphen): M \times M \rightarrow \mathbb{Z}$ be the symmetric bilinear form. Then we define
\begin{defn}[Primitive Positive Roots]
 The set of primitive positive roots $\sum_{M,(\hyphen,\hyphen)}$ is defined to be the set of $m \in M, m \neq 0$ such that $2-(m,m) \geq 0$ and $2-(m,m) > \sum_{j=1}^{s} (2-(m_j,m_j))$ for any non-trivial decomposition $m= \sum_{j=1}^{s} m_j$ with $m_j \in M, m_j \neq 0$.
\end{defn}

\begin{defn}[Simple Positive Roots]
The set of simple positive roots $\Phi^{+}_{M,(\hyphen,\hyphen)}$ is defined to be the union of the set of primitive positive roots $\sum_{M,(\hyphen,\hyphen)}$ and $lm, l \geq 1$ where $m \in \sum_{M,(\hyphen,\hyphen)}$ and $(m,m) =0$. 
\end{defn}

Then define subsets
\begin{align*}
\Phi^{+,\mathrm{real}}_{M,(\hyphen,\hyphen)} &:= \{ m \in \Phi^{+}_{M,(\hyphen,\hyphen)} \mid (m,m) = 2 \} \mathrm{\  (\ Real \ positive \ roots \ ) } \\
\Phi^{+,\mathrm{im}}_{M,(\hyphen,\hyphen)} &:= \{ m \in \Phi^{+}_{M,(\hyphen,\hyphen)} \mid (m,m) \leq 0 \} \mathrm{\  (\ Imaginary \ positive \ roots \ ) } \\
\Phi^{+,\mathrm{iso}}_{M,(\hyphen,\hyphen)} &:= \{ m \in \Phi^{+}_{M,(\hyphen,\hyphen)} \mid (m,m) = 0 \} \mathrm{\  (\ Isotopic \ positive \ roots \ )} \\
\Phi^{+,\mathrm{hyp}}_{M,(\hyphen,\hyphen)} &:= \{ m \in \Phi^{+}_{M,(\hyphen,\hyphen)} \mid (m,m) < 0 \} \mathrm{\  (\ Hyperbolic \ positive \ roots \ ) }
\end{align*}

\begin{defn}[Cartan Matrix]
Let \[ A_{M,(\hyphen,\hyphen)} = (a_{m,n} 
= (m,n))_{m,n \in \Phi^{+}_{M,(\hyphen,\hyphen)}}\] to be the Cartan Matrix for $M, (\hyphen,\hyphen)$. 
\end{defn}

The final piece of the data is the weight function. 

\begin{defn}[Weight Function]
Let \[ P \colon \Phi^{+}_{M,(\hyphen,\hyphen)} \rightarrow \mathbb{N}[t^{\pm 1/2}]\] given by $P \mapsto P_m(t^{1/2})$ be the weight function. We assume that $P_{m}(t^{1/2}) = 1$ whenever $(m,m)=2$. 
\end{defn}

 Let $P_m(t^{1/2}) = \sum_{j \in \mathbb{Z}} p_{m,j} t^{j/2}$ with $p_{m,j} \in \mathbb{N}$. 

\begin{defn}[Generalized Kac Moody Lie algebra for Monoid and Weight Function]\label{generalizedKACMOODY}
  Given a monoid $M$ with symmetric bilinar form $(\hyphen,\hyphen)$, such that $(m,n) \leq 0$ if $m \neq n$ and $(m,m) \in 2 \mathbb{Z}_{\leq 1}$. The generalized Kac Moody Lie algebra $\mathfrak{g}_{M,(\hyphen,\hyphen),P}$ is defined to be the Lie algebra generated by $e_{m,j,l}, f_{m,j,l}, h_m$ with $m  \in \Phi^{+}_{M,(\hyphen,\hyphen)}, j \in \mathbb{Z}$ and $1 \leq l \leq p_{m,j}$, with the relations 
  \begin{align*}
      h_{m+n} &= h_m+h_n \\
      [h_m,h_n]&=0 \\ 
      [h_m, e_{n,j,l}] &= a_{m,n}e_{n,j,l} \\
      [h_m,f_{n,j,l}] &= -a_{m,n}f_{n,j,l} \\
      [e_{m,j,l},f_{n,j^{\prime},l^{\prime}}] &= \delta_{m,n}\delta_{j,j^{\prime}} \delta_{l,l^{\prime}}h_{n}  \\
      \mathrm{ad}(e_{m,j,l})^{1-a_{m,n}}(e_{n,j^{\prime},l^{\prime}}) &= 0 \mathrm{\ if \ } a_{m,n}=2 \\
      \mathrm{ad}(f_{m,j,l})^{1-a_{m,n}}(f_{n,j^{\prime},l^{\prime}}) &= 0 \mathrm{\ if \ } a_{m,n}=2 \\
    [e_{m,j,l},e_{n,j^{\prime},l^{\prime}}] &= [f_{m,j,l},f_{n,j^{\prime},l^{\prime}}] = 0 \mathrm{ \ if \ } a_{m,n} =0
\end{align*}
\end{defn}
We place $e_{m,j,l}$ in cohomological degree $j$, each $f_{m,j,l}$ in cohomological degree $-j$ and $h_m$ of cohomological degee $0$. We define positive half $\mathfrak{g}^{+}_{M,(\hyphen,\hyphen),P} \subset \mathfrak{g}_{M,(\hyphen,\hyphen),P} $ by Lie subalgebra generated by $e_{m,j,l}$. Similarly define Cartan $\mathfrak{h}_{M,(\hyphen,\hyphen),P}$ as the Lie subalgebra generated by $h_m$. We then have the following theorem.

\begin{thm}[\cite{davison2024bps}]\label{dhsm1}
Let $Q$ be any quiver. Let $M = \mathbb{N}^{Q_0}$ and let $(\bd,\be)_{\mathrm{\Pi}_Q} = \chi_Q(\bd,\be)+\chi_Q(\be,\bd)$ be the bilinear form. Let  
\begin{align*}
P(\bd) =  \begin{cases}
\sum_{j \in \mathbb{Z}} \dim \HH^{j}(\mathcal{IC}_{\mathcal{M}_{\bd}(\mathrm{\Pi}_Q)})t^{j/2} \mathrm{ \ when \ } \bd \in \Phi^{+,\mathrm{real}}_{\mathbb{N}^{Q_0},(\hyphen,\hyphen)_{\mathrm{\Pi}_Q}} \backslash \Phi^{+,\mathrm{iso}}_{\mathbb{N}^{Q_0},(\hyphen,\hyphen)_{\mathrm{\Pi}_Q}}  \\
\sum_{j \in \mathbb{Z}} \dim \HH^{j}(\mathcal{IC}_{\mathcal{M}_{\bd^{\prime}}(\mathrm{\Pi}_Q)})t^{j/2} \mathrm{ \ when \ } \bd = l \bd^{\prime}  \in \Phi^{+,\mathrm{iso}}_{\mathbb{N}^{Q_0},(\hyphen,\hyphen)_{\mathrm{\Pi}_Q}}, \bd^{\prime} \in   \sum^{+}_{\mathbb{N}^{Q_0},(\hyphen,\hyphen)_{\mathrm{\Pi}_Q}} 
\end{cases}
\end{align*}
be the weight function. Then there is an isomorphism of Lie algebras: 

\[ \mathfrak{g}^{\BPS}_{\tilde{Q},\tilde{W}} \simeq \mathfrak{g}^{+}_{\mathbb{N}^{Q_0},(\hyphen,\hyphen)_{\mathrm{\Pi}_Q},P}.\]
\end{thm}

We will refer to the Lie algebras $\mathfrak{g}_{\mathbb{N}^{Q_0},(\hyphen,\hyphen)_{\mathrm{\Pi}_Q},P}$ by $\mathfrak{g}^{\GKM}_{Q}$. 

\begin{example}[Cyclic Quivers] \label{dhsmcyclic}
Let $Q$ be the cyclic quiver $Q^K$. Then the Symmetrized euler form is given by \[ 2 \left(\sum_{i=0}^{K} (d_i-d_{i+1})^2 .\right)\]

So only solutions to $2-(\bd,\bd)_{\mathrm{\Pi}_Q} \geq 0$ is given by strings of type $(n,\cdots,n,n+1,n+1,\cdots,n+1,n,\cdots,n)$, i.e $\bd = n \cdot \delta + [i,m)$ or $(n+1) \cdot \delta$ for some $n \geq 0$, $i \in [0,K]$ and $0 \leq m \leq K-1$ (See section \ref{indices} for Notation).

Moreover if $\bd = n \cdot \delta + [i,m)$, then it cannot be a primitive root since one can write $2-(\bd,\bd) = 0$, but $2-(n\cdot \delta,n\cdot \delta)_{\mathrm{\Pi}_{Q^K}} = 2$. Similarly if $\bd  = [i,m)= \delta_i+ \delta_{i+1}+\cdots+\delta_{i+m-1}$ then also it can't be a primitive root as then $2-(\bd,\bd)_{\mathrm{\Pi}_{Q^K}}=0$ but then so does $2-(\delta_i,\delta_i)_{\mathrm{\Pi}_{Q^K}}$ which contradicts the primitiveness of $\bd$. Thus, we have real positive roots $\delta_i$ for $i \in [0,K]$ and imaginary positive roots $  l\delta, l\geq 1$. Now we consider the weight function. 
\begin{align*}
P(\delta_i) &= \sum \dim(\HH^{j}(\mathbb{Q}_{\pt}))t^{j/2} = 1 \\
P(l\delta) &= \sum \dim(\HH^{j}(\mathcal{IC}_{\mathcal{M}_{\delta}(\mathrm{\Pi}_{Q^K})}))t^{j/2} = t^{-1} 
\end{align*}
So we have generators $e_i,f_i,h_i$ for $i \in [0,K]$ of cohomological degree $0$ and $e_{l\delta},f_{l\delta},h_{l\delta}$ of cohomological degree $-2,0,2$ respectively satisfying the relations above, i.e the affine Lie algebra $\hat{\mathfrak{gl}}(K+1)$. 

\end{example}

We can also define the notion of a Verma Module.

\begin{defn}[Verma Module of Generalized Kac Moody Lie algebra ]
Let $\mathfrak{g}^{\GKM,\leq 0}_{Q}$ be the negative Borel, i.e., Lie subalgebra generated by $f_{m,j,l}$ and $h_m$. For $\bff \in \Hom(\mathbb{Z}^{Q_0} \rightarrow \mathbb{Z})$, we define the Verma module as \[ V_{\mathfrak{g}^{\GKM}_{Q}, \bff} = \mathbf{U}(\mathfrak{g}^{\GKM}_{Q}) \otimes_{\mathbf{U}(\mathfrak{g}^{\GKM, \leq 0}_{Q})} \mathbb{C}\]
where $\mathbb{C}$ becomes a $\mathbf{U}(\mathfrak{g}^{\GKM,\leq 0}_{Q})$ module via the quotient $\mathbf{U}(\mathfrak{g}^{\GKM,\leq 0}_{Q}) \rightarrow \mathbf{U}(\mathfrak{h}_M)$ and $h_m \cdot 1 = \mathbf{f}(m)$. 
\end{defn}

Then, analogous to the usual case of semisimple Lie algebras, one defines 
\begin{defn}
The Lowest Weight module $L_{\mathfrak{g}^{\GKM}_{Q}, \bff}$ is defined to be quotient of $V_{\mathfrak{g}^{\GKM}_{Q}, \bff}$ by the unique maximal submodule of $V_{\mathfrak{g}^{\GKM}_{Q}, \bff}$ which doesn't contain $1 \otimes 1$.
\end{defn}

\section{Affinized BPS Lie Algebra and Factorization Coproduct}
For any symmetric quiver $Q$ and potential $W$, we saw in the Section \ref{perversefiltrationandbps}, that there exists a Lie algebra $\mathfrak{g}^{\BPS}_{Q,W}$ such that $\Sym(\mathfrak{g}^{\BPS}_{Q,W}[u]) \simeq \mathcal{A}^{\psi}_{Q,W}$ as vector spaces. This makes it natural to wonder if the vector subspace $\mathfrak{g}^{\BPS}_{Q,W}[u] \subset \mathcal{A}^{\psi}_{Q,W}$ is itself closed under the commutator Lie bracket coming from the associative algebra structure on $\mathcal{A}^{\psi}_{Q,W}$. The answer to this question is no for a general quiver with potential $Q,W$\footnote{This can be seen when $Q,W$ comes from the derived category of resolved conifold and will be explained in future work.}. However, the answer becomes yes in the case of the tripled quiver with potential $\tilde{Q},\tilde{W}$. This is obtained by constructing a coproduct which uses factorization structure on the stack of representations of the Jacobi algebra for the tripled quiver with potential and is outlined by Davison in \cite{davison2022affine}. This is proved rigorously in author's PhD thesis \cite{mythesis}[Chapter 6] and will be used subsequently in this paper. The construction for more general 3-Calabi–Yau completions
of 2-Calabi–Yau categories will appear in the joint work of Davison, Hennecart, Kinjo, Schiffmann and Vasserot. 

In the case of tripled quiver with canonical potential, there is an isomorphism of algebras $\Jac(\tilde{Q},\tilde{W}) \simeq \mathrm{\Pi}_{Q}[\omega]$ and thus a representation of $\Jac(\tilde{Q},\tilde{W})$ is the same as a representation of preprojective algebra $\mathrm{\Pi}_Q$ with an endomorphism. This allows us to define 
\begin{defn}
Let $U \subset \C$ be any open ball. Then we define  $\mathfrak{M}^{U}_{\bd}(\tilde{Q}) \subset \mathfrak{M}_{\bd}(\tilde{Q})$ to be open substack of representations such that generalized eigenvalues of $\omega_i$ are all inside $U$. a
\end{defn}

Using the support lemma of the BPS sheaves \ref{supportlemma}, it is proved in \cite{mythesis}[Proposition 6.1.3, 6.1.4] that there is an isomorphism 
\[ \HH(\mathfrak{M}_{\bd}(\tilde{Q}),\pPhi_{\Tr(W)_{\bd}}) \simeq \HH(\mathfrak{M}^{U}_{\bd}(\tilde{Q}),\pPhi_{\Tr(W)_{\bd}|_{U}}). \]

Furthermore, for two disjoint open balls $U_1$ and $U_2$, since any point $\rho \in \mathfrak{M}^{U_1 \coprod U_2}(\Jac(\tilde{Q},\tilde{W}))$ admits a canonical direct sum decompotion $\rho_1 \oplus \rho_2$ such that generalized eigenvalues of $\omega$ acting on $\rho_1$ are in in $U_1$ and of $\omega$ acting on $\rho_2$ are in $U_2$, we have an isomorphism of stacks 
\begin{align} \label{step1}
\mathfrak{M}_{\bd}^{U_1 \coprod U_2}(\Jac(\tilde{Q},\tilde{W})) \simeq \mathfrak{M}_{\bd}^{U_1}(\Jac(\tilde{Q},\tilde{W}))  \times \mathfrak{M}_{\bd}^{U_2}(\Jac(\tilde{Q},\tilde{W})). 
\end{align}
Then using Thom-Sebastiani isomoprhism and deformed dimensional reduction, it is proved in \cite{mythesis}[Proposition 6.1.8] that we have isomorphism 
\begin{align} \label{step2}
\HH(\mathfrak{M}^{U_1 \coprod U_2}(\tilde{Q}),\pPhi_{\Tr(W)|_{U_1 \coprod U_2}})  \simeq \HH(\mathfrak{M}^{U_1}(\tilde{Q}),\pPhi_{\Tr(W)|_{U_1}}) \otimes \HH(\mathfrak{M}^{U_2}(\tilde{Q}),\pPhi_{\Tr(W)|_{U_2}}). 
\end{align}  
Combining (\ref{step1}) and (\ref{step2}),  we get a morphism 
\[ \mathrm{\Delta}: \mathcal{A}^{\chi}_{\tilde{Q},\tilde{W}} \rightarrow \mathcal{A}^{\chi}_{\tilde{Q},\tilde{W}}  \otimes \mathcal{A}^{\chi}_{\tilde{Q},\tilde{W}}. \]
It is proved in \cite{mythesis}[Proposition 6.1.13] that this morphism is compatible with the CoHA structure. Furthermore, it is co-commutative \cite{mythesis}[Proposition 6.1.12], since $U_1$ and $U_2$ can be exchanged via a homotopy. Thus by the Milnor-Moore theorem $\mathcal{A}^{\chi}_{\tilde{Q},\tilde{W}} \simeq \bU(\mathcal{P})$ where $\mathcal{P}$ are the primitive elements with respect to the coproduct. Furthermore, it is proved in \cite{mythesis}[Proposition 6.1.15] that $\mathfrak{g}^{\BPS}_{\tilde{Q},\tilde{W}} \otimes \C[u]$ is primitive under this coproduct and by the PBW theorem (Proposition \ref{PBWtheorem}), $\mathcal{A}^{\chi}_{\tilde{Q},\tilde{W}} \simeq \Sym(\mathfrak{g}^{\BPS}_{\tilde{Q},\tilde{W}} \otimes \C[u])$ and so by considering the size of $\mathcal{P}$, it follows that $\mathcal{P} = \mathfrak{g}^{\BPS}_{\tilde{Q},\tilde{W}} \otimes \C[u]$. This shows that the sub vector space $\mathfrak{g}^{\BPS}_{\tilde{Q},\tilde{W}} \otimes \C[u] \hookrightarrow \mathcal{A}^{\chi}_{\tilde{Q},\tilde{W}}$ is closed under the Lie bracket. 

\begin{defn}[The affinized BPS Lie algebra] \label{affinizedBPSliealgebra}
For any quiver $Q$, the affinized BPS Lie algebra $\abps{}$ is defined to be the Lie algebra structure on the $\N^{Q_0} \times \Z$ graded vector space \[ \abps{}:= \mathfrak{g}^{\BPS}_{\tilde{Q},\tilde{W}} \otimes \C[u] \hookrightarrow \mathcal{A}^{\chi}_{Q,W} \] where the Lie bracket is defined by taking the commutator Lie bracket coming from the associative algebra structure on $\mathcal{A}_{\tilde{Q},\tilde{W}}$. 
\end{defn}

\begin{convention}
    Given an element $\alpha \in \mathfrak{g}^{\BPS}_{\QTC{},\WTC{}}$, we denote by \[\alpha^{(n)} := u^n \cdot \alpha \in \abps{} \]
\end{convention}

Let $\textbf{w}: \QTC{} \rightarrow N$ be a torus weighting such that $\textbf{w}(\omega_i)=0$ for all $i \in Q_0$\footnote{This condition is cruicial. This construction doesn't work, for example, in the case of tripled ADE quiver with canonical potential with torus weighting $\textbf{w}(a)=1, \textbf{w}(a^{*})=1$ and $\textbf{w}(\omega_i) = -2$ for all $i \in Q_0$.}. Let $T^{\prime} = \Hom(N,\mathbb{C}^{*})$. Then since $\textbf{w}(\omega)=0$, where $\omega = \sum_{i \in Q_0} \omega_i$, the morphism \[ \lambda \colon \mathfrak{M}_{\bd}(\QTC{}) \rightarrow \Sym(\mathbb{A}^1) \] which records the generalized eigenvalue, is $T^{\prime}$ equivariant. This induces a morphism \[ \lambda^T \colon \mathfrak{M}^{T'}_{\bd}(\QTC{}) \rightarrow \Sym(\mathbb{A}^1).\] We can apply the construction of factorization coproduct to $\mathcal{A}^{T',\chi}_{\QTC{},\WTC{}}$, where we consider $\mathcal{A}^{T',\chi}_{\QTC{},\WTC{}}$ as an object in the category of $\HH_{T'}(\pt)$ modules, i.e. we have a $\HH_{T'}(\pt)$ module morphism 
\[ \mathrm{\Delta}^{T} \colon 
\mathcal{A}^{T',\chi}_{\QTC{},\WTC{}} \rightarrow \mathcal{A}^{T',\chi}_{\QTC{},\WTC{}} \otimes_{\HH_{T'}(\pt)} \mathcal{A}^{T',\chi}_{\QTC{},\WTC{}}.\]
Considering the primitive elements of the cocommutative coproduct $\mathrm{\Delta}^T$, shows that the $\HH_{T}(\pt)$ module \[ \widehat{\mathfrak{g}}^{\BPS,T'}_{\QTC{},\WTC{}} := \mathfrak{g}^{\BPS,T^{\prime}}_{\QTC{},\WTC{}} \otimes \C[u] \subset \mathcal{A}^{T',\chi}_{\QTC{},\WTC{}} \] has a structure of Lie algebra, where the Lie bracket is defined by taking the commutator Lie bracket, coming from the associative algebra structure on $\mathcal{A}^{T',\chi}_{\QTC{},\WTC{}}$. Furthermore, there is an isomorphism of $\HH_{T'}(\pt)$ algebras 
\[ \bU_{\HH_{T'}(\pt)}(\widehat{\mathfrak{g}}^{\BPS,T'}_{\QTC{},\WTC{}}) \simeq \mathcal{A}^{T',\chi}_{\QTC{},\WTC{}}. \]

We call $\widehat{\mathfrak{g}}^{\BPS,T'}_{\QTC{},\WTC{}}$ the $T'\hyphen$deformed affinized BPS Lie algebra. We calculate $\widehat{\mathfrak{g}}^{\BPS}_{\QTC{},\WTC{}}$ and it's deformation, for the case of the cylic quivers in Section \ref{sectionaffinized} and \ref{sectiondeformedaffinized} respectively. 

\section{Heis Lie algebra action} \label{HeisLiealgebra}
The action of the tautological bundle on the vector space $\mathcal{A}^{\psi}_{Q,W}$ increases the cohomological degree by $2$. We can describe the other half of this action, which acts by decreasing the cohomological degree. This structure has been exploited in \cite{joyce2021enumerative} as a translation operator in the construction of vertex algebras and by \cite{davison2022bps} for the calculation of the affinized BPS Lie algebra for the Jordan quiver. We consider the action by scaling by the automorphisms of $\mathfrak{M}^{T,\zeta}_{\bd}(Q)$. More precisely, we have a morphism
\begin{align*}
\act: \BCu \times \mathfrak{M}^{T,\zeta}_{\bd}(Q) \rightarrow \mathfrak{M}^{T,\zeta}_{\bd}(Q)
\end{align*} 
defined by sending the tautological line bundle $\mathcal{L}$ and family of representations $\mathcal{F}$ to $\mathcal{F \otimes L}$. Doing the pullback (\ref{Pullback}), gives a map of sheaves
\begin{equation} \label{act*morphism}
\pPhi_{\Tr(W_{\bd})} (\Q_{\mathfrak{M}^{T,\zeta-\sss}_{\bd}(Q)}) \rightarrow \act_{*}(\Q_{\BCu} \boxtimes \pPhi_{\Tr(W_{\bd})}(\Q_{\mathfrak{M}^{T,\zeta-\sss,\psi}_{\bd}(Q)})).
\end{equation} 
After shifting and taking global sections, this gives a map of cohomologically graded vector spaces
\[ \act^{*}: \HH(\mathfrak{M}^{T,\zeta-\sss}_{\bd}(Q),\pPhi_{\Tr(W_{\bd})})  \rightarrow \HH( \BCu,\Q) \otimes \HH(\mathfrak{M}^{T,\zeta-\sss}_{\bd}(Q),\pPhi_{\Tr(W_{\bd})}).
\]
In particular, for any $\alpha \in \mathcal{A}^{T,\zeta,\psi}_{Q,W,\bd}$ we have \[\act^{*}(\alpha) = \sum_{n \geq 0} u^n \otimes \alpha_{(n)}. \] We define \[\partial \colon \mathcal{A}^{T,\zeta,\psi}_{Q,W} \rightarrow \mathcal{A}^{T,\zeta,\psi}_{Q,W}\] by $\alpha \mapsto \alpha_{(1)}$. Since the morphism $\act^{*}$ preserves the cohomological degree, $\partial$ is an operator which decreases the cohomological degree by $-2$. We then have

\begin{prop}{\cite{davison2022affine}} \label{heisliealgebraction}
    The morphisms  
    \begin{align*}
     \partial: \mathcal{A}^{T,\zeta,\psi}_{Q,W} & \rightarrow \mathcal{A}^{T,\zeta,\psi}_{Q,W} \\
     u: \mathcal{A}^{T,\zeta,\psi}_{Q,W} & \rightarrow \mathcal{A}^{T,\zeta,\psi}_{Q,W}
    \end{align*} define derivations on $\mathcal{A}^{T,\zeta,\psi}_{Q,W}$ i.e. \begin{align*}
    \partial(a \star b) &= (\partial a) \star b + a \star (\partial b) \\ u(a \star b) &= (u \cdot a) \star b + a \star (u \cdot b).
    \end{align*} Furthermore, for any $\alpha \in \mathcal{A}^{T,\zeta,\psi}_{Q,W,\bd}$, we have $[\partial,u]\alpha = |\bd| \alpha$. 
\end{prop}

We now show that the operators $\partial$ and $u$ interact with the perverse filtration $\mathfrak{P}^{}$ on $\mathcal{A}^{T,\zeta,\psi}_{Q,W}$ in a nice way.

\begin{prop}\label{partialuder}
The action $u$ increases the perverse degree by $2$ while $\partial$ decreases the perverse degree by $2$, i.e. 
\begin{align*}
u \cdot \fP^{i} & \subset \fP^{i+2} \\ 
\partial \cdot \fP^{i} & \subset \fP^{i-2}
\end{align*} 
\end{prop}
\begin{proof}

Applying $\JH^{T,\zeta}_{\bd}$ to morphism (\ref{act*morphism}), gives a morphism of complexes in $\Db(\mathcal{M}^{T,\zeta-\sss}_{\bd}(Q))$

\[ a: (\JH^{T,\zeta}_{\bd})_{*}\pPhi_{\Tr(W_{\bd})} (\Q^{\vir}_{\mathfrak{M}^{T,\zeta-\sss}_{\bd}(Q)}) \rightarrow (\JH^{T,\zeta}_{\bd})_{*} \act_{*}(\Q_{\BCu} \boxtimes \pPhi_{\Tr(W_{\bd})}(\Q^{\vir}_{\mathfrak{M}^{T,\zeta-\sss}_{\bd}(Q)})). \]

After taking semisimplification, the action map collapses, and so we have the following commutative square:

\[\begin{tikzcd}[ampersand replacement=\&]
	{\BCu \times \mathfrak{M}_{\bd}^{T,\zeta-\sss}(Q)} \& {\mathfrak{M}_{\bd}^{T,\zeta-\sss}(Q)} \\
	{\mathfrak{M}^{T,\zeta-\sss}_{\bd}(Q)} \& {\mathcal{M}^{T,\zeta-\sss}_{\bd}(Q)}
	\arrow["{\JH_{\bd}^{T,\zeta}}", from=1-2, to=2-2]
	\arrow["\act", from=1-1, to=1-2]
	\arrow["{\pi_{2}}", from=1-1, to=2-1]
	\arrow["{\JH^{T,\zeta}_{\bd}}", from=2-1, to=2-2]
\end{tikzcd}\]
where $\pi_2$ is the projection to the second component. Thus we have
\[   (\JH^{T,\zeta}_{\bd})_{*} \act_{*}(\Q_{\BCu} \boxtimes \pPhi_{\Tr(W_{\bd})}(\Q^{\vir}_{\mathfrak{M}^{T,\zeta-\sss}_{\bd}(Q)}))  \simeq \HH(\BCu,\Q) \otimes (\JH^{T,\zeta-\sss}_{\bd})_{*}\pPhi_{\Tr(W_{\bd})} (\Q^{\vir}_{\mathfrak{M}^{T,\zeta-\sss}_{\bd}(Q)}). \]
But since 
\[ \tau^{\leq n}( \HH^{2k}(\BCu,\Q) \otimes (\JH^{T,\zeta}_{\bd})_{*}\pPhi_{\Tr(W_{\bd})} (\Q^{\vir}_{\mathfrak{M}^{T,\zeta-\sss}_{\bd}(Q)})) = \tau^{\leq n-2k}((\JH^{T,\zeta}_{\bd})_{*}\pPhi_{\Tr(W_{\bd})} (\Q^{\vir}_{\mathfrak{M}^{T,\zeta-\sss}_{\bd}})), \] 
the morphism $\tau^{\leq n}(a)$ factors through \[ \bigoplus_{k \geq 0} \HH^{2k}(\BCu,\Q) \otimes \tau^{n-2k} ((\JH^{T,\zeta}_{\bd})_{*}\pPhi_{\Tr(W_{\bd})}(\Q^{\vir}_{\mathfrak{M}^{T,\zeta-\sss}_{\bd}})).\]
So for any $\alpha \in \fP^{i}$, $\act^{*}(\alpha) = \sum_{n \geq 0} u^n \alpha_{(n)} $ where $\alpha_{n} \in \fP^{i-2n}$ and thus in particular $\alpha_{1} \in \fP^{\leq i-2}$.  Finally, the action of $u$ increases the perverse degree, which is exactly \cite{davison2020cohomological}[Lemma 5.8] applied to the determinant bundle. 

\end{proof}

The above argument also shows that if $\alpha_{\bd} \in \mathfrak{g}^{\BPS,T,\zeta}_{Q,W,\bd}$ then $\partial(\alpha_{\bd})=0$, since the perverse filtration starts from degree $1$ (Section \ref{perversefiltrationandbps}). Now, by repetitively applying the Proposition \ref{partialuder}, it follows that 
\begin{prop} \label{derivationbyp}
Let $\alpha_{\bd} \in \mathfrak{g}^{\BPS,T,\zeta}_{Q,W,\bd}$, then $u^n \alpha_{\bd} \in \fP^{2n+1}$ and we have $\partial(u^n \alpha_{\bd}) = n|\bd| u^{n-1} \alpha_{\bd}.$
\end{prop}

The action of $u,\partial$ can be interpreted as an action of the 3-dimensional Heisenberg $\Heis$ on $\mathcal{A}^{T,\zeta}_{Q,W}$, as considered in \cite{davison2022affine}. 
\begin{defn}
    Let $\textbf{Heis}$ be the Lie algebra over $\Q$ having basis $p,q,c$ such that \[[q,p]=c,\] $c$ is central. We give $\textbf{Heis}$ a `cohomological' grading by setting elements $p,q,c$ as elements of cohomological degree $2,-2,0$ respectively. Then $M$ is said to be a graded $\textbf{Heis}$ module, if $M = \oplus_{i \in \mathbb{Z}} M_{2i}$ and the $\Heis$ action respects the grading, i.e.  $p(M_{2i}) \subset M_{2i+2}, q(M_{2i}) \subset M_{2i-2}, c(M_{2i}) \subset M_{2i}$.
\end{defn}

If a graded $\Heis$ module $\mathfrak{L}$ has a Lie algebra structure, then it is said to be a $\Heis$ Lie algebra if $p$ and $q$ act by Lie algebra derivations on $\mathfrak{L}$. In many cases, a much smaller subspace controls $\Heis$ Lie algebras. For any $\Heis$ Lie algebra $\mathfrak{L}$, let $\mathfrak{L}^{\leq 0} = \oplus_{i \leq 0} \mathfrak{L}_{2i}$ be the the subspace of non-positive even graded elements. Clearly $\mathfrak{L}^{\leq 0} \subset \mathfrak{L}$ forms a Lie subalgebra.  

\begin{defn}[Negatively determined $\Heis$ Lie algebras] 
We say that a $\Heis$ Lie algebra $\mathfrak{L}$ is negatively determined if $\mathfrak{L}_{i} =0$ for all odd $i$ and the map $q: \mathfrak{L}_{2i} \rightarrow \mathfrak{L}_{2i-2}$ is an isomorphism of vector spaces for any $i >0$. 

\end{defn}

If we have two negatively determined Lie algebras, then a morphism between them is also negatively determined.

\begin{prop} \label{mapofheisLiealgebras}
Let $\mathfrak{L}$ and $\mathfrak{H}$ be two negatively determined $\Heis$ Lie algebras such that there is a homomorphism of Lie algebras \[\phi^{\leq 0} \colon \mathfrak{L}^{\leq 0} \rightarrow \mathfrak{H}^{\leq 0}\] which commutes with the action of $\Heis$, i.e, where for any $\alpha_{-2i} \in \mathfrak{L}_{-2i},i >0$, $\phi^{\leq 0}(p \alpha_{-2i}) = p \phi^{\leq 0}(\alpha_{-2i})$, for any $\alpha_{-2i+2} \in \mathfrak{L}_{-2i+2}, i>0$, $ \phi^{\leq 0}(q \alpha_{-2i+2}) = q  \phi^{\leq 0}(\alpha_{-2i+2})$ and for any $\alpha_0 \in \mathfrak{L}_0, \phi^{\leq 0}(c\alpha_0)=c\phi^{\leq 0}(\alpha_0)$. Then, the map $\phi^{\leq 0}$ extends uniquely to map $\phi \colon \mathfrak{L} \rightarrow \mathfrak{H}$ of graded $\Heis$ modules. 
Furthermore, $\phi$ is a Lie algebra homomorphism if and only if $\phi([\alpha_{2i},\phi_{2j}]) = [\phi(\alpha_{2i}),\phi(\alpha_{2j})]$ for $i+j \leq 0$. 
% which extends to homomorphism of modules $\phi^{\leq k}: \mathcal{L}^{\leq k} \rightarrow \mathcal{H}^{\leq k}$ then the map extends to a homorphism of Lie algebras $\phi: \mathcal{L} \rightarrow \mathcal{H}$
\end{prop}

\begin{proof}

For any element $\alpha_{2i} \in \mathfrak{L}_{2i}$ with $i >0$, we define $\phi(\alpha_{2i})$ to be the unique element in $\mathfrak{H}_{2i}$ such that $q^{i}\phi(\alpha_{2i}) = \phi^{\leq 0}(q^i\alpha_{2i})$ and for any $\alpha_{2i} \in \mathfrak{L}_{2i}$ with $i \leq 0$, we set $\phi( \alpha_{2i}) = \phi^{\leq 0}(\alpha_{2i})$. The map is well defined by the injectivity of $q$ on positive components. 

The morphism $\phi$ is a map of $\Heis$ modules since when $i=1$, any $\alpha_2 \in \mathfrak{L}_2$ satisfies $q\phi(\alpha_2) = \phi^{\leq 0}(q\alpha_2)$ and if $i>1$, then $\phi(q \alpha_{2i})$ is defined to be unique element in $\mathfrak{H}$ such that $q^{i-1}\phi(q\alpha_{2i}) = \phi^{\leq 0}(q^i\alpha_{2i})$, which is clearly satisfied by $q \phi(\alpha_{2i})$. Similarly, since $q^ic\phi(\alpha_{2i}) = cq^i\phi(\alpha_{2i}) = c\phi^{\leq 0}(q^i\alpha_{2i}) = \phi^{\leq 0}(cq^i\alpha_{2i}) = \phi^{\leq 0}(q^ic\alpha_{2i})$ implies that $\phi(c\alpha_{2i}) = c \phi(\alpha_{2i})$. Finally succesive application of $[q,p]=c$ implies that for any $\alpha \in \mathfrak{H}$ or $\mathfrak{L}$, $q^{i+1}p \alpha = pq^{i+1} \alpha + (i+1)q^{i}c \alpha$. So $q^{i+1}(p \phi(\alpha_{2i})) = (pq^{i+1}+(i+1)q^ic)\phi(\alpha_{2i}) = pq \phi^{\leq 0}(q^i\alpha_{2i}) + (i+1)c\phi^{\leq 0}(q^i\alpha_{2i}) = \phi^{\leq 0}((pq^{i+1}+(i+1)cq^{i})\alpha_{2i}) = \phi^{\leq 0}(q^{i+1}p\alpha_{2i})$. Thus $p\phi(\alpha_{2i}) = \phi(p\alpha_{2i})$. 

It suffices to check that it is indeed a Lie algebra homomorphism if $\phi([\alpha_{2i},\phi_{2j}]) = [\phi(\alpha_{2i}),\phi(\alpha_{2j})]$ for $i+j \leq 0$. For this, we consider $\alpha_{2i} \in \mathfrak{L}_{2i}$ and $\beta_{2i} \in \mathfrak{L}_{2j}$. If $i+j \leq 0$, then $\phi([\alpha_{2i},\alpha_{2j}]) = ([\phi(\alpha_{2i}), \phi(\alpha_{2j})])$ by the assumption. So, we can assume $i+j>0$. Then 
$\phi([\alpha_{2i},\alpha_{2j}])$ is the unique element such that \[ q^{i+j} \phi([\alpha_{2i},\alpha_{2j}])= \phi^{\leq 0}(q^{i+j}[\alpha_{2i},\beta_{2j}]) = \sum_{r = 0}^{i+j}\binom{i+j}{r} \phi^{\leq 0}[q^{r}\alpha_{2i},q^{i+j-r}\alpha_{2j}]\]
By assumption, we have
%Now, in this sum, we may assume that at least one of $q^r \alpha_{2i}$ or $q^{i+j-r}\alpha_{2j}$ are in $\mathcal{L}^{\leq k}$ and the other one in $\mathcal{L}^{\leq 0}$. So then 

\begin{align*}
\sum_{r=0}^{i+j}\binom{i+j}{r} \phi^{\leq 0}[q^{r}\alpha_{2i},q^{i+j-r}\alpha_{2j}] &= \sum_{r = 0}^{i+j}\binom{i+j}{r} [\phi(q^{r}\alpha_{2i}),\phi(q^{i+j-r}\alpha_{2j})] \\ &= \sum_{r = 0}^{i+j}\binom{i+j}{r} [q^r \phi(\alpha_{2i}), q^{i+j-r} \phi( \alpha_{2j})]  \\ &= q^{i+j}[\phi(\alpha_{2i}),\phi(\beta_{2j})].
\end{align*}
and thus $\phi([\alpha_{2i},\alpha_{2j}]) = [\phi(\alpha_{2i}),\phi(\alpha_{2j})]$ for all $i,j$ and we are done. 
\end{proof}

\begin{defn}
We say that a $\Heis$ module $M$ is integrable if it admits a $\Heis$ module decomposition $M = \oplus_{n \in\mathbb{Z}} M^{[n]}$ such that each the central element $c$ acts on $M^{[n]}$ by scaling with $n$. We correspondingly say that $M$ is integrable $\Heis$ Lie algebra if it it integrable $\Heis$ module, $\Heis$ Lie algebra, and if the Lie algebra structure is also $\mathbb{Z}$ graded (i.e $[M^{[n]},M^{[m]}] \subset M^{[n+m]}$).
\end{defn}

An example of an integrable $\Heis$ Lie algebra is as follows. 

\begin{defn} \label{currentliealgebra}
For any Lie algebra $\mathfrak{g}$ we define $\mathfrak{g}[D] := \fg \otimes_{\Q} \Q[D]$ to be a Lie algebra given by basis $\alpha D^n$ where $n \geq 0$ and $\alpha \in \fg$ with the  relations $[\alpha D^n, \alpha^{'} D^m] = [\alpha,\alpha^{'}] D^{n+m}$ where $\alpha,\alpha^{'} \in \fg$ and $n,m \in \mathbb{Z}_{\geq 0}$.
\end{defn}

\begin{prop}\label{integrableliealgebraexample}
Let $\mathfrak{g}$ be a $\mathbb{Z}$ graded Lie algebra. Then \[\mathfrak{g}[D] = \bigoplus_{m \in \mathbb{Z}}  \mathfrak{g}_{m}[D] \]is negatively determined integrable $\Heis$ Lie algebra defined by $q(\alpha_{m} D^n) = n \alpha_{m} D^{n-1}$, $c(\alpha_{m} D^n) =m \alpha_{\bd} D^n$ and $p(\alpha_{m} D^n)= m \alpha_{\bd} D^{n+1}$ for any $\alpha_{m} \in \mathfrak{g}_{m}$. 
\end{prop}

We then have the following consequence of Proposition \ref{mapofheisLiealgebras}.

\begin{prop} \label{negativelydetermined} Let $\mathfrak{L}$ be a negatively determined integrable $\Heis$ Lie algebra such that $\mathfrak{L}_{i}=0$ for all $i<0$. Then there is an isomorphism of Lie algebras $\mathfrak{L}_{0}[D] \simeq \mathfrak{L}$ where $\mathfrak{L}_{0} \subset \mathfrak{L}$ is the degree $0$ Lie subalgebra. 

%Let $\mathcal{L}^{\leq 0} = \oplus_{i \leq 0} \mathcal{L}_{2i} \subset \mathcal{L}$ be the Lie subalgebra. Suppose that for $i >0$, $q: L_{2i} \rightarrow L_{2i-2}$ is injective. 
%%are determined by relations Lie algebra $\mathcal{L}^{\leq 0}$ and it is action on it is module \[\mathcal{L}^{\leq k} := \bigoplus_{i=-k}^k \mathcal{L}_{2i} \] 
\end{prop}

\begin{proof}
Since $\mathfrak{L}$ is integral, it admits a decomposition $\mathfrak{L} = \oplus_{m \in \mathbb{Z}} \mathfrak{L}^{[m]}$ such that $\mathfrak{L}_{0} = \oplus_{m \in \mathbb{Z}} \mathfrak{L}^{[m]}_{0}$ is $\mathbb{Z}$ graded Lie algebra. By Proposition \ref{integrableliealgebraexample}, the Lie algebra $\mathfrak{L}_{0}[D]$ is integrable $\Heis$ Lie algebras. Note that $\mathfrak{L}_{0}[D]$ and $\mathfrak{L}$ are both negative determined Lie algebras. Thus by Proposition \ref{mapofheisLiealgebras}, the morphism $\id: \mathfrak{L}_{0} \rightarrow \mathfrak{L}_0$ extends to a morphism of Lie algebras $\phi: \mathfrak{L}_{0}[D] \rightarrow \mathfrak{L}$. This morphism is an injection, since for a non zero $\alpha \in \mathfrak{L}_{0}$, $q^i(\phi(\alpha D^i)) = \alpha \neq 0.$ Similarly it is surjection since for any $\alpha_{2i} \in \mathfrak{L}_{2i}$, $\phi((q^i \alpha_{2i}) D^i) = \alpha_{2i}$.
\end{proof}

Next, we see that affinized BPS Lie algebras are also negatively determined integrable $\Heis$ Lie algebras. 

\begin{prop} \label{negativelydeterminedaffinizedbps}
    For any quiver $Q$, the affinized BPS Lie algebra $\abps{}$ is negatively determined integrable $\Heis$ Lie algebra.
\end{prop}

\begin{proof}
From Proposition \ref{partialuder}, it follows that \[\abps{} = \bigoplus_{i \in \mathbb{Z}, \bd \in \mathbb{N}^{Q_0}}  \widehat{\mathfrak{g}}^{\BPS,i}_{\tilde{Q},\tilde{W},\bd}  = \bigoplus_{m \in \mathbb{Z}}  \left( \bigoplus_{i \in \mathbb{Z}, \bd \in \mathbb{N}^{Q_0}, |\bd|=m} \widehat{\mathfrak{g}}^{\BPS,i}_{\tilde{Q},\tilde{W},\bd} \right)\] is integrable $\Heis$ Lie algebra.  The cohomology degrees of the BPS Lie algebra $\mathfrak{g}^{\BPS}_{\tilde{Q},\tilde{W}}$ is determined by the Kac polymomial (\ref{Kacspolynomial}) and in particular, it implies that $\mathfrak{g}^{\BPS,i}_{\tilde{Q},\tilde{W}} =0$ if $i>0$ and for all odd integers $i \in \mathbb{Z}$.  But then since there is an isomorphism of cohomologically graded vector spaces $\abps{} \simeq  \mathfrak{g}_{\tilde{Q},\tilde{W}}[u]$, $\abps{}$ is spanned by $u^n \cdot \alpha$ where $\alpha \in \mathfrak{g}^{\BPS}_{\tilde{Q},\tilde{W}}$ and in particular if the cohomological degree of $u^n \cdot \alpha>0$ then $n \geq 1$. For a fixed dimension vector $\bd$, the derivation $\partial$ is a morphism of vector spaces $\partial^{i}_{\bd} \colon \widehat{\mathfrak{g}}^{\BPS,i}_{\tilde{Q},\tilde{W},\bd}\rightarrow  \widehat{\mathfrak{g}}^{\BPS,i-2}_{\tilde{Q},\tilde{W},\bd}$. When $i$ is an even positive integer then by Proposition \ref{derivationbyp}, the map $\partial^{i}_{\bd}$ is surjective since it sends the elements $u^n \alpha$ for $n>0$ to $n|\bd| u^{n-1} \alpha$. But by above argument or Equation \ref{kacpolynomialandrelationwithbpsliealgebra},  $\widehat{\mathfrak{g}}^{\BPS,i}_{\tilde{Q},\tilde{W},\bd}$ and $\widehat{\mathfrak{g}}^{\BPS,i-2}_{\tilde{Q},\tilde{W},\bd}$ has the same dimension when $i>0$ and thus $\partial$ is in fact isomorphism and we are done. 
\end{proof}

\section{Spherical subalgebra of CoHA}
We define the spherical part of a $\mathbb{N}^{Q_0}$ graded algebra as a subalgebra generated by the smallest possible dimension vectors. 

\begin{defn}[Spherical Generation]
For any quiver $Q$, let $\mathcal{A}$ be a $\N^{Q_0}$ graded algebra over a ring $R$. Then we say that $\mathcal{A}$ is spherically generated if $A$ is generated as an $R$ algebra by $ \langle \mathcal{A}_{\delta_i} \mid i \in Q_0 \rangle$. The spherical part of algebra $\mathcal{A}$ is the subalgebra generated by $\langle \mathcal{A}_{\delta_i} \mid i \in Q_0 \rangle $ and is denoted by $\mathcal{SA}$.
\end{defn}

We then have the following interesting consequence of the above abstraction. We determine $\mathcal{S}\mathcal{A}_{\tilde{Q},\tilde{W}}$ for any quiver without loops. 

\begin{thm}\label{sphericalsubalgebra}
For any quiver $Q$, let $Q'$ be the subquiver formed after removing all the loops of $Q$. Then 
\begin{enumerate}
    \item  The subvector space $\mathfrak{g}^{\BPS,0}_{\tilde{Q},\tilde{W}}\otimes \HH(\BCu,\Q) \subset \abps{}$ is closed under the Lie bracket and we have a Lie algebra isomorphism \[\mathfrak{g}^{\BPS,0}_{\tilde{Q},\tilde{W}}\otimes \HH(\BCu,\Q) \simeq \mathfrak{n}_{Q'}^{+}[D] .\] 
    \item  For any $Q$, there is an embedding of algebras 
    \[ \bU(\mathfrak{n}^{+}_{Q'}[D]) \hookrightarrow \mathcal{A}^{\chi}_{\tilde{Q},\tilde{W}}\] which is an isomorphism if and only if the quiver $Q$ is of finite type. 
    \item For quiver $Q$ without loops, the image of above embedding is the spherical subalgebra $\mathcal{S}\mathcal{A}^{\chi}_{\tilde{Q},\tilde{W}}$, so that we have an isomorphism of algebras \[ \bU(\mathfrak{n}^{+}_{Q'}[D]) \simeq \mathcal{S} \mathcal{A}^{\chi}_{\tilde{Q},\tilde{W}}. \]
\end{enumerate}

\end{thm}
   
\begin{proof}

For part (1), we note that given $\alpha^{(n)}= u^n \cdot \alpha$ and $\beta^{(m)}:= u^m \cdot \beta$ in $\mathfrak{g}^{\BPS,0}_{\tilde{Q},\tilde{W}} \otimes \HH(\BCu,\Q)$ where $\alpha,\beta \in \mathfrak{g}^{\BPS,0}_{\QTC{},\WTC{}}$, we have

\begin{equation}\label{summation0}
 [\alpha^{(n)},\beta^{(m)}] = \sum_{i \geq 0} u^{n+m+i} \gamma_i 
 \end{equation} for some $\gamma_i \in \mathfrak{g}^{\BPS,-2i}_{\QTC{},\WTC{}}$, since $[\alpha^{(n)},\beta^{(m)}]$ is of cohomological degree $2n+2m$ and $\mathfrak{g}^{\BPS}_{\QTC{},\WTC{}}$ is all supported in non-positive cohomological degrees (see proof of Proposition \ref{negativelydeterminedaffinizedbps}). We claim that almost all the terms in the above summation are superfluous, i.e., we claim that $\gamma_i=0$ for all $i>0$. Since the summation in (\ref{summation0}) is a finite sum, let $N>0$ be the smallest such that $\gamma_i =0$ for all $i>N$. 
 We apply the derivation $\partial$, $n+m+N$ times to the Equation \ref{summation0} and using Proposition \ref{derivationbyp} gives that 
 \[ 0 = (n+m+N)!(|\bd|)^{n+m+N} \gamma_N\] where $\bd$ is the dimension vector on which $[\alpha^{(n)},\beta^{(m)}]$ are supported.  Thus $\gamma_{N}=0$. Contradicting the minimality of $N$. Thus $N=0$ and the claim follows. 
 
 Thus $ \mathfrak{g}^{\BPS,0}_{\tilde{Q},\tilde{W}}\otimes \HH(\BCu,\Q)\subset \abps{}$ is closed under the Lie bracket. It is a negatively determined integrable $\Heis$ Lie algebra which starts from degree $0$. Thus by Proposition \ref{negativelydetermined}, there is an isomoprhism of Lie algebras $\mathfrak{g}^{\BPS,0}_{\tilde{Q},\tilde{W}}\otimes \HH(\BCu,\Q) \simeq \mathfrak{g}^{\BPS,0}_{\tilde{Q},\tilde{W}}[D]$. Further, by the Proposition \ref{zerocohomologybps}, there is an isomorphism of Lie algebras $\mathfrak{g}^{\BPS,0}_{\tilde{Q},\tilde{W}}\simeq \mathfrak{n}^{+}_{Q^{\prime}}$ and so we are done. Part (2) follows from the fact that Kac polynomials are constant if and only if $Q$ is of finite type. Finally for part (3), we observe that the subalgebra generated by $\mathcal{A}^{\chi}_{\tilde{Q},\tilde{W},\delta_i}$ is exactly the universal enveloping algebra of the Lie subalgebra of $\abps{}$ generated by $\widehat{\mathfrak{g}}^{\BPS}_{\tilde{Q},\tilde{W} \delta_i}$ for $i \in Q_0$. But then for a quiver without loops, we have $\mathfrak{g}^{\BPS}_{\tilde{Q},\tilde{W},\delta_i} =  \mathfrak{g}^{\BPS,0}_{\tilde{Q},\tilde{W},\delta_i}$ and so we are done by the above argument. 

%% Then since $\HH^{i}(\bps{}) = 0 \forall i >0$, we must have \[ u^n \alpha, u^m \alpha^{'}] = \sum u^{n+m+k_i} \beta_{k_i}\] for some $k_1<k_2 \cdots <k_r$ and $\beta_{k_i} \in \bps{}$. 

\end{proof}

\section{Nakajima Quiver Varieties for Cyclic Quivers} \label{chapter: quivervarieties}
In this chapter, we will consider the action of cohomological Hall algebras on the cohomology of Nakajima quiver varieties. These actions have been defined in \cite{Yang_2018}, \cite{davison2022bps}, \cite{schiffmanncohagenerators} in the case of stability condition $\zeta^{-}$, which demands that the image of the framed vector generates the entire representation (Section \ref{stabilityinfinity}). When the quiver is cyclic, this gives an action of CoHA on the cohomology of the equivariant Hilbert Scheme. However, we can do similar consturction for a different stability condition (Section \ref{kapranovstability}) $\zeta^{n}$ which allows us to define action of a subalgebra of CoHA on the cohomology of Hilbert scheme of points on the resolution of Kleinian singularity $\mathbb{C}^2/\mathbb{Z}_{K}$ (Section \ref{leftactionhilbertscheme}). We study these two actions by comparing them with the work of Nakajima, where they defined the action of Kac-Moody Lie algebra on the cohomology of the equivariant Hilbert scheme and the action of Heisenberg Lie algebra on the cohomology of the Hilbert scheme of points.

Given any quiver $Q$, fix an identification of the vertex set $Q_0$ with the set $[0,1, \dots, K]$ where $K= |Q_0|-1$. Let $\bff \in \N^{Q_0}$ be a dimension vector, called the framing vector. Then we consider the space of representations:

\[ \mathrm{Rep}_{\bd,\bff}(Q) := \prod_{a \in Q_1}\Hom(\mathbb{C}^{\bd(s(a))},\mathbb{C}^{\bd(t(a))}) \times \prod_{i \in Q_0} \Hom(\mathbb{C}^{\bff_i},\mathbb{C}^{\bd_i}).\]

Then the cotangent bundle $T^{*}\Rep_{\bd,\bff}(Q)$ carries the action of $\mathrm{GL}_{\bd}$ which preserves the symplectic form. So we have the moment map

\begin{equation}\label{momentmaporigin}
\mu_{\bd,\bff}: T^{*}\mathrm{Rep}_{\bd,\bff}(Q) \rightarrow \mathfrak{gl}_{\bd}.
\end{equation}

Then Nakajima quiver variety is defined by taking the GIT quotient of $\mu_{\bd,\bff}^{-1}(0)$. By a trick due to Crawley-Boevey \cite{CB1}, we can look at the space $T^{*}\Rep_{\bd,\bff}(Q)$ as a representation of another quiver $Q_{\bff}$ of dimension $(\bd,1)$. $Q_{\bff}$ is defined by adding a vertex $\infty$ to the vertex set $Q_0$, and an arrows $\alpha_{i,n}$ for all $i \in Q_0, 1 \leq n \leq \bff_i$ where the source $s(\alpha_{i,n}) = \infty$ and $t(\alpha_{i,n}) = i$. For any dimension vector $\bd \in \mathbb{N}^{Q_0}$ and number $n \in \mathbb{N}$, let $(\bd,n) \in \N^{Q_{\bff}}$ be the dimension vector where we denote by $(\bd,n)_{\infty} = n$ and $(\bd,n)_i = \bd_i$ for all $i \in Q_0$. Here to get a representation of quiver $Q_{\bff}$ from $\Rep_{\bd,\bff}(Q)$, we are essentially identifiying $\Hom(\mathbb{C}^{\bff_i},\mathbb{C}^{\bd_i})$ with $\Hom(\mathbb{C},\mathbb{C}^{\bd_i})^{\bff_i}$.

Note that the moment map for $(\bd,1)$ dimensional representation of the preprojective algebra gives a morphism
\[ \mu_{Q_{\bff},(\bd,1)}: \Rep_{(\bd,1)}(\overline{Q_{\bff}})\rightarrow \mathfrak{gl}_{(\bd,1)}\]

The morphism \ref{momentmaporigin} is the composition of $\mu_{Q_{\bff},(\bd,1)}$ and the projection $\mathfrak{gl}_{(\bd,1)} \rightarrow \mathfrak{gl}_{\bd}$.
\[\begin{tikzcd}[ampersand replacement=\&]
	{\mu_{\bd,\bff} \colon \Rep_{(\bd,1)}(\overline{Q_{\bff}})} \& {\mathfrak{gl}_{(\bd,1)}} \& {\mathfrak{gl}_{\bd}}
	\arrow[from=1-1, to=1-2]
	\arrow["\pi", from=1-2, to=1-3]
\end{tikzcd}\]

So we have an extra relation \[ \sum_{i \in Q_0, 1 \leq n \leq \bff_i} \rho(\alpha_{i,n}^{*})\rho(\alpha_{i,n}) = 0\] in $\mu_{Q_{\bff},(\bd,1)}^{-1}(0)$. 
However, since the dimension of the vertex $\infty$ is $1$, this is same condition as having 
\begin{equation}\label{extrarelmoment}
 \sum_{i \in Q_0, 1 \leq n \leq \bff_i} \Tr( \rho(\alpha_{i,n}^{*})\rho(\alpha_{i,n})) = 0. 
\end{equation}However this is superflous since the defining relations for $\mu_{(\bd,\bff)}^{-1}(0)$ are \[\sum_{a \in Q_0} [\rho(a),\rho(a^*)] + \sum_{i \in Q_0, 1 \leq n \leq \bff_i}  \rho(\alpha_{i,n})\rho(\alpha_{i,n}^{*})=0 \] but then $\Tr([A,B])=0$ so this implies that \[\sum_{i \in Q_0, 1 \leq n \leq \bff_i}  \Tr(\rho(\alpha_{i,n})\rho(\alpha_{i,n}^{*})) = 0 \]
which is same as equation \ref{extrarelmoment}. Thus there are essentially no new equations and so \[\mu_{\bd,\bff}^{-1}(0) = \mu_{Q_{\bff},(\bd,1)}^{-1}(0) .\] So from now on we will working with later point of view. Taking the affine GIT quotient defines
\begin{defn}[Affine Nakajima Quiver Variety]
The affine Nakajima quiver variety is defined to be the affine quotient \[ \Nak^0_{\bd,\bff}(Q) = \C[\mu^{-1}_{Q_{\bff},(\bd,1)}(0)]^{\mathrm{GL}_{\bd}}].\]
\end{defn}

The varieties $\Nak^0_{\bd,\bff}(Q)$ are singular. There is a way to define an open subset of $\mu^{-1}_{Q_{\bff},(\bd,1)}(0)$ where the action of $\GL_{\bd}$ is free.  Given a dimension vector $\bd$ and a stability condition $\zeta \in \Q^{Q_0}$, we extend this to a stability condition $\hat{\zeta} \in \Q^{(Q_{\bff})_{0}}$ by setting $\zeta_{\infty} = - \sum \bd_i \zeta_i $. Assume that $\hat{\zeta}$ is generic for dimension vector $(\bd,1)$. This defines an open subset $\mu^{-1}(0)^{\zeta-\sss}_{Q_{\bff},(\bd,1)} \subset \mu^{-1}_{{Q_{\bff},(\bd,1)}}(0)$ of $\hat{\zeta}$ stable representations. 

\begin{example}[$\zeta^{\infty}$ stability condition]
For any quiver $Q$, let $\zeta \in \mathbb{Q}^{Q_0}$ be any tuple where $\zeta_i <0 $ for all $i \in Q_0$. Then $\hat{\zeta}$ is generic for any dimension $(\bd,1)$ since $\chi_{\hat{\zeta}}(\bd^{\prime},0) = \sum \zeta_i \bd_i \neq 0$ for all $(\bd^{\prime},0) < (\bd,1)$ unless $\bd_i=0$ and $\chi_{\hat{\zeta}}(\bd^{\prime},1) = \sum_{i \in Q_0} \zeta_i (\bd^{\prime}_i- \bd_i) $ unless all $(\bd^{\prime},1) = (\bd,1)$. Infact since $\chi_{\hat{\zeta}}(\bd^{\prime},1)$ is always positive unless $\bd^{\prime}=\bd$ implies that for any $\zeta$ stable representation of quiver $Q$, it cannot be stable unless $\bd^{\prime}=\bd$. Also, $\chi_{\hat{\zeta}}(\bd^{\prime},0)$ is always negative, so it doesn't impose any condition on the stability. Thus, a representation of $\zeta$ is semistable iff it is generated by the image of a vector at the vertex $\infty$. We shall denote this stability condition by $\zeta^{\infty}$. 
\end{example}

\begin{defn}[Nakajima quiver variety]
Given quiver $Q$, dimension vector $\bd$, framing vector $\bff$ and generic stability condition $\zeta$, the Nakajima quiver variety is defined to be the quotient

\[\Nak^{\zeta}_{\bd,\bff}(Q) =  \mu^{-1}_{Q_{\bff},(\bd,1)}(0)^{\zeta-\sss}/\mathrm{GL}_{\bd} \]
\end{defn}

We have the following fundamental theorem of Nakajima.

\begin{thm}
The open subset $\mu^{-1}(0)^{\zeta-\sss}_{Q_{\bff},(\bd,1)}$ is a non-singular complex variety of dimension $\dim(\Rep_{(\bd,1)}(Q_{\bff})) - \dim(\mathrm{GL}_{\bd}) = \sum_{a \in Q_0}(\bd_{s(a)} \bd_{t(a)}) + 2 \sum_{i \in Q_0} \bd_i\bff_i - \sum_{i \in Q_0} \bd_i^2$ and the Nakajima quiver variety $\Nak^{\zeta}_{\bd,\bff}(Q)$ is a non-singular complex variety of dimension $\dim(\Rep_{(\bd,1)}(Q_{\bff})) - 2 \dim(\mathrm{GL}_{\bd}) = 2\sum_{a \in Q_0}(\bd_{s(a)} \bd_{t(a)}) + 2 \sum_{i \in Q_0} \bd_i\bff_i - 2\sum_{i \in Q_0} \bd_i^2 = - 2\chi_{Q}(\bd,\bd)+ 2 \bd \cdot \bff$.
\end{thm}

Note that for any weighting $\mathbf{w}:(\overline{Q_{\bff}})_{1} \rightarrow \mathbb{Z}^{r}$ for which $\sum_{a \in Q_{\bff}} [a,a^{*}]$ is homogenous. Then the corresponding $T$ action preserves $\mu^{-1}_{Q_{\bff},(\bd,1)}(0)$ and thus defines an action on Nakajima quiver variety. 
By the definition of GIT, there exists a projective morphism \begin{equation} \label{affinizationmorphism}
q: \Nak^{\zeta}_{\bd,\bff}(Q) \rightarrow \Nak^{0}_{\bd,\bff}(Q)
\end{equation} which is a symplectic resolution whenever surjective.

By definition, Nakajima quiver varieties come with non-trivial tautological bundles for every vertex: 

\begin{defn}[Tautological Bundles on Nakajima quiver Variety]
 For each vertex $i$ of the quiver, we define tautological bundle $\Vcal_i$ to be the vector bundle associated with the $\GL_{d_i}$ principal bundle 
\begin{align*}
(\mu_{Q_{\bff},(\bd,1)}^{-1}(0) \cap \Rep^{\hat{\zeta}-\sss}_{(\bd,1)}(\overline{Q_{\bff}}))/\GL^{i}_{\bd} \rightarrow \Nak^{\zeta}_{\bd,\bff}(Q)
\end{align*} where by $\GL^{i}_{\bd}$ we mean $\prod_{j \neq i} \GL_{d_j}$, and the action on $\mu_{(\bd,1)}^{-1}(0) \cap \Rep^{\hat{\zeta}-\sss}_{(\bd,1)}(\overline{Q_{\bff}})/\GL^{i}_{\bd}$ is given by the injection $\prod_{j \neq i} \GL_{d_j} \rightarrow \GL_{\bd}$.  

Later, we will also consider topologically trivial vector bundles $\mathcal{W}_i$ of rank $\bff_i$ defined by $\Nak^{\zeta}_{\bd,\bff}(Q) \times \mathbb{C}^{\bff_i} \rightarrow \Nak^{\zeta}_{\bd,\bff}(Q) $. Nakajima quiver variety carries an action of the torus $T$, and the equivariant Chern class of $\mathcal{W}_i$ will then carry non-trivial information. 

\end{defn}

In Section \ref{moduliofnilpotent}, we defined the Lusztig nilpotent stack of preprojective algebra representations. Analogously, we can define

\begin{defn}
Given any quiver $Q$, dimnesion $\bd$, framing $\bff$ and stability condition $\zeta$, we define nilpotent quiver variety $\mathrm{L}^{\zeta}_{\bd,\bff}(Q)$ as the closed subvariety $q^{-1}(0)$ where $q: \Nak^{\zeta}_{\bd,\bff}(Q) \rightarrow \Nak^{0}_{\bd,\bff}(Q)$ is the GIT projective morphism defined in equation \ref{affinizationmorphism}. 
\end{defn}

Note that since $q$ is a projective morphism, $\mathrm{L}^{\zeta}_{\bd,\bff}(Q)$ is always projective. 

\begin{comment}
\begin{defn}
Given any quiver $Q$, dimnesion $\bd$, framing $\bff$ and stability condition $\zeta$, we define nilpotent quiver variety as the closed subvariety $\mathrm{N}^{\zeta,\mathcal{N}}_{\bd,\bff}(Q) \subset \mathrm{N}^{\zeta}_{\bd,\bff}(Q)$, where underlying $\overline{Q}$ representation is nilpotent.
\end{defn}
\end{comment}
\subsection{Critical locus description of Nakajima quiver varieties}

In \cite{davison2022bps}, Davison provides a critical locus description of the Nakajima quiver varieties $\Nak^{\zeta^{<}}_{\bd,\bff}(Q)$. Their method doesn't work for an arbitrary generic stability condition $\zeta$. In this section, using dimension reduction, we will show that the statement is true at the level of cohomology.

Given framing $\bff$, we consider the tripled framed quiver $\tilde{Q}_{\bff}$. This is the usual doubled framed quiver $\overline{Q_{\bff}}$ that we used to define the Nakajima quiver
varieties, but with an extra loop at every vertex, including a loop at the framing vertex. We then consider the canonical cubic potential
\[ \tilde{W}_{\bff} =  \left(\sum_{i \in (Q_{\bff})_0} \omega_i \right) \left(\sum_{a \in (Q_{\bff})_1} [a,a^{*}] \right).\]
Assume that $\hat{\zeta}$ is generic for dimension vector $(\bd,1)$. Then the 3d Nakajima quiver variety is defined to be smooth variety  \[\mathcal{M}^{\zeta,3d}_{\bd,\bff}(Q): = \Rep_{(\bd,1)}^{\hat{\zeta}-\sss}(\tilde{Q}_{\bff})/\GL_{\bd}, \] where $\Rep_{(\bd,1)}^{\hat{\zeta}-\sss}(\tilde{Q}_{\bff}) \subset \Rep_{(\bd,1)}(\tilde{Q}_{\bff})$ the open subset of $\zeta$ semistable representations. We can similarly define $\Rep^{\hat{\zeta}-\sss}_{(\bd,1)}(\overline{Q_{\bff}}) \subset \Rep_{(\bd,1)}(\overline{Q_{\bff}})$ the open subset of $\hat{\zeta}$ stable $\overline{Q_{\bff}}$ representations. Let $\Tr^{\zeta}(\tilde{W}_{\bff})$ denote the restriction of the trace function. We then have \begin{comment}and $\mu_{\bd,\fr}: \Rep_{(\bd,1)}(\overline{Q_{\fr}}) \rightarrow \mathfrak{gl}_{\bd}$ be composition of the moment map $\mu_{(\bd,1)}: \Rep_{(\bd,1)}(\overline{Q_{\fr}}) \rightarrow \mathfrak{gl}_{(\bd,1)}$ and the projection map. Then
the Nakajima quiver variety \cite{Nak_quivers_kac_moody} is defined to be the variety
\[M^{\zeta}_{\bd,\fr}(Q) := \mu_{\bd,\fr}^{-1}(0) \cap \Rep_{(\bd,1)}(\overline{Q_{\fr}})^{\hat{\zeta}-\sss}/\GL_{\bd} = \mu_{(\bd,1)}^{-1}(0) \cap \Rep_{(\bd,1)}(\overline{Q_{\fr}}_)^{\hat{\zeta}-\sss}/\GL_{\bd}. \]
\end{comment}

%%Notice that if $\zeta$ is such that for representations of $\C \tilde{Q_{\fr}}$, if $\zeta$ stability is same as semi-stability then so does for $M^{\zeta}_{\bd,\fr}(Q)$ and so $M^{\zeta}_{\bd,\fr}$ is also 

\begin{lemma}
Let $\pi_{Q_{\bff},\bd}: \Rep_{(\bd,1)}(\tilde{Q}_{\bff}) \rightarrow \Rep_{(\bd,1)}(\overline{Q_{\bff}})$ be the projection map. Then \[ \crit(\Tr^{\zeta}(\tilde{W}_{\bff})) \cap \pi_{Q_{\bff},\bd}^{-1}(\Rep^{\hat{\zeta}-\sss}_{(\bd,1)}(\overline{Q_{\bff}})) = \crit(\Tr^{\zeta}(\tilde{W}_{\bff})) \cap \Rep_{(\bd,1)}^{\hat{\zeta}-\sss} (\tilde{Q}_{\bff}) \]
\end{lemma}

\begin{proof}  This proof uses same argument as in \cite{davison2022integrality}[Lemma 6.3]. Let $  \rho \in   \crit(\Tr^{\zeta}(\tilde{W}_{\bff})) \cap \Rep_{(\bd,1}^{\hat{\zeta}-\sss} (\tilde{Q}_{\bff}) \backslash \crit(\Tr^{\zeta}(\tilde{W}_{\bff})) \cap \pi_{Q_{\fr},\bd}^{-1}(\Rep^{\hat{\zeta}-\sss}_{(\bd,1)}(\overline{Q_{\bff}}))$. Then since $\Jac( \tilde{Q}_{\bff},\tilde{W}_{\bff}) \simeq \mathrm{\Pi}_{Q_{\bff}}[\omega]$, we have that $\rho = (\rho^{\prime}, f)$ where $\rho^{\prime}$ is a representation of $\mathrm{\Pi}_{Q_{\fr}}$ with an endomorphism $f \in \End_{\mathrm{\Pi}_{Q_{\bff}}}(\rho^{\prime})$. Our assumtion on $\rho$ is equivalent to  $\rho^{\prime}$ being not a semistable representation of $\mathrm{\Pi}_{Q_{\bff}}$. But then by the existence of Harder-Narasimhan filtrations, there exists a filtration of $\rho^{\prime}$ as a $\mathrm{\Pi}_{Q_{\bff}}$ module \[ 0 = \rho_0' \subset \rho_1' \subset \rho_2' \cdots \subset \rho_n= \rho^{\prime}, \] such that the slopes $\mu(\rho_i'/\rho_{i-1}')$ are strictly decreasing. Each of $\rho_i'/\rho_{i-1}'$ is $\hat{\zeta}$ semistable. Consider the restriction $f_{|\rho_1'}: \rho_1' \rightarrow \rho_n'$. Since $\mu(\rho_1')> \mu(\rho_n'/\rho_{n-1}')$, $f_{|\rho_1'} \colon \rho_1' \rightarrow \rho_n' \rightarrow \rho_{n}'/\rho_{n-1}'$ is the $0$ morphism. Thus $f_{|\rho_1'}$ factors through $\rho_{n-1}'$. Since $\mu(\rho_1') > \mu(\rho_{i}'/\rho_{i-1}')$ for all $i$, we may continue this argument inductively to conclude that $\rho_{1}'$ is invariant under $f$. Thus $(\rho_1',f_{|\rho_0'})$ forms a sub representation of $\rho$ with higher slope, contradicting the semistability of $\rho$.   
\end{proof}

Note that \[\pi_{Q_{\bff},\bd}^{-1}(\Rep^{\hat{\zeta}-\sss}_{(\bd,1)}(\overline{Q_{\bff}})) = \Rep^{\hat{\zeta}-\sss}_{(\bd,1)}(\overline{Q_{\bff}}) \times \bigoplus_{i \in Q_{\bff}} \Hom(\C^{\bd_i},\C^{\bd_i}) \] is a $\GL_{\bd}$ equivariant decomposition. Thus, by the above lemma and the dimensional reduction theorem (\ref{dimensionreduction}) applied in the case when \begin{align*}
\overline{X} &= \pi_{Q_{\bff},\bd}^{-1}(\Rep^{\hat{\zeta}-\sss}_{(\bd,1)}(\overline{Q_{\bff}})) \\  X &= \Rep^{\hat{\zeta}-\sss}_{(\bd,1)}(\overline{Q_{\bff}}) \\\mathbb{A}^n &= \bigoplus_{i \in Q_{\bff}} \Hom(\C^{\bd_i},\C^{\bd_i}),
\end{align*}it follows that
\[ \mathrm{DR}: \textrm{H}^{T}_{c}(\mathcal{M}^{\zeta,3d}_{\bd,\bff}(Q),\pPhi_{\Tr^{\zeta}(\tilde{W}_{\bff})} \Q_{\mathcal{M}^{\zeta}_{\bd,\bff}(Q)}[2(\bd,1)\cdot (\bd,1)])  \simeq  \textrm{H}^{T}_{c}(\Nak^{\zeta}_{\bd,\bff}(Q), \Q). \]

%%\[ \text{H}_{c}(\mathfrak{M}_{(\bd,1)}(\tilde{Q_{\bff}})^{\hat{\zeta}-\sss}, \phi_{\Tr(\tilde{W_{\bff}})}\Q) \simeq \text{H}_{c}(\mu^{-1}_{Q_{\bff},(\bd,1)}(0) \cap \Rep^{\hat{\zeta}-\sss}_{(\bd,1)}(\overline{Q_{\bff}})/\GL_{(\bd,1)}, \Q) \otimes \L^{((\bd,1)\cdot(\bd,1)} \] But since $\mu^{-1}_{Q_{\bff},(\bd,1)}(0) = \mu^{-1}_{\bff,\bd}$, we have \[ \text{H}_{c}(\mathfrak{M}_{(\bd,1)}(\tilde{Q_{\bff}})^{\hat{\zeta}-\sss}, \phi_{\Tr(\tilde{W_{\bff}})}\Q) \simeq \text{H}_{c}(M^{\zeta}_{(\bd,\bff)}(Q) \otimes H_{c}(\text{BG}_{m},Q) \otimes \L^{(\bd,1)\cdot(\bd,1)} \]

Thus, by applying Verdier duality, since $\dim(\mathcal{M}^{\zeta,3d}_{\bd,\bff}(Q) - \dim(\Nak^{\zeta}_{Q}(\bd,\bff) - 2(\sum_{i \in Q_{\bff}} \bd_i^2) = -1$, it follows that 

\begin{prop}\label{criticallocusnakajima} 
Let $\zeta \in \Q^{Q_0}$ by any stability condition such that $\hat{\zeta}$ is generic for dimension vector $(\bd,1)$. Then we have a natural isomorphism of graded vector spaces  
\begin{equation}
(-1)^{\binom{\bd}{2}}\mathrm{DR}: \HH^{T}(\mathcal{M}^{\zeta,3d}_{\bd,\bff}(Q), \pPhi_{\Tr^{\zeta}(\tilde{W_{\bff}})}\mathcal{IC}_{\mathcal{M}^{\zeta,3d}_{\bd,\bff}(Q)}[-1]) \simeq \HH^{T}(\Nak^{\zeta}_{\bd,\bff}(Q),\mathbb{Q}^{\vir}).
\end{equation}  

Note that here we have twisted dimension reduction isomorphism by the sign $(-)^{\binom{\bd}{2}}$, this is so that we have compatibility with the action of preprojective algebra, as we shall later see. 

%%\[ \text{H}(\mathcal{M}^{\hat{\zeta}-\sss}_{\bff,\bd}(Q), \phi_{\Tr(\tilde{W_{\bff}})} \mathcal{IC}_{\mathcal{M}^{\hat{\zeta}-\sss}_{\bff,\bd}(Q)})  \simeq \text{H}(M^{\zeta}_{(\bd,\bff)}(Q),\mathcal{IC}_{M^{\zeta}_{(\bd,\bff)}(Q)}) \] 
\end{prop}

\section{CoHA Action} \label{cohaactionequivarianthilbert}

In this section, we will construct the action of COHA on the Nakajima quiver variety for the stability condition $\zeta^{\infty}$. This is done in \cite{davison2022affine} in the formulation we explain next and in \cite{yang2017cohomological}, \cite{schiffmanncohagenerators} for preprojective CoHA's. We will also see the compatibility with the cup product of the tautological line bundle. Later, we will see how these ideas can be used to define an action of a subalgebra of CoHA for a different stability condition. 

Given two dimension vectors $\bd_1,\bd_2$ and framing $\bff$, let \[\mathcal{M}^{\zeta^{\infty},3d}_{(\bd_1,\bd_2,\bff)}(Q) = \{ 0 \rightarrow (\rho_1)_{( \bd_1,0)} \rightarrow (\rho_2)^{\zeta^{\infty}-\sss}_{(\bd_1+\bd_2,1)} \rightarrow (\rho_3)_{(\bd_2,1)} \rightarrow 0 \} \] denote the stack of short exact sequences where $\rho_2$ is $\zeta^{\infty}$ semistable. We then have the usual correspondence diagram
\[\begin{tikzcd}[ampersand replacement=\&]
	{} \& {\mathfrak{M}_{\bd_1}(\tilde{Q}) \times \mathcal{M}^{\zeta^{\infty},3d}_{\bd_2,\bff}(Q)} \& {\mathcal{M}^{\zeta^{\infty},3d}_{(\bd_1,\bd_2,\bff)}(Q) } \& {\mathcal{M}_{\bd_1+\bd_2,\bff}^{\zeta^{\infty},3d}(Q)}
	\arrow["{\pi_1 \times \pi_3 }", from=1-3, to=1-2]
	\arrow["{\pi_2}"', from=1-3, to=1-4],
\end{tikzcd}\] where the morphism $\pi_1 \times \pi_3$ is projection to $(\rho_1,\rho_3)$ and $\pi_2$ is projection to $\rho_2$. $\pi_2$ is a proper map, since the $\zeta^{\infty}$ stability condition implies that every $(\bd_2,1)$ dimensional quotient of a $(\bd_1+\bd_2,1)$ dimensional $\zeta^{\infty}$ semistable representation is also semistable. Then, by doing push-pull along vanishing cycle cohomology, it follows that we have a left action of $\mathcal{A}_{\tilde{Q},\tilde{W}}$ on 
\[ \bigoplus_{\bd \in \N^{Q_0}} \HH(\mathcal{M}^{\zeta^{\infty}, 3d}_{\bd,\bff}(Q), \pPhi_{\Tr(\tilde{W_{\bff}})} \mathcal{IC}_{\mathcal{M}^{\zeta^{\infty}-\sss}_{\bd,\bff}(Q)}[-1]) \simeq  \bigoplus_{\bd \in \mathbb{N}^{Q_0}} \mathrm{H}(\Nak^{\zeta^{\infty}}_{\bd,\bff}(Q),\Q^{\vir}) =: \Nak^{\zeta^{\infty}}_{\bff}(Q)\] 
Here, the second isomorphism is defined in Proposition \ref{criticallocusnakajima}. 

\begin{prop} \label{cohaactionnakajimainfinity}
For any quiver $Q$ and framing $\bff$, the cohomological Hall algebra $\mathcal{A}_{\tilde{Q},\tilde{W}}^{\chi}$ acts on the cohomology of Nakajima quiver variety $\Nak^{\zeta^{\infty}}_{\bff}(Q)$.
\end{prop}
The proof that the above defines an action is the same for associativity of cohomological Hall algebra as proved in \cite{kontsevich2011cohomological}. We denote this action by $\bullet$.

Since we have defined the action of CoHA, there is an action of $\mathfrak{g}^{\BPS}_{\tilde{Q},\tilde{W}}$ on $\Nak^{\zeta^{\infty}}_{\bff}(Q)$. In theorem \ref{dhsm1}, we saw that $\mathfrak{g}^{\BPS}_{\tilde{Q},\tilde{W}}$ can be realized as a positive Half of the Generalized Kac Moody Lie algebra. Infact the action of $\mathfrak{g}^{\BPS}_{\tilde{Q},\tilde{W}}$ can be lifted to action of full $\mathfrak{g}^{\GKM}_{Q}$.

\begin{thm}[\cite{davison2024bps}, {Theorem 10.5}]\label{dhsm2}
The $\mathfrak{g}^{\BPS}_{\tilde{Q},\tilde{W}}$ action on $\Nak^{\zeta^{\infty}}_{\bff}(Q)$ extends to $\mathfrak{g}^{\GKM}_{Q}$ action and there is decomposition of $\mathfrak{g}^{\GKM}_{Q}$ modules 
\[ \Nak^{\zeta^{\infty}}_{\bff}(Q) \simeq  \bigoplus_{(\bd,1) \in \Phi^{+}_{\mathrm{\Pi}_{Q_{\bff}}}} \HH^{*}(\mathcal{IC}(\mathcal{M}_{\mathrm{\Pi}_{Q_{\bff}},(\bd,1)})) \otimes L_{\mathfrak{g}^{\GKM}_{Q},\hat{\bd}}.\]
where $\hat{\bd} \in \Hom(\mathbb{Z}^{Q_{0}}, \mathbb{Z})$ given by $\bd^{\prime} \mapsto (\bd,\bd^{\prime})_{\mathrm{\Pi}_Q} - \bff \cdot \bd^{\prime}$. 
\end{thm}
\begin{example} \label{peiceofdhsm}
$(\mathbf{0},1)$ is always a real positive root. Thus $L_{\mathfrak{g}^{\GKM}_Q,\hat{\mathbf{0}}}$ where $\hat{\mathbf{0}}: \bd^{\prime} \mapsto \bff \cdot \bd^{\prime}$ is always a summand in $\Nak^{\zeta^{\infty}}_{\bff}(Q)$.
\end{example}

On the other hand, there is the action of preprojective CoHA $\mathcal{A}^{\chi}_{\mathrm{\Pi}_Q}$ on $\Nak^{\zeta^{\infty}}_{\bff}(Q)$ defined in \cite{Yang_2018}[Section 5] and \cite{schiffmanncohagenerators}[Section 5.6]. However, it is shown in \cite{Yang_2019}[Theorem A] that these actions are compatible under the twisted dimension reduction morphism $\tilde{\mathrm{DR}}$ in Section \ref{dimensionreductiontripled}.

Although we will not be pursuing this until section \ref{moyangianconjecture}, the perspective we want to give is that $\mathcal{A}^{\mathcal{\overline{N}}}_{\tilde{Q},\tilde{W}} \simeq \mathcal{A}^{\mathcal{N}}_{\mathrm{\Pi}_Q}$ should be thought of as the opposite half of the CoHA in the following way. First, it is shown in \cite{schiffmanncohagenerators} that we can lift the above action to the case of nilpotent quiver variety(One can also prove using the above construction). 
\begin{prop}[\cite{schiffmann2023cohomological}] \label{actionondual}
    For any quiver $Q$ and framing $\bff$, the nilpotent cohomological Hall algebra $\mathcal{A}^{\mathcal{N},\chi}_{\tilde{Q},\tilde{W}}$ acts on the cohomology of nilpotent Nakajima quiver variety\footnote{Note that the nilpotent quiver variety $\Lus^{\zeta^{\infty}}_{\bd,\bff}(Q)$ aren't always smooth.}
    \[ \mathrm{L}^{\zeta^{\infty},\mathcal{N}}_{\bff}(Q) := \bigoplus_{\bd \in \mathbb{N}^{Q_0}} \HH^{\mathrm{BM}}(\Lus^{\zeta^{\infty}}_{\bd,\bff}(Q),\mathbb{Q}[2(\bd,\bd)- 2\bff \cdot \bd]). \]
\end{prop}

Now, we claim that there is duality \[ \HH(\Nak^{\zeta^{\infty}}_{\bd,\bff}(Q),\mathbb{Q}^{\vir}) \simeq \HH^{\mathrm{BM}}(\Lus^{\zeta^{\infty}}_{\bd,\bff}(Q),\mathbb{Q}[2(\bd,\bd)- 2\bff \cdot \bd])^{*}\]

This is simply because the varieties $\Nak^{\zeta^{\infty}}_{\bd,\bff}(Q)$ and $\Lus^{\zeta^{\infty}}_{\bd,\bff}(Q)$ are homotopic. But $\Lus^{\zeta^{\infty}}_{\bd,\bff}(Q)$ is also a projective variety. Hence, its usual cohomology is the same as the compactly supported cohomology, which is Poincaré dual to the Borel-Moore homology. Thus by Proposition \ref{actionondual}, we get a \textit{right} action of $\mathcal{A}^{\mathcal{N}}_{\tilde{Q},\tilde{W}}$ on $\Nak^{\zeta^{\infty}}_{\bff}(Q)$.

\begin{remark}
Note that there is addition action of $\GL_{\bff}$ on $\mathcal{M}^{\zeta^{\infty},3d}_{\bd,\bd}(Q)$ for any quiver $Q$ which reparametrize the framing of $\tilde{Q}_{\bff}$. Then the action of CoHA also works equivariantly. We get a left action of $\mathcal{A}^{T}_{\tilde{Q},\tilde{W}}$ on \[\bigoplus_{\bd \in \N^{Q_0}} \HH^{T \times \mathrm{GL}_{\bff}}(\Nak^{\zeta^{\infty}}_{\bd,\bff}(Q),\mathbb{Q}^{\vir}).\] and similarly a right action of $\mathcal{A}^{T,\tilde{\mathcal{N}}}_{\tilde{Q},\tilde{W}}$. In the case when $T$ scales the symplectic form non-trivially, due to a theorem of \cite{botta2023okounkovs} and \cite{schiffmann2023cohomological}, these CoHA can be realized as positive and negative half of MO Yangian respectively, which implies that after taking direct sum across all framings, these actions are faithful(See Chapter \ref{moyangianconjecture}). However, faithfulness is conjectural when $T$ acts trivially on the symplectic form. 
\end{remark}
\begin{comment}

To see this, we consider the closed subvariety $\mathrm{L}^{\zeta^{\infty}}_{\bd,\bff}(Q) \subset \Nak^{\zeta^{\infty}}_{\bd,\bff}(Q)$ defined by  $q^{-1}(0)$ where $q: \Nak^{\zeta^{\infty}}_{\bd,\bff}(Q) \rightarrow \Nak^{0}_{\bd,\bff}(Q)$ is the GIT projective morphism defined in equation \ref{affinizationmorphism}. Note that since $q$ is a projective morphism, $\mathrm{L}^{\zeta^{\infty}}_{\bd,\bff}(Q)$ is always projective. Note that we can identify $\mathrm{L}^{\zeta^{\infty}}_{\bd,\bff}(Q)$ with representations $\rho$ where the underlying $\overline{Q}$ representation is nilpotent and $\rho(\alpha_{i,l}^{*})=0$ where for $i \in Q_0, 1 \leq l \leq \mathbf{f}_i$, $\alpha_{i,l}^{*}$ are arrows with source $i$ and target $\infty$. But then this is same as $\Nak^{\zeta^{\infty}}_{\bd,\bff}(Q)$ as if underlying $\overline{Q}$ representation is nilpotent, then

\end{comment}

\subsubsection{Compatibility with Cup product}
Note that, in the same way, we constructed the action of tautological bundles on $\mathcal{A}_{Q,W}$ in Section \ref{tautologicalbundle}, we can construct analogus action of the tautological line bundle $\otimes_{i \in Q_0} \mathcal{V}_i$ on $\mathrm{N}^{\zeta^{\infty}}_{\bff}(Q)$ by $u \cdot v := \sum_{i \in Q_0} c_1(\mathcal{V}_i) \cup v$. The following proposition will be useful later to understand the induced action of the affinized $\BPS$ Lie algebra.

\begin{prop} \label{uactionnakajima}
The action of $\mathcal{A}_{\tilde{Q},\tilde{W}}$ on $\mathrm{\Nak}^{\zeta^{\infty}}_{\bff}(Q)$ satisfies 
\[ (u \cdot \alpha) \bullet v = u \cdot (\alpha \bullet v) - \alpha \bullet (u \cdot v)   \]
\end{prop}

\begin{proof}
We have 
\begin{align*}
u \cdot (\alpha \bullet v) &= \sum_{i \in Q_0} c_1(\mathcal{V}_i) \cup ((\pi_2)_{*} (\pi_1 \times \pi_3)^{*}(\alpha \otimes v) \\ &= \sum_{i \in Q_0}  (\pi_2)_{*}((\pi_2)^{*}(c_1(\mathcal{V}_i)) \cup (\pi_1 \times \pi_3)^{*}(\alpha \otimes v)) \\ &= \sum_{i \in Q_0} (\pi_2)_{*}((\pi_1 \times \pi_3)^{*}((c_1(\mathcal{V}_i) \otimes 1 + 1 \otimes c_1(\mathcal{V}_i)) \cdot (\alpha \otimes v) ) \\& = \sum_{i \in Q_0}  (u \cdot \alpha) \bullet v +  \alpha \bullet (u \cdot v)
\end{align*}
where the third equality is true since for a short exact sequence $0 \rightarrow L_1 \rightarrow L_2 \rightarrow L_3 \rightarrow 0$, $c_1(L_2) = c_1(L_1)+c_1(L_3)$. 
\end{proof}

\section{Kleinian  singularity and McKay correspondence}
Let $G \subset \SL_2(\C)$ be a finite subgroup. The McKay correspondence associates to every finite subgroup of $\textrm{SL}_2(\C)$ a Dynkin graph of affine type ADE. Let $\rho_0,\rho_1, \cdots, \rho_{d-1}$ be the set of all the isomorphism classes of irreducible representations of the group $G$ where $\rho_0$ is assumed to be the trivial representation. Since $G \subset \SL_2(\C) \subset \GL_2(\C)$, we have the tautological representation of $G$ on $\C^2$. 
Let $a_{ij}$ be the multiplicities in the decomposition of the tensor product $\rho_i \otimes \C^2$ into the sum of irreducible representations, that is \[ \rho_i \otimes \C^2 \simeq \bigoplus_{j=0}^{d-1} \rho_{j}^{a_{ij}} . \] Then the graph with the set of vertices $0,1,\cdots, d-1$ and $a_{ij}$ edges between the vertices $i$ and $j$ is the corresponding McKay graph. We now consider the finite cylic group of order $K+1$ $\mathbb{Z}_{K+1} = \langle \omega \rangle \subset \SL_2(\C)$ embedded in $\SL_2(\C)$ by $\omega \mapsto \begin{pmatrix}
  \omega & 0 \\ 0 & \omega^{-1}
\end{pmatrix} $ where $\omega$ is a primitive $K+1$-th root of unity. The McKay graph corresponding to the group $\mathbb{Z}_{K+1}$ gives the doubled cyclic quiver $\overline{Q^{K}}$.

On the other hand, we may consider the affine quotient $\C^2/\Z_{K+1}$. It is a singular space with a singularity at the origin. Interestingly, the Nakajima quiver variety associated with the quiver cyclic quiver $Q^K$ gives a minimal resolution of this singularity. In the case we are interested in, these varieties can be understood as equivariant Hilbert schemes.  

\subsection{Equivariant Hilbert scheme}
Let $\Hilb(\C^2)$ denote the Hilbert scheme of points on $\C^2$. Since $\mathbb{Z}_{K+1}$ acts on $\C^2$, we have an induced action of $\Z_{K+1}$ on $\Hilb(\C^2)$. For any finite group $G$, let $R(G)$ be the monoid of the isomorphism classes of finite-dimensional representations of the group $G$. For every representation $V \in R(\mathbb{Z}_{K+1})$, let \[ \Hilb^{V}(\C^2) = \{ J \subset \C[x,y], \C[x,y]/J \simeq_{\mathbb{Z}_{K+1}} V \} \] where $\simeq_{\mathbb{Z}_{K+1}}$ denote isomorphism as representations of the finite group $\mathbb{Z}_{K+1}$. This gives a decomposition into components 
\begin{align}\label{components}
(\Hilb(\C^2))^{\Z_{K+1}} = \bigoplus_{V \in R(\Z_{K+1})} \Hilb^{V}(\C^2).
\end{align} where $(\Hilb(\C^2))^{\Z_{K+1}}$ is the fixed point locus of the $\mathbb{Z}_{K+1}$ action on $\Hilb(\C^2)$.  Note that since $\mathbb{Z}_{K+1}$ is abelian, all the irreducible representations are $1$ $1$-dimensional. Thus for every representation $V$, $V \simeq \oplus \rho_i^{\oplus \bd_i}$ as $\Z_{K+1}$ representations, where $\rho_i$ is the $1$ dimensional representation of weight $i$, given by $v \mapsto \omega^i v$. This gives an isomorphism of monoids $\bd: R(\Z_{K+1}) \mapsto \N^{K+1}$. 
\begin{example}
Consider the ideal $I$ of the functions which vanish at the points of a free $\Z_{K+1}$-orbit. The orbit consists of $K+1$ elements and the corresponding $\Z_{K+1}$-module $\C[x,y]/I$ is a regular representation $\rho_{\reg}$ of $\Z_{K+1}$. In fact $\bd(\rho_{\reg}) = \underbrace{(1,1,\cdots,1)}_{\text{$K+1$ times}} = \delta$. The action of $\Z_{K+1}$ on $\C^2$ is free outside the origin. The restriction of the Hilbert-Chow map (Proposition \ref{minimalresolution}) gives the minimal resolution $\Hilb^{\rho_{\reg}}(\C^2) \rightarrow \C^2/\Z_{K+1}$. 
\end{example}

\subsubsection{Tutological bundles}
By the way we have defined them, $\Hilb^{V}(\C^2)$ is an $\Z_{K+1}$ invariant subvariety of $\Hilb^{\dim(V)}(\C^2)$. Denote by $\O_{\C^2}^{n}$, the tautological rank $\dim(V)$ bundle on the Hilbert scheme $\Hilb^{\dim(V)}(\C^2)$ (See Definition \ref{hilbertschemetautological}). Consider the restriction $(\O_{\C^2}^{\dim(V)})|_{\Hilb^{V}(\C^2)}$ over $\Hilb^{V}(\C^2)$. Then, each fibre carries a structure of representation of $G$. We thus have \[ (\O_{\C^2}^n)_{\mid \Hilb^{V}(\C^2)} \simeq \bigoplus_{k \in \mathbb{Z}_{K+1}} \mathcal{H}_{k} \otimes \rho_k,\] giving tautological bundles $\mathcal{H}_k$ on $\Hilb^{V}(\C^2)$. Note that we can equivalently write $\mathcal{H}_k = \Hom(\rho_{i},(\O_{\C^2}^n)_{\mid \Hilb^{V}(\C^2)}) $

\subsubsection{Nakajima quiver variety description} \label{stabilityinfinity}
The equivariant Hilbert scheme is an example of a Nakajima quiver variety.

\begin{prop}[\cite{nakajima2002geometric}] \label{dimensionofnkajimaquiver}
 We have an isomorphism of smooth connected algebraic varieties  \[ \Hilb^{V}(\C^2) \simeq \Nak^{\zeta^{\infty}}_{\bd(V),\delta_0}(Q^{K}) \] and so their dimension is $ 2\bd_0(V) - \sum_{i \in \mathbb{Z}_{K+1}}(\bd_i(V)-\bd_{i+1}(V))^2$. 
\end{prop}

On closed points, given an ideal $J$, the above isomorphism sends $J$ to a representation of $\overline{Q^K_{\delta_0}}$, given by \[V_i = \Hom_{\mathbb{Z}_{K+1}}(\rho_i,\C[z_1,z_2]/J) \] for $i \in [0,K]$ and $V_{\infty} = \C$ and the morphisms between these vector spaces are given by composition with the morphisms $z_1: \C[z_1,z_2]/J \rightarrow \C[z_1,z_2]/J$ and $z_2: \C[z_1,z_2]/J \rightarrow \C[z_1,z_2]/J$ given by multiplication by $z_1$ and $z_2$ respectively. In particular, this means that the above isomorphism sends the tautological bundles $\mathcal{H}_i$ on $\Hilb^{V}(\C^2)$ to the tautological bundles $\mathcal{V}_i$ on the Nakajima quiver variety $\Nak^{\zeta^{\infty}}_{\bd(V),\delta_0}(Q^K)$.

\begin{convention}
Given any smooth variety $X$ and an irreducible closed subvariety $V$, let \[[V] \in \HH^{2(\dim(X)-\dim(V))}(X)\] be the fundamental class \cite{Anderson_Fulton_2023}[Appendix A]. 
\end{convention}
\subsubsection{Nakajima action} \label{nakajimliealgebraction}
We now briefly explain the action of the positive half $\mathfrak{n}_{K+1}^{+}$ of the affine Kac-Moody Lie algebra $\tilde{\mathfrak{sl}}_{K+1}$ on \[ \Nak^{\zeta^{\infty}}_{\delta_0}(Q^K) \simeq \bigoplus_{V \in R(\Z_{K+1})}  \text{H}(\Hilb^{V}(\C^2),\Q^{\vir}) \] due to \cite{Nak_quivers_kac_moody}\footnote{There is an action for general Nakajima quiver variety for quiver without loops. Here we only focus on a special case.}. For each $i$ and $V$, the Hecke correspondence is defined to be the subvariety
\[\mathfrak{B}_i(V) := \{ (I_1,I_2) \in \Hilb^{V - \rho_i}(\C^2) \times \Hilb^{V}(\C^2) | I_2 \subset I_1 \}  \subset \Hilb^{V - \rho_i}(\C^2) \times \Hilb^{V}(\C^2). \]
Note that $x,y$ act trivially on $I_1/I_2$. By \cite{Nak_quivers_kac_moody}[Section 5], $\mathfrak{B}_i(V)$ are known to be smooth Lagrangian submanifolds inside the product $\Hilb^{V-\rho_i}(\C^2) \times \Hilb^{V}(\C^2).$ Let $e_i$ be the standard Chevalley generators of $\mathfrak{n}^{+}_{K+1}$ as in (\ref{negativehalfLiealgebra}). We consider the usual correspondence diagram

\[\begin{tikzcd}[ampersand replacement=\&]
	{\mathfrak{B}_i(V)} \& {\Hilb^{V-\rho_i}(\C^2) \times \Hilb^{V}(\C^2)} \\
	{\Hilb^{V-\rho_i}(\C^2)} \&\& {\Hilb^{V}(\C^2)}
	\arrow[hook, from=1-1, to=1-2]
	\arrow["{p_1}"', from=1-2, to=2-1]
	\arrow["{p_2}", from=1-2, to=2-3]
\end{tikzcd}\]

Let $v \in \text{H}(\Hilb^{V-\rho_i}(\C^2),\Q^{\vir})$. The action of $\mathfrak{n}_{K+1}^{+}$ on $\Nak^{\zeta^{\infty}}_{\delta_0}(Q^K)$ is defined by 
\begin{align} \label{nakajimaformula}
e_i(v) & := (-1)^{d_{i}(V-\rho_i)+d_{i+1}(V-\rho_i)} (p_2)_{*}(p_1^{*}(v) \cup [\mathfrak{B}_i(V)]). 
\end{align}
Strictly speaking, in the original work, Nakajima constructed operators on a Lagrangian subvariety of $\Hilb^{V}(\mathbb{C}^2)$, which is dual to the cohomology of $\Hilb^{V}(\mathbb{C}^2)$. This formulation is present in the work of Nagao in \cite{nagao2007quiver}[Section 5.3.1]. Later in section \ref{moyangianconjecture}, we will see how these operators are generalized to the action of Yangians on the Equivariant cohomology of these varieties by Varagnolo \cite{varagnolo2000}. 

\subsubsection{Minimal resolution as Nakajima quiver variety} \label{componentsoffiber}
\begin{thm}[\cite{kronheimer1989construction}] \label{minimalresolution}
The affine Nakajima quiver variety $\Nak^{0}_{\delta,\delta_0}(Q^K) $ is isomorphic to the affine quotient $\C^2/\mathbb{Z}_{K+1}$ and the quiver variety $\Nak^{\zeta^{\infty}}_{\delta,\delta_0}(Q) \simeq \Hilb^{\rho_{\reg}}(\C^2)$ is the minimal resolution via the GIT map\footnote{There is more general result for any finite subgroup of $\mathrm{SL}_2(\mathbb{C})$.} \[ \pi \colon \Nak^{\zeta^{\infty}}_{(\delta,\delta_0)}(Q^K) \rightarrow \Nak^{\zeta^{0}}_{(\delta,\delta_0)}(Q^K) \simeq \C^2/\mathbb{Z}_{K+1}.\]
\end{thm}%%Thus $S_{K}:= M^{\zeta^{-}}_{\delta,\delta_0}$ is a smooth surface with the trivial canonical bundle and the tautological bundles $\mathcal{R}_i$ gives line bundles $\mathcal{L}_i$ on $S$. 

The exceptional fibre $\pi^{-1}(0)$ components can be understood as Nakajima correspondences defined in the previous section. Consider $\mathfrak{B}_{i}(\rho_{\reg})$ for $i \neq 0$. Then by Proposition \ref{dimensionofnkajimaquiver}, the dimension of $\Hilb^{\rho_{\reg}- \rho_i}(\C^2)$ is $0$ and thus $\Hilb^{\rho_{\reg}- \rho_i}(\C^2)= \{ \pt_i \}$, for some point $\pt_i \in \Hilb^{\rho_{\reg}- \rho_i}(\C^2)$. Thus by definition, we may view $\mathfrak{B}_i(\rho_{\reg})$ as a subvariety of $\Hilb^{\rho_{\reg}}(\C^2)$: \begin{equation} \label{identification} \mathfrak{B}_{i}(\rho_{\reg}) = \{\pt_i \subset I \} \subset \Hilb^{\rho_{\reg}} \simeq S_{K} \end{equation} 
The subvarieties $\mathfrak{B}_i(\rho_{\reg})$ are smooth Lagrangian manifold inside the surface $S_{K}$. These are isomorphic to $\mathbb{P}^1$ \cite{hakkaido}[Proposition 6.2].  In \cite{hakkaido}[Example 6.3], it is shown that the exceptional fiber $\pi^{-1}(0)$ is the union of $\mathfrak{B}_i(\rho_{\reg})$. For any representation $\rho$ of $\overline{Q^{K}_{\delta_0}}$ of dimension $(\bd,1)$ and for any arrow $a \in Q^{K}_{\delta_0}$, let $\rho_{a}: \C^{s(a)} \rightarrow \C^{t(a)}$ be the corresponding linear map and $\rho^{*}_{a}: \C^{t(a)} \rightarrow \C^{s(a)}$ be the map in the opposite direction. 
\begin{example}\label{examplelocus}
We can observe that $\mathfrak{B}_i(\rho_{\reg})$ correspond to the isomorphism classes of representations of dimension $\delta$ where $\rho_a = 1, \rho_a^{*}=0$ where $a:j-1 \rightarrow j$ for $j=1,\cdots, i-1$; $\rho_a = 0, \rho_a^{*}=1$ where $a: j-1 \rightarrow j$ for $j=i+2,\cdots, K+1$ and $\rho_{a}=1,\rho_{a}^{*}=0$ where $a: \infty \rightarrow 0$, while the choice of the maps $\rho_{i-1 \rightarrow i}$ and $\rho_{i \rightarrow i+1}^{*}$ correspond to $\mathbb{P}^1$ (See Figure \ref{figure:1} for $K=3,i=2$). 
\end{example}

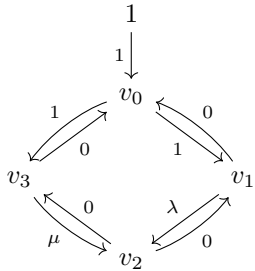
\begin{figure}[h!]
    \centering
    \[\begin{tikzcd}[ampersand replacement=\&] 
	\& 1 \\
	\& {v_0} \\
	{v_3} \&\& {v_1} \\
	\& {v_2}
	\arrow["1"', from=2-2, to=3-3]
	\arrow["0"', curve={height=6pt}, from=3-3, to=2-2]
	\arrow["0"', from=4-2, to=3-1]
	\arrow["0"', curve={height=6pt}, from=4-2, to=3-3]
	\arrow["\mu"', curve={height=6pt}, from=3-1, to=4-2]
	\arrow["1"', curve={height=6pt}, from=2-2, to=3-1]
	\arrow["\lambda"', from=3-3, to=4-2]
	\arrow["0"', from=3-1, to=2-2]
	\arrow["1"', from=1-2, to=2-2]
\end{tikzcd}\]
    \caption{$\mathfrak{B}_i(\delta)$ for $K=3,i=2$}
    \label{figure:1}
\end{figure}

The cohomology $\HH(S_K,\Q)$ is spanned by $[S_K] \in \HH^{0}(S_K,\Q)$ and irreducible components of exceptional fibers $[\mathfrak{B}_{i}(\rho_{\reg})] \in \HH^{2}(S_K,\Q)$. On the other hand, since Nakajima quiver varieties are tautologically generated \cite{McGerty_2017}, it follows that the Chern class of the tautological bundles $c_1(\mathcal{R}_i)$ also spans $\HH^{2}(S_K,\Q)$. For any finite subgroup $G \subset \SL_2(\C)$, the geometric McKay correspondence \cite{ASENS_1983_4_16_3_409_0} gives a duality between the irreducible components of the exceptional fiber $\pi: S_G \rightarrow \C^2/G$ and the tautological bundles associated to the irreducible representation of the group $G$. Reinterpretating this in terms of the Nakajima quiver variety gives
\begin{prop}\label{duality}
For all $i \in \{1,\cdots,K \}$, we have  \[ [\mathfrak{B}_i(\rho_{\reg})] \simeq \sum_{k=1}^{K} c_{ik} c_1(\Vcal_k), \] where $c_{ik} \in \End_{K}(\mathbb{Z})$ is the Cartan matrix for the $A_K$ quiver, i,e $c_{ik}= 2\delta_{ik} - \delta_{i,k-1}-\delta_{i,k+1}$ where $1 \leq i,j \leq K$.  
\end{prop}
\begin{proof}
For each edge $i$ where $ i \neq 0$, we have tautological morphisms $\rho_{i,i+1}: \mathcal{V}_i \rightarrow \mathcal{V}_{i+1}$ and $\rho_{i-1,i}^{*}: \mathcal{V}_{i} \rightarrow \mathcal{V}_{i-1}$. So the preprojective relations give rise to a complex of tautological bundles \[ \Vcal_i \xrightarrow[]{\rho_{i-1,i}^{*} \oplus  \rho_{i,i+1}}   \Vcal_{i-1} \oplus \Vcal_{i+1} \xrightarrow[]{\rho_{i-1,i} - \rho_{i,i+1}^{*}} \Vcal_i \] Thus we have a map between vector bundles \[  \mathcal{V}_i \xrightarrow[]{\rho_{i-1,i}^{*}  \oplus \rho_{i,i+1}} \textrm{Ker}(\rho_{i-1,i} - \rho_{i,i+1}^{*}). \]
Consider the locus $Z_{i}$ where the above map is not an isomorphism. It follows from Example \ref{examplelocus} or \cite{hakkaido}[Example 6.3] that $Z_i = \mathfrak{B}_{i}(\rho_{\reg})$. But note that the morphism \[ \Vcal_{i-1} \oplus \Vcal_{i+1} \xrightarrow[]{\rho_{i-1,i} - \rho_{i,i+1}^{*}} \Vcal_i\] is a a surjection. Since it is not a surjection, then it must be $0$. But then this gives a subrepresentation of $\overline{Q}^K$ which is not $\infty$ generated, contradicting the stability. Thus, this morphism is a surjection. But this implies that the fundamental class $[\mathfrak{B}_i(\rho_{reg})] = 2 c_1(\mathcal{V}_i)- c_1(\mathcal{V}_{i+1})-c_1(\mathcal{V}_{i-1})$.
\end{proof}

By Proposition \ref{cohaactionnakajimainfinity}, we have action of $\bps{K}$ on 
\[ \Nak^{\zeta^{\infty}}_{\delta_0}(Q^K) = \bigoplus_{\bd \in \N^{Q_0}} \text{H}(\mathcal{M}^{\zeta^{-}-\sss}_{\bd,\fr}(Q), \pPhi_{\Tr(\tilde{W_{\bff}})} \mathcal{IC}_{\mathcal{M}^{\hat{\zeta}-\sss}_{\bd,\fr}(Q^K)}[1]) \simeq  \bigoplus_{V \in R(\Z_{K+1})}  \text{H}(\Hilb^{V}(\C^2),\Q^{\vir}) .\]

The induced action of the $\mathfrak{n}_{Q^{K}}^{+} \subset \mathfrak{g}^{\BPS}_{\QTC{K},\WTC{K}}$ on $ \Nak^{\zeta^{\infty}}_{\delta_0}(Q^K)$ can be identified with Nakajima's operators $e_i$. We have  
\begin{prop}[] \label{relationwithnakajimaoperators}
 The action of generator $\alpha_i \subset \mathfrak{g}^{\BPS}_{\QTC{K},\WTC{K},\delta_i}$ 
\[ \alpha_i:  \HH(\Hilb^{V- \rho_i}(\C^2),\Q^{\vir}) \rightarrow  \HH(\Hilb^{V}(\C^2),\Q^{\vir}), \] coincides with the action of Nakajima operator $e_i$.

\end{prop}

\begin{proof}
    It is shown in \cite{Yang_2019}[Theorem 5.1], that dimension reduction from $\mathcal{A}^{\chi}_{\tilde{Q},\tilde{W}} \simeq \mathcal{A}^{\chi}_{\mathrm{\Pi}_Q}$ is compatible with action on Nakajima quiver variety. Furthermore, in \cite{Yang_2018}[Theorem 5.6], it is shown that the action of $x_{i,1}^{0} \in \mathbb{C}[x_{i,1}] = \mathcal{A}^{\chi}_{\mathrm{\Pi}_Q,\delta_i}$ coincides with the action of Nakajima operators. 
\end{proof}

We also consider the action of any non-zero $\gamma_{\delta} \in \mathfrak{g}^{\BPS,-2}_{\QTC{K},\WTC{K},\delta}$, this make sense since by Proposition \ref{bpsLiealgebracyclic}, $\mathfrak{g}^{\BPS,-2}_{\QTC{K},\WTC{K},\delta}$ is a 1 dimensional vector space and $u \cdot \gamma_{\delta}$ on $ \Nak^{\zeta^{\infty}}_{\delta_0}(Q^K)$. Choose a non zero element $\ket{0} \in \HH(\Hilb^{0}(\C^2),\Q^{\vir}) \subset  \Nak^{\zeta^{\infty}}_{\delta_0}(Q^K)$. We then have 
\begin{prop} \label{imaginaryactiononequivarianthilbert}
For some non-zero $\lambda^{\prime} \in \Q$, we have an equality
\[\gamma_{\delta}\ket{0} = \lambda^{\prime} [S_K]\] and \begin{align*}
(u \cdot \gamma_{\delta}) \ket{0} = \lambda^{\prime} \left( \sum_{i=1}^{K} c_1(\Rcal_i) \right) = \lambda^{\prime} \left( \sum_{i=1}^{K} \sum_{j=1}^{K} A_{ij} [\mathfrak{B}_j(\rho_{\reg})] \right)
\end{align*}
where $A = C^{-1}$, is the inverse of the Cartan matrix of $A_K$.

\end{prop}
\begin{proof}
The action of $\gamma_{\delta}$ gives a map \[ \gamma_{\delta} \colon \HH(\Hilb^{0}(\C^2),\Q)) \rightarrow \HH(\Hilb^{\rho_{\reg}}, \Q^{\vir}) \simeq   \Q[2] \oplus \Q^{K}[0]\]

Since the element $\gamma_{\delta}$ is of cohomological degree $-2$, we have $\gamma_{\delta}\ket{0} = \lambda^{\prime} [S]$ for some $\lambda^{\prime}$. We claim that $\lambda^{\prime} \neq 0$.

By Theorem \ref{dhsm1}, the action of $\mathfrak{g}^{\BPS}_{\tilde{Q^K},\tilde{W^K}}$ lifts to an action of double $\mathfrak{g}^{\GKM}_{Q^K}$. By Example \ref{peiceofdhsm}, $\Nak^{\zeta^{\infty}}_{\delta_0}(Q^K)$ contains the Lowest weight module $L_{\mathfrak{g}^{\GKM}_{Q^K},\hat{\mathbf{0}}}$ of $\mathfrak{g}^{\GKM}_{Q^K}$ where $\hat{\mathbf{0}}: \bd^{\prime} \mapsto - \delta_0 \cdot \bd^{\prime}$. In Example \ref{dhsmcyclic}, we calculated $\mathfrak{g}^{\GKM}_{Q^K}$ and in particular we can identity $\gamma_{\delta}$ with Chevalley generator $e_{\delta}$. We also have annihilation operator $f_{\delta}$ and element $h_{\delta}$ in the Cartan. We can identify the vaccum element $\ket{0}$ as $1 \otimes 1 \in L_{\mathfrak{g}^{\GKM}_{Q^K},\hat{\mathbf{0}}}$. Since $L_{\mathfrak{g}^{\GKM}_{Q^K},\hat{\mathbf{0}}}$ is a Lowest weight module, by definition $f_{\delta} \ket{0} = 0$ and $h_{\delta} \ket{0}= - \delta_i \cdot \delta  \ket{0} = - \ket{0}$. But then $e_{\delta}  \ket{0}$ can't be $0$, since otherwise $h_{\delta}\ket{0} = [e_{\delta},f_{\delta}] \ket{0} = 0$. Thus $e_{\delta} \ket{0} \neq 0$ and hence $\gamma_{\delta} \ket{0} \neq 0$ and thus $\lambda^{\prime} \neq 0$. 

Next, by Proposition \ref{uactionnakajima}, it follows that $(u \cdot \gamma_{\delta}) \star [0] = u \cdot ([S]) = c_1(\mathcal{R})= \sum_{i=1}^{K} c_1(\mathcal{R}_i)$. Then, the claim follows from Proposition \ref{duality}. 

\end{proof}

\begin{comment}
The inverse of the Cartan matrix of $A_K$ can be described explicitly. We note down the matrix.

\begin{prop}
    The inverse of the Cartan matrix of $A_K$ is given by Matrix $A_{k,j} = \min(k,j)-\frac{kj}{K+1}$.
\end{prop}

\end{comment}

\subsection{Hilbert schemes}
For the minimal resolution $S_{K} \rightarrow \C^2/\Z_{K+1}$, we consider the $\N \times \mathbb{Z}$ graded vector space 
\[ V_K := \bigoplus_{n \in \mathbb{N}} \HBM(\Hilb^{n}(S_K),\Q^{\vir}) \] and introduce a similarly defined action of a subalgebra of $\mathcal{A}_{\QTC{K},\WTC{K}}$ and study their relationship with Nakajima's operators. Note that since $\Hilb^n(S_K)$ is smooth, there is an isomorphism between the Borel-Moore homology and the usual cohomology $\HBM(\Hilb^{n}(S_K),\Q^{\vir}) \simeq \HH(\Hilb^{n}(S_K),\Q^{\vir})$.
\subsubsection{Nakajima operators} \label{Nakajimaoperatorhilbertscheme}
In \cite{nakajima1995heisenberg}, for any quasi-projective surface $S$, the author constructs the action of certain operators $p_m(\alpha)$ for $m \in \Z $ and $\alpha \in \HH(S,\Q)$ on \[ V(S) := \bigoplus_{n \in \N} \HBM(\Hilb^n(S),\Q^{\vir}) \] such that they satisfy infinite dimensional Heisenberg Lie algebra type relations. We explain the construction for $m>0$, since that's what we will be using later. Let \[ \Hilb^{n,n+m}(S) = \{ I \subset J, \Supp(I/J) = \pt \in  S  \}  \subset S \times \Hilb^{n+m}(S) \times \Hilb^{n}(S).  \] Note that, we have a map $s: \Hilb^{n,n+m}(S) \rightarrow S$ given by taking $s(I \subset J) = \Supp(I/J)$. For any $I \subset \{1,2,3 \}$, let $\pi_I$ be the projection maps from $S \times \Hilb^{n+m}(S)  \times \Hilb^{n}(S) $ to components in $I$. Then \[ p_m(\alpha): \HBM(\Hilb^{n}(S),\Q^{\vir}) \rightarrow \HBM(\Hilb^{n+m}(S),\Q^{\vir}) \] is defined to be $p_m(\alpha)(c) = (\pi_2)_{*}((\pi_{1,3})^{*}((PD([S]) \cap \alpha) \otimes c) \cap [\Hilb^{n,n+m}(S)])$ where $PD: \HH(S,\Q^{\vir}) \simeq \HBM(S,\Q^{\vir})$ is the Poincare duality map. The action of $p_{m}(\alpha)$ for $m <0$ is defined similarly. Furthermore, the action of operators $p_n(\alpha)$ for any $n$ and $\alpha$ is faithful on $V(S)$ \cite{nakajima1995heisenberg}. 

\subsubsection{Lehn and Lin-Qing-Wang relations}

In \cite{Lehn_1999}, the author constructs operators on the Hilbert schemes of points on surfaces by combining Nakajima's operators and the action of the Chern classes of tautological bundles. 

\begin{defn}[Tautological bundles on Hilbert scheme of points] \label{hilbertschemetautological}

Let $F$ be any locally free sheaf on the smooth quasi-projective surface $S$. For any $n \geq 0$, let $Z^{n}_S \subset \Hilb^{n}(S) \times S$ be the universal family. The associated tautological bundle $F^{[n]}$ is defined to be $(\pi_1)_{*}(\O_{Z^{n}_S} \otimes (\pi_2)^{*}F)$ where $\pi_1,\pi_2$ are the projection maps from $\Hilb^{n}(S) \times S$. It is a locally free sheaf of rank $rn$ on $\Hilb^{n}(S)$ where the rank of $F$ is $r$. By definition, the fiber $F^{[n]}(Z)$ for $Z$ a point in $\Hilb^{n}(S)$ can be naturally identified with $\HH^{0}(Z,F \otimes \O_Z)$. When the context is clear, we sometimes abuse notation and refer to the associated tautological bundle $F^{[n]}$ as $F$ itself. 
\end{defn}

\begin{defn}[Lehn operators]
    Given cohomology class $\alpha \in \HH(S,Q)$, let $p_n(\alpha)$ be the Nakajima operator. Then we define an operator $p_n^{1}(\alpha) = [\partial, p_n(\alpha)]$ on $V(S)$ where $\partial \cdot v = c_1(\O_{S}^{[m]}) \cup v$ for any $v \in \HBM(\Hilb^m(S),\Q^{\vir}) \subset V(S)$ and more generally we define \[p_n^{k}(\alpha) := \textrm{ad}^{k}_{\partial}(p_n(\alpha)) = [\underbrace{\partial, [\partial,[ \cdots , [\partial}_{k \textrm{ times }}, p_n(\alpha)]]]]. \] 
\end{defn}

For a smooth projective surface $S$, the relations between the operators $[p_n^{k_n}(\alpha),p_m^{k_m}(\beta)]$ have been computed in \cite{Lehn_1999}[Theorem 3.10] and \cite{li2002hilbert}[Theorem 5.5]. The formula involves contributions of the first Chern of the canonical bundle $c_1(\mathcal{K}_S)$ and the Euler class $c_2(\mathcal{T}_S)$.
However, the surface $S_K$ is only \textit{quasi}-projective in our setting. However, due to the result of \cite{McGerty_2017}, the cohomology of $S_K$ is still pure. Furthermore, $S_K$ is Calabi-Yau, i.e. $\mathcal{K}_{S_K} = \O_{S_K}$. Let $i: S_{K} \hookrightarrow \overline{S_K}$ be a smooth compactification. Then the pullback $\HH^{*}(\Hilb^{n}(\overline{S_K}),\Q) \rightarrow \HH^{*}(\Hilb^{n}(S_K),\Q)$ is a surjection of cohomology rings. Furthermore, the Chern classes of tautological bundles are supported on the pure cohomology, so that $c_i(V) \cup \alpha = c_i(V)|_{S_K} \cup \alpha$. Since $i^{*}(c_1(\mathcal{K}_{\overline{S_K}})) = c_1(\mathcal{K}_{S_K}) = c_1(\mathcal{O}_{S_K})= 0$ and $i^{*}(c_2(\mathcal{T}_{\overline{S_K}})) = 0$ by cohomological degree reasons, we conclude from \cite{Lehn_1999}[Theorem 4.2] and \cite{li2002hilbert}[Theorem 5.5, 4.2, Definition 5.1] that:

\begin{prop}\label{Lehnweicommute}
For any $v \in V_K$, we have:  
\begin{enumerate}
    \item For any line bundle $F$ on $S_K$,  we have:  \[ [c_1(F), p_1(\alpha)] v = p_1^{1}(\alpha) v+ p_1(c_1(F)\alpha) v. \]
    \item  For any $n_1,n_2 \geq 1$, $a_1,a_2 \geq 0$ and $\alpha,\beta \in \HH(S_K,\Q)$, we have: 
    \begin{align*}
    [p^{a_1}_{n_1}(\alpha),p^{a_2}_{n_2}(\beta)] v = n_1^{a_1}n_2^{a_2}\frac{a_2 n_1-n_2 a_1}{(n_1+n_2)^{a_1+a_2-1} } p^{a_1+a_2-1}_{n_1+n_2}(\alpha \beta) v.
    \end{align*}
\end{enumerate}
\end{prop}

\subsubsection{CoHA subalgebra action}

We now consider a subalgebra \[\mathcal{A}_{\QTC{K},\WTC{K}}^{\Im} := \bigoplus_{n \in \Z_{\geq 0} \cdot \delta} \mathcal{A}_{\QTC{K},\WTC{K},n \cdot \delta}\] of  $\mathcal{A}_{\QTC{K},\WTC{K}}$. 
We show that there is an action of $\mathcal{A}_{\QTC{K},\WTC{K}}^{\Im}$ on $V_K$. 
To do so, we use the work in \cite{kapranov1999kleinian}, which describes $\Hilb^n(S_K)$ as the Nakajima quiver variety on a cyclic quiver with a different stability condition than what we considered before.

\begin{defn}[Kuznetsov's stability condition] \label{kapranovstability}
For any $n \geq  1$, given a dimension vector $n \cdot \delta$, we consider the stability condition $\zeta^{n}$ on the cyclic quiver $Q^{K}$ defined by $\zeta^{n}_i = -2n$ for all $i=1, \cdots,K$ and $\zeta^{n}_0 = 2Kn-1$.    
\end{defn}

Note that this stability condition lies in a chamber different from $\zeta^{-}$ \cite{craw2021introduction}[Theorem 4.2]. 
We now show that $\zeta^{n}$ stability fits the criterion to apply the Proposition \ref{criticallocusnakajima}. 

\begin{prop}
The stability condition $\hat{\zeta^{n}}$ is generic for the dimension $(n \cdot \delta,1)$. 
\end{prop}
\begin{proof}
We have to show that $\mu_{\hat{\zeta^n}}(\bd^{\prime}) \neq \mu_{\hat{\zeta^n}}((n\delta,1))$ for any $\bd^{\prime} < (n \delta,1)$.
    \begin{itemize}
        \item Assume that $\bd^{\prime} = (\bd,0)$ for some $\bd < n \delta$. Then we have $\sum_{i=0}^{K} \zeta_i^{n} \bd_i=0$. Thus \begin{equation} \label{equationstability}
        (2Kn-1) \bd_0 = 2n \left( \sum_{i=1}^{K} \bd_i \right).
        \end{equation} This implies that $2n$ divides $\bd_0$, which is not possible since $\bd_0 \leq n$ unless $\bd_0=0$. But if $\bd_0=0$ then by the equation \ref{equationstability}, $\sum_{i=1}^{K} \bd_i=0$ and so each of $\bd_i=0$. \item Now assume $\dim(\rho^{\prime}) = (\bd,1)$ for some $\bd < n \delta$. Then again, by the same argument, we see that it implies $\bd= n \cdot \delta$, which contradicts that $\rho^{\prime}$ is a subrepresentation. 
    \end{itemize}
\end{proof}

The reason why we consider this peculiar stability condition is due to the following theorem.
%Note that this stability condition is also generic in the sense \cite{kapranov1999kleinian}
\begin{prop}{\cite{kuznetsov2001quiver} [Theorem 4.3]} \label{Kuznetsovtheorem} 
For any $n \geq 0$, 
there is a canonical isomorphism of varieties \[g_n: \Hilb^{n}(S_{K}) \simeq \Nak^{\zeta^{n}}_{n \cdot \delta, \delta_0}(Q^{K}). \] 
\end{prop}

Let $\mathcal{V}_i^{\zeta^n}$ be the tautological bundles on the Nakajima quiver variety $\Nak^{\zeta^{n}}_{n \cdot \delta,\delta_0}(Q^K)$. Note that since $S_K \simeq \Nak^{\zeta^{\infty}}_{\delta, \delta_0}(Q^K)$ there are tautological bundles $\mathcal{V}_i$ on $S_K$, which by Definition \ref{hilbertschemetautological} gives tautological bundles $[\mathcal{V}_i]^{[n]}$ on $\Hilb^{n}(S_K)$. We then have 
\begin{prop} \label{pullbackofdeterminantlinebundleonhilbertscheme}
Under the above isomorphism, the tautological vector bundles $\Vcal^{\zeta^n}_i$ on $\Nak^{\zeta^{n}}_{n \cdot \delta,\delta_0}(Q^K)$ are pulled back to vector bundles $[\Vcal_i]^{[n]}$ on $\Hilb^n(S_K)$. 
\end{prop}

\begin{proof}
This is really in the construction of the isomorphism. Let $Z$ be a closed point in $\Hilb^n(S_K)$. It is a length $n$ subscheme in $S_K$. Let $\O_Z$ be its structure sheaf. Then the map $g_n$ associates a closed point on $\Nak^{\zeta^{n}}_{n \cdot \delta,\delta_0}(Q^K)$ by assigning vector space $V_i(Z) := \HH^{0}(S_K, \Vcal_i \otimes \O_Z)$ to the vertex $i$ where $i=0,\cdots,K$, $\HH^{0}(S_K,\mathcal{O}_{S_K})$ to the vertex $\infty$. Since the fiber of $(\mathcal{V}_i)^{[n]}$ on any point $Z$ is exactly $\HH^{0}(S_K, \mathcal{V}_i \otimes \O_Z)$ the statement follows.  %and the homomorphism to each arrow $h: i \rightarrow j$ by the map $\HH^{0}(S_K, \Rcal_i \otimes \O_Z) \rightarrow \HH^{0}(S_K, \Rcal_j \otimes \O_Z)$ induced by maps $B_h \otimes \id : \Rcal_i \otimes \O_Z \rightarrow \Rcal_j \otimes \O_Z$ while the morphism $V_{\infty} \rightarrow V_0$ is induced by $\O_S \rightarrow \O_Z$ and the morphism $V_{0} \rightarrow V_{\infty}$ is 0
\end{proof}

By Proposition \ref{criticallocusnakajima}, there is an isomorphism of graded vector spaces: 

\[ \bigoplus_{ n \geq 0}  \mathrm{H}(\mathcal{M}^{\zeta^{n},3d}_{n \cdot \delta,\delta_0}(Q^K), \pPhi_{\Tr^{\zeta^n}(\tilde{W}_{\delta_0})} \mathcal{IC}_{\mathcal{M}^{\zeta^{n},3d}_{n \cdot \delta, \delta_0}(Q^K)}[-1])  \simeq \bigoplus_{n \geq 0}  \text{H}^{\text{BM}}(\Nak^{\zeta^{n}}_{(n \cdot \delta,\delta_0)}(Q^K),\Q^{\vir}) \simeq V_K. \]

We now mimic the construction as in Section \ref{cohaactionequivarianthilbert}. For any $n,m \geq 0$, we consider the stack of short exact sequences \[\mathcal{M}^{\zeta^{n+m},3d}_{n \cdot \delta, m \cdot \delta,\delta_0}(Q^K) = \{ 0 \rightarrow (\rho_1)_{(n \cdot \delta,0)} \rightarrow (\rho_2)^{\zeta^{n+m}-\sss}_{((m+n) \cdot \delta,1)} \rightarrow (\rho_3)_{(m \cdot \delta,0)} \rightarrow 0 \}, \]
where $\rho_1,\rho_2,\rho_3$ are representations of $\tilde{Q}_{\delta_0}$ of dimension $(n\cdot \delta,0),((m+n) \cdot \delta,1), (m \cdot \delta,1)$ respectively such that $\rho_2$ is $\zeta^{n+m}$-semistable. We then have \begin{lemma} \label{stablesurjective}
Let $\rho$ be a representation of dimension $((n+m) \cdot \delta,1)$ such that $\rho$ is $\zeta^{n+m}$-semistable then any quotient $\rho^{\prime}$ of dimension $(m \cdot \delta,1)$ is $\zeta^m$-semistable.  
\end{lemma}

\begin{proof}
Let $\pi: \rho \rightarrow \rho^{\prime}$ be the surjection. Let $V \subset \rho^{\prime}$ be a non-zero proper subrepresentation such that $V \neq \rho^{\prime}$.  Define $\chi_{m}((\bd,a)) = (\sum_{i=0}^{K} \zeta^m_i \bd_i) + \zeta^{m}_{\infty} a$. Recall(Section \ref{semistability}) that the slope of the representation $V$ is defined as $\mu_{\zeta^m}(V) = \chi_{m}((\bd,a))/(\sum_{i=1}^{i=K} \bd_i + a)$. Thus, it suffices to show that $\chi_{m}(V) \leq  0$. We consider the representation $\pi^{-1}(V)$. 
\begin{itemize}
    \item Assume that $\dim(V) = (\bd, 1)$. Then $\dim(\pi^{-1}(V)) = ( \textbf{d}+n \cdot \delta,1)$ and is a non-zero proper subrepresentation of $\rho$. Since $\rho$ is $\zeta^{n+m}$-semistable, it follows that $\chi_{n+m}(\pi^{-1}(V)) <0$. So %\\ &= \sum_{i>0}^{K} (-2(n+m))(\bd_i-m) + (2nK+2mK-1)(\bd_0-m) 
    \begin{align*}
    0 > \chi_{n+m}(\pi^{-1}(V)) &= \sum_{i=0}^{K} \left( \zeta^{n+m}_i(\bd_i+n) - \zeta^{m+n}_i (m+n) \right) \\& = \sum_{i=1}^{K} \left( -2(n+m)(\bd_i-m) \right) + (2K(m+n)-1)(\bd_0-m)  \\&=  -2(n+m) \left(\sum_{i =1}^{K}(\bd_i-\bd_0) \right)  +m - \bd_0
    \end{align*} 
    But then since $V \subset \rho^{\prime}$ is a proper subrepresentation, $\bd < m \cdot \delta$. The above inequality implies that 
    \[ \sum_{i =1}^{K} (\bd_i-\bd_0) > (m-\bd_0)/2(n+m) >0 \] But since $\sum_{i=1}^{K} (\bd_i-\bd_0)$ is an integer, above inequality implies that in fact $\sum_{i =1}^{K} (\bd_i-\bd_0) \geq 1$. 
     So \begin{align*}
    \chi_{m}(V) &= \sum_{i=0}^{K} \zeta_i^{m}(\bd_i-m) \\ &= (-2m) \sum_{i=1}^{K}(\bd_i-\bd_0) + m-\bd_0  \leq  -m-\bd_0 <0. 
    \end{align*} 
    \item Now assume that $\dim(V) = ( \textbf{d},0)$. Then again  % &= \sum_{i>0}^{K} (-2(n+m))(\bd_i)+ (2(n+m)K-1)(\bd_0) + n(\sum_{i=0}^{K} \zeta_i^{n+m}) \\
    \begin{align*}
    0 > \chi_{n+m}(\pi^{-1}(V)) &= \sum_{i=0}^{K} \zeta^{n+m}_i(\bd_i+n) \\&= \sum_{i=1}^{K} \left( -2(n+m)(\bd_i+n) \right) + (2K(n+m)-1)(\bd_0+n) \\& = -2(n+m)(\sum_{i=1}^{K} (\bd_i-\bd_0)) -n- \bd_0
    \end{align*} and so $\sum_{i=1}^{K} (\bd_i-\bd_0) > -(n+\bd_0)/(2(n+m)) \geq -(n+m)(2(n+m))> -1/2$, since $\bd_0 \leq n$.  But then since $\sum_{i=1}^{K} (\bd_i-\bd_0) $ is an integer, in fact we have $\sum_{i=1}^{K} (\bd_i-\bd_0) \geq 0$. Now 
    \begin{align*}
    \chi_{m}(V) = \sum_{i=0}^{K} \zeta_i^{m}(\bd_i) =  -2m( \sum_{i =1}^{K}(\bd_i-\bd_0))- \bd_0 \leq 0.
    \end{align*} 
\end{itemize}
So for every non-zero proper subrepresentation $V \subset \rho^{\prime}$, $\chi_m(V) \leq 0$ and so $\rho^{\prime}$ is $\zeta^{m}$-semistable. 
\end{proof}
 Thus we have the usual correspondence \[\begin{tikzcd}[ampersand replacement=\&]
	{} \& {\mathfrak{M}_{n \cdot \delta}(Q^K) \times \mathcal{M}^{\zeta^{m},3d}_{m \cdot \delta,\delta_0}(Q^K)} \& {\mathcal{M}^{\zeta^{n+m},3d}_{n \cdot \delta, m \cdot \delta,\delta_0}(Q^K) } \& {\mathcal{M}_{(n+m) \cdot \delta, \delta_0}^{\zeta^{n+m},3d}(Q^K)}
	\arrow["{\pi_1 \times \pi_3 }", from=1-3, to=1-2]
	\arrow["{\pi_2}"', from=1-3, to=1-4]
\end{tikzcd}.\]

The map $\pi_1 \times \pi_3$ is smooth since it is a fibre bundle. Unlike the $\zeta^{-}$ stability condition, it is not entirely obvious why the morphism $\pi_2$ is a proper map. However, by the previous lemma, fibres of the map $\pi_2$ are Quot schemes, and so, in particular, the morphism $\pi_2$ is proper. Thus again, just as in the definition of cohomological Hall algebra, we do the pushforward (\ref{Pushforward}) along $\pi_2$ and the pullback (\ref{Pullback}) along $\pi_1 \times \pi_3$ in the vanishing cycle cohomology and finally apply the Thom-Sebastiani isomorphism (\ref{ThomSebastini}) to get that 
%We then have canonical maps 

%\[ \alpha: \phi_{\Tr(W)} \Q_{\mathfrak{M}_{n \cdot \delta}(Q) \times \mathcal{M}^{\zeta^{-}}_{(\delta_0,(m \cdot \delta))}(Q)} \rightarrow \phi_{\Tr(W)}(\pi \times \pi_3)_{*} \Q_{\mathcal{M}^{\zeta^{-}}_{\delta_0,m \cdot \delta, n \cdot \delta}(Q) } \rightarrow (\pi \times \pi_3)_{*} \phi_{\Tr(W)}\Q_{\mathcal{M}^{\zeta^{-}}_{\delta_0,m \cdot \delta, n \cdot \delta}(Q) }  \]

%and similarly

%\[ \beta: (\pi_2)_{*} \phi_{Tr(W)} \Q_{\mathcal{M}^{\zeta^{-}}_{\delta_0,m \cdot \delta, n \cdot \delta}(Q) } \rightarrow  \phi_{Tr(W)} \Q_{\mathcal{M}^{\zeta^{-}}_{\delta_0,(m+n) \cdot \delta}(Q)}   \]

%Composition of $H(\alpha)$ and $H(\beta)$ with Thom-Sebastiani gives
\begin{prop}\label{leftactionhilbertscheme}
There is a left action $\sqbullet$ of $\mathcal{A}^{\Im}_{\QTC{K},\WTC{K}}$ on the cohomologically graded vector space $V_K$. It satisfies $(u \cdot \alpha) \sqbullet v = u \cdot (\alpha \sqbullet v) - \alpha \sqbullet (u \cdot v)$ where action of $u$ on $v \in \HBM(\Hilb^n(S_K),\Q^{\vir}) \subset V_K$ is given by cup product with $\sum_{i=0}^{K} c_1(\mathcal{V}^{[n]}_i)$. 
\end{prop}

\begin{proof}
The associativity of the action can be proven in the same way as the associativity of the cohomological Hall algebra in  \cite{kontsevich2011cohomological}. The action of $u$ is induced by the determinant line on the Nakajima quiver variety realization in Proposition \ref{Kuznetsovtheorem}. Thus, the claim follows from Proposition \ref{pullbackofdeterminantlinebundleonhilbertscheme}. 
\end{proof}
Thus we have an induced left action of $\mathfrak{g}^{\Im}_{\QTC{K},\WTC{K}}$ on $V_K$ where \[ \mathfrak{g}^{\Im}_{\QTC{K},\WTC{K}} := \bigoplus_{n \geq 1} \mathfrak{g}^{\BPS}_{\QTC{K},\WTC{K}, n \cdot \delta}.\]
\begin{prop}\label{faithfulhilbertscheme}
The action of $ \mathfrak{g}^{\Im}_{\QTC{K},\WTC{K}}$ on the graded vector space $V_K$ is faithful. 
\end{prop}

\begin{proof}
We first claim that there is isomorphism of graded vector spaces \[\HBM(\Hilb^n(S_K),\Q^{\vir}) \simeq \mathfrak{g}^{\BPS}_{\tilde{Q^K_{\delta_0}},\tilde{W^K_{\delta_0}},(n \cdot \delta,1)}.\]
We have a commutative diagram
\[\begin{tikzcd}[ampersand replacement=\&]
	{\mathfrak{M}^{\zeta^n-\sss}_{(n \cdot \delta,1)}(\tilde{Q^K_{\delta_0}})} \& {\mathfrak{M}_{(n \cdot \delta,1)}(\tilde{Q^K_{\delta_0}})} \\
	{\mathcal{M}^{\zeta_n,3d}_{n \cdot \delta,\delta_0}(Q^K)} \& {\mathcal{M}_{(n \cdot \delta,1)}(\tilde{Q^K_{\delta_0}})}
	\arrow["{\textrm{JH}^{\zeta^n}_{(n \cdot \delta,1)}}", from=1-1, to=2-1]
	\arrow[hook, from=1-1, to=1-2]
	\arrow["{q_{(n\cdot \delta,1)}^{\zeta^n}}", from=2-1, to=2-2]
	\arrow["{\textrm{JH}_{(n \cdot \delta,1)}}", from=1-2, to=2-2],
\end{tikzcd}\]

where $q$ is the affinization map and $\JH$ are the Jordan-Holder maps. Note that $\JH^{\zeta^n}_{(n\cdot \delta,1)}$ is a $\pt/\C^{*}$ fibre bundle and thus we have
\[(\JH^{\zeta^n}_{(n \cdot \delta,1)})_{*}\pPhi_{\Tr^{\zeta^n}(\tilde{W^K_{\delta_0}})} \mathcal{IC}_{\mathfrak{M}^{\zeta^n-\sss}_{(n \cdot \delta,1)}(\tilde{Q^K_{\delta_0}})} \simeq \pPhi_{\Tr^{\zeta^n}(\tilde{W^K_{\delta_0}})} \mathcal{IC}_{\mathcal{M}^{\zeta^n,3d}_{n \cdot \delta,\delta_0}(Q^K)}[-1] \otimes \HH(\BCu, \Q).\]

So, in particular, 
\begin{align*}\label{bpssheafsemistable}
\mathcal{BPS}^{\zeta^n}_{\tilde{Q^K_{\delta_0}},\tilde{W^K_{\delta_0}}} = {}^{\mathfrak{p'}}\!\tau^{\leq 1}((\JH^{\zeta^n}_{(n \cdot \delta,1)})_{*}\pPhi_{\Tr^{\zeta^n}(\tilde{W^K_{\delta_0}})} \mathcal{IC}_{\mathfrak{M}^{\zeta^n-\sss}_{(n \cdot \delta,1)}(\tilde{Q^K_{\delta_0}})}[1]) \simeq \pPhi_{\Tr^{\zeta^n}(\tilde{W^K_{\delta_0}})} \mathcal{IC}_{\mathcal{M}^{\zeta^n,3d}_{n \cdot \delta,\delta_0}(Q^K)}.
\end{align*}
By the theorem of Toda (Proposition \ref{todatheorem}), we have a natural isomorphism 
\begin{align*}
\HH(\mathcal{M}^{\zeta^n,3d}_{n \cdot \delta,\delta_0}(Q^K), \mathcal{BPS}^{\zeta^n}_{\tilde{Q^K_{\delta_0}},\tilde{W^K_{\delta_0}}}[-1])  \simeq \mathfrak{g}^{\BPS}_{\tilde{Q^K_{\delta_0}},\tilde{W^K_{\delta_0}},(n \cdot \delta,1)}. 
\end{align*}Finally, by Proposition \ref{criticallocusnakajima}, there is an isomorphism 
\[  \HH(\mathcal{M}^{\zeta^n,3d}_{n \cdot \delta,\delta_0}(Q^K), \pPhi_{\Tr^{\zeta_n}(\tilde{W^K_{\delta_0}})} \mathcal{IC}_{\mathcal{M}^{\zeta^n}_{n \cdot \delta,\delta_0}(Q^K)}[-1]) \simeq \HBM(\Nak^{\zeta^n}_{(n \cdot \delta,\delta_0)}(Q^K),\Q^{\vir}). \] and thus the claim follows from Proposition \ref{Kuznetsovtheorem}.

So we may identify the action of $\mathfrak{g}^{\Im}_{\QTC{K},\WTC{K}}$ on $V_K \simeq \oplus_{n \geq 0} \mathfrak{g}^{\BPS}_{\tilde{Q^K_{\delta_0}},\tilde{W^K_{\delta_0}},(n \cdot \delta,1)}$ by the natural identification $\mathfrak{g}^{\Im}_{\QTC{K},\WTC{K}} \simeq  \oplus_{n \geq 0} \mathfrak{g}^{\BPS}_{\tilde{Q^K_{\delta_0}},\tilde{W^K_{\delta_0}},(n \cdot \delta,0)}$. Let $\alpha \in \mathfrak{g}^{\BPS}_{\tilde{Q^K_{\delta_0}},\tilde{W^K_{\delta_0}},(m \cdot \delta,0)}$ for some $m \geq 0$. By Theorem \ref{dhsm1}, we may identify $\mathfrak{g}^{\BPS}_{\tilde{Q^K_{\delta_0}},\tilde{W^K_{\delta_0}}}$ as a positive Half of the generalized Kac Moody Lie algebra $\mathfrak{g}^{\GKM}_{\mathrm{\Pi}_{Q^K_{\delta_0}}}$. Note that $(0 \cdot \delta,1)$ is a primitive positive root of $\mathfrak{g}^{\GKM}_{\mathrm{\Pi}_{Q^K_{\delta_0}}}$, since the Symmetrized Euler form $((0 \cdot \delta,1),(0 \cdot \delta,1))_{\mathrm{\Pi}_{Q^K_{\delta_0}}}=2$. Thus there exist Chevalley generators for dimension $(0 \cdot \delta,1)$ in the subspace $\oplus_{n \geq 0} \mathfrak{g}^{\BPS}_{\tilde{Q^K_{\delta_0}},\tilde{W^K_{\delta_0}},(n \cdot \delta,1)} \subset  \mathfrak{g}^{\GKM}_{\mathrm{\Pi}_{Q^K_{\delta_0}}}$. We denote them by $e_{\infty},f_{\infty}$ and $h_{\infty}$ where $e_\infty \in \mathfrak{g}^{\BPS}_{\tilde{Q^K_{\delta_0}},\tilde{W^K_{\delta_0}},(0\cdot \delta,1)}$. We claim that $[\alpha,e_{\infty}] \neq 0$, which proves the assertion. Note that $[\alpha,f_{\infty}] = 0$ as there are no elements in dimension $(m \cdot \delta,-1)$. Also by definition of Generalized Kac Moody Lie algebra, $[\alpha, h_{\infty}] = ((m \cdot \delta,0),(0 \cdot \delta,1))_{\mathrm{\Pi}_{Q^K_{\delta_0}}} = - h_{\infty} \neq 0$. But then if $[\alpha,e_{\infty}]=0$, then $[\alpha,[f_{\infty},e_{\infty}]]= 0 = [\alpha,h_{\infty}]=0$ which we checked is non-zero. 
\begin{comment}Now one can proceed as in the proof of the faithfulness of $\mathfrak{g}^{\BPS}_{\tilde{Q},\tilde{W}}$ on the Nakajima quiver variety with $\zeta^{-}$ stability condition \cite{davison}[Corollary 10.6]. Let $e_{\infty} \in \mathfrak{g}^{\BPS}_{\tilde{Q^K_{\fr}},\tilde{W^K_{\fr}},(0,1)}$ be the lowest weight vector. 

\textcolor{red}{assuming compatibility}

Then it follows that for any $\alpha \in \mathfrak{g}^{\BPS}_{\QTC{K},\WTC{K}, n \cdot \delta}$, we have $[e_{\infty},\alpha] \neq 0$. But $e_{\infty} \cdot \alpha =0$ as \textcolor{red}{ss}
za%Then theorem of Toda tells us that $q^{\zeta_n}_{*} \mathcal{BPS}^{\zeta_n}_{Q_{\fr},W_{\fr},(1,n\cdot \delta)} \simeq \mathcal{BPS}_{Q_{\fr},W_{\fr},(1,n \cdot \delta)}$.

\end{comment}

\end{proof}

\begin{prop}\label{actionforn=1}
Up to a non-zero scalar, the operator $\gamma_{\delta}$ acts on $V_K$ as the Nakajima raising operator $p_1([S_K])$.  
\end{prop}

\begin{proof}
Applying dimension reduction (\ref{dimensionreduction}) by forgetting loops, i.e. the morphism of representation spaces $\Rep_{\delta}(\tilde{Q}) = \Rep_{\delta}(\overline{Q^K}) \times \Rep_{\delta}(L^K) \rightarrow  \Rep_{\delta}(\overline{Q^K})$, where $L^K$ is the quiver obtained by removing all the arrows $a$ and $a^{*}$ from the tripled quiver $\tilde{Q^K}$, The action of $\mathfrak{g}^{\BPS}_{\QTC{K},\WTC{K},\delta}$ on $V_K$ gives a map of cohomologically graded vector spaces \[ a_{1,n}: \HBM(\mathfrak{M}_{\delta}(\mathrm{\Pi}_{Q^K}),\Q^{\vir}) \times \HBM(\Hilb^n(S_K),\Q^{\vir}) \rightarrow \HBM(\Hilb^{n+1}(S_K),\Q^{\vir}). \]  We consider the stability condition $\zeta^{1}$(Definition \ref{kapranovstability}) so that $\mathfrak{M}_{\delta}^{\zeta^1}(\mathrm{\Pi}_{Q^K}) \simeq \Coh_{\pt}(S) \simeq S/\C^{*}$. Then by Toda's theorem (Proposition \ref{todatheorem}),  $\mathfrak{g}^{\BPS,\zeta^1}_{\QTC{K}, \WTC{K}, \delta} \simeq \mathfrak{g}^{\BPS}_{\QTC{K},\WTC{K}\delta}$, we have an isomorphism of vector spaces $\mathfrak{g}^{\BPS}_{\QTC{K},\WTC{K},\delta} \simeq \HBM(S_K,\Q^{\vir})$. By observing the cohomological degree, we can identify $\gamma_{\delta}$ with $[S_K] \in \HBM(S_K,\Q^{\vir})$. Let  \[\mathcal{M}_{(n,n+1)}^{\text{st}}(\mathrm{\Pi}_Q) : = \{ 0 \rightarrow (\rho_1)_{\delta} \rightarrow (\rho_2)_{(n+1) \cdot \delta,1} \rightarrow (\rho_3)_{(n \cdot \delta,1)} \rightarrow 0 \} \]  be the variety of short exact sequences on representations of $\mathrm{\Pi}_{Q_{\delta_0}}$ such that  $\rho_1$ is $\zeta^1$-semistable, $\rho_2$ is $\zeta^{n+1}$-semistable and $\rho_3$ is $\zeta^n$ stable. Given a point $(Z_1 \subset Z_2) \in \Hilb^{n,n+1}(S_K)$, it induces a short exact sequence of sheaves  \[ 0 \rightarrow \O_{s} \rightarrow \O_{Z_2} \rightarrow \O_{Z_1} \rightarrow 0\] for some $s \in S_K$ which equivalently after applying $g$ from Proposition \ref{Kuznetsovtheorem}, gives an exact sequence of representations
\[ 0 \rightarrow \widehat{g_1(s)} \rightarrow g_{n+1}(Z_2) \rightarrow g_n(Z_1) \rightarrow 0\] where $\widehat{g_1(s)}$ is the kernel of the surjection $g_1(s) \rightarrow \rho_0$ where $\rho_0$ is unique $(\textbf{0},1)$ dimensional representation. The same can be done the other way, i.e., given a short exact sequence $0 \rightarrow \rho_1 \rightarrow \rho_2 \rightarrow \rho_3 \rightarrow 0$ in $\mathcal{M}^{\mathrm{st}}_{(n,n+1)}(\mathrm{\Pi}_{Q^K})$, we can apply the inverse of isomorphism $g$ to get a closed point $g_n^{-1}(\rho_3) \subset g_{n+1}^{-1}(\rho_2)$ of $\Hilb^{n,n+1}(S_K)$. Thus, it follows that $\Hilb^{n,n+1}(S_K) \simeq \mathcal{M}^{\text{st}}_{(n,n+1)}(\mathrm{\Pi}_Q)$. We then consider the commutative diagram  \[\begin{tikzcd}[ampersand replacement=\&]
	{S_K \times \Hilb^n(S_K) } \& {\mathcal{M}^{\text{st}}_{(n,n+1)}(\mathrm{\Pi}_Q) \simeq \Hilb^{n,n+1}(S_K) } \& {\Hilb^{n+1}(S_K)} \\
	\& {S_K \times \Hilb^{n+1}(S_K) \times \Hilb^{n}(S_K)}
	\arrow["{p_1 \times p_3}"', from=1-2, to=1-1]
	\arrow["{p_2}", from=1-2, to=1-3]
	\arrow["i"', hook, from=1-2, to=2-2]
	\arrow["{\pi_{1,3}}"{pos=0.6}, from=2-2, to=1-1]
	\arrow["{\pi_2}"'{pos=0.6}, from=2-2, to=1-3]
\end{tikzcd} \] where the morphism $p_1 \times p_3$ is projection to $(\rho_1,\rho_3)$ and $p_2$ is projection to $\rho_2$. In particular, $p_1 \times p_3$ is already quasi-smooth, and $p_2$ is proper. However, since dimension reduction is compatible with the action of the cohomological Hall algebra of preprojective algebra on Nakajima quiver varieties \cite{Yang_2019}\footnote{In particular, since $\pi_1 \times \pi_3$ is smooth, we use the compatibility between dimension reduction with pullback and pushforward.}), we have 
\begin{align*}a_{1,n}([S_K] \otimes \alpha) = (p_2)_{*}(p_1 \times p_3)^{!}([S_K] \otimes \alpha) &= (\pi_2)_{*}(\pi_{1,3}^{*}([S_K] \otimes \alpha) \cap [\Hilb^{n,n+1}(S_K)]) \\ &=  p_1([S_K]) \alpha, \end{align*} as we wanted.

%the action is equivalent to the action
%\[ \HBM(\Coh_{\pt}(S),\Q^{\vir}) \times \HBM(\Hilb^n(S),\Q^{\vir}) \rightarrow \HBM(\Hilb^{n+1}(S),\Q^{\vir})\] A
%and let $\mathcal{L}$ be the universal line bundle of $\Hilb^{n,n+1}$. 

%Note that $\Hilb^{n,n+1}(S) \simeq $

%\[\begin{tikzcd}[ampersand replacement=\&]
%	\& {\mathcal{L}} \\
%	S \& {\Hilb^{[n,n+1]}(S)} \& {\Hilb^{n+1}(S)} \\
%	\& {\Hilb^n(S)} \\
%	\& {}
%	\arrow["\rho", from=2-2, to=2-1]
%	\arrow["\phi"', from=2-2, to=3-2]
%	\arrow["\psi"', from=2-2, to=2-3]
%	\arrow[dotted, from=1-2, to=2-2]
%\end{tikzcd}\]

%Then the action of $\alpha z^i$ where $\alpha \in \HBM(S,\Q^{\vir})$ can be reinterpreted as \[ (\alpha z^i) \cdot v  = \psi_{*}(\rho^{*}(\alpha) \cdot c_{1}(\mathcal{L})^{i}) \cap \phi^{!}(\alpha))\] which is exactly the Nakajima's action for $i=0$ and the statement follows. 

\end{proof}

\section{Integral Matrix $\mathcal{W}_{1+\infty}$ Lie algebras} \label{algebras}

In the introduction, we defined the associative algebras $D_{\hbar}(\C^{*})$  and $D_{\hbar}(\C^{*}) \otimes \mathfrak{gl}_K$ for any $K \geq 1$ (Definition \ref{matrixwLiealgebra}). We further defined a subspace $\mathcal{W}_K \subset (D_{\hbar}(\C^{*}) \otimes \mathfrak{gl}_{K}) \otimes_{\mathbb{C}[\hbar]} \C[\hbar^{\pm 1}]$. We check that 
\begin{prop}\label{integralsubLiealgebra}
The subspace $\mathcal{W}_K \subset (D_{\hbar}(\C^{*}) \otimes \mathfrak{gl}_{K}) \otimes_{\mathbb{C}[\hbar]} \C[\hbar^{-1}]$ forms a Lie subalgebra.
\end{prop}

\begin{proof}
We recall that $\mathcal{W}_K$ is the $\C[\hbar]$ linear subspace spanned by $T_{m,a}(X)= z^m D^a X$ where $m \in \mathbb{Z}, a \geq 0$, and $t_{m,a} = T_{m,a}(1)/\hbar$ where $1 \in \mathfrak{gl}_K$ is the identity matrix. 
Since $[D, z] = \hbar z$ and $[D,z^{-1}]= -\hbar z^{-1}$ it follows inductively that 
\begin{align}\label{commutor1}
[z^{m}D^a,z^{n}D^{b}] = z^{m+n}((D+n \hbar)^{a}D^b - D^a(D+m \hbar)^b).
\end{align} Since the expression on the right is of the form $\hbar \cdot (\textrm{some linear expression in } t_{p,q})$, it shows that the Lie brackets $[t_{m,a},t_{n,b}]$ are closed. We can similarly check for the Lie brackets of type $[t_{m,a},T_{n,b}(X)]$ since
\begin{align}\label{commutor2}
[z^{m}D^a,z^{n}D^{b}X] = z^{m+n}((D+n \hbar)^{a}D^b - D^a(D+m \hbar)^b)X.
\end{align}
The other Lie brackets are clearly closed since $D_{\hbar}(\C^{*}) \otimes \mathfrak{gl}_{K}$ forms a Lie algebra. 
\end{proof}

We now define the positive half of $\mathcal{W}_K$, by mimicking the definition of standard positive half of the affine Lie algebra. 

\begin{defn}[Positive half of $\mathcal{W}_K$] \label{positivehalfofdeformedmatrixdifferentialoperators}

The positive half $(\mathcal{W}_K)^{+}$ is defined to be $\C[\hbar]$ linear subalgebra generated by $T_{k,a}(X)$ where $k \geq 1, a \geq 0, X \in \mathfrak{gl}_{K}$ or $k=0, a \geq 0, X \in \mathfrak{n}_{K}$ where $\mathfrak{n}_{K} \subset \mathfrak{gl}_{K}$ is the Lie subalgebra generated by upper diagonal matrices and $t_{k,a}$ where $k \geq 1, a \geq 0$. 
\end{defn}

We will also consider the classical limit of $\mathcal{W}_K$. Note, however, that the way it is defined, we cannot directly set $\hbar=0$ since $t_{n,a} = T_{n,a}(1)/\hbar$. But using relations (\ref{commutor1}) and (\ref{commutor2}) and more generally
\begin{align}\label{commutor3}
[z^{m}D^aX,z^{n}D^{b}Y] = z^{m+n}(D+n \hbar)^{a}D^b XY - z^{m+n}D^a(D+m \hbar)^bYX,
\end{align}
for any $X,Y \in \mathfrak{gl}_K$, it follows that 
\begin{prop} \label{generatorrelationswkliealgebra}
Let $\tilde{W}_{K}$ be the Lie algebra generated by $T_{k,a}(X)$ where $X \in \mathfrak{sl}_K$ and $t_{k,a}$ where $k \in \mathbb{Z}$ and $a \geq 0$ with the relations 
\begin{align}\label{Relations_full_Liealgebra}
[t_{m,a},T_{n,b}(X)] &= (na-mb)T_{m+n,a+b-1}(X) \\ [t_{m,a},t_{n,b}] &= (na-mb)t_{m+n,a+b-1} \\ [T_{m,a}(X),T_{n,b}(Y)] &= T_{m+n,a+b}([X,Y])
\end{align} 
Then there is an isomorphism of Lie algebras \[ (\mathcal{W}_{K})/(\hbar=0) \simeq \tilde{W}_K \]

\end{prop}

Recall the Lie algebra $\mathfrak{po}(\TT^{*}(\C^{*})) \ltimes (\O(\TT^{*}(\C^{*}))$ defined in Section \ref{matrixwLiealgebra}. We note that \[\mathfrak{po}(\TT^{*}(\C^{*})) \ltimes (\O(\TT^{*}(\C^{*})) \otimes \mathfrak{sl}_{K})) \simeq \tilde{W}_K, \] since $\mathfrak{po}(\TT^{*}(\C^{*}))$, the Lie algebra of the functions on the cotangent bundle $\TT^{*}(\C^{*})$ is generated by $D = z \frac{\partial}{\partial z}$ and $z^{\pm 1}$ with the Poisson bracket, determined by $\{D,z \}=z$. It is spanned by $z^kD^a$ where $k \in \mathbb{Z}, a \geq 0$ and the morphism \[\mathrm{A} \colon \mathfrak{po}(\TT^{*}(\C^{*})) \ltimes (\O(\TT^{*}(\C^{*})) \rightarrow \tilde{W}_K, \] given by $\mathrm{A}(z^{k_1}D^{a_1}, z^{k_2}D^{a_2}X) \mapsto t_{k_1,a_1} + T_{k_2,a_2}(X)$ is an isomorphism of Lie algebras. We now define the positive half of $\tilde{W}_K$ in the same way as before. 

\begin{defn}[Positive half of $\tilde{W}_K$] \label{positivehalfofclassicallimit}
Let $\tilde{W}_K^{+} \subset \tilde{W}_K$ be the Lie subalgebra spanned by $T_{k,a}(X)$ where $k \geq 1, a \geq 0, X \in \mathfrak{sl}_{K}$ or $k=0, a \geq 0, X \in \mathfrak{n}_{K}$ where $\mathfrak{n}_{K} \subset \mathfrak{sl}_{K}$ is the Lie subalgebra generated by upper diagonal matrices and $t_{k,a}$ where $k \geq 1, a \geq 0$. 
\end{defn}

We now show that $\tilde{W}_K$ is, in fact, a Lie algebra in the category of $\textbf{Heis}$ modules. Let \[H_K = \sum (i-1/2-K/2)E_{i,i} \] be a diagonal element inside $ \mathfrak{sl}_{K}$. We defined it so that $[H_{K},E_{i,j}] = (i-j)E_{i,j}$ for $E_{i,j} \in \mathfrak{sl}_K$.

\begin{defn}[Grading on $\tilde{W}_K^{+}$]
The Lie algebra $\tilde{W}_K^{+}$ has two gradings. We define the \textit{cohomological grading} by $\mathrm{CG}(T_{m,a}(X)) = 2a, \mathrm{CG}(t_{m,a}) = 2a-2$ and \textit{dimension grading} induced by $\dim (T_{m,a}(E_{ij})) = Km+j-i$ and $\dim (t_{m,a}) = Km$. 
\end{defn}

\begin{prop} \label{wKnegativelydetermined}
    There is an action of $\Heis$ on $\tilde{W}_{K}^{+}$ by derivations, given by 
    \begin{align*}
        p & \rightarrow \left[\frac{t_{0,2}}{2} - \frac{ T_{0,1}(H_K)}{K}, - \right] \\ 
        q & \rightarrow  \partial_{D}
    \end{align*}
    where $\partial_{D}(T_{m,a}(X))= aK T_{m,a-1}(X), \partial_{D}(t_{m,a})= aK t_{m,a-1}$. For every graded piece of fixed dimension $\dim$, the central charge is the dimension $\dim$ , and the Lie algebra $\tilde{W}_K^{+}$ is negatively determined. 
\end{prop}

\begin{proof}
Since $p$ acts as a commutator, it is a derivation. The same can be checked for $q$ using Lie algebra relations (\ref{Relations_full_Liealgebra}). On the other hand, 
    \begin{align*}
    [q,p](T_{m,a}(E_{i,j})) &= (mK+(j-i))T_{m,a}(E_{i,j}) \\ [q,p](t_{m,a}) &= mKt_{m,a}
    \end{align*} 

The negatively determined assumption holds since the $\Heis$ grading is the cohomological grading and $\partial_{D}$ is injective for $a \geq 1$. 
    
\end{proof}

Next, we see that $\tilde{W}_{K}^{+}$ is generated by a much smaller set of generators.

\begin{prop} \label{sphericalgenerationofLK+1Liealgebra}
The Lie algebra $\tilde{W}_K^{+}$ is generated as a Lie algebra by the elements $t_{1,a}$ for $a \geq 0$, $T_{0,a}(E_{i,i+1})$ for $i=1,2, \cdots, K-1; a \geq 0$ and $T_{1,a}(E_{1,K}), a \geq 0$. 

\end{prop}

\begin{proof}
Since $[t_{1,0},t_{n,b}]= (-b)t_{n+1,b-1}$, the elements $t_{n,a}, n \geq 1; a \geq 0$ are in the subalgebra generated by $t_{1,a}$ where $a \geq 0$. The elements $E_{i,i+1}$ for $i \in [1,K-1]$ and $zE_{K,1}$ are the generators of the positive half of the affine Lie algebra $\tilde{\mathfrak{sl}}_{K}$, which is  $\mathfrak{n}_{K} \oplus z \Q[z] \mathfrak{sl}_{K}$. Since $[T_{m,a}(X),T_{n,b}(Y)] = T_{m+n,a+b}([X,Y])$, it follows that all $T_{k,a}(X)$ where $k \geq 1, a \geq 0, X \in \mathfrak{sl}_{K}$ or $k=0, a \geq 0, X \in \mathfrak{n}_{K}$ are in the algebra generated by $T_{0,a}(E_{i,i+1})$ for $i=1,2, \cdots, K-1; a \geq 0$ and $T_{1,a}(E_{1,K}), a \geq 0$ and so we are done. 
\end{proof}

\section{Affinized BPS Lie algebra of Cyclic Quivers} \label{sectionaffinized}
We are ready to calculate the affinized BPS Lie algebra $\abps{K}$ for $K \geq 1$. The affinized BPS Lie algebra for the cyclic quiver $Q^K$ is spanned by $\alpha_{\bd}^{(n)}$ where $\alpha_{\bd} \in \bps{K}$. It suffices to understand the Lie bracket between these generators. By Proposition \ref{bpsLiealgebracyclic}, we have an isomorphism of Lie algebras $\bps{K} \simeq \mathfrak{n}_{Q^{K}}^{+} \oplus s \Q[s]$. Let $\mathfrak{re}_{K} \subset \abps{K}$ be the subvector space spanned by the elements $\alpha_{\bd}^{(n)}$ for $n \geq 0$, where $\alpha_{\bd} \in \mathfrak{g}^{\BPS,0}_{\QTC{K},\WTC{K}} \simeq \mathfrak{n}^{+}_{Q^{K}}$. It follows from Theorem \ref{sphericalsubalgebra} that 
\begin{prop} \label{realLiesubalgebra}
    We have an isomorphism of Lie algebras \[ \mathfrak{re}_{K} \simeq \mathfrak{n}^{+}_{Q^{K}}[D]. \]
\end{prop}

Next, let $\mathfrak{im}_K \subset \abps{K}$ be the subvector space generated by the elements $\alpha^{(n)}_{m \cdot \delta}$ where $\alpha_{m \cdot \delta}$ spans $ \mathfrak{g}^{\BPS,-2}_{\tilde{Q},\tilde{W}, m \cdot \delta} \simeq \C$. Since $ \mathfrak{g}^{\BPS,-2}_{\tilde{Q},\tilde{W}, m \cdot \delta}$  is one dimensional, $\alpha_{m \cdot \delta}$ is uniquely determined up to a non-zero scalar. 

\begin{prop} \label{relation1}
For any $n \geq 1$, there exist some non-zero scalar $\lambda_n \in \Q$ such that we have an equality \[[\alpha_{\delta}^{(1)},\alpha_{n \cdot \delta}] = \lambda_n \alpha_{(n+1) \cdot \delta}. \]
\end{prop}

\begin{proof}
Note that the Lie bracket $[\alpha_{\delta}^{(1)}, \alpha_{n \cdot \delta}]$ is an element of cohomological degree $-2$ in the dimension $(n+1) \cdot \delta$. Since $\hat{\mathfrak{g}}^{\BPS,-2}_{\QTC{K},\WTC{K},(n+1) \cdot \delta} \simeq \Q$ is a one dimensional vector space, it is either $0$ or up to a non-zero scalar, it is $\alpha_{(n+1) \cdot \delta}$. We claim that it is non-zero. We consider the action $\sqbullet$ on the cohomology of Hilbert scheme of $S_{K}$ constructed in Proposition \ref{leftactionhilbertscheme} and show that the Lie bracket $[\alpha_{\delta}^{(1)}, \alpha_{n \cdot \delta}]$ acts non-trivially. We claim that for any $n \geq 1$, up to a factor of a non-zero scalar, there is an equality \[\alpha_{n \cdot \delta} \sqbullet v = p_n([S_K]) v \] where $p_n([S_K)])$ is Nakajima's operator (See Section \ref{Nakajimaoperatorhilbertscheme}) and $v \in V_K$. The case when $n=1$ is proved in Proposition \ref{actionforn=1}. Then by Proposition \ref{Lehnweicommute}
\begin{align} \label{delta 1}
\alpha_{\delta}^{(1)} \sqbullet v = \sum_{i=0}^{K} [c_1(\Rcal_i), p_1([S_K])]  v &= \sum_{i=0}^{K}(p_1^{1}([S_K])v+p_1(c_1(\Rcal_i))v)  \\ &= (K+1)p_1^{1}([S_K])v + \sum_{i=0}^{K} p_1(c_1(\Rcal_i))v \nonumber
\end{align}
where for each $i$, $\mathcal{R}_i$ is the tautological bundle on $S_K$. But then, assuming the claim for $n=N$, we have
\begin{align*}
[\alpha_{\delta}^{(1)}, \alpha_{N \cdot \delta}] \sqbullet v &= (K+1)[p_1^{1}([S_K]),p_{N}([S_K])]v + \sum_{i=0}^{K} [p_1(c_1(\mathcal{R}_i)),p_{N}([S_K])] v \\ &=
-(N)(K+1)p_{N+1}([S_K]) v, 
\end{align*}
The second equality is again a consequence of Proposition \ref{Lehnweicommute}. But then for any $n$, there exists some $v \in V_K$ such that $p_{N+1}([S_K]) v \neq 0$, since Nakajima's operators act faithfully (see Section \ref{Nakajimaoperatorhilbertscheme}), and thus $[\alpha_{\delta}^{(1)}, \alpha_{N \cdot \delta}] \neq 0$ and thus up to a factor of non-zero scalar $[\alpha_{\delta}^{(1)}, \alpha_{N \cdot \delta}] = \alpha_{(n+1)\cdot \delta}$ and both the claim and proposition follows by induction. 
\end{proof}

From now on, we fix $\gamma_{\delta}$ as some renormalization of $\alpha_{\delta}$, so that for the vacuum vector $\ket{0} = [0] \in \HBM(\Hilb^{0},\Q_0) \in \Nak^{\zeta^{\infty}}_{\delta_0}(Q^K)$ we have $\gamma_{\delta} \bullet \ket{0} = [S_K]$. This is allowed by Proposition \ref{imaginaryactiononequivarianthilbert}. Then we have $\gamma_{\delta}^{(r)} = u^r \cdot \gamma_{\delta}$. Finally for all $k \geq 1$,  we fix 
\begin{equation}
\gamma_{k \cdot \delta} :=   \ad^{k-1}_{\gamma_{\delta}^{(1)}}(\gamma_{\delta})
\label{defnofimaginaryelements}
\end{equation}
which by the previous propositon, is just the renormalization of $\alpha_{k \cdot \delta}$ and it spans $\mathfrak{g}^{\BPS,-2}_{\QTC{K},\WTC{K}}$. Similarly, we have \[\gamma^{(r)}_{k \cdot \delta} := u^r \ad^{k-1}_{\gamma_{\delta}^{(1)}}(\gamma_{\delta}^{(0)}). \] 
Now we study how an element of $\mathfrak{im}_{K}$ interacts with an element of $\mathfrak{re}_{K}$. By $[a_1,a_2,a_3,\cdots,a_N]$ we mean the iterated Lie bracket $[a_1,[a_2,[a_3,[\cdots, a_N]]]]$. 

\begin{prop} \label{strongrationality}
For any $i \in [0,K]$, the Lie bracket $[\alpha_i, \gamma_{\delta}^{(1)}] \neq 0$. 
\end{prop}

\begin{proof}
We consider the action of $\mathcal{A}_{\tilde{Q^{K}},W_K}$ on $\Nak^{\zeta^{\infty}}_{\delta_0}(Q^K)$ as constructed in Section \ref{cohaactionequivarianthilbert}. We show that $[\alpha_i,\gamma_{\delta}^{(1)}] \bullet \ket{0} \neq 0$. By Proposition \ref{relationwithnakajimaoperators}, we can identify the action of $\alpha_i$ with the action of the generator $e_i$ of the affine Lie algebra $\tilde{\mathfrak{sl}}_{K+1}$ on $\Nak^{\zeta^{\infty}}_{\delta_0}(Q^K)$, due to Nakajima (Section \ref{nakajimliealgebraction}).
\par
Note that $\Hilb^{\rho_0}(\C^2) = \pt$ as $I = \langle x,y \rangle$ is the only $\mathbb{Z}_{K+1}$ equivariant ideal with $\C[x,y]/I \simeq \rho_0$. Thus $\mathfrak{B}_{0}(\rho_{\reg}) = \Hilb^{\rho_0}(\C^2)$. Let $[0_{[0,1)}] \in \HH(\Hilb^{\rho_0}(\C^2),\Q^{\vir})$ be the fundamental class of the point. Thus by (\ref{nakajimaformula}), we have \[\alpha_0 \bullet \ket{0} = [0_{[0,1)}]. \]

Similarly, since the closed points of $\Hilb^{1}(\C^2)$ are just the maximial ideals $(x-t_1,y-t_2)$, only $\mathbb{Z}_{K+1}$ equivariant point is the ideal $(x,y)$. It follows that $\Hilb^{\rho_i}(\C^2) =0$ for all $i \neq 0$ and so by (\ref{nakajimaformula}), for all $i \neq 0 $ we have \[ \alpha_i \bullet \ket{0} = 0. \] 

Now for any $i \neq K$, we consider the equivariant Hilbert scheme $\Hilb^{\oplus_{j=0}^{i} \rho_j}(\C^2)$. By Proposition \ref{dimensionofnkajimaquiver}, $\Hilb^{\oplus_{j=0}^{i} \rho_j}(\C^2)$ is $0$ dimensional connected scheme and thus just a point. Let $[0_{[0,i)}] \in \HH(\Hilb^{\oplus_{j=0}^{i} \rho_j}(\C^2),\Q^{\vir})$ be the fundamental class of the point. Consider $\alpha_{1} \bullet (\alpha_0 \bullet \ket{0})$. Again since each of the schemes $\Hilb^{\rho_0}(\C^2)$ and $\Hilb^{\rho_0 \oplus \rho_1}(\C^2)$ are points and $\rho_0 \subset \rho_0 \oplus \rho_1$, it follows that $\mathfrak{B}_{1}(\rho_0) = \Hilb^{\rho_0 \oplus \rho_1}(\C^2)$. We thus have by (\ref{nakajimaformula}) that \[ \alpha_{1} \bullet (\alpha_0 \bullet \ket{0}) = [0_{[0,2)}].\] We may continue this argument succesively to calculate $\alpha_{i-1} \bullet (\alpha_{i-2} \bullet( \cdots \alpha_0 \sqbullet \ket{0}))$ for any $i \in [1,K]$. By associativity of the action of the $\mathfrak{n}_{Q^K}^{+}$, we conclude that for any $i \in [1,K]$, 
\[ (\alpha_{i-1} \star \alpha_{i-2} \star \cdots \star \alpha_{0}) \bullet \ket{0} = [0_{[0,i)}]  \in \HH(\Hilb^{(\oplus_{j=0}^{i} \rho_j)}(\C^2),\Q). \] 

Now consider the equivariant Hilbert scheme $\Hilb^{(\oplus_{j=0}^{i} \rho_j) \oplus \rho_K}(\C^2)$. Again, by Proposition \ref{dimensionofnkajimaquiver}, it is a $0$-dimensional connected scheme and hence a point. Let $[0_{[K,i+1)}] \in \HH(\Hilb^{(\oplus_{j=0}^{i} \rho_j)\oplus \rho_K}(\C^2),\Q^{\vir})$ be the fundamental class of the point. Then again by (\ref{nakajimaformula}), we have 
\[ \alpha_{K} \bullet ((\alpha_{i-1} \star \alpha_{i-2} \star \cdots \star \alpha_{0}) \bullet \ket{0})= (-1)[0_{[K,i+1)}] \in \HH(\Hilb^{(\oplus_{j=0}^{i} \rho_j) \oplus \rho_K}(\C^2),\Q^{\vir}). \] Note that the sign comes from the sign in the formula (\ref{nakajimaformula}). For any $i \in [1,K]$, we now consider $\Hilb^{(\oplus_{j=0, j \neq i}^{K} \rho_j)}(\C^2)$. Again, $\Hilb^{(\oplus_{j=0, j \neq i}^{K} \rho_j)}(\C^2)$ is a point by Proposition \ref{dimensionofnkajimaquiver}. Let $[0_{[i+1,K)}] \in \HH(\Hilb^{(\oplus_{j=0, j \neq i}^{K} \rho_j)}(\C^2),\Q^{\vir})$ be the fundamental class of the point. Then by the same argument as above, we conclude for any $i \in [1, K]$, that 
\[ (\alpha_{i+1} \star \alpha_{i+2} \star \cdots \alpha_K \star \alpha_{i-1} \star \cdots \alpha_0) \bullet \ket{0} = (-1)^{K-i}[0_{[i+1,K)}] \in \HH(\Hilb^{(\oplus_{j=0, j \neq i}^{K} \rho_j)}(\C^2),\Q^{\vir}). \] Note that $\rho_{\reg}- \rho_i = \oplus_{j=0,j \neq i}^{K} \rho_j$. By (\ref{nakajimaformula}), for $i \in [1,K]$, we  have $\alpha_i \bullet [0_{[i+1,K)}] = (-1) [\mathfrak{B}_{i}(\rho_{\reg})]$, where by \ref{identification}, we may view $\mathfrak{B}_i(\rho_{\reg})$ as a Lagrangian inside the surface $S_K$. In particular $[\mathfrak{B}_{i}(\rho_{\reg})] \in \HH(S_K,\Q^{\vir})$. We thus have for $i \in [1,K]$, 
\[
(\alpha_i \star \alpha_{i+1} \cdots \star \alpha_K \star \alpha_{i-1} \cdots \star \alpha_0) \bullet \ket{0} = (-1)^{K+1-i}[\mathfrak{B}_i(\rho_{\reg})]. \]
But since $\alpha_i \bullet\ket{0} = 0$ for $i \neq 0 $, we can write above as 
\begin{align} \label{commutatoractiononvaccum}
[\alpha_i, \alpha_{i+1} \cdots , \alpha_K ,  \alpha_{i-1} \cdots, \alpha_0] \bullet \ket{0} =  (-1)^{K+1-i}[\mathfrak{B}_i(\rho_{\reg})].
\end{align}

Now for $i \neq 0$, we consider the commutator $[\alpha_i, \gamma_{\delta}^{(1)}]$. Again since $\alpha_i \bullet \ket{0}=0$, we have by Proposition \ref{imaginaryactiononequivarianthilbert} that 
\begin{align*}
[\alpha_i, \gamma_{\delta}^{(1)}] \ket{0} =\alpha_i \bullet \gamma_{\delta}^{(1)} \ket{0} = \alpha_i \star (\sum_{j} (\sum_{k} A_{kj}) [\mathfrak{B}_j(\rho_{\reg})]).  
\end{align*}
But by the Equation \ref{commutatoractiononvaccum}, this is the same as 
\[[\alpha_i, \gamma_{\delta}^{(1)}] \ket{0} =  \sum_{j} (-1)^{K+1-j} \sum_k A_{kj} [\alpha_i;\mathcal{B}_j] \bullet \ket{0}, \]
where $[\alpha_i, \mathcal{B}_j]$ represents the commutator 
\[[\alpha_i, \alpha_j , \alpha_{j+1} \cdots,  \alpha_K, \alpha_{j-1} \cdots , \alpha_0].\]
and $A_{k,j}$ are the entries of the inverse of the Cartan matrix of $A_K$. It is easy to calculate the inverse of the Cartan matrix of $A_K$. We have $A_{kj} = \min(k,j) - \frac{kj}{K+1}$, which shows that $S_j= \sum_{k} A_{kj} = j(K+1-j)/2$. Similarly, since $\mathfrak{n}_{Q^K}^{+} \simeq (\mathfrak{sl}_{K+1}[z,z^{-1}])^{+}$ after identifying $e_i$ with $E_{i,i+1}$ and $e_0$ with $zE_{K+1,1}$; we may identify $[\alpha_i,\mathcal{B}_i] = 2 ((-1)^{K-i}zE_{i,i+1})$ while $[\alpha_i,\mathcal{B}_{i-1}] = [\alpha_i,\mathcal{B}_{i+1}] = (-1)^{K-i} zE_{i,i+1}$ and $[\alpha_i,\mathcal{B}_j]=0$, otherwise. In particular $[\alpha_i,\mathcal{B}_i] = 2[\alpha_i,\mathcal{B}_{i-1}] = 2[\alpha_i, \mathcal{B}_{i+1}]$. Thus we have
\begin{align*}
[\alpha_i, \gamma_{\delta}^{(1)}] \bullet \ket{0} = (-1)^{K+1-i} (S_{i}-\frac{S_{i-1}+S_{i+1}}{2})[\alpha_i, \mathcal{B}_i] \bullet \ket{0} = (-1)^{(K+1-i)}/2 [\alpha_i,\mathcal{B}_i] \bullet \ket{0}. 
\end{align*}

But $[\alpha_i,\mathcal{B}_i] \bullet \ket{0} \neq 0$ by the same argument for faithfulness as in Proposition \ref{imaginaryactiononequivarianthilbert}.  This in particular means that $[\alpha_i, \gamma_{\delta}^{(1)}] \neq 0$. By symmetry, we also get the same relation for $i=0$, since we can construct the representation by adding the framing to vertex $1$ instead of $0$. Thus we have 
\begin{align*}
[\alpha_{0}, \gamma_{\delta}^{(1)}] \bullet \ket{0} = \frac{(-1)^{K}}{2} [\alpha_0, \alpha_0, \alpha_1, \cdots, \alpha_K] \bullet\ket{0} \neq 0
\end{align*} 
and so we are done. 
\end{proof}

Note that $[\alpha_{i},\gamma_{\delta}^{(1)}]$ is an element of cohomological degree $0$, in dimension $\delta+\delta_i$. From proposition \ref{strongrationality}, it follows that $[\alpha_{i},\gamma_{\delta}^{(1)}] = \lambda  [\alpha_i, \mathcal{B}_i]$ for some non-zero $\lambda$, since $\widehat{\mathfrak{g}}^{\BPS,0}_{\QTC{K},\WTC{K},\delta+\delta_i}$ is $1$ dimensional. But then, the above calculation even precisely determines $\lambda$, and so finally, we have
\begin{equation}
[\alpha_i, \gamma_{\delta}^{(1)}] = \frac{(-1)^{K+1-i}}{2} [\alpha_i, \mathcal{B}_i]. 
\label{relationsbetweenrealandimaginary}
\end{equation} Furthermore, since $[\alpha_{i},\gamma_{\delta}^{(0)}]=0$, we have $[\gamma_{\delta}^{(0)}, \alpha_i^{(1)}]= \frac{(-1)^{K+1-i}}{2} [\alpha_i,\mathcal{B}_i]$.

\begin{corollary}[Strong rationality for Kleinian  singularity] \label{strongrationalityconjectureproof}
For any $\bd \in \N^{Q_{0}}$, the commutator\[[\gamma_{\delta}^{(1)}, \hyphen]: \mathcal{A}^{\chi}_{\QTC{K},\WTC{K}} \rightarrow \mathcal{A}^{\chi}_{\QTC{K},\WTC{K}} \] maps $\mathfrak{g}^{\BPS}_{\QTC{K},\WTC{K},\bd}$ isomorphically to $\mathfrak{g}^{\BPS}_{\QTC{K},\WTC{K},\bd+ \delta}$. 
\end{corollary}
\begin{proof}
The Lie algebra generated by $\alpha_i$ where $i \in [0,K]$ is isomorphic to the Lie algebra $\mathfrak{n}^{+}_{Q^K} \simeq (\mathfrak{sl}_{K+1}[z,z^{-1}])^{+}$ by Proposition \ref{zerocohomologybps}. The isomorphism is given by the identification $\alpha_i$ with $E_{i,i+1}$ for $i \neq 0$ and $\alpha_0$ with $zE_{K+1,1}$. Thus after this identification, equation (\ref{relationsbetweenrealandimaginary}) implies that \[ [\gamma_{\delta}^{(1)}, E_{i,i+1}] = zE_{i,i+1}, \  [\gamma_{\delta}^{(1)}, zE_{K+1,1}] = z^2E_{K+1,1}\]
But then one can write $\alpha_{\bd} \in \mathfrak{g}^{\BPS,0}_{\tilde{Q},\tilde{W},\bd}$ as linear combination of commutators of $E_{i,i+1}$ and $zE_{K+1}$. Thus if $\alpha_{\bd} \neq 0$ then $[\gamma_{\delta}^{(1)},\alpha_{\bd}] \neq 0$. While for $\alpha_{\bd} \in \mathfrak{g}^{\BPS,-2}_{\tilde{Q},\tilde{W}}$, $[\gamma_{\delta}^{(1)},\alpha_{\bd}] \neq 0$ is Proposition \ref{relation1}. Since $\gamma_{\delta}^{(1)}$ is an element of cohomological degree $0$, for any $\alpha_{\bd} \in \mathfrak{g}^{\BPS}_{\tilde{Q},\tilde{W},\bd}$, $[\gamma_{\delta}^{(1)},\alpha_{\bd}] \in \mathfrak{g}^{\BPS}_{\tilde{Q},\tilde{W},\bd+\delta}$. So we have checked that the morphism \[ [\gamma_{\delta}^{(1)}, \hyphen]: \mathfrak{g}^{\BPS}_{\tilde{Q},\tilde{W},\bd} \rightarrow \mathfrak{g}^{\BPS}_{\tilde{Q},\tilde{W},\bd+\delta} \] is injective. Since the dimensions are the same, it is an isomorphism. %we may identify $\alpha_{\bd} = z^{n}E_{i,j}$ for some $j,i \in [1,K], n \geq 0$ such that $j>i$ or $z^{n+1}E_{i,j}$ where $i,j \in [1,K]$ and $i \neq j, n \geq 0$. But then inductively, it follows that $[\gamma_{\delta}^{(1)},z^{n}E_{j,i}] = z^{n+1}E_{j,i}$ and $[\gamma_{\delta}^{(1)},z^{n+1}E_{i,j}] = z^{n+2}E_{i,j}$ up to some non-zero scalar. 

\end{proof}

%Consider the Lie automorphism $z: H^{0}(R) \rightarrow R$ given by sending generators $f_i  \rightarrow zf_i$. 

%We recall that there is a derived equivalence
%\[ \mathrm{\Psi}: \Db(\C \times S_K) \rightarrow \Db(\Jac(\QTC{K},\WTC{K})) \]

%We may identify the class $\gamma_{\delta}^{(1)}$ as $u 1_{C}$. 

\begin{prop} \label{imaginaryLiealgebra}
    The subspace $\mathfrak{im}_K$ forms a Lie subalgebra of $\abps{K}$ and is generated by $\gamma_{\delta}^{(r)}$ for $r \geq 0$. 
\end{prop}

\begin{proof}
The claim about generation follows from the definition of $\gamma_{k \cdot \delta}^{(r)}$ and the fact that $u$ acts by derivation. It suffices to show that for $k \geq 1, n,m \geq 0$, the Lie bracket $[\gamma_{\delta}^{(n)}, \gamma_{k \cdot \delta}^{(m)}] \in \mathfrak{im}_{K} $. Since for all $r \geq 1$,  $\gamma_{r \cdot \delta}^{(n)}$ is of cohomological degree $2n-2$, there must exist $\lambda^{n,m}_{k} \in \Q$ and $\mu^{n,m}_k \in \mathbb{Q}$ such that 
\begin{equation}\label{equalityclaim} [\gamma_{\delta}^{(n)}, \gamma_{k \cdot  \delta}^{(m)}]  = \lambda^{n,m}_{k} \gamma_{(1+k) \cdot \delta}^{(n+m-1)}+ \mu_{k}^{n,m}\beta^{(n+m-2)}_{(1+k)\cdot \delta} \end{equation} where $\beta_{(1+k) \cdot \delta} \in \mathfrak{g}^{\BPS,0}_{\QTC{K},\WTC{K}}$. It suffices to show that $\mu^{n,m}_{k} =0$. Note that showing this for $m=0, n \geq 0$, i.e. to show that $\mu^{n,0}_k=0$ for all $n \geq 0$ is enough, since if for some $i \geq 0$, $[\gamma_{\delta}^{(n)},\gamma_{k \cdot \delta}^{(i)}] \in \mathfrak{im}_{K}$ then applying $u$ gives that $[\gamma_{\delta}^{(n+1)},\gamma_{k \cdot \delta}^{(i)}] + [\gamma_{\delta}^{(n)},\gamma_{k \cdot \delta}^{(i+1)}] \in \mathfrak{im}_K$. So if $[\gamma_{\delta}^{(n+1)},\gamma_{k \cdot \delta}^{(i)}] \in \mathfrak{im}_K$ then so does $[\gamma_{\delta}^{(n)},\gamma_{k \cdot \delta}^{(i+1)}]$. We consider the action $\sqbullet$ of $\mathcal{A}^{\Im}_{\tilde{Q^{K}},\tilde{W}}$ on $V_K$ as constructed in Proposition \ref{leftactionhilbertscheme}. We then have 
%\vspace{.5cm}
\begin{claim}
For any $m \geq 1$, there exist a non-zero scalar $\lambda^{\prime \prime}_m \in \Q$, such that for any $v \in V_K$, we have the equality
\begin{align} 
\gamma_{m \cdot \delta}^{(n)} \sqbullet v = \lambda^{\prime \prime}_{m} \left( (K+1) p_{m}^{n}([S_K])v + mn \sum p_m^{n-1}(c_1(\mathcal{R}_i)) v \right)
\label{formulaactiongeneral}
\end{align}
%For some non-zero $\lambda^n_m$, \[ \gamma_{m \cdot \delta}^{n} \sqbullet v = \lambda^{n}_{m} (K+1)^{n-1}((K+1) p_{m}^{n}(1)v + mn p_m^{n-1}(\alpha)v)\]
\end{claim}

\begin{claimproof}
We use the same argument as in Proposition \ref{relation1}. We identify the action $\gamma_{\delta} \sqbullet v$ with $p_1([S_K])v$, and then the result follows from the repetitive use of formulas in Proposition \ref{Lehnweicommute} applied to the definition of $\gamma_{m \cdot \delta}^{(n)}$ (equation \ref{defnofimaginaryelements}). 
\end{claimproof}

It follows from the above claim (\ref{formulaactiongeneral}), equation (\ref{equalityclaim}), and Proposition \ref{Lehnweicommute}, that up to a factor of a non-zero scalar, for all $v \in V_K$, we have
\[ [\gamma_{\delta}^{(n)}, \gamma_{k \cdot \delta}^{(0)}] \sqbullet  v = \gamma_{(1+k) \cdot \delta}^{(n-1)} \sqbullet  v\] 
To connlude, tt suffices to show that for some $v \in V_K$, $u^{n-2}\alpha_{(1+k)\cdot \delta} \sqbullet v \neq 0$. By Proposition \ref{faithfulhilbertscheme}, the action of $\mathfrak{g}^{\Im}_{\QTC{K},\WTC{K}}$ on $V_K$ is faithful, thus the statement that there exists $v \in V_K$ such that $u^{n-2}\alpha_{(1+k)\cdot \delta} \sqbullet v \neq 0$ is true for $n=2$. Now, assume otherwise for some $n>2$. Then for all $v$, $[\gamma_{\delta}^{0}, u^{n-2}\alpha_{n \cdot \delta}] \sqbullet v =0$. But by Proposition \ref{strongrationalityconjectureproof}, $[\gamma_{\delta}, u^{n-2}\alpha_{k \cdot \delta}] = u^{n-3} \alpha^{\prime}_{(k+1) \cdot \delta}$ for some $\alpha^{\prime}_{(k+1) \cdot \delta}$. Thus action of $u^{n-3} \alpha^{\prime}_{(k+1) \cdot \delta}$ isn't faithful either. Thus, we inductively conclude that action of some non-zero $\beta_{(k+n-2)\cdot \delta} \in \mathfrak{g}^{\Im}_{\tilde{Q^{K}},\tilde{W}}$ on $V_K$ isn't faithful, which is a contradiction to Proposition \ref{faithfulhilbertscheme}.   
\end{proof}

Note that as vector spaces $\mathfrak{im}_K + \mathfrak{re}_{K} = \abps{K}$. Thus it remains to understand $[\mathfrak{im}_K,\mathfrak{re}_K]$. 

\begin{prop}
We have
\[ [\mathfrak{im}_K,\mathfrak{re}_K] \subset \mathfrak{re}_K\] and is determined by $[\gamma_{\delta}^{(n)}, \alpha_i^{(m)}] $, $n \geq 0$ and $i \in [0,K], m \geq 0$. 
\end{prop}

\begin{proof}
Since $\mathfrak{im}$ is generated by $\gamma_{\delta}^{r}, r \geq 0$ it suffices to show that $[\gamma_{\delta}^{r}, \mathfrak{re}] \subset \mathfrak{re}$. But $\mathfrak{re}$ is also generated by $e_i^n$ by the Proposition \ref{realLiesubalgebra} so, in fact we need to consider the commutators between $[\gamma_{\delta}^{r}, e_i^{n}]$. The commutator $[\gamma_{\delta}^{r}, f_i^{n}]$ is an element of dimension $\delta+\delta_i$ in $\abps{K}$. Since $\mathfrak{im}_K$ only contains elements whose dimension vector is a multiple of $\delta$, the commutator must be in $\mathfrak{re}_K$. 
\end{proof}

We now sieve back everything together. Recall the Lie algebra $\tilde{W}_{K+1}^{+}$ from Definition \ref{positivehalfofclassicallimit}. By Proposition \ref{wKnegativelydetermined}, $\tilde{W}_{K+1}^{+}$ is a negatively determined $\Heis$ Lie algebra. We consider the Lie subalgebra $\tilde{W}_{K+1}^{+,\leq 0} \subset \tilde{W}_{K+1}^{+}$. It is a Lie subalgebra spanned by $T_{0,0}(X)$ where $ X \in \mathfrak{n}^{+}_{K+1}$, $T_{m+1,0}(X)$ where $ \in \mathfrak{gl}_{K+1}, m \geq 0$ and $t_{n,1},t_{n,0}$ for all $n \geq 1$. Similarly, by Proposition \ref{negativelydeterminedaffinizedbps}, $\abps{K}$ is a negatively determined $\Heis$ Lie algebra. Let $(\abps{K})^{\leq 0} \subset \abps{K}$ be the non positive degree subspace of $\abps{K}$. Clearly it is generated by $\alpha_{i}^{(0)}, i \in [0,K]$ and $\gamma_{n \cdot \delta}^{(0)},\gamma_{n \cdot \delta}^{(1)}$ for $n \geq 1 $.

\begin{thm}\label{Theorem1}
For any $K \geq 1$, the morphism $F: (\abps{K})^{\leq 0} \rightarrow \tilde{W}_{K+1}^{+,\leq 0}$ given by 
    \begin{align*}
        \gamma_{n \cdot \delta}^{(0)} &\rightarrow (-1)^{n}(n-1)!(K+1)^{n}t_{n,0} \\ 
         \gamma_{n \cdot \delta}^{(1)} &\rightarrow (-1)^{n}(n)!(K+1)^{n-1} \left((K+1)t_{n,1} - T_{n,0}(H_{K+1}) \right) \\
        \alpha_{i}^{(0)} &\rightarrow T_{0,0}(E_{i,i+1}), \forall i \in [1,K]  \\ \alpha_{0}^{(0)} &\rightarrow T_{1,0}(E_{K+1,1})
    \end{align*}
    extends to an isomorphism of Lie algebras \[ \hat{F}: \abps{K} \rightarrow \tilde{W}^{+}_{K+1}.\]
\end{thm}

\begin{proof}
We check that $F$ satisfies the conditions in Proposition \ref{mapofheisLiealgebras}. We first check the compatibility with $\Heis$ actions. We have 
\begin{align*}
F(u \cdot \gamma^{(0)}_{n \cdot \delta}) &= F(\gamma_{n \cdot \delta}^{(1)}) \\ &= (-1)^{n}(n)!(K+1)^{n-1} \left((K+1)t_{n,1} - T_{n,0}(H_{K+1}) \right) \\ &= \left[ \frac{t_{0,2}}{2} - \frac{T_{0,1}(H_{K+1})}{K+1}, (-1)^{n}(n-1)!(K+1)^{n}t_{n,0} \right] \\ &= p F(\gamma^{(0)}_{n \cdot \delta}),
\end{align*} where the last equality is the consequence of the definition of $p$ in Proposition \ref{wKnegativelydetermined} and similarly $F(\partial \cdot \gamma_{n \cdot \delta}^{(1)}) = n F(\gamma_{n \cdot \delta}^{(0)}) = q F(\gamma_{n \cdot \delta}^{(1)}) $ by Proposition \ref{wKnegativelydetermined}. 
We now check that $F$ is a morphism of  Lie algebras. Since the  structure constants  between $[\gamma_{n \cdot \delta}^{(1)}, \gamma^{(0)}_{m \cdot \delta}]$ and $[\gamma_{n \cdot \delta}^{(1)}, \gamma_{m \cdot \delta}^{(1)}]$ for $m,n>1$ are determined by relations of $[\gamma_{\delta}^{(1)},\gamma^{(0)}_{m \cdot \delta}]$ and $[\gamma_{\delta}^{(1)},\gamma_{m \cdot \delta}^{(1)}]$, we only consider the latter.  Clearly $F([\gamma^{(0)}_{\delta},\gamma^{(0)}_{n \cdot \delta}]) = 0 = -(-1)^{n+1}[(K+1)t_{1,0},(n-1)!(K+1)^{n}t_{n,0}]=0$. Next, we see that 
\begin{align*}
[F(\gamma_{\delta}^{(1)}),F(\gamma^{(0)}_{n \cdot \delta})] &= [(-1)((K+1)t_{1,1}-T_{1,0}(H_{K+1})), (-1)^{n}(n-1)!(K+1)^{n}t_{n,0}]  \\ &= (-1)^{n+1}n!(K+1)^{n+1}t_{n+1,0} \\ &=  F(\gamma_{(n+1) \cdot \delta}^{(0)})\\ &= F([\gamma_{\delta}^{(1)},\gamma^{(0)}_{n \cdot \delta}])
\end{align*} where the last equality is Equation (\ref{defnofimaginaryelements}). The same argument holds for the relation $[\gamma_{n \cdot \delta}^{(1)},\gamma^{(0)}_{\delta}] = \gamma_{(n+1)\cdot \delta}$.  On the other hand via Proposition \ref{imaginaryLiealgebra} we have relation $[\gamma_{\delta}^{(1)},\gamma_{n \cdot \delta}^{(1)}] = (n-1)/(n+1) \gamma_{(n+1) \cdot \delta}^{(1)}$ whose compatibility can be checked in the same way. We now consider relations among $\alpha_i^{(0)}$ and $\gamma_{n \cdot \delta}$. For $i \neq 0$, we have 
\begin{align*}
[F(\gamma_{\delta}^{(1)}), &F(\alpha^{(0)}_{i})] = [(-1)((K+1)t_{1,1}-T_{1,0}(H_{K+1})), T_{0,0}(E_{i,i+1})] \\ &= T_{1,0}(E_{i,i+1}) \\ &= \frac{(-1)^{K-i}}{2} [T_{0,0}(E_{i,i+1}),T_{0,0}(E_{i,i+1}), \cdots, T_{0,0}(E_{K,K+1}), T_{0,0}(E_{i-1,i}), \cdots, T_{1,0}(E_{K+1,1})] \\ &= F([\gamma_{\delta}^{(1)},\alpha^{(0)}_{i}]) 
\end{align*} via Equation \ref{relationsbetweenrealandimaginary}. The same argument holds when $n>1$ and when $i=0$. In the same way, for $i \neq 0$
\begin{align*}
[F(\gamma^{(0)}_{n \cdot \delta}),F(\alpha^{(0)}_{i})] &=  [ (-1)^n(n-1)!(K+1)^n t_{n,0}, T_{0,0}(E_{i,i+1})] \\
&= 0 = F([\gamma^{(0)}_{n \cdot \delta},\alpha^{(0)}_{i}])
\end{align*}
and for $i=0$, we have
\begin{align*}
[F(\gamma^{(0)}_{n \cdot \delta}),F(\alpha^{(0)}_{0})] &=  [ (-1)^n(n-1)!(K+1)^n t_{n,0}, T_{1,0}(E_{K+1,1})] \\
&= 0 = F([\gamma^{(0)}_{n \cdot \delta},\alpha^{(0)}_{0}])
\end{align*}
where the last equalities are because of Proposition \ref{bpsLiealgebracyclic}. Finally, the relations between $[\alpha_{i}^{(0)},\alpha_{j}^{(0)}]$ are also satisfied since by Proposition \ref{bpsLiealgebracyclic}, the Lie algebra generated by $\alpha^{(0)}_{i}$ is exactly the positive half $\mathfrak{n}^{+}_{K+1}$ and the above map is just the identification of  $\mathfrak{n}^{+}_{K+1}$ with $\tilde{\mathfrak{sl}}^{+}_{K+1}$.

By Propositon \ref{mapofheisLiealgebras}, $F$ extends to a map of $\Heis$ modules $\hat{F}$. To check if the map extends to the whole Lie algebra, it suffices to check the relations for $[\gamma_{n \cdot \delta}, \gamma_{m \cdot \delta}^{(2)}]$, $[\gamma_{n \cdot \delta}^{(1)}, \gamma_{m \cdot \delta}^{(1)}]$, $[\gamma_{n \cdot \delta}, \alpha_i^{(1)}]$. Since by definition, it is a map of $\textbf{Heis}$ modules, this boils down to checking the relation for $[\gamma_{\delta}^{(2)}, \gamma_{n \cdot \delta}]$. But again it can be checked by definition that \[\hat{F}(\gamma_{\delta}^{(2)}) = (-1)((K+1)t_{1,2}-2T_{1,1}(H_{K+1})) \] and so 
\begin{align*}
[\hat{F}(\gamma_{\delta}^{(2)}),\hat{F}(\gamma_{n \cdot \delta})] &= [(-1)((K+1)t_{1,2}-2T_{1,1}(H_{K+1})), (-1)^{n}(n-1)!(K+1)^{n}t_{n,0}]  \\ &= (-1)^{n+1}(n)!(K+1)^{n}((K+1)t_{n+1,1}-T_{n+1,0}(H_{K+1})) \\ &= 2/(n+1) F(\gamma_{(n+1) \cdot \delta}^{0})\\ &= \hat{F}([\gamma_{\delta}^{(2)},\gamma_{n \cdot \delta}]).
\end{align*}
where the last equality follows from Proposition \ref{imaginaryLiealgebra}. Thus $\hat{F}$ is a map of Lie algebras. By Proposition \ref{sphericalgenerationofLK+1Liealgebra}, the map $\hat{F}$ is surjective. The map $\hat{F}$ preserves cohomological grading and total dimensions. We note by Proposition \ref{bpsLiealgebracyclic}, that 

\begin{align*}
\dim((\widehat{\mathfrak{g}}^{\BPS,i}_{\tilde{Q^{K}},\tilde{W^{K}}})_{|\bd|}) = \begin{cases}
1  & \textrm{if } i=-2, |\bd| \equiv 0 \mod{K+1} \\ K+1 & \textrm{if } i \in 2 \mathbb{Z}_{\geq 0}, |\bd| >0 
\\ 0 & \textrm{ otherwise } 
\end{cases}
\end{align*}

and 
\begin{align*}
\dim((\tilde{W}^{+}_{K+1})^{\mathrm{CG}}_{\dim}) = \begin{cases}
1  & \textrm{if } \mathrm{CG}=-2, \dim \equiv 0 \mod{K+1} \\ K+1 & \textrm{if } \mathrm{CG} \in 2 \mathbb{Z}_{\geq 0}, \dim  >0 
\\ 0 & \textrm{ otherwise } 
\end{cases}
\end{align*}

Since they are the same, it follows that the morphism $\hat{F}$ is, in fact, an isomorphism of Lie algebras. 
\end{proof}

\subsection{Application: CoHA of Kleinian surface} We recall $\pi: S_K \rightarrow \C^2/\mathbb{Z}_{K+1}$ is the minimial resolution of the Kleinian singularity $\C^2/\mathbb{Z}_{K+1}$. By \cite{kapranov1999kleinian}[1.4,3.4], there is an equivalence of derived categories
\[ \mathrm{\Psi} \colon \Db(\Coh(S_K)) \rightarrow \Db(\mathrm{\Pi}_{Q^{K}} \hyphen \textrm{Mod}) \] which restricts to an equivalence 
\[ \mathrm{\Psi} \colon \Db(\Coh_{\ps}(S_K)) \rightarrow \Db({\mathrm{\Pi}}_{Q^{K}} \hyphen \textrm{mod}) \] where $\Coh_{\ps}(S_K)$ is the abelian category of compactly supported coherent sheaves on $S_K$. This equivalence doesn't preserve the heart given by the usual bounded derived $t$ structures on the derived category. It is known that the abelian subcategory on the left, which is mapped to the representation category of the finite dimensional representations of the preprojective algebra of the cyclic quiver $\mathrm{\Pi}_{Q^K}$ is the abelian category of compactly supported perverse coherent sheaves associated to $\pi \colon S_K \rightarrow \C^2/\mathbb{Z}_{K+1}$ \cite{10.1215/S0012-7094-04-12231-6}[3.1]. 

Recall that for $i \in [1,K]$, $\mathfrak{B}_{i}(\rho_{\text{reg}})$ are the irreducible components of the exceptional fiber $\pi^{-1}(0)$ (Section \ref{componentsoffiber}).  Let $\omega \in \textrm{Pic}_{\Q}(S_K):= \textrm{Pic}(S_K) \otimes_{\Z} \mathbb{Q}$ be a $\Q$ polarization of $S_K$. By polarization, we mean that $\langle \omega, \mathfrak{B}_{i}(\rho_{\text{reg}}) \rangle >0$ for all $i \in [1,K]$, where $\langle \hyphen, \hyphen \rangle$ is the intersection pairing on the Picard group of $S_K$. The slope of a sheaf $\mathcal{F} \in \Coh_{\textrm{ps}}(S_K)$ is defined to be 
\begin{align*}
    \mu_{\omega}(\mathcal{F}) := \begin{cases}
        \frac{\chi(F)}{\text{ch}_1(\mathcal{F}) \cdot \omega} \textrm{ if } \text{ch}_1(\mathcal{F}) \cdot \omega \neq 0 \\ 
        \infty \textrm{ otherwise }
    \end{cases}
\end{align*}

For any $\mu \in \Q_{>} \cup \infty$, let $\Coh^{\omega}_{\mu}(S_K)$ be the classical truncation of the derived stack of $\omega$-semistable properly supported coherent sheaves of $S_K$ of slope $\mu$. Then in \cite{porta2022twodimensional}[2.3.2] and \cite{diaconescu2020mckay}[3.1], the authors define a cohomological Hall algebra $\mathcal{A}^{\zeta}_{\mu}(S_K)$ on the Borel-Moore homology $\HBM(\Coh^{\omega}_{\mu}(S_K),\Q)$ of the stack $\Coh^{\omega}_{\mu}(S_K)$. Furthermore, it is shown that under the above derived equivalence, this algebra is isomorphic to the semistable cohomological Hall algebra of the preprojective algebra of the cyclic quiver. More precisely, 
\begin{prop}\cite{diaconescu2020mckay} \label{salatheorem}[Corollary 3.4] Let $\zeta_i = \langle \omega, \mathfrak{B}_i(\rho_{\reg}) \rangle$ for $1 \leq i \leq K$ and \[ \zeta_0 := \begin{cases}
\frac{1}{\mu} - \sum_{i=1}^{K} \zeta_i &, \textrm{if } \mu \neq \infty \\ -\sum_{i=1}^{K} \zeta_i &, \textrm{if } \mu = \infty
\end{cases}. \]
The derived equivalence $\mathrm{\Psi}$ induces an algebra isomorphism \[\mathcal{A}^{\omega}_{\mu}(S_K) \simeq \mathcal{A}^{\zeta}_{\mathrm{\Pi}_{Q^{K}},0}.\] where $\mathcal{A}^{\zeta}_{\mathrm{\Pi}_{Q^{K}},0}$ is the $\zeta$-semistable, slope $0$ preprojective cohomological Hall algebra of the cyclic quiver $Q^K$. 
\end{prop}

In what follows, we will calculate the algebra $\mathcal{A}^{\zeta}_{\mathrm{\Pi}_{Q^{K}},0}$, which by the dimension reduction theorem, \cite{davison2022affine}[Section 3.2.2] is isomorphic to the algebra $\mathcal{A}^{\zeta,\chi}_{\QTC{K},\WTC{K},0}$. We first consider a more general setting. Given any symmetric quiver $Q$, stability condition $\zeta$ and slope $\mu$, consider the Lie subalgebra \[ \widehat{\mathfrak{g}}^{\BPS, \zeta}_{\mathrm{\Pi}_Q,\mu} := \bigoplus_{\bd \in \Lambda^{\zeta}_{\mu} \backslash \{ \textbf{0} \}}  \mathfrak{g}^{\BPS}_{\tilde{Q},\tilde{W},\bd}[u]  \subset \abps{}.\]
We then consider its universal enveloping algebra
 \[\mathcal{C}^{\zeta}_{\tilde{Q},\tilde{W},\mu} := \bU(\widehat{\mathfrak{g}}^{\BPS, \zeta}_{\mathrm{\Pi}_Q,\mu}) \subset \mathcal{A}^{\chi}_{\tilde{Q},\tilde{W}}. \] 

The following proposition gives a general way to realize the algebra $\mathcal{A}^{\zeta, \chi}_{\tilde{Q^K},\tilde{W^K},\mu}$ as a subalgebra of $\mathcal{A}^{\chi}_{\tilde{Q^K},\tilde{W^K},\mu}$ explicitly.
\begin{prop}\label{semistableaffinebps}
There is an isomorphism of algebras 
\[\mathcal{C}^{\zeta}_{\tilde{Q},\tilde{W},\mu}  \simeq \mathcal{A}^{\zeta, \chi}_{\tilde{Q},\tilde{W},\mu} . \]
\end{prop}

\begin{proof}
For every $\bd \in \Lambda^{\zeta}_{\mu}$, we have open inclusion $\mathfrak{M}_{\bd}^{\zeta}(\tilde{Q}) \hookrightarrow \mathfrak{M}_{\bd}(\tilde{Q})$,  which after restriction, induces a map of algebras $(i^{\zeta}_{\mu})^{*} \colon \mathcal{C}^{\zeta}_{\tilde{Q},\tilde{W},\mu} \rightarrow \mathcal{A}^{\zeta}_{\tilde{Q},\tilde{W}}$ such that \[(i^{\zeta}_{\mu})^{*}_{|\mathfrak{g}^{\BPS}} \colon \bigoplus_{\bd \in \Lambda^{\zeta}_{\mu} \backslash \{ \textbf{0} \}} \mathfrak{g}^{\BPS}_{\tilde{Q},\tilde{W},\bd} \rightarrow  \bigoplus_{\bd \in \Lambda^{\zeta}_{\mu} \backslash \{ \textbf{0} \}} \mathfrak{g}^{\BPS,\zeta}_{\tilde{Q},\tilde{W},\bd}\] is an isomorphism by Toda's theorem (Equation \ref{todaLiealgebras}).  Now since the cup product with the tautological bundle commutes with the open inclusion, it follows that the morphism \[(i^{\zeta}_{\mu})^{*}_{|}: \bigoplus_{\bd \in \Lambda^{\zeta}_{\theta} \backslash \{ \textbf{0} \}} \mathfrak{g}^{\BPS}_{\tilde{Q},\tilde{W},\bd}[u] \rightarrow  \bigoplus_{\bd \in \Lambda^{\zeta}_{\mu} \backslash \{ \textbf{0} \}} \mathfrak{g}^{\BPS,\zeta}_{\tilde{Q},\tilde{W},\bd}[u]\] is an isomorphism of Lie algebras. But then by the PBW theorem (\ref{PBWtheorem}) for $\mathcal{A}^{\zeta, \chi}_{\tilde{Q},\tilde{W},\mu}$, it follows that in fact $(i^{\zeta}_{\mu})^{*}$ is an isomorphism. 
\end{proof}
We now consider the algebras $\mathcal{A}^{\omega}_{\mu}(S_K) \simeq \mathcal{A}^{\zeta,\chi}_{\QTC{K},\WTC{K},0}$. When $\mu= \infty$, the only $\omega$-semistable properly supported coherent sheaves on $S_K$ are the $0$ dimensional sheaves. The cohomological Hall algebra $\mathcal{A}(S)$ of $0$-dimensional sheaves on any quasi-projective surface $S$ is studied and explicitly calculated in \cite{mellit2023coherent}. We calculate this algebra for $S = S_K$ using Proposition \ref{semistableaffinebps}. Note that for $\zeta_i >0$ and $\zeta_0= - \sum \zeta_i$, having the slope $\mu_{\zeta}(\bd)=0$ implies that $\bd_i=\bd_0 = n$ for some $n \geq 0$. Thus the cohomolgical Hall algebra is the universal enveloping algebra of the Lie subalgebra $\oplus_{n \geq 1} \mathfrak{g}^{\BPS}_{\QTC{K},\WTC{K},n \cdot \delta}[u] \subset \widehat{\mathfrak{g}}^{\BPS}_{\mathrm{\Pi}_Q^K}$. So by Theorem \ref{Theorem1}, it follows that

\begin{corollary} \label{W_sLiealgebra}
Let $W(S_K) \subset \tilde{W}_{K+1}$ be the Lie subalgebra on the subspace spanned by $T_{n,a}(H_i)$ where $ i \in [1,K], n \geq 1, a \geq 0$ and $t_{n,a}$ where $n \geq 1, a \geq 0$ with the relations
\begin{align*} 
    [t_{m,a},t_{n,b}] &= (na-mb)t_{m+n,a+b-1} \\ 
    [t_{m,a}, T_{n,b}(H_i)] &= (na-mb) T_{m+n,a+b-1}(H_i) \\
    [T_{m,a}(H_i),T_{n,b}(H_j)] &= T_{m+n,a+b}([H_i,H_j]) = 0
\end{align*}
Then we have an isomorphism of algebras 
\[\mathcal{A}(S_K) \simeq \bU(W(S_K)). \]
\end{corollary}

When $\mu \neq \infty$, then $\mu_{\zeta}(n \cdot \delta) \neq 0$ and in particular the Lie subalgebra $\oplus_{\bd \in \Lambda^{\zeta}_0} \mathfrak{g}^{\BPS}_{\tilde{Q},\tilde{W},\bd}[u]$ is always a Lie subalgebra of the real part $\mathfrak{re}_K \subset \abps{K}$. Given the stability condition $\zeta = \{ \zeta_i \mid i \in [0,K] \}$, we are interested in representations of slope $0$. We solve for dimension vectors $\bd$, where $\bd$ is a real positive root of the affine Lie algebra $\tilde{\mathfrak{sl}}_{K+1}$ which satisfies the equality \begin{align}\label{equation}
\mu_{\zeta}(\bd)= \sum_{i=1}^{K} \zeta_i (\bd_i -\bd_0) + \bd_0/\mu = 0.
\end{align} Since $\zeta_i>0$, any solution would be of type 
\[ n \cdot \delta + [K+1-j,i+j) = (\underbrace{n+1,\cdots,n+1}_{\text{$i$ times}},n,\cdots,n,\underbrace{n+1,\cdots,n+1}_{\text{$j$ times}}) \] for $i> 0$. Note that if $\bd^{\prime},\bd^{\prime \prime}$ satisfies the equation (\ref{equation}) then so does $\bd^{\prime}+\bd^{\prime \prime}$. Since $\mathfrak{g}^{\BPS}_{\tilde{Q},\tilde{W},n \cdot \delta + [K+1-j,i+j)}$ is 1 dimensional, we can identify the generator with $T_{n+1,0}(E_{K+2-j,i+1})$ when $j >0$ and with $T_{n,0}(E_{1,i+1})$ when $j=0$. 

\begin{defn} \label{Womegamuliealgebra}
Given $\omega$ a $\Q\hyphen$polarization of $S_K$ and slope $\mu \in \Q_{>}$, let $\zeta_i = \langle \omega, \mathfrak{B}_{i}(\rho_{\reg}) \rangle$ for $1 \leq i \leq K$ and $\zeta_0 = 1/\mu - \sum_{i=1}^{K} \zeta_i$. Let $\mathbf{L} \subset \Z_{\geq 0} \times \Z_{>0} \times \Z_{\geq 0}$ be the subset of tuples $(n,i,j)$ such that $ \bd = n \cdot \delta + [K+1-j,i+j)$ satisfies equation (\ref{equation}). Let $W^{\omega}_{\mu}(S_{K}) \subset \tilde{W}_{K+1}$ be the Lie algebra spanned by $T_{n+1,a}(E_{K+2-j,i})$ for all $(n,i,j) \in \mathbf{L}, j>0, a \in \mathbb{Z}_{\geq 0}$ and $T_{n,0}(E_{1,i+1})$ for all $(n,i,0) \in \mathbf{L}, a \in \mathbb{Z}_{\geq 0}$ satisfying \[[T_{m,a}(X),T_{n,b}(Y)] = T_{m+n,a+b}([X,Y]).\] for any $X,Y \in \mathfrak{gl}_{K+1}$. 
\end{defn}

We then have by Theorem \ref{Theorem1}, Proposition \ref{semistableaffinebps} and Proposition \ref{salatheorem} that

\begin{corollary} \label{womegamuLiealgebra}
Let $\omega$ be any $\Q$ polarization of $S_{K}$ and let $\mu \in \Q_{> 0} $ be any slope.  Then there is an isomorphism of algebras 
\[ \mathcal{A}^{\omega}_{\mu}(S_{K}) \simeq \bU(W^{\omega}_{\mu}(S_{K})). \] 
\end{corollary}

%%Given any stability condition $\zeta$ and slope $\mu$, we recall

%\begin{thm}[\cite{toda2022gopakumarvafa}]
%For any symmetric quiver $Q$ and any potential $W$, $\zeta \in \Q^{Q_0}$ a stability condition, $\bd \in \N^{Q_0)}$ a dimension vector and a GIT proper map \[ q_{\bd}: \mathcal{M}_{\bd}^{\zeta-\sss}(Q) \rightarrow \mathcal{M}_{\bd}(Q)\] We have natural isomorphism \[ (q_{\bd})_{*} \BPS^{\zeta}_{Q,W,\bd} \simeq \BPS_{Q,W,\bd} \]
%\end{thm}

%Since the Lie-bracket on the BPS Lie algebra $\mathfrak{g}^{BPS,\zeta}_{\mu}$ is induced from the Lie bracket \[[,]: \BPS^{\zeta}_{Q,W,\mu} \otimes \mathcal{L}^{1/2} \boxtimes \BPS^{\zeta}_{Q,W,\mu} \mathcal{L}^{1/2} \rightarrow \BPS^{\zeta}_{Q,W,\mu} \otimes \mathcal{L}^{1/2}`\] Pushing forward along the map $q_{\bd}$ and taking cohomology implies that we have isomorphism of graded Lie algebras

%\[ \mathfrak{g}^{\BPS, \zeta}_{\tilde{Q},\tilde{W},\mu}  \simeq \bigoplus_{\zeta(\bd)=\mu} \mathfrak{g}^{\BPS}_{\tilde{Q},\tilde{W},\bd}  \]

%Since the action of $u$ commutes with open restriction, we have an isomorphism of affinized counterparts. We conclude that 

%\begin{prop}
%    We have isomorphism of Lie algebras 
% \[ \hat{\mathfrak{g}}^{\BPS, \zeta}_{\tilde{Q},\tilde{W},\mu}  \simeq \bigoplus_{\zeta(\bd)=\mu} \hat{\mathfrak{g}}^{\BPS}_{\tilde{Q},\tilde{W},\bd}  \]
    
%\end{prop}

\section{Affine Yangians and deformed CoHA}

We will now consider cohomological Hall algebras $\mathcal{A}^{T}_{\tilde{Q},\tilde{W}}$ in the case when the torus $T$ doesn't preserve the symplectic form, i.e., in the case when the construction of the affinized BPS Lie algebra doesn't work. In this case,  we get a deformation of the universal enveloping algebra $\textbf{U}(\mathfrak{g}^{\BPS}_{\tilde{Q},\tilde{W}}[u])$, i.e., Yangians. We calculate these algebras by embedding them into cohomological Hall algebras $\mathcal{A}^{T}_{\tilde{Q}}$, which carry a description using the Shuffle product.

\subsection{Injection to shuffle algebra} \label{injectiontoshuffle}

Given any $T$, with $W=0$, the cohomological Hall algebra can be explicitly described in terms of the shuffle product. Since $\Rep_{\bd}(Q)$ is homotopic to a point, we have as a graded vector space \[ \mathcal{A}^{T}_{Q} = \bigoplus_{\bd \in \N^{Q_0}} \HH_{T}(\pt)[x_{i,n}| i \in Q_0, 1 \leq n \leq \bd_i]^{S_{\bd}}\]
where $S_{\bd} := \prod S_{\bd_i}$ is the product of symmetric groups and $S_{\bd_i}$ acts by permuting the variables $x_{i,1},\dots,x_{i,d_i}$. Then it is proved in \cite{kontsevich2011cohomological}[Theorem 2.2]) that\footnote{Originally this is proved only in the case when $T$ is trivial by localization formula, but then the same proof works in the case when there is $T$ equivariance.}
\begin{prop}
Let $f \in \mathcal{A}^{T}_{Q,\bd^{\prime}}$ and $g \in \mathcal{A}^{T}_{Q,\bd^{\prime \prime}}$ then the product is given by
\begin{align} \label{shuffleproductformula}
& f(x_{0,1},\dots, x_{|Q_0|-1,\bd^{\prime}_{|Q_0|-1}}) \star g(x_{0,1},\dots, x_{|Q_0|-1,\bd^{\prime \prime}_{|Q_0|-1}}) = \\ & \sum_{\sigma \in \Sh_{\bd^{\prime},\bd^{\prime \prime}}} \sigma( f(x_{0,1},\dots, x_{|Q_0|-1,\bd^{\prime}_{|Q|_0-1}})g(x_{0,\bd^{\prime}_0+1},\dots, x_{|Q_0|-1,\bd^{\prime}_{|Q_0|-1}+\bd^{\prime \prime}_{|Q_0|-1}})K_{Q}(\textbf{x}_{\bd^{\prime}}; \textbf{x}_{\bd^{\prime \prime}})) \nonumber, \end{align}

where $K_{Q}(\textbf{x}_{\bd^{\prime}}; \textbf{x}_{\bd^{\prime \prime}})$ is called the shuffle kernel, and it depends on the quiver precisely by
\begin{equation}
K_{Q}(\textbf{x}_{\bd^{\prime}}; \textbf{x}_{\bd^{\prime \prime}}) := \frac{\displaystyle \prod_{a \in Q_1} 
\prod_{\substack{1 \leq m \leq \bd^{\prime}_{s(a)} \\ \bd^{\prime}_{t(a)} < n \leq \bd^{\prime}_{t(a)}+ \bd^{\prime \prime}_{t(a)}}} (x_{t(a),n}-x_{s(a),m} + \bt(a))}{ \displaystyle \prod_{i \in Q_0} \prod_{\substack{1 \leq m \leq \bd^{\prime}_i \\ \bd^{\prime}_{i} < n \leq \bd^{\prime}_i+ \bd^{\prime \prime}_i} }(x_{i,n}-x_{i,m})} 
\label{shufflekernel}
\end{equation}
and $\Sh_{\bd^{\prime},\bd^{\prime \prime}} \subset S_{\bd^{\prime}+\bd^{\prime \prime}}$ is defined to be the subset of permutations $(\sigma_i) \in S_{\bd^{\prime}+\bd^{\prime \prime}}$ where $i \in Q_0$ such that 
\begin{align*}
\sigma_i(1) < \cdots< \sigma_i(\bd^{\prime}_i)  \textrm{    and   }  \sigma_i(\bd^{\prime}_{i}+1) < \cdots < \sigma_i(\bd^{\prime}_{i}+\bd^{\prime \prime}_i) 
\end{align*}
\end{prop}
This product is often referred to as the shuffle product. 
\begin{remark} \label{ontwoshufflealgebras}
For tripled cyclic quiver $\tilde{Q}_K$, also incorporating the twist $\chi$ in (\ref{signtwist}) the shuffle kernel can be written as 
\begin{equation}
K_{\QTC{K}}(x_{\bd^{\prime}}; x_{\bd^{\prime \prime}}):=  (-1)^{\sum_{i \in \mathbb{Z}_{K+1}} \bd^{\prime}_i \bd^{\prime \prime}_i - \bd^{\prime}_i \bd^{\prime \prime}_{i+1}}  \prod_{i,j \in \mathbb{Z}_{K+1}} \prod_{\substack{1 \leq m \leq  \bd^{\prime}_i \\ \bd^{\prime}_{j} <  n < \bd^{\prime}_{j} + \bd^{\prime  \prime}_{j}}} \overline{\omega}_{i,j}^{\text{CoHA}}(x_{i,m},x_{j,n})
\end{equation}
where 
\begin{equation}
    \overline{\omega}_{i,j}^{\text{CoHA}}(z,w) = \left( w-z+t_1 \right)^{\delta_{j,i+1}} \left( w-z+t_2 \right)^{\delta_{j,i-1}} \left( \frac{w-z-t_1-t_2}{w-z} \right)^{\delta_{j,i}}
\end{equation}
In \cite{Bershtein_2019}[Section 5.2], a similarly defined algebra structure on the ring of Laurent symmetric functions is considered as an additive shuffle algebra of cyclic type. Let $\Sh^{\BT}_{Q^K} = \oplus_{\bd \in \N^{K+1}} \Sh^{\BT}_{Q^K,\bd}$ be a $\N^{K+1}$ graded vector space, where 
\[ \Sh^{\BT}_{Q^K,\bd} := \C(t_1,t_2)(x_{i,n}| i \in \mathbb{Z}_{K+1}, 1 \leq n \leq \bd_i)^{S_{\bd}}\]

This vector space is endowed with a associative algebra product $\star_{\BT}$ given by 
\begin{align*}\label{shuffleproductformulaquiver}
& f(x_{0,1},\dots, x_{K,\bd^{\prime}_{K}}) \star_{\BT} g(x_{0,1},\dots, x_{K,\bd^{\prime \prime}_{K}}) := \\ & \Sym \left(  f(x_{0,1},\dots, x_{K,\bd^{\prime}_{K}})g(x_{0,\bd^{\prime}_0+1},\dots, x_{K,\bd^{\prime}_{K}+ \bd^{\prime \prime}_{K}}) \prod_{i,j \in \mathbb{Z}_{K+1}} \left( \prod_{\substack{1 \leq m \leq  \bd^{\prime}_i \\ \bd^{\prime}_{j} <  n < \bd^{\prime}_{j} + \bd^{\prime  \prime}_{j}}} \overline{\omega}_{i,j}(x_{i,m},x_{j,n}) \right) \right) \end{align*}
where\footnote{We do the substitution $\hbar_1 = -t_2,\hbar_2 = t_1+t_2,\hbar_3 = -t_1$}. 
\begin{equation*}
\overline{\omega}_{i,j}(z,w) = \left( \frac{w-z+t_1}{w-z} \right)^{\delta_{j,i+1}} \left(\frac{w-z+t_2}{w-z} \right)^{\delta_{j,i-1}} \left( \frac{w-z-t_1-t_2}{w-z} \right)^{\delta_{j,i}}.
\end{equation*} and for a function $f \in \C(t_1,t_2)(x_{i,n}| i \in \mathbb{Z}_{K+1}, 1 \leq n \leq \bd_i)$, it's symmetrization is defined to be \[\Sym(f) := \left(\prod_{i \in \mathbb{Z}_{K+1}} \frac{1}{\bd_i!} \right) \left( \sum_{\sigma \in S_{\bd}} \sigma(f(x_{0,1},\dots, x_{K,\bd_{K}})) \right). \]
Let $\mathcal{A}^{\T,\chi}_{\QTC{K},\textrm{Loc}}$ be the same vector space as $\Sh^{\BT}_{Q^K}$ but with the algebra product induced by the algebra structure on $\mathcal{A}^{\T,\chi}_{\QTC{K}}$. Then the morphism
\[ S: \mathcal{A}^{T}_{\QTC{K}, \textrm{Loc}} \rightarrow \Sh^{\BT}_{Q^K} \] given by
\[ S( f(x_{\bd})) = (-1)^{\sum_{i \in \mathbb{Z}_{K+1}} \binom{\bd_i}{2} + \bd_i \bd_{i+1}} f(x_{\bd}) \left( \prod_{i \in \mathbb{Z}_{K+1}} \frac{1}{\bd_i!} \left(\prod_{\substack{1 \leq n \leq \bd_i \\ 1 \leq m \leq \bd_{i+1}}} \frac{1}{x_{i,n}-x_{i+1,m}} \right) \right) \] is an algebra isomorphism and so in particular $\mathcal{A}^{\T,\chi}_{\QTC{K}} \hookrightarrow \mathcal{A}^{\T,\chi}_{\QTC{K},\textrm{Loc}} \rightarrow \Sh^{\BT}_{\QTC{K}}$ is an injection of algebras. 

\end{remark}

\subsubsection{Homomorphism to shuffle algebra} Given any quiver $Q$ with potential $W$ and a torus $T$ preserving the potential $W$, assume that there exist a weighting $\mathbf{w}^{\prime}: Q_1 \rightarrow \mathbb{Z}_{\geq 0}$, for which $W$ is homogeneous and of strictly positive weight. Let $Q^{\prime} \subset Q$ be the subquiver consiting only those edges $a$ for which $\textbf{w}^{\prime}(a)=0$. Then there is a homomorphism from the cohomological Hall algebra $\mathcal{A}^{T,\chi}_{Q,W}$ to $\mathcal{A}^{T,\chi}_{Q}$. %We describe the morphism in the setting of quiver $\QTC{K}$ with potential $\WTC{K}$.    
Let $i_{\bd,T} \colon Z^{T}_{\bd}(Q) := \Tr(W)_{\bd}^{-1}(0) \rightarrow \mathfrak{M}^{T}_{\bd}(Q) $ be the inclusion where $\Tr(W)_{\bd} \colon \mathfrak{M}^{T}_{\bd}(Q) \rightarrow \C$. By considering the torus action associated with the weighting $\textbf{w}^{\prime}$, one can construct a homotopy between $Z^{T}_{\bd}(Q)$ and $\mathfrak{M}^{T}_{\bd}(Q^{\prime})$. 
The vanishing cycle functor comes with a natural map 
\[ \pPhi_{W} \Q^{\vir}_{\mathfrak{M}^{T}_{\bd}(Q)} \rightarrow (i_{\bd,T})_{*}(i_{\bd,T})^{*} \Q^{\vir}_{\mathfrak{M}^{T}_{\bd}(Q)}   \] which after taking global sections, gives map of graded vector spaces
\[ i_{\bd,T}: \mathcal{A}^{T}_{Q,W,\bd} \rightarrow \HH(Z^{T}_{\bd}(Q),(i_{\bd})^{*}\Q^{\vir}_{\mathfrak{M}^{T}_{\bd}(Q)}) \simeq  \mathcal{A}^{T}_{Q,\bd}. \]
Here, the second isomorphism is simply because both $\mathfrak{M}^{T}_{\bd}(Q^{\prime})$ and $\mathfrak{M}^{T}_{\bd}(Q)$ are homotopic to a point. This gives a map of $\N^{Q_0} \times \Z$ graded vector spaces
\[ i^{T}: \mathcal{A}^{T,\psi}_{Q,W} \rightarrow \mathcal{A}^{T,\psi}_{Q}.\]
Then it is proved in \cite{botta2023okounkovs}[Prop 4.4] that 

\begin{prop} \label{morphismshuffle}
The morphism \[ i^{T}: \mathcal{A}^{T,\psi}_{Q,W} \rightarrow \mathcal{A}^{T,\psi}_{Q}.\] is a homorphism of $\N^{Q_{0}} \times \Z$ graded algebras. 
\end{prop}

We now show that this morphism respects the action of $\Heis$ constructed in Proposition \ref{heisliealgebraction}. 

\begin{prop}\label{heisqinjection}
The morphism $i^{T}$ is a morphism of $\Heis$ modules.
\end{prop}

\begin{proof}
Since the diagram 
\[\begin{tikzcd}[ampersand replacement=\&]
	{\mathfrak{M}^{T}_{\bd}(Q)} \& {\BCu \times \mathfrak{M}^{T}_{\bd}(Q) } \\
	{Z^{T}_{\bd}(Q) } \& {\BC_u \times Z^{T}_{\bd}(Q) }
	\arrow["{\det \times \id }", from=1-1, to=1-2]
	\arrow["{\det \times \id }"', from=2-1, to=2-2]
	\arrow["{\id \times i^{T}_{\bd}}"', from=2-2, to=1-2]
	\arrow["{i^{T}_{\bd}}", from=2-1, to=1-1]
\end{tikzcd}\]
commutes, it follows from the naturality of the morphism (\ref{exacttriangle}) that the diagram
\[\begin{tikzcd}[ampersand replacement=\&]
	{\pPhi_{0 \boxplus W_{\bd}} \Q_{\BCu   \times \mathfrak{M}_{\bd}^{T}(Q)}} \& {(\Det \times \id)_{*}\pPhi_{W_{\bd}}\Q_{\mathfrak{M}_{\bd}^{T}(Q)}} \\
	{(\id \times i^{T}_{\bd})_{*}\Q_{\BCu \times Z^{T}_{\bd}(Q)} } \& {(i^{T}_{\bd})_{*}\Q_{Z^{T}_{\bd}(Q) }}
	\arrow[from=1-1, to=1-2]
	\arrow[from=2-1, to=2-2]
	\arrow[from=1-1, to=2-1]
	\arrow[from=1-2, to=2-2]
\end{tikzcd}\]
commutes. By taking global sections, it follows that the action by $u$ commutes. The compatibility with the $\partial_u$ action follows in the same way since the diagram 
\[\begin{tikzcd}[ampersand replacement=\&]
	{\BCu \times \mathfrak{M}_{\bd}^{T}(Q)} \& {\mathfrak{M}^{T}_{\bd}(Q)} \\
	{\BCu \times Z^{T}_{\bd}(Q) } \& {Z^{T}_{\bd}(Q)}
	\arrow["\act", from=1-1, to=1-2]
	\arrow["\act"', from=2-1, to=2-2]
	\arrow["{\id \times i^{T}_{\bd}}", from=2-1, to=1-1]
	\arrow["{i_{\bd}^{T}}"', from=2-2, to=1-2]
\end{tikzcd}\]
commutes as $\Tr(W)(\rho) =0 \implies \Tr(W)(\lambda \rho)=0$ for any $\lambda \in \C$. 

\end{proof}

The morphism $i^{T}$ is not, in general, an injection. We can already see this case for tripled cyclic quivers. 

\subsubsection{Not injection for small torus} \label{notinjection} Since $\bps{K}$ is non-trivial by Proposition \ref{bpsLiealgebra}, it follows that neither of algebras $\mathcal{A}^{\chi}_{\QTC{K},\WTC{K}}, \mathcal{A}^{\C^{*},\chi}_{\QTC{K},\WTC{K}}$ are commutative. However by the shuffle formula above, $\mathcal{A}^{\chi}_{\QTC{K}}$ and $\mathcal{A}^{\C^{*},\chi}_{\QTC{K}}$ are commutative. This, in particular, means that the algebra morphisms $i$ and $i^{\C^{*}}$ cannot be injective. 

\subsubsection{Injection for generic torus}

However for the torus $\T$, it is shown in \cite{davison2022integrality}[Theorem 10.2] and \cite{schiffmann2012cherednik} that $\HH(\mathfrak{M}_{\bd}^{\T}(\QTC{K}) \pPhi_{\Tr(\WTC{K})})$ is a torsion free $\HH_{\T \times \GL_{\bd}}$ module, which in particular implies that 
\begin{prop}\label{injectiontoshufflealgebra}
The natural morphism \[ i^{\T}: \mathcal{A}^{\T,\chi}_{\QTC{K},\WTC{K}} \rightarrow \mathcal{A}^{\T,\chi}_{\QTC{K}}\] is an injection of $\N^{Q_{0}} \times \Z$ graded algebras. 
\end{prop}

 By Proposition \ref{flatdeformationofcoha}, we have an isomorphism of cohomologically graded vector spaces \[\mathcal{A}_{\QTC{K},\WTC{K}}^{\T,\chi} \simeq \mathcal{A}^{\chi}_{\QTC{K},\WTC{K}} \otimes \C[t_1,t_2]\] This morphism restricts to an isomorphism of cohomologically graded $\HHT$ modules $\mathfrak{g}^{\BPS, \T}_{\QTC{K},\WTC{K}} \simeq \mathfrak{g}^{\BPS}_{\QTC{K},\WTC{K}} \otimes \C[t_1,t_2]$. Thus for every element $\alpha_{i} \in \mathfrak{g}^{\BPS}_{\QTC{K},\WTC{K}}, i \in [0,K]$ and $\gamma_{n \cdot \delta} \in \mathfrak{g}^{\BPS}_{\QTC{K},\WTC{K}}, n \geq 1$, we have a unique (up to scalar), non-canonical lift $\tilde{\alpha_i}:= \alpha_{i} \otimes 1$ and $\tilde{\gamma_{n \cdot \delta}}: = \gamma_{n \cdot \delta} \otimes 1$, since both $\HH^{-2}(\mathfrak{g}^{\BPS,\T}_{\QTC{K},\WTC{K},n \cdot \delta})$ and $\HH^{0}(\mathfrak{g}^{\BPS,\T}_{\QTC{K},\WTC{K},\delta_i})$ are $1$ dimensional vector spaces. We then have

%\begin{prop}

%Suppose that $\CoHA$ is generated in $\alpha_{\bd_i} \in \mathfrak{g}_{\BPS}$ and their higher modes $\alpha_{\bd_i}^{(k)}:= u^k \cdot \alpha_{\bd_i}$ for $k \geq 0$ then so is $\CoHA^{T}$.
%\end{prop}

\begin{prop}\label{imageofcoha}
 The $\C[t_1,t_2]$ algebra $\mathcal{A}_{\tilde{Q^{K}},W^{K}}^{\mathbb{T}, \chi}$ is generated as a $\C[t_1,t_2]$ algebra by $u^k \cdot \tilde{\alpha}_{i}$ for all $i \in [0,K], k \geq 0$ and $u^k \cdot \tilde{\gamma}_{\delta}$ for all $ k \geq 0$. 
   
\end{prop}

\begin{proof}
There is an isomorphism of Lie algebras 
\[ \mathfrak{g}^{\BPS,\T}_{\QTC{K},\WTC{K}} \simeq \mathfrak{g}^{\BPS} _{\QTC{K},\WTC{K}} \otimes \C[t_1,t_2]\] by \cite{davison2024bps}[Corollary 11.9]. For every $\alpha \in \mathfrak{g}^{\BPS}_{\QTC{K},\WTC{K}}$ we set $\tilde{\alpha} = \alpha \otimes 1 \in \mathfrak{g}^{\BPS,\T}_{\QTC{K},\WTC{K}}$ via this isomorphism and $\tilde{\alpha}^{(m_i)} := u^m \cdot \tilde{\alpha}$. %By Proposition \ref{bpsLiealgebracyclic}, $\mathfrak{g}^{\BPS}_{\QTC{K},\WTC{K}} \simeq \mathfrak{n}^{-}_{Q^K} \oplus s \Q[s]$. 
By Theorem \ref{Theorem1}, we may identify $\mathfrak{g}^{\BPS}_{\QTC{K},\WTC{K}}$ as a Lie subalgebra of $\tilde{W}_{K+1}$, spanned by $T_{k,0}(X)$ where $k \geq 1, X \in \mathfrak{sl}_{K}$ or $k=0, X \in \mathfrak{n}_{K}$ where $\mathfrak{n}_{K} \subset \mathfrak{sl}_{K}$ is the Lie subalgebra generated by upper diagonal matrices and $t_{k,0}$ where $k \geq 1$. For any $k \geq 1$, up to a factor of non-zero scalar, $t_{k,0}$ can be written as $t_{k,0} = \ad^{k-1}_{t_{1,1}}(t_{1,0})$. Similarly, each of $T_{k,0}(X)$ where $k \geq 1, X \in \mathfrak{sl}_{K}$ or $k=0, X \in \mathfrak{n}_{K}$ can be written as succesive Lie brackets on $T_{0,0}(E_{i,i+1}), i \in [1,K]$ and $T_{1,0}(E_{K+1,1})$. Thus every basis element $\alpha_{\bd} \in \mathfrak{g}^{\BPS}_{\QTC{K},\WTC{K},\bd}$ can be written as a linear combination of products of $\alpha^{(0)}_{0}, \alpha^{(0)}_{1}, \cdots, \alpha^{(0)}_{K}, \gamma_{\delta}^{(1)}, \gamma_{\delta}^{(0)}$, i.e. there exists an expression $g_{\alpha_{\bd}}$ such that \[g_{\alpha_{\bd}}(\alpha^{(0)}_{0}, \alpha^{(0)}_{1}, \cdots, \alpha^{(0)}_{K}, \gamma_{\delta}^{(0)}, \gamma_{\delta}^{(1)}) = \alpha_{\bd}. \] We fix such an expression for every basis element $\alpha_{\bd} \in \mathfrak{g}^{\BPS}_{\QTC{K},\WTC{K},\bd}$. We then define a map of $\C[t_1,t_2]$ modules   
\[ \psi_{\bd}: \mathfrak{g}^{\BPS}_{\QTC{K},\WTC{K},\bd} \otimes \C[t_1,t_2] \otimes \C[u] \rightarrow \mathcal{A}^{\T,\chi}_{\QTC{K},\WTC{K}}\] given by \[\psi_{\bd}(\alpha_{\bd} \otimes f(t_1,t_2) \otimes u^m) = f(t_1,t_2) u^m \cdot g_{\alpha_{\bd}}(\tilde{\alpha}^{(0)}_{0}, \tilde{\alpha}^{(0)}_{1}, \cdots, \tilde{\alpha}^{(0)}_{K}, \tilde{\gamma}_{\delta}^{(0)}, \tilde{\gamma}_{\delta}^{(1)} ) \] and thus we have map of  free $\C[t_1,t_2]$ modules

\[ \Psi: \Sym \left(\bigoplus_{\bd} \mathfrak{g}^{\BPS}_{\QTC{K},\WTC{K},\bd} \otimes \C[t_1,t_2] \otimes \C[u] \right) \rightarrow \mathcal{A}^{\T}_{\QTC{K},\WTC{K}} \]

However, $\Psi \otimes_{\C[t_1,t_2]} \C[t_1,t_2]/(t_1,t_2)$ is an isomorphism by the PBW theorem (Theorem \ref{PBWtheorem}) and thus $\Psi$ is an isomorphism of $\C[t_1,t_2]$ modules by the graded Nakayama lemma and so the statement follows.

\end{proof}

\subsubsection{Image in the Shuffle algebra} 

We now determine the image of $\mathcal{A}^{\T,\chi}_{\QTC{K},\WTC{K}}$ under the map $i^{\T}$. By Proposition \ref{imageofcoha}, it suffices to understand the image of $\tilde{\alpha}_{i}^{(k)}, \tilde{\gamma}_{\delta}^{(k)}$ for all $ k \geq 0$.

\begin{prop} \label{imageimaginary}
The image $i^{\T}(\mathcal{A}^{\T,\chi}_{\QTC{K},\WTC{K}}) \subset \mathcal{A}^{\T,\chi}_{\QTC{K},\WTC{K}}$ is the $\C[t_1,t_2]$ subalgebra generated by $x_{i,1}^r \in \mathcal{A}^{\T,\chi}_{\QTC{K}, \delta_i}$ for all $ r \geq 0, i \in [0,\cdots,K]$ and $(t_1+t_2)^{K} (x_{0,1}+x_{1,1}+\cdots + x_{K,1})^r \in \mathcal{A}^{\T,\chi}_{\QTC{K}, \delta}$ for all $r \geq 0$. We have $i^{\T}(\tilde{\alpha}_{i}^{(r)}) = x_{i,1}^r$ and $i^{\T}(\tilde{\gamma}_{\delta}^{(r)}) = \lambda (t_1+t_2)^{K} (x_{0,1}+x_{1,1}+\cdots + x_{K,1})^r$ for some non-zero scalar $\lambda$. 
\end{prop}

\begin{proof}
Since $i^{\T}_{\delta_i}: \CoHA^{T}_{(0,\cdots,1,\cdots,0)} \simeq \C[t_1,t_2][x_{i,1}]$ and $\tilde{\alpha}_i = (i^{\T})^{-1}(1)$, Proposition \ref{heisqinjection} implies that 
\[i^{\T}(\tilde{\alpha}_i^{(k)}) = x_{i,1}^{k} \] for all $k \geq 0$. Let $i^{\T}(\tilde{\gamma}_{\delta}^{0}) = f$ where $f \in \Q[t_1,t_2][x_{0,1},x_{1,1},x_{2,1},\cdots, x_{K,1}]$. Then by the shuffle formula (\ref{shuffleproductformula}), it follows that  

{ \small \begin{align} \label{com1}
\tilde{\gamma}_{\delta}^{0} \star \tilde{\alpha}_i &=  f(x_{0,1},x_{1,1},\cdots,x_{i,1},\cdots, x_{K,1}) \frac{(x_{i,2}-x_{i,1}-t_1-t_2)(x_{i,2}-x_{i-1,1}+t_1)(x_{i,2}-x_{i+1,1}+t_2)}{(x_{i,2}-x_{i,1})} + \\ & f(x_{0,1},x_{1,1},\cdots,x_{i,2},\cdots, x_{K,1}) \frac{(x_{i,1}-x_{i,2}-t_1-t_2)(x_{i,1}-x_{i-1,1}+t_1)(x_{i,1}-x_{i+1,1}+t_2)}{(x_{i,1}-x_{i,2})} \nonumber
\end{align}}
and 
{\small \begin{align} \label{com2}
\tilde{\alpha}_i \star \tilde{\gamma}_{\delta}^{0} &= f(x_{0,1},\cdots,x_{i,2},\cdots,x_{K,1}) \frac{(x_{i,2}-x_{i,1}-t_1-t_2)(x_{i+1,1}-x_{i,1}+t_1)(x_{i-1,1}-x_{i,1}+t_2)}{(x_{i,2}-x_{i,1})} + \\ &  f(x_{0,1},\cdots,x_{i,1},\cdots,x_{K,1}) \frac{(x_{i,1}-x_{i,2}-t_1-t_2)(x_{i+1,1}-x_{i,2}+t_1)(x_{i-1,1}-x_{i,2}+t_2)}{(x_{i,1}-x_{i,2})} \nonumber
\end{align}}
Since $\tilde{\gamma}_{\delta}^{(0)}$ is of cohomological degree $-2$ and $\tilde{\alpha}_{i}^{(0)}$ is of cohomological degree $0$, the Lie bracket $[\tilde{\gamma}_{\delta}^{0}, \tilde{\alpha}_i]=0$ is of cohomological degree $-2$. Since $\HH^{-2}(\mathfrak{g}^{\BPS,\T}_{\QTC{K},\WTC{K}, \delta+\delta_i})=0$, we must have $[\tilde{\gamma}_{\delta}^{0}, \tilde{\alpha}_i]=0$. Since {\tiny \begin{align*}
  \frac{(x_{i,2}-x_{i,1}+t_3)(x_{i,2}-x_{i-1,1}+t_1)(x_{i,2}-x_{i+1,1}+t_2)}{(x_{i,2}-x_{i,1})}  -  &   \frac{(x_{i,1}-x_{i,2}+t_3)(x_{i+1,1}-x_{i,2}+t_1)(x_{i-1,1}-x_{i,2}+t_2)}{(x_{i,1}-x_{i,2})} = \\  \frac{(x_{i,2}-x_{i,1}+t_3)(x_{i+1,1}-x_{i,1}+t_1)(x_{i-1,1}-x_{i,1}+t_2)}{(x_{i,2}-x_{i,1})} -  &\frac{(x_{i,1}-x_{i,2}+t_3)(x_{i,1}-x_{i-1,1}+t_1)(x_{i,1}-x_{i+1,1}+t_2)}{(x_{i,1}-x_{i,2})},
\end{align*}} equation (\ref{com1}), (\ref{com2}) and the condition that $\tilde{\alpha}_i \star \tilde{\gamma}_{\delta}^{0} =\tilde{\gamma}_{\delta}^{0} \star \tilde{\alpha}_i$ implies that for all $i \in \{ 0,1,\cdots,K \}$ we must have 
\begin{align*}
f(x_{0,1},\cdots, x_{i-1,1},x_{i,1}, x_{i+1,1}, \cdots,x_{K,1}) = f(x_{0,1},\cdots,x_{i-1,1}, x_{i,2}, x_{i+1,1}, \cdots,x_{K,1}) 
\end{align*}

This means that for any $i$, $f$ is independent of the variable $x_{i,1}$, i.e. $f \in \C[t_1,t_2]$. Since $\tilde{\gamma}_{\delta}^{(0)}$ is of cohomological degree $-2$, this implies that the degree of polynomial $f(t_1,t_2)$ must be $K$. Now $[\gamma_{\delta}^{0},\alpha_i^{(1)}]$ is an element of cohomological degree $0$ in the dimension vector $\delta_i + \delta$. Since the $\C[t_1,t_2]$ action only increases the cohomological degree, it means that up to a scalar, the only non-zero term of cohomological degree $0$ in dimension $\delta_i+ \delta$ in $\mathfrak{g}^{\BPS,\T}_{\QTC{K},\WTC{K}}$ is $[\tilde{\alpha}_i,\tilde{\alpha}_i,\cdots,\tilde{\alpha}_K,\tilde{\alpha}_{i-1},\cdots,\tilde{\alpha}_0]$ (Proposition \ref{bpsLiealgebracyclic}). Again by the shuffle product description of $\mathcal{A}^{\T,\chi}_{\QTC{K}}$, we can explictly calculate $[i^{\T}(\tilde{\alpha}_i),i^{\T}(\tilde{\alpha}_i),\cdots,i^{\T}(\tilde{\alpha}_K),i^{\T}(\tilde{\alpha}_{i-1}),\cdots,i^{\T}(\tilde{\alpha}_0)] $. We see that for some non-zero constant $\lambda_i \in \Q$ we have 

\begin{equation}
[i^{\T}(\tilde{\alpha}_i),i^{\T}(\tilde{\alpha}_i),\cdots,i^{\T}(\tilde{\alpha}_K),i^{\T}(\tilde{\alpha}_{i-1}),\cdots,i^{\T}(\tilde{\alpha}_0)] = \lambda_i \frac{(t_0+t_1)^K [i^{\T}(\tilde{\gamma}_{\delta}^{(0)}),i^{\T}(\tilde{\alpha}_{i}^{(1)})]}{f(t_1,t_2)}
 \label{commutatordegree01}
\end{equation}

Now Proposition \ref{strongrationality} implies that there  exist some $\mu_i \neq 0$ such that in $\mathcal{A}^{\T,\chi}_{\QTC{K},\WTC{K}}$, 

\begin{equation*}
 [\tilde{\gamma_{\delta}}^{(0)},\tilde{\alpha}_{i}^{(1)}] = \mu_i [\tilde{\alpha}_i,\tilde{\alpha}_i,\cdots,\tilde{\alpha}_K,\tilde{\alpha}_{i-1},\cdots,\tilde{\alpha}_0] + (t_1,t_2) i(\mathcal{A}^{\T,\chi}_{\QTC{K},\WTC{K}}).  
\end{equation*}

This implies 

\[ \left( \frac{f(t_1,t_2)}{\lambda_i (t_0+t_1)^K} - \mu_i \right) [\tilde{\alpha}_i,\tilde{\alpha}_i,\cdots,\tilde{\alpha}_K,\tilde{\alpha}_{i-1},\cdots,\tilde{\alpha}_0] \in (t_1,t_2) i(\mathcal{A}^{\T,\chi}_{\QTC{K},\WTC{K}}). \]

By the PBW theorem (Proposition \ref{PBWtheorem}), the only element of cohomological degree $-2$ in dimension vector $\delta+\delta_i$ could be $\tilde{\gamma}_{\delta}^{(0)} \star \tilde{\alpha}_i^{(0)}$. We again apply the shuffle product formula to calculate $i^{\T}(\tilde{\gamma}_{\delta}^{(0)}) \star i^{\T}(\tilde{\alpha}_i^{(0)})$. We see that \[[i^{\T}(\tilde{\alpha}_i),i^{\T}(\tilde{\alpha}_i),\cdots,i^{\T}(\tilde{\alpha}_K),i^{\T}(\tilde{\alpha}_{i-1}),\cdots,i^{\T}(\tilde{\alpha}_0)] \notin (t_1,t_2) \cdot i^{\T}(\tilde{\gamma}_{\delta}^{(0)}) \star i^{\T}(\tilde{\alpha}_i^{(0)})\] and thus $[\tilde{\alpha}_i,\tilde{\alpha}_i,\cdots,\tilde{\alpha}_K,\tilde{\alpha}_{i-1},\cdots,\tilde{\alpha}_0] \notin (t_1,t_2) i(\mathcal{A}^{\T,\chi}_{\QTC{K},\WTC{K}})$. But $[\tilde{\alpha}_i,\tilde{\alpha}_i,\cdots,\tilde{\alpha}_K,\tilde{\alpha}_{i-1},\cdots,\tilde{\alpha}_0] \neq 0 $ by the calculation of $\mathfrak{g}^{\BPS,\T}_{\QTC{K},\WTC{K}}$ in Proposition \ref{bpsLiealgebracyclic}. Thus, we must have 
\[ f(t_1,t_2) = \lambda_i \mu_i (t_0+t_1)^K .\]

So $i^{\T}(\gamma_{\delta}^{(0)}) = \lambda (t_0+t_1)^{K}$ up to a non-zero scalar $\lambda$. Then by Proposition \ref{heisqinjection}, $i^{\T}(\gamma_{\delta}^{(r)}) = \lambda (t_0+t_1)^{K}(x_{01}+x_{11}+\cdots + x_{K,1})^r$ and we are done.  
%for $i=0$, 

%\[  [i(\gamma_{\delta}^{(0)}),i(\alpha_{0}^{(0)})] = x_{01}^2+x_{02}^2-x_{01}x_{11}-x_{11}x_{21}-x_{01}x_{21}-x_{02}x_{21}+2x_{11}x_{21}+(t_0-t_1)(x_{21}-x_{11})-(t_0^2+t_1^2)
%\]

\end{proof}

This already allows us to identify a commutative subalgebra inside the image $i^{\T}(\mathcal{A}^{\T}_{\QTC{K},\WTC{K}}) \subset \mathcal{A}^{\T}_{\QTC{K}}$. 

\begin{corollary} \label{elementsLn}
Let $e = \sum_{i \in [0,K]} x_{i,1} \in \mathcal{A}^{\T,\chi}_{\QTC{K},  \delta}$. For each $n \geq 1$, we define  \[L_n := [\underbrace{e,[ e, [\cdots[e,}_{n-1 \text{ times }}, 1]]]] \in \mathcal{A}^{\T,\chi}_{\QTC{K}, n \cdot \delta}\] 
Then, the subalgebra generated by $L_n, n \geq 1$ is a polynomial subalgebra isomorphic to $\C[L_1,L_2, \cdots]$.

%and for all $i \in [1,K]$ we define
%\[ L_n^{i} := [\underbrace{e,[ e, [\cdots}_{n-1 \text{ times }}, 2x_{i,1}-x_{i+1,1}-x_{i-1,1}]]] \]
%\[ L_n^{i} := [\underbrace{e,[ e, [\cdots}_{n-1 \text{ times }}, 2x_{i,1}-x_{i+1,1}-x_{i-1,1}]]] \]
%Then $L^i_n, i \in [0,K], n \geq 1$ generates a polynomial subalgebra inside $\mathcal{A}^{\T}_{\QTC{K}}$
\end{corollary}
\begin{proof}
Consider the Lie subalgebra generated by $\tilde{\gamma}_{n \cdot \delta}, n \geq 1$. Then, by Proposition \ref{bpsLiealgebracyclic}, it is a trivial Lie algebra, and the subalgebra generated by this Lie algebra is a polynomial subalgebra inside $\mathcal{A}^{\T,\chi}_{\QTC{K},\WTC{K}}$. Let us call this algebra $A^K$. We show that the image of this under $i^{\T}$ up to a factor of polynomials in $\C[t_1,t_2]$ is exactly the subalgebra generated by the above generators. For any $n \geq 1$ and $\tilde{\gamma}_{n \cdot \delta} \in \HH^{-2}(\mathfrak{g}^{\BPS,\T}_{\QTC{K},\WTC{K},n \cdot \delta})$ consider the commutator $[\tilde{\gamma}_{\delta}^{(1)},\tilde{\gamma}_{n \cdot \delta}]$. Then since $\tilde{\gamma}_{\delta}^{(1)} \in \mathfrak{P}^{\leq 3}$ and $\tilde{\gamma}_{n \cdot \delta} \in \mathfrak{P}^{\leq 1}$, it follows that $[\tilde{\gamma}_{\delta}^{(1)},\tilde{\gamma}_{n \cdot \delta}] \in \mathfrak{P}^{\leq 3}$ where we recall that $\mathfrak{P}^{\leq i}$ denote the perverse filtration defined in Section \ref{perversefiltrationandbps}. Thus, it should be a linear combination of elements of the form $\mathfrak{P}^{\leq 1}, \mathfrak{P}^{\leq 1} \star \mathfrak{P}^{\leq 1}$ or $u \cdot \mathfrak{P}^{\leq 1}$. However minimal cohomological degree of $u \cdot \mathfrak{P}^{\leq 1}$ is $0$ and thus it follows that we must have 
\begin{align}\label{equationdelta1n}
[\tilde{\gamma}_{\delta}^{(1)},\tilde{\gamma}_{n \cdot \delta}] = \lambda \tilde{\gamma}_{(n+1)\cdot \delta} + \sum_{n_i+n_j=n+1} \lambda_{n_i,n_j} \tilde{\gamma}_{n_i \cdot \delta} \star \tilde{\alpha}_{n_j \cdot \delta} + \sum_{k_i+k_j=n+1} f_{k_i,k_j}(t_1,t_2) \tilde{\gamma}_{k_i \cdot \delta} \star \tilde{\gamma}_{k_j \cdot \delta}
\end{align} where $\tilde{\alpha}_{n_j \cdot \delta} \in \mathfrak{n}^{-}_{Q^{K}} \subset \HH^{0}(\mathfrak{g}^{\BPS}_{\QTC{K},\WTC{K},n_j \cdot \delta})$ for some $\lambda, \lambda_{n_i,n_j} \in \C$ and linear functions $f_{k_i,k_j}(t_1,t_2)$. We claim that each of $\lambda_{n_i,n_j}=0$. For any $\tilde{\alpha}_{n_j \cdot \delta}$, by Proposition \ref{bpsLiealgebracyclic}, there exist $\tilde{\alpha}_j$ in dimension $\delta_j$ such that $[\tilde{\alpha}_{n_j \cdot \delta}, \tilde{\alpha}_j] \neq 0$. But then since $\tilde{\gamma}_{n \cdot \delta}$ commutes with $\mathfrak{n}^{-}_{Q^K}$, applying $[-, \tilde{\alpha}_j]$ to the  above equation (\ref{equationdelta1n}) gives that \[ 0 = \sum_{n_i+n_j=n+1} \lambda_{n_i,n_j} \tilde{\gamma}_{n_i \cdot \delta} \star [\tilde{\alpha}_{n_j \cdot \delta}, \tilde{\alpha}_j ]. \] But each of non-zero $\tilde{\gamma}_{n_i \cdot \delta} \star [\tilde{\alpha}_{n_j \cdot \delta}, \tilde{\alpha}_j ]$ can be extended to a basis of $\mathcal{A}^{\T}_{\QTC{K},\WTC{K}}$ by the PBW theorem (Proposition \ref{PBWtheorem}). Since the above equation gives a relation, it implies that $\lambda_{n_i,n_j}=0$. Furthemore $\lambda \neq 0$ by Proposition \ref{relation1}. Thus it follows that we must have $[\tilde{\gamma}_{\delta}^{(1)},\tilde{\gamma}_{n \cdot \delta}]\in A^K$. By equation \ref{equationdelta1n}, it follows that the algebra generated by $[\tilde{\gamma}_{\delta}^{(1)},\tilde{\gamma}_{n \cdot \delta}], n \geq 1$ and $\tilde{\gamma}_{\delta}$ contains all the generators $\tilde{\gamma}_{n \cdot \delta}$ of $A^K$. But then, so does 
\begin{align}
\tilde{L}_n : = [\underbrace{\tilde{\gamma}_{\delta}^{(1)},[ \tilde{\gamma}_{\delta}^{(1)}, [\cdots [\tilde{\gamma}_{\delta}^{(1)}}_{n-1 \text{ times }}, \tilde{\gamma}_{\delta}]]]]\in A^K \subset \mathcal{A}^{\T}_{\QTC{K},\WTC{K}}
\end{align} Thus algebra generated by $\tilde{L}_n, n \geq 1$ is $A^K$, the commutative subalgebra. Thus, we are done from Proposition \ref{imageimaginary}.

\end{proof}

Proposition \ref{imageimaginary} also gives a slight refinement of a particular case of \cite{schiffmann2017cohomological} and \cite{neguţ2023generators} about spherical generation of localized cohomological Hall algebra. We note it down as the argument will be helpful later. 

\begin{corollary}  \label{sphgeneration}
The $\C[t_1,t_2]$ linear algebra $\mathcal{A}^{\T,\chi}_{\QTC{K},\WTC{K}}$ is not a spherically generated $\C[t_1,t_2]$ algebra, while the localized algebra $\mathcal{A}^{\T,\chi}_{\QTC{K},\WTC{K}} \otimes \C[t_1^{-1},t_2^{-1}]$ is spherically generated. 
\end{corollary}

\begin{proof}
The fact that $\mathcal{A}^{\T,\chi}_{\QTC{K},\WTC{K}}$ is not spherically generated follows from the fact that the cohomological degree of $\tilde{\alpha}_{i}^{(r)} \geq 0$, while cohomological degree of $\tilde{\gamma}_{\delta}^{0} = -2$. Since the product structure preserves cohomological grading, it is not possible to reach this element by taking products of $\tilde{\alpha}_{i}^{(r)}$, nor by multiplication with any element of $\C[t_1,t_2]$, since that would also increase cohomological degree. 

On the other hand by Proposition \ref{imageimaginary}, up to a factor of non-zero scalar, $i^{\T}(\gamma_{\delta}^{(0)}) = (t_1+t_2)^K \in \mathcal{A}^{\T,\chi}_{\QTC{K},\WTC{K},\delta}$. By the argument above $i^{\T}(\gamma_{\delta}^{(0)})$ cannot be written as the $\C[t_1,t_2] \hyphen$linear combination of products of $i^{\T}(\tilde{\alpha}_{i}^{(r)})$. However, we can write it as a $\C[t_1^{\pm 1},t_2^{\pm 1}]\hyphen$linear combination of the products of $i^{\T}(\tilde{\alpha}_{i}^{(r)})$. Let \begin{align}
\mathcal{Y}_2 &=  \tilde{\alpha}^{(0)}_{0}  \talpha^{(0)}_1 + \talpha^{(0)}_1  \tilde{\alpha}^{(0)}_{0} \\ \mathcal{Z}_{2} &= [\talpha^{(0)}_{0},\talpha^{(1)}_{1}] + [\talpha^{(0)}_{1},\talpha^{(1)}_{0}]
\end{align} for $K=1$ and 

\begin{align}
\mathcal{Y}_{K+1} &= \sum_{i \in \mathbb{Z}_{K+1}} \tilde{\alpha}_{i}^{(0)} \cdot \left( [\tilde{\alpha}_{i+1}^{(0)},[\tilde{\alpha}_{i+2}^{(0)},\cdots, [ \tilde{\alpha}^{(0)}_{i+K-1}, \tilde{\alpha}^{(0)}_{i+K}]]] \right) \\ \mathcal{Z}_{K+1} &= \sum_{i \in \mathbb{Z}_{K+1}} [\tilde{\alpha}^{(0)}_{i},[\tilde{\alpha}^{(1)}_{i+1}, [\tilde{\alpha}^{(0)}_{i+2},[ \cdots, [\tilde{\alpha}^{(0)}_{i+K-1},\tilde{\alpha}^{(0)}_{i+K}]]]] 
\end{align} for $K \geq 2$.

Note that second term in $\mathcal{Z}_{K+1}$ is $\tilde{\alpha}^{(1)}_{i+1}$. Using the shuffle product formula (\ref{shuffleproductformula}), it follows that for \[i^{\mathbb{T}}(\mathcal{Y}_{2}) = (-2)((x_{1,1}-x_{0,1})^2+t_1t_2) \] and \[ i^{\mathbb{T}}(\mathcal{Y}_{K+1}) = (-1)^K (t_1+t_2)^{K-1} \sum_{i \in \mathbb{Z}_{K+1}} \left(x_{i+1,1}x_{i-1,1}-2x_{i+1,1}x_{i,1}+x_{i,1}^2 + t_1t_2 \right) \]  for $K \geq 2$. Similarly, we calculate that \[ i^{\mathbb{T}}(\mathcal{Z}_2) = (-2)(t_1+t_2)(x_{1,1}-x_{0,1})^2\] and 
\[ i^{\mathbb{T}}(\mathcal{Z}_{K+1}) = (-1)^{K}(t_1+t_2)^K \sum_{i \in \mathbb{Z}_{K+1}} \left( x_{i+1,1}x_{i-1,1} - 2x_{i+1,1}x_{i,1} + x_{i,1}^2 \right)\] for $K \geq 2$.

Thus for all $K \geq 1$, it follows that 

\[ (t_1+t_2)i^{\mathbb{T}}(\mathcal{Y}_{K+1}) - i^{\mathbb{T}}(\mathcal{Z}_{K+1}) = (-1)^{K}(K+1)(t_1+t_2)^{K}(t_1t_2)\] and thus it follows that

\[\frac{1}{(-1)^K(K+1)} \left(\left(\frac{t_1+t_2}{t_1t_2} \right) i^{\T}(\mathcal{Y}_{K+1})- \left(\frac{1}{t_1t_2} \right) i^{\T}(\mathcal{Z}_{K+1}) \right) = (t_1+t_2)^K \]

Thus by Proposition \ref{heisqinjection}, it follows that for any $K \geq 1$, we can write $\tilde{\gamma}_{\delta}^{r}$ as a non-zero scalar multiple of 
\begin{equation} \label{imageofydelta}
\mathcal{K}^{r}_{K+1}:= u^r \cdot (\frac{(t_1+t_2)\mathcal{Y}_{K+1} - \mathcal{Z}_{K+1}}{t_1t_2})
\end{equation} 

\end{proof}

Now that we know the image of $\mathcal{A}^{\T,\chi}_{\QTC{K},\WTC{K}}$ in the shuffle algebra, we show that the image is an integral form of an affine Yangian.
\subsection{Affine Yangian} \label{mapyangian}
Let $n \geq 2$. Let $\ddot{\mathcal{Y}}_{\hbar_1,\hbar_2}(\mathfrak{gl}(n))$ be the affine Yangian of $\mathfrak{gl}_n$ considered in \cite{Bershtein_2019}, which is a slight reparametrization of the affine Yangian defined by \cite{GUAY2007436}. The definition of the affine Yangian for $n =2$ was first given in \cite{koderayangian} and \cite{Bershtein_2019}. Recall $\mathcal{Y}^{(n),\CoHA}_{\hbar_1,\hbar_2}$ from the Definition \ref{cohayangiandefn}, which is slight modification of the affine Yangian.

\begin{defn}\label{usualyangianrelations}
For any $n \geq 2$, let the positive half $\ddot{\mathcal{Y}}^{+}_{\hbar_1,\hbar_2}(\mathfrak{gl}(n))$ of the affine Yangian $\ddot{\mathcal{Y}}_{\hbar_1,\hbar_2}(\mathfrak{gl}(n))$ to be the $\C[\hbar_1,\hbar_2]$ sub algebra of $\mathcal{Y}^{(n),\CoHA}_{\hbar_1,\hbar_2}$ generated by the elements $X^{+}_{i,r}$ where $i=0,\dots,n-1$. 
\begin{comment}
 
% general [X_{i,r+1}^{+},X^{+}_{j,s}] - [X^{+}_{i,r},X^{+}_{j,s+1}] &= \frac{\hbar_2}{2}c_{i,j}(X^{+}_{i,r}X^{+}_{j,s}+X^{+}_{j,s}X^{+}_{i,r}) \ \textrm{for } i,j \in \{1,\cdots,n \} \\ 
\begin{subequations}
\begin{align}
    [X_{i,r+1}^{+},X^{+}_{i+1,s}] - [X^{+}_{i,r},X^{+}_{i+1,s+1}] &= (\hbar_1+\hbar_2/2)[X^{+}_{i,r},X^{+}_{i+1,s}] - \hbar_2/2 \{X^{+}_{i,r},X^{+}_{i+1,s} \} \label{rel11} \\ 
    [X_{i,r+1}^{+},X^{+}_{i-1,s}] - [X^{+}_{i,r},X^{+}_{i-1,s+1}] &= -(\hbar_1+\hbar_2/2)[X^{+}_{i,r},X^{+}_{i-1,s}] - \hbar_2/2 \{X^{+}_{i,r},X^{+}_{i-1,s} \}  \label{rel12}\\ 
    [X_{i,r+1}^{+},X^{+}_{i,s}] - [X^{+}_{i,r},X^{+}_{i,s+1}] &= \hbar_2(X^{+}_{i,r}X^{+}_{i,s}+X^{+}_{i,s}X^{+}_{i,r}) \label{rel13} \\ 
    [X^{+}_{i,r},X^{+}_{j,s}] &= 0 \ \forall \ |i-j|> 1  \label{rel14} \\
    \sum_{\sigma \in S_2} [X^{+}_{i,r_{\sigma(1)}},[X^{+}_{i,r_{\sigma(2)}},X^{+}_{i+1,s}]] &= 0 \label{serre1}
\end{align}
\end{subequations}
when $n \geq 2$ and 
\begin{subequations}
\begin{align}
[X^{+}_{i,r+1},X^{+}_{i,s}] - [X^{+}_{i,r},X^{+}_{i,s+1}] = \hbar_2 \{ X^{+}_{i,r},X^{+}_{i,s}\} \label{rel21}\\
[X^+_{i,r+2},X^{+}_{i+1,s}] - 2[X^{+}_{i,r+1},X^{+}_{i+1,s+1}]+[X^{+}_{i,r},X^{+}_{i+1,s+2}] = \nonumber \\  \hbar_1(\hbar_1+\hbar_2)[X^+_{i,r},X^{+}_{j,s}] - \hbar_2(\{X^{+}_{i,r+1},X^{+}_{i+1,s} \}- \{X^{+}_{i,r},X^{+}_{i+1,s+1} \}) \label{rel22}\\
\mathrm{Sym}_{r_1,r_2,r_3} [X^{+}_{i,r_1},[X^{+}_{i,r_2},[X^{+}_{i,r_3},X^{+}_{i+1,s}]]] = 0 \label{serre2}
\end{align} 
\end{subequations}
when $n=2$. 
\end{comment}

\end{defn}

Note that we can give a \textit{dimension} grading to the algebra $\ddot{\mathcal{Y}}^{+}_{\hbar_1,\hbar_2}(\mathfrak{gl}(n))$ by declaring $\textbf{dim}(X^{+}_{i,r}) = \delta_i$ where $\delta_i = (0,\cdots,1,\cdots,0) \in \N^{n}$. We also have a \textit{cohomological} grading given by declaring $\CG(X^{+}_{i,r}) = 2r$ and $\CG(\hbar_1)=\CG(\hbar_2)=2$. Then the size of this algebra is known due to earlier work. We have

\begin{prop}\label{sizeofaffineyangian}
There is an isomorphism of free $\C[\hbar_1,\hbar_2]$ modules
\[ \ddot{\mathcal{Y}}^{+}_{\hbar_1,\hbar_2}(\mathfrak{gl}(n)) \simeq \Sym(\mathcal{L}_n) \otimes\mathbb{C}[\hbar_1,\hbar_2]\] where $\mathcal{L}_n$ is $\N^{n} \times \mathbb{Z}$ graded vector space of graded dimension
\begin{align} \label{sizeofbpsu}
\dim((\mathcal{L}_n)^{\CG}_{\bdim}) = \begin{cases}
n-1  & \textrm{if } \CG = 0, \bdim \in \mathbb{Z}_{\geq 1} \cdot \delta \\
n  & \textrm{if } \CG \geq 1, \bdim \in \mathbb{Z}_{\geq 1} \cdot \delta \\ 1 & \textrm{if } \CG  \in \mathbb{Z}_{\geq 0}, \bdim \text{ is a positive real root of } \tilde{\mathfrak{sl}}_{n}
\\ 0 & \textrm{ otherwise } 
\end{cases}
\end{align}
\end{prop}

\begin{proof}
We define a filtration $F$ of  $\ddot{\mathcal{Y}}^{+}_{\hbar_1,\hbar_2}(\mathfrak{gl}(n))$ by assigning $X^{+}_{i,r}$ degree $2r$ and rest $0$. Then the associated graded algebra $\mathrm{Gr}_{F}( \ddot{\mathcal{Y}}^{+}_{\hbar_1,\hbar_2}(\mathfrak{gl}(n)) ) \simeq \mathbf{U}(\mathcal{L}_n) \otimes \mathbb{C}[\hbar_1,\hbar_2]$ where $\mathcal{L}$ is Lie algebra generated by relations  $\mathcal{L}_n$ is a Lie algebra over $\mathbb{C}$, generated by $X^{+}_{i,r}, r \geq 0, i \in [0,n-1]$ with the defining relations  
\begin{align*}
[X^{+}_{i,r+1},X^{+}_{i,s}] - [X^{+}_{i,r},X^{+}_{i,s+1}] &= 0 \\
[X^+_{i,r+2},X^{+}_{i+1,s}] - 2[X^{+}_{i,r+1},X^{+}_{i+1,s+1}]+[X^{+}_{i,r},X^{+}_{i+1,s+2}] &=  0  \\
\mathrm{Sym}_{r_1,r_2,r_3} [X^{+}_{i,r_1},[X^{+}_{i,r_2},[X^{+}_{i,r_3},X^{+}_{i+1,s}]]] &= 0 
\end{align*} when $n=2$
and 
\begin{align*}
[X_{i,r+1}^{+},X^{+}_{i+1,s}]  &= [X^{+}_{i,r},X^{+}_{i+1,s+1}] \\ 
[X_{i,r+1}^{+},X^{+}_{i-1,s}]  &= [X^{+}_{i,r},X^{+}_{i-1,s+1}] \\ 
[X_{i,r+1}^{+},X^{+}_{i,s}] &= [X^{+}_{i,r},X^{+}_{i,s+1}] \\ 
[X^{+}_{i,r},X^{+}_{j,s}] &= 0 \ \forall |i-j|> 1\\
\Sym_{r_1,r_2}[X^{+}_{i,r_1},[X^{+}_{i,r_2},X^{+}_{i\pm 1,s}]] &= 0
\end{align*} when $n>2$. 

Given any commutative algebra $A$, we can consider the Lie algebra $\mathfrak{sl}_n(A)$, defined by the Lie bracket $[X \otimes a, Y \otimes b] = [X,Y] \otimes ab$ where $X,Y \in \mathfrak{sl}_n$ and $a,b \in A$ are arbitrary elements. Let $\widehat{\mathfrak{sl}_n}(A)$ be the universal central extension of $\mathfrak{sl}_n(A)$ which is as vector spaces $\mathfrak{sl}_n(A) \oplus \mathrm{\Omega}^{1}(A)/dA$, where $\mathrm{\Omega}^{1}(A)/dA$ is the space of $1$ forms on the affine variety $\Spec(A)$, modulo the exact forms (We refer to \cite{MR1066569} and \cite{GUAY2007436} for details). The Lie algebra is defined by 
\[ [z_1 \otimes a_1,z_2 \otimes a_2] = [z_1,z_2] \otimes a_1 \cdot a_2 + \langle z_1,z_2 \rangle  a_2 da_1\] where $\langle \cdot, \cdot \rangle$ is the Killing form.
%, $E_{i,i+1}v^n, n \geq 0; E_{K+1,1} uv^n$, $n \geq 0$ where $E_{ij} \in \mathfrak{sl}_{K+1}$ 
We consider the setting when $A = \C[u^{\pm 1},v]$. Let $\mathfrak{b}^{+}_n \subset \widehat{\mathfrak{sl}_{n}}[u^{\pm 1},v]$ be the positive half. It can be identified as the Lie algebra spanned by $Xu^av^b$ where $X \in \mathfrak{sl}_{n}, a \geq 1, b \geq 0$ and $Xv^b$ where $X \in \mathfrak{n}^{+}_{n}$ and $K_{b}[a]$ where $a,b \geq 1$, such that they satisfy the relations
\[ [Xu^{m_1}v^{n_1}, Yu^{m_2}v^{n_2}] = [X,Y] \otimes u^{m_1+m_2}v^{n_1+n_2} + \langle X,Y \rangle (m_2n_1-m_1 n_2) K_{m_1+m_2}[n_1+n_2], \] with $K_{b}[a]$ central.  

For $n>2$, by \cite{GUAY2007436}[Lemma 3.4] and for $n=2$, by \cite{salalong}[Proposition IV.3.10]; we have an isomoprhism of Lie algebras $\mathcal{L}_n \simeq \mathfrak{b}^{+}_n$ given by $X^{+}_{i,r} \mapsto E_{i,i+1}v^r$ and $X^{+}_{0,r} \mapsto E_{n,1} uv^r$. The advantage is that now we have an explicit spanning set of the Lie algebra $\mathcal{L}_n$. The $\dim$ and $\CG$ grading makes $\mathcal{L}_n$ a $\mathbb{N}^n \times \mathbb{Z}$ graded vector spaces. Computing these graded dimensions proves the above theorem. 
\end{proof}

%This is essentially an extension of \cite{yang2018pbw} for the case with two parameters, but the proof techniques are entirely different. 

\begin{prop}\label{Isomorphismloc}
Assume $K \geq 1$. Then we have a homomorphism of algebras\footnote{A 1-parameter version of this proposition is also proved in \cite{yang2017cohomological} for any loop-free quiver $Q$. }

\[ \mathrm{\Phi}: \ddot{\mathcal{Y}}^{+}_{-t_2,t_1+t_2}(\mathfrak{gl}(K+1)) \rightarrow \mathcal{A}^{\T,\chi}_{\QTC{K},\WTC{K}} \] given by the assignment $X^{+}_{i,s} \mapsto \tilde{\alpha}_{i}^{(s)}$. 
\end{prop}

\begin{proof}
By Proposition \ref{injectiontoshufflealgebra}, there is an injection of algebras $i^{\T}: \mathcal{A}^{\T,\chi}_{\QTC{K},\WTC{K}} \hookrightarrow \mathcal{A}^{\T,\chi}_{\QTC{K}}$. 
Since $i^{\T}(\mathrm{\Phi}(X^{+}_{i,s})) = x^{s}_{i,1}$. We can use the explicit shuffle product (Equation \ref{shufflekernel}) on $\mathcal{A}^{\T,\chi}_{\QTC{K}}$ to check that the defining relations (\ref{usualyangianrelations}) of the positive half of the affine Yangian are satisfied. We refer author's thesis \cite{mythesis} for details. 
\end{proof}

We now show that in fact the morphism $\mathrm{\Phi}$ is an injection\footnote{In an earlier version of the paper, for $K>1$, we proved this by comparing graded dimension after taking Perverse filtration.}.
\begin{prop}\label{Isomorphismloc}
Assume $K \geq 1$. Then we have an injection of algebras:

\[ \mathrm{\Phi}: \ddot{\mathcal{Y}}^{+}_{-t_2,t_1+t_2}(\mathfrak{gl}(K+1)) \rightarrow \mathcal{A}^{\T,\chi}_{\QTC{K},\WTC{K}} \] given by the assignment $X^{+}_{i,s} \mapsto \tilde{\alpha}_{i}^{(s)}$.  
\end{prop}

\begin{proof}
We show that, in fact, we have an isomorphism 
\[ \mathrm{\Phi}: \ddot{\mathcal{Y}}^{+}_{-t_2,t_1+t_2}(\mathfrak{gl}(K+1)) \simeq \mathcal{S}\mathcal{A}^{\T,\chi}_{\QTC{K},\WTC{K}} \hookrightarrow \mathcal{A}^{\T,\chi}_{\QTC{K},\WTC{K}} \]

Clearly by definition, $\mathrm{\Phi}$ is a surjection on its image. The spherical subalgebra $\mathcal{S}\mathcal{A}^{\T,\chi}_{\QTC{K},\WTC{K}}$ is cohomologically and dimensionally graded. It is shown in \cite{schiffmanncohagenerators}[Theorem 5.18] that $\mathcal{S}\mathcal{A}^{\T,\chi}_{\QTC{K},\WTC{K}}$ is isomorphic to the nilpotent cohomological Hall algebra $\mathcal{A}^{\mathbb{T},\mathcal{N},\chi}_{\mathrm{\Pi}_Q} \simeq \mathcal{A}^{\mathbb{T},\tilde{\mathcal{N}},\chi}_{\tilde{Q},\tilde{W}} $. We now compute the graded dimension of nilpotent CoHA. By Theorem \ref{nilpotentfree}, we have an isomorphism of free $\mathbb{C}[t_1,t_2]$ modules 

\[ \mathcal{A}^{\mathbb{T},\tilde{\mathcal{N}}}_{\tilde{Q},\tilde{W}} \simeq \mathcal{A}^{\tilde{\mathcal{N}}}_{\tilde{Q},\tilde{W}} \otimes \mathbb{C}[t_1,t_2].\]

In Example \ref{nilpotentkacpolynomialandrelationwithbpsliealgebra}, we proved that there is an isomorphism of graded vector spaces:

\[ \mathcal{A}^{\tilde{\mathcal{N}}}_{\tilde{Q},\tilde{W}} \simeq \Sym(\HH(\mathcal{M}_{\bd}^{\mathcal{N}_{\overline{Q}}}(\overline{Q}), i^{!}\mathcal{BPS}_{\mathrm{\Pi}_Q})[u]). \]

and \[
 \sum_{i } (-1)^{i} \dim(\HH^{i}(\mathcal{M}_{\bd}^{\mathcal{N}_{\overline{Q}}}(\overline{Q}), i^{!}\mathcal{BPS}_{\mathrm{\Pi}_Q})  )q^{i/2} = a_{Q,\bd}(q).
\] and thus by calculation of Kac polynomial for cyclic quiver in Equation \ref{kacpolynomialcyclicquiver}, it follows that
\begin{align} \label{sizeofbpsu}
\dim(\HH(\mathcal{M}_{\bd}^{\mathcal{N}_{\overline{Q}}}(\overline{Q}), i^{!}\mathcal{BPS}_{\mathrm{\Pi}_Q}[u])^{i}) = \begin{cases}
K  & \textrm{if } i = 0, \bd \in \mathbb{Z}_{\geq 1} \cdot \delta \\
K+1  & \textrm{if } i \geq 1, \bd \in \mathbb{Z}_{\geq 1} \cdot \delta \\ 1 & \textrm{if } i  \in \mathbb{Z}_{\geq 0}, \bd \text{ is a positive real root of } \tilde{\mathfrak{sl}}_{K+1}
\\ 0 & \textrm{ otherwise } 
\end{cases}
\end{align}
which matches the graded dimensions of the affine Yangian in Proposition \ref{sizeofaffineyangian}. Thus, the morphism $\mathrm{\Phi}$ is an injection. 
\end{proof}

\begin{remark} \label{nilpotent}
From the above argument, it follows that there is an isomorphism of $\mathbb{C}[t_1,t_2]$ linear algebras $\ddot{\mathcal{Y}}^{+}_{-t_2,t_1+t_2}(\mathfrak{gl}(K+1)) \simeq \mathcal{A}^{\mathbb{T},\tilde{\mathcal{N}},\chi}_{\tilde{Q^K},\tilde{W^K}}$. This statement is now also proved in \cite{salalong}[Theorem IV.3.11]. 
\end{remark}

%\begin{remark}
%We note that the proof of injection in \cite{Rap_k_2019}[Theorem 7.1.1] and \cite{Bershtein_2019}[Theorem 5.4] is missing, but it follows from the same argument solving BPS Lie algebras.
%\end{remark}

We now identify the integral form precisely.

\begin{lemma}
There exists a derivation $T: \ddot{\mathcal{Y}}^{+}_{\hbar_1,\hbar_2}(\mathfrak{gl}(n)) \rightarrow  \ddot{\mathcal{Y}}^{+}_{\hbar_1,\hbar_2}(\mathfrak{gl}(n))$ such that $T(X^{+}_{i,r}) = X^{+}_{i,r+1}$
\end{lemma}

\begin{proof}
We define $T(X^{+}_{i,r}) = X^{+}_{i,r+1}$ and extend it by $T(ab) := T(a)b+aT(b)$ for $a,b \in \ddot{\mathcal{Y}}^{+}_{\hbar_1,\hbar_2}(\mathfrak{gl}(n))$ and note that it satisfies the relations in Definition \ref{usualyangianrelations} and hence is well defined\footnote{Alternatively one can use \cite{Guay_2018}[Lemma 2.22] to construct such a operator in $\ddot{\mathcal{Y}}^{(0)}_{\hbar_1,\hbar_2}(\mathfrak{gl}(n))$.}. 

%Alternatively, from \cite{Guay_2018}[Lemma 2.22], for each $i \geq 0$, there exist operator $\tilde{h}_{i,1} \in \ddot{\mathcal{Y}}^{(0)}_{\hbar_1,\hbar_2}(\mathfrak{gl}(n))$ such that $[\tilde{h}_{i,1}, X^{+}_{j,s}] = (2\delta_{i,j}-\delta_{i,j-1}-\delta_{i,j+1}) X^{+}_{j,s+1}$. Let $\tilde{T} = \sum \lambda_i \tilde{h}_{i,1}$, where $\lambda_i = \frac{(i+1)(n-i)}{2}$. Then, by definition, it is a derivation, and it satisfies $[\tilde{T}, X^{+}_{i,r}] = X^{+}_{i,r+1}$, which is what we want. 
\end{proof}

\begin{defn}
Assume $n \geq 2$. Let $\mathcal{Y}^{(n),+, \CoHA}_{\hbar_1,\hbar_2} \subset \ddot{\mathcal{Y}}^{+}_{\hbar_1,\hbar_2}(\mathfrak{gl}(n)) \otimes_{\C[\hbar_1,\hbar_2]} \C[\hbar_{1}^{\pm 1},(\hbar_1+\hbar_2)^{\pm 1}]$ be a $\mathbb{C}[\hbar_1,\hbar_2]$ linear sub algebra generated by $X^{+}_{i,r}$ and $\mathfrak{K}_{n}^{(r)}$ for $r \geq 0$ such that for all $r \geq 0$, we have an extra relation 
\begin{align*}
(\hbar_1)(\hbar_1+\hbar_2) \mathfrak{K}^{(r)}_{n} &= \TT^r (\mathfrak{Z}_n-\hbar_2 \mathfrak{Y}_{n}) 
\end{align*}
where  
\begin{align*}
    \mathfrak{Y}_2 & := X^{+}_{0,0}X^{+}_{1,0}+ X^{+}_{1,0}X^{+}_{0,0} \\ \mathfrak{Z}_2 & := [X^{+}_{0,0},X^{+}_{1,1}]+ [X^{+}_{1,0},X^{+}_{0,1}]
\end{align*} and 
\begin{align*}
\mathfrak{Y}_n & := \sum_{i \in \mathbb{Z}/n\mathbb{Z}} X^{+}_{i,0} \cdot \left([X^{+}_{i+1,0},[X^{+}_{i+2,0},\cdots,[X^{+}_{i+n-2,0} X^{+}_{i+n-1,0}]]] \right) \\ \mathfrak{Z}_n & := \sum_{i \in \mathbb{Z}/n\mathbb{Z}} [X^{+}_{i,0},[X^{+}_{i+1,1}, [X^{+}_{i+2,0},[ \cdots,[X^{+}_{i+n-2,0},X^{+}_{i+n-1,0}]]]]. 
\end{align*} for $n \geq 2$. 
Note that second term in the sum defining $\mathfrak{Z}_n$ is $X^{+}_{i+1,1}$. 
\end{defn}

Note that these extra elements are defined purely via cohomological Hall algebras, but quite remarkably they satisfy interesting algebraic simplification when $\hbar_2=0$. See Proposition \ref{positivedegisom} in particular.

\begin{thm}\label{equivcoha}
We have an isomorphism of $\C[t_1,t_2]$ algebras
\[ \tilde{\mathrm{\Phi}} : \mathcal{Y}^{(K+1),+, \CoHA}_{-t_2,t_1+t_2} \rightarrow \mathcal{A}^{\T,\chi}_{\QTC{K},\WTC{K}}, \] given by the assignment $X^{+}_{i,s} \mapsto \tilde{\alpha}_{i}^{(s)}$ and $\mathfrak{K}^{(r)}_{K+1} \mapsto \tilde{\gamma}_{\delta}^{(r)}$. 
    
\end{thm}

\begin{proof}
The map is well defined by Proposition \ref{Isomorphismloc}, Equation \ref{imageofydelta}, and Proposition \ref{imageimaginary}. By Corollary \ref{imageofcoha}, the morphism $\tilde{\mathrm{\Phi}}$ is a surjection. The algebra $\mathcal{A}^{\T}_{\QTC{K},\WTC{K}}$ is a free $\C[t_1,t_2]$ module by Theorem \ref{free}. After localizing, $\tilde{\mathrm{\Phi}} \otimes_{\C[t_1,t_2]} \C(t_1,t_2)$ is just the morphism $\mathrm{\Phi} \otimes_{\C[t_1,t_2]} \C(t_1,t_2)$ in the Proposition \ref{Isomorphismloc}, which is an injection. So $\tilde{\mathrm{\Phi}}$ is an injection, and we are done. 
\end{proof}

\section{Application: Maulik-Okounkov Yangian for Cyclic Quiver} \label{moyangianconjecture}

In this section, we calculate $\HHT$ linear algebras Maulik-Okounkov Yangians $\mathbf{Y}^{T,\MO}_Q$ in the case of cyclic quivers. We note the emphasis on $\HHT$ linearity, as the Nakajima quiver varieties are not proper, so after localizing, i.e., $\mathbf{Y}^{T,\MO}_Q \otimes_{\HHT} \mathrm{Frac}(\HHT)$ forgets a lot of the geometry of Nakajima quiver varieties. 

\subsection{Maulik-Okounkov Yangian}

Following \cite{schiffmann2023cohomological}, we first quickly recall the definition of the MO Yangian. Let $Q$ be any quiver and Let $T$ be any torus, defined by weighting $\mathbf{w}: \overline{Q}_1 \rightarrow \mathbb{Z}^r$ such that $\rho= \sum_{a \in Q_1} [a,a^{*}]$ is homogenous. We recall that the torus associated to this is given by $\Hom(\mathbb{Z}^r,\mathbb{C}^{*})$. In particular for $\rho$, we get a homomorphism $T \rightarrow \mathbb{C}^{*}$ given by $t \mapsto t(\mathbf{w}(\rho))$, which by taking pullback, defines a cohomology class $\mathbf{t}(\rho) \in \HHT$. We define $\hbar:= \mathbf{t}(\rho)$ and we assume that $\hbar \neq 0$\footnote{The assumption $\hbar \neq 0$ is crucial as only in this case MO Yangians are defined.}.Given any framing $\bff \in \mathbb{N}^{Q_0}$, we can extend weighting $\mathbf{w}^{\prime}$ to a weighting $\mathbf{w}: \overline{Q_{\bff}}_1 \rightarrow \mathbb{Z}^r$ defined by $\mathbf{w}^{\prime}(a)=\mathbf{w}(a)$ for all $a \in \overline{Q}_1$ and $\mathbf{w}^{\prime}(\alpha_{i,m})=\mathbf{w}(\rho),\mathbf{w}^{\prime}(\alpha_{i,m}^{*})=0$. Let  \[\Nak^{\tilde{\mathrm{GL}}_{\bff},\zeta^{\infty}}_{\bff}(Q) := \bigoplus_{\bd \in \N^{Q_0}} \HH^{\tilde{\mathrm{GL}}_{\bff}}(\Nak^{\zeta^{\infty}}_{\bd,\bff}(Q),\mathbb{Q}^{\vir}).\] be the equivariant cohomology of Nakajima quiver varieties, where $\tilde{\mathrm{GL}}_{\bff} = \mathrm{GL}_{\bff} \times T$, and $\mathrm{GL}_{\bff}= \prod_{i \in Q_0} \mathrm{GL}_{\bff_i}$ is the gauge group which reparametrize the framing. Then in \cite{maulik2018quantum}, by developing theory of Stable envelopes, authors define certain morphisms \[R_{\bff_1,\bff_2}(Q) \in  \mathrm{End}_{\HH_{\tilde{\mathrm{GL}}_{\bff}}}(\Nak^{\tilde{\mathrm{GL}}_{\bff_1},\zeta^{\infty}}_{\bff_1}(Q) \otimes_{\HHT} \Nak^{\tilde{\mathrm{GL}}_{\bff_2},\zeta^{\infty}}_{\bff_2}(Q))\otimes_{\HH_{\tilde{\mathrm{GL}}_{\bff}}(\pt)} \mathrm{Frac}(\HH_{\tilde{\mathrm{GL}}_{\bff}}(\pt)) \]
for any choice of framing $\bff_1,\bff_2 \in \mathbb{N}^{Q_0}$, where $\bff= \bff_1+\bff_2$. 

By the way the theory is structured, these morphisms automatically satisfy the Yang-Baxter equations. They prove that for any $\bff=\bff_1+\bff_2+\bff_3$, we have

\[ R^{12}_{\bff_1,\bff_2}(Q) R^{13}_{\bff_1,\bff_3}(Q)  R^{23}_{\bff_2,\bff_3}(Q) = R^{23}_{\bff_2,\bff_3}(Q)  R^{13}_{\bff_1,\bff_3}(Q)  R^{12}_{\bff_1,\bff_2}(Q)  \]

Let $S$ be any ring containing $\mathbb{Q}$. Given a collection of $S$ modules $F_i$ and spectral R matrices $R_{ij} \in \mathrm{End}(F_i \otimes F_j)(u)$, satisfying the Yang-Baxter Equation, there is a general way called the FRT formalism, named after Faddeev, Reshetikhin, and Takhtadzhyan, which constructs a quasi-triangular Hopf algebra which naturally acts on $F_i \otimes S(u)$. This was exploited by Maulik and Okounkov to construct the Maulik-Okounkov Yangian.  

For any quiver $Q$, let \[A^{T,\zeta^{\infty}}_Q =  \bigoplus_{\bff \in \mathbb{N}^{Q_0}} \End_{\tilde{\mathrm{GL}_{\bff}}} \left(\Nak^{\tilde{\mathrm{GL}}_{\bff},\zeta^{\infty}}_{\bff}(Q) \right).\]

Then the MO Yangians are defined to be a subalgebra 
\[ \mathbf{Y}^{T,\MO}_{Q} \subset A^{T,\zeta^{\infty}}_Q \] defined by certain operators using above $R$ matrices, as explained in \cite{maulik2018quantum}[Section 5.2.6]. Due to the work of Botta-Davison in \cite{botta2023okounkovs} and Schiffmann-Vasserot in \cite{schiffmann2012cherednik}, there is now an understanding of this subalgebra in terms of cohomological Hall algebras, which we now recall. 

\subsection{Relation with Cohomological Hall algebras}

For any Nakajima quiver variety $\mathrm{N}^{\zeta^{\infty}}_{\bd,\bff}(Q)$, we have tautological bundles $\mathcal{V}_i$ and $\mathcal{W}_i$ as defined in Section \ref{tautologicalbundle}. These defines operators $\ch_l(\mathcal{V}_i), \ch_{l}(\mathcal{W}_i)$ acting by cup product on there cohomology. We see them as elements inside $A^{T}_Q$. In Section \ref{cohaactionnakajimainfinity}, we saw that there is action of $\mathcal{A}^{T}_{\tilde{Q},\tilde{W}}$ and $\mathcal{A}^{T,\tilde{\mathcal{N}}}_{\tilde{Q},\tilde{W}}$ on $\Nak^{\tilde{\mathrm{GL}}_{\bff}, \zeta^{\infty}}_{\bff}(Q)$, and so we have morphisms
\begin{align}
   \rho_Q \colon \mathcal{A}^{T,\chi}_{\tilde{Q},\tilde{W}} & \rightarrow  A^{T,\zeta^{\infty}}_Q \\
   \rho^{\mathcal{N}}_Q  \colon\mathcal{A}^{T,\tilde{\mathcal{N}},\chi}_{\tilde{Q},\tilde{W}} & \rightarrow A^{T,\zeta^{\infty}}_Q 
\end{align}

Assume that the weighting $\mathbf{w}: \overline{Q}_1 \rightarrow \mathbb{Z}^r$ is defined so that it factors through weighting $\mathbf{w}^{\prime \prime}: \overline{Q}_1 \rightarrow \mathbb{Z}^2$ given by $\mathbf{w}^{\prime \prime}(a)=(1,0), \mathbf{w}^{\prime \prime}(a^{*})=(0,1)$. Then it is proved in \cite{botta2023okounkovs}[Corollary 1.4] that the morphism $\rho_Q$ is an injection, while the same is proved in \cite{schiffmann2017cohomological} for $\rho^{\mathcal{N}}_{Q}$. Combining Theorem 4.13 and Lemma 3.18 of \cite{schiffmann2023cohomological}, and Corollary 1.4 of \cite{botta2023okounkovs}, it follows that
\begin{thm} \label{moyangiancoha}
The Maulik-Okounkov Yangian can be identified by the subalgebra of the cohomology ring of Nakajima quiver variety generated by the image of $\rho_Q,\rho^\mathcal{N}_Q$ and tautological operators $\mathrm{ch}_{k}(\mathcal{V}_i)$ and $\mathrm{ch}_k(\mathcal{W}_i)$. i.e we have \[\mathbf{Y}^{T, \MO}_{Q} = \langle \mathcal{A}^{T}_{\tilde{Q},\tilde{W}},\mathcal{A}^{T,\tilde{\Ncal}}_{\tilde{Q},\tilde{W}},\ch_k(\mathcal{V}_i), \ch_{k}(\mathcal{W}_i) \rangle \subset A^{T,\zeta^{\infty}}_Q  \]  as $\HH_{T}(\pt)$ algebras. 
\end{thm}

We now specialize to the case when $Q=Q^K$ is a cyclic quiver and torus $T=\mathbb{T}$. 

By Remark \ref{nilpotent}, we may identify $\mathcal{A}^{\mathbb{T},\tilde{\mathcal{N}},\chi}_{\tilde{Q^K},\tilde{W^K}}$ with the positive Half of the affine Yangian $\ddot{\mathcal{Y}}^{+}_{-t_2,t_1+t_2}(\mathfrak{gl}(K+1))$. However, in the view that the action on the Nakajima quiver variety is the right action, it is more natural to identify the nilpotent CoHA with the opposite algebra of the affine Yangian, i.e., with the negative half. This statement is also proved in \cite{salalong}[Theorem III.7.3]. We have

\begin{prop}
Let $\ddot{\mathcal{Y}}^{-}_{\hbar_1,\hbar_2}(\mathfrak{gl}(n))$ be the negative Half of the affine Yangian, i.e. the subalgebra $\langle X^{-}_{i,r} \mid i \in \mathbb{Z}/n\mathbb{Z}, r \geq 0 \rangle$ of $ \mathcal{Y}^{(n),\CoHA}_{\hbar_1,\hbar_2}$ defined in Definition \ref{cohayangiandefn}. Then for any $K \geq 1$, we have an isomorphism of algebras $\mathcal{A}^{\mathbb{T},\tilde{\mathcal{N}}}_{\QTC{K},\WTC{K}} \simeq \ddot{\mathcal{Y}}^{+}_{-t_2,t_1+t_2}(\mathfrak{gl}(K+1))$ such that for any framing $\bff$, this isomorphism induces left action of $\ddot{\mathcal{Y}}^{-}_{-t_2,t_1+t_2}(\mathfrak{gl}(K+1))$ on $\Nak^{\tilde{\mathrm{GL}}_{\bff},\zeta^{\infty}}_{\bff}(Q^K)$. 
\end{prop}

In Theorem \ref{equivcoha}, we proved that $\mathcal{A}^{\mathbb{T}}_{\tilde{Q^K},\tilde{W^K}}$ is isomorphic to an integral form $\mathcal{Y}^{+,(K+1),\CoHA}_{-t_2,t_1+t_2}$ of the affine Yangian and thus inducing action of the integral form on $\Nak^{\tilde{\mathrm{GL}}_{\bff},\zeta^{\infty}}_{\bff}(Q^K)$ for any framing $\bff$. It suffices to understand how the positive and negative halves glue, which we can do by the work of Varagnolo.

\subsection{Varagnolo Action of Yangian}

We consider the action of the full affine Yangian on the Nakajima quiver variety, due to Varagnolo \cite{varagnolo2000}, which is based on work of Nakajima in the context of equivariant $K$ theory. In the case of two-parameter deformation and $K=1$, this action was also considered in \cite{koderayangian} and in \cite{salalong}. The action is defined by convolution, in a similar way to Nakajima's action of Kac-Moody Lie algebras we considered earlier.

Let $\bff$ be any framing vector. Given two dimension vectors $\bd_1$ and $\bd_2$, %let $\pi_i$ be the composition of maps $\Nak^{\zeta^{\infty}}_{\bd_i,\bff} \rightarrow \Nak^{0}_{\bd_i,\bff} \hookrightarrow \Nak^{0}_{\bd_1+\bd_2,\bff}$. This allows to define the Steinberg Variety \[ Z(\bd_1,\bd_2,\bff) =  \Nak^{\zeta^{\infty}}_{\bd_1,\bff} \times_{\Nak^{0}_{\bd_1+\bd_2,\bff}} \Nak^{\zeta^{\infty}}_{\bd_2,\bff}.\]

Let $\bd_2=\bd_1+ \delta_i$ for some $i \in Q_0$. Then we define subvariety $C^{+}_i(\bd,\bff) \subset \Nak^{\zeta^{\infty}}_{\bd,\bff}(Q) \times \Nak^{\zeta^{\infty}}_{\bd+\delta_i,\bff}(Q) $ given by pairs $(\rho_1(a),\rho_1(a^{*}),\rho_1(\alpha_{i,l}),\rho_1(\alpha^{*}_{i,l}))$ and $(\rho_2(a),\rho_2(a^{*}),\rho_2(\alpha_{i,l}),\rho_2(\alpha^{*}_{i,l}))$ such that the representation $\rho_2$ is quotient of representation $\rho_1$. It is proved in \cite{nakajimaloop} that these subvarieties are Lagrangian and smooth. We can then define line bundle $\mathcal{L}_i$ on $C^{+}_i(\bd,\bff)$ by $p_2^{*}\mathcal{V}_{i}/p_1^{*}\mathcal{V}_{i}$ where $p_1: C^{+}_i(\bd,\bff) \rightarrow \Nak^{\zeta^{\infty}}_{\bd,\bff}(Q)$ and $p_2: C^{+}_i(\bd,\bff) \rightarrow \Nak^{\zeta^{\infty}}_{\bd+\delta_i,\bff}(Q) $ are the projection maps and $\mathcal{V}_i$ are the tautological bundles at vertex $i$.  Then for any $\bd$, one define morphisms $x^{+}_{i,r}: \HH^{\tilde{\mathrm{GL}}_{\bff}}(\Nak^{\zeta^{\infty}}_{\bd,\bff}(Q),\mathbb{Q}^{\vir}) \rightarrow \HH^{\tilde{\mathrm{GL}}_{\bff}}(\Nak^{\zeta^{\infty}}_{\bd+\delta_i,\bff}(Q),\mathbb{Q}^{\vir})$ by $(-1)^{(\delta_i,\bd)} (p_2)_{*}((c_1(\mathcal{L}_i)^r \cup p_1^{*}(\alpha))$ and $x^{-}_{i,r}: \HH^{\tilde{\mathrm{GL}}_{\bff}}(\Nak^{\zeta^{\infty}}_{\bd,\bff}(Q),\mathbb{Q}^{\vir}) \rightarrow \HH^{\tilde{\mathrm{GL}}_{\bff}}(\Nak^{\zeta^{\infty}}_{\bd-\delta_i,\bff}(Q),\mathbb{Q}^{\vir})$ by  $(-1)^{(\delta_i,\bd-\delta_i)} (p_1)_{*}((c_1(\mathcal{L}_i)^r \cup p_2^{*}(\alpha))$.

Given any $(\bd,\bff)$ and $i \in [0,K], s \in \mathbb{Z}_{\geq 0}$. let $h_{i,s}$ be the element in the tautological ring of $\Nak^{\zeta^{\infty}}_{\bd,\bff}(Q)$ defined by

\[ \mathrm{Res}_{z=0} \left( z^{s} \left(-1+  \frac{\lambda_{-1/z}(\mathcal{F}_i(\bd,\bff))}{\lambda_{-1/z}(q_1q_2\mathcal{F}_i(\bd,\bff))} \right) \right)/(t_1+t_2) \]  where $\mathcal{F}_{i}(\bd,\bff)$ is the element in the equivariant $K$ theory $K^{T \times \mathrm{GL}_{\bff}}(\Nak^{\zeta^{\infty}}_{\bd,\bff}(Q))$ given by \[\mathcal{F}_{i}(\bd,\bff) =  (q_1q_2)^{-1}\mathcal{W}_i-(1+(q_1q_2)^{-1})\mathcal{V}_i +  q_1^{-1}\mathcal{V}_{i-1} + q_2^{-1} \mathcal{V}_{i+1} \]

Here $q_1,q_2$ are line bundles associated to action of $\mathbb{T}$, i.e $c_1(q_1)=t_1$ and $c_1(q_2)=t_2$ and $\lambda_{1/z}(V)$ is the equivariant chern polynomial $\lambda_{1/z}(V) = \sum_{i \geq 0} c_i(V)/z^i$. Then we have 

\begin{thm}[\cite{varagnolo2000}]
For any framing $\bff$, there is an action of affine Yangian $\ddot{\mathcal{Y}}_{-t_2,t_1+t_2}(\mathfrak{gl}(n))$ on $\Nak^{\tilde{\mathrm{GL}}_{\bff},\zeta^{\infty}}_{\bff}(Q^K)$, given by $X^{\pm}_{i,r} \mapsto x^{\pm}_{i,r}$ and $H_{i,s} \mapsto h_{i,s} \cup $
\end{thm}

Thanks to Theorem \cite{yang2017cohomological}[Theorem 5.6], we can see the operators $x^{\pm}_{i,r}$ as an action of the spherical part of CoHA. 

\begin{prop}
Under the morphism $\ddot{\mathcal{Y}}^{+}_{-t_2,t_1+t_2}(\mathfrak{gl}(K+1)) \rightarrow \mathcal{A}^{\mathbb{T},\chi}_{\QTC{K},\WTC{K}}$, the Varagnolo's action of $\ddot{\mathcal{Y}}^{+}_{-t_2,t_1+t_2}(\mathfrak{gl}(K+1))$ on the $\Nak^{\tilde{\mathrm{GL}}_{\bff},\zeta^{\infty}}_{\bff}(Q^K)$ coincides with  $\mathcal{A}^{\mathbb{T},\chi}_{\QTC{K},\WTC{K}}$ action defined in Section \ref{cohaactionequivarianthilbert}. Similarly the Varagnolo's action of $\ddot{\mathcal{Y}}^{-}_{-t_2,t_1+t_2}(\mathfrak{gl}(K+1))$ can be identifed with the action of $\mathcal{A}^{\mathbb{T},\mathcal{N},\chi}_{\QTC{K},\WTC{K}}$ under the isomorphism  $\ddot{\mathcal{Y}}^{-}_{-t_2,t_1+t_2}(\mathfrak{gl}(K+1)) \rightarrow \mathcal{A}^{\mathbb{T},\tilde{\mathcal{N}},\chi}_{\QTC{K},\WTC{K}}$. 
\end{prop}

Furthemore, by \cite{salalong}[Theorem III.7.8], the subalgebra $\langle \ch_{k}(\mathcal{V}_i), \ch_k(\mathcal{W}_i) \rangle \subset A^{\mathbb{T},}_{Q^K }$, generated by tautological classes is isomoprhic to the subalgebra $\langle h_{i,s}, \ch_{k}(\mathcal{W}_i) \rangle \subset A^{\mathbb{T}}_{Q^K}$. This is simply because universal polynomials exist that relate the Chern polynomial and Chern characters.

Let $\overline{\mathcal{Y}}^{n, \CoHA}_{\hbar_1,\hbar_2}$ be the subalgebra of $ \mathcal{Y}^{n,\CoHA}_{\hbar_1,\hbar_2}$ which is generated by generators $\langle \mathfrak{K}^{(r)}_{+}, X^{\pm}_{i,r}, H_{i,r} \mid i \in \mathbb{Z}/n\mathbb{Z}, r \geq 0 \rangle$ (So we only exclude the generators $\mathfrak{K}^{(r)}_{-}$ for $r \geq 0$). Let $\overline{\mathcal{Y}}^{n, \CoHA,\mathfrak{e}}_{\hbar_1,\hbar_2} = \overline{\mathcal{Y}}^{n, \CoHA}_{\hbar_1,\hbar_2}[\langle k_{i,r} \rangle]$ be the extended algebra, where we have added central elements $\langle k_{i,r} \mid i \in \mathbb{Z}/n\mathbb{Z}, r \geq 0 \rangle$ to the list of generators. 

Since the cup product with tautolgical classes $\ch_k(\mathcal{W}_i)$ commutes, it follows from Theorem \ref{moyangiancoha} that there is a $\mathbb{C}[t_1,t_2]$ linear algebra surjection

\begin{comment}This is simply because for any vector bundle $\mathcal{F}$ with chern roots $x_i(\mathcal{F}), 1 \leq i \leq \mathrm{rank}(\mathcal{F})$, we have \[\frac{\lambda_{-1/z}(\mathcal{F})}{\lambda_{-1/z}(q_1q_2\mathcal{F})} = \frac{\prod_{1 \leq i \leq \mathrm{\rank}(\mathcal{F}} (z - x_{i}(\mathcal{F}))}{\prod_{1 \leq i \leq \mathrm{\rank}(\mathcal{F}} (z-x_{i}(\mathcal{F})-\hbar_1-\hbar_2)} = \prod_{1 \leq i \leq \mathrm{\rank}(\mathcal{F}} \left(1+ \sum_{r \geq 0} \frac{(\hbar_1+\hbar_2)^r}{(z-x_i(\mathcal{F}))^r} \right) \] is invertible? 
\end{comment}

\[\mathrm{\Psi}: \overline{\mathcal{Y}}^{K+1, \CoHA,\mathfrak{e}}_{-t_2,t_1+t_2}  \rightarrow \mathrm{Y}^{\MO,\mathbb{T}}_{Q^K} \subset A^{\mathbb{T}}_{Q^K}\] where the morphism sends $k_{i,r} \mapsto \ch_{r}(\mathcal{W}_i)$.  

\begin{thm}
    The morphism $\mathrm{\Psi}$ is an isomorphism of $\C[t_1,t_2]$ linear algebras. 
\end{thm}

\begin{proof}
It suffices to show that this morphism is an injection. The affine Yangian is a free $\mathbb{C}[t_1,t_2]$ module, and so does the Maulik-Okounkov Yangian by \cite{maulik2018quantum}[Theorem 5.5.1]. Thus, it suffices to show that:

\[ \mathrm{\Psi} \colon (\overline{\mathcal{Y}}^{K+1, \CoHA,\mathfrak{e}}_{-t_2,t_1+t_2} )_{\mathbf{K}} \rightarrow (\MOY^{\MO,\mathbb{T}}_{Q^K})_{\mathbf{K}} \] is an isomorphism, where for any any algebra $A$, $A_{\mathbf{K}}$ denote the localization $ A \otimes_{\mathbb{C}[t_1,t_2]}\mathbb{C}(t_1,t_2)$. Clearly $(\overline{\mathcal{Y}}^{K+1, \CoHA,\mathfrak{e}}_{-t_2,t_1+t_2} )_{\mathbf{K}}  = (\ddot{\mathcal{Y}^{\mathfrak{e}}}_{-t_2,t_1+t_2}(K+1))_{\mathbf{K}}$, where $\ddot{\mathcal{Y}^{\mathfrak{e}}}_{-t_2,t_1+t_2}(K+1))$ denote the extended affine Yangian (We add central elements $k_{i,r}, i \in [0,K], r \geq 0$). Let $(\ddot{\mathcal{Y}^{0,\mathfrak{e}}}_{-t_2,t_1+t_2}(K+1))_{\mathbf{K}}$ be the subalgebra generated by generators $H_{i,s}$ and $k_{i,l}$. Then it shown in  \cite{salalong}[Theorem III.5.6], that the extended affine Yangian $(\ddot{\mathcal{Y}}^{\mathfrak{e}}_{-t_2,t_1+t_2}(K+1)_{\mathbf{K}}$ has a triangular decomposition, i.e the morphism
{\small
\[ m: (\ddot{\mathcal{Y}}^{+}_{-t_2,t_1+t_2}(K+1)_{\mathbf{K}} \otimes (\ddot{\mathcal{Y}}^{0,\mathfrak{e}}_{-t_2,t_1+t_2}(K+1))_{\mathbf{K}} \otimes (\ddot{\mathcal{Y}}^{-}_{-t_2,t_1+t_2}(K+1))_{\mathbf{K}} \rightarrow (\ddot{\mathcal{Y}}^{\mathfrak{e}}_{-t_2,t_1+t_2}(K+1))_{\mathbf{K}} \]} is an isomorphism of vector spaces. It is proved in \cite{schiffmann2023cohomological} that the Maulik-Okounkov Yangian also admits a triangular decomposition, and so we have 

\[ m: (\MOY^{\MO,\mathbb{T},+}_{Q^K})_{\mathbf{K}} \otimes (\MOY^{\MO,\mathbb{T},0}_{Q^K})_{\mathbf{K}} \otimes (\MOY^{\MO,\mathbb{T},-}_{Q^K})_{\mathbf{K}} \rightarrow (\MOY^{\MO,\mathbb{T}}_{Q^K})_{\mathbf{K}},\]

such that there is isomorphism of algebras $(\MOY^{\MO,\mathbb{T},+}_{Q^K})_{\mathbf{K}} \simeq (\mathcal{A}^{\T}_{\QTC{K},\WTC{K}})_{\mathbf{K}}$ and $(\MOY^{\MO,\mathbb{T},-}_{Q^K})_{\mathbf{K}} \simeq (\mathcal{A}^{\T,\tilde{\mathcal{N}}}_{\QTC{K},\WTC{K}})_{\mathbf{K}}$ and $(\MOY^{\MO,\mathbb{T},0}_{Q^K})_{\mathbf{K}} \simeq \langle \ch_k(\mathcal{V}_i), \ch_{k}(\mathcal{W}_i) \rangle$. Thus, the morphism $\mathrm{\Psi}_{\mathbf{K}}$ restricted to each of the parts of the triangular decomposition is an isomorphism, and so we are done by the following lemma. 

\end{proof}

\begin{lemma} \label{triangulardecomisomor}
Suppose $F: A \rightarrow B$ a morphism of algebras where both $A$  and $B$ admits a triangular decomposition $A^{+} \otimes A^{0} \otimes A^{-} \xrightarrow[]{m} A$, $B^{+} \otimes B^{0} \otimes B^{-} \xrightarrow[]{m} B$. Suppose that $F$ restricts to isomorphism of algebras $F^{+}\colon A^{+} \rightarrow B^{+}$, $F^{0} \colon A^{0} \rightarrow B^{0}$ and $F^{-}: A^{-} \rightarrow B^{-}$ then $F$ is an isomorphism of algebras. 
\end{lemma}

\begin{proof}
We have commutative diagram \[\begin{tikzcd}[ampersand replacement=\&]
	A \& B \\
	{A^{+} \otimes A^{0} \otimes A^{-}} \& {B^{+} \otimes B^{0} \otimes B^{-}}
	\arrow["F", from=1-1, to=1-2]
	\arrow["{m_A}", from=2-1, to=1-1]
	\arrow["{F^{+}\otimes F^{0} \otimes F^{-}}"', from=2-1, to=2-2]
	\arrow["{m_B}"', from=2-2, to=1-2]
\end{tikzcd}\] and thus $F$ is isomorphism of vector spaces. Since $F$ is a morphism of algebras, we are done. 
\end{proof}

%Let $\mathcal{Y}^{(K+1),-, \CoHA}_{\hbar_1,\hbar_2} \subset \mathcal{Y}^{(K+1) \CoHA}_{\hbar_1,\hbar_2}$ be the negative half defined as the subalgebra generated by $X^{-}_{i,r}$ and $\mathfrak{K}_{-,K+1}^{(r)}$. Then note that we also have isomorphism $\mathcal{Y}^{(K+1),-, \CoHA}_{\hbar_1,\hbar_2} 
%\simeq \textbf{Y}^{\textrm{MO},-}_{A_{K}^{'}}$ Thus, together, we have an algebra map from the sum 
%\[ \mathcal{Y}^{(K+1),+, \CoHA}_{\hbar_1,\hbar_2} \coprod \mathcal{Y}^{(K+1),-, \CoHA}_{\hbar_1,\hbar_2} \rightarrow \textbf{Y}^{\textrm{MO}}_{A_{K}^{'}}  \]
%which is surjective since the positive and negative half generates the whole algebra. 

%Idea: Use minimalistic presentation as in Guay-Nakajima-W for coproduct

%\subsection{Representation Theory}
%It's possible to show that the algebra $\mathcal{A}_{\QTC{K},\WTC{K}}$ also acts on the cohomology of affine Laumon spaceinstanton moduli space. An instanton on the surface $S$ is a torsion-free sheaf $\Fcal$ on the compactified resolution $\tilde{S} \rightarrow \PP^2/G$ along with a choice of framing $\Fcal|_{\PP} \simeq \O_{\PP}^{\oplus r}$ on the line at infinity $\PP \subset \PP^2$. Let $\Mfrak^{n}_{r}(S)$ be the moduli of rank $r$ instantons on $S$ of instanton number $n$. We then have by \cite{MR2372210}, that \[ \Mfrak^{n}_{r}(S) \simeq \Mcal^{\zeta_{m}}(n \cdot \delta, r \delta_{0})\]
%\end{remark}
%This possibly gives a new/conjectural answer to a question of Kuznetsov concering action of $\hat{\mathfrak{gl}}_K$ on the the Affine Laumon space. We pursue this direction in future work with Tommaso Botta. 

\section{Deformed Affinized BPS Lie algebra for Cyclic Quivers} \label{sectiondeformedaffinized}

We recall that in Section \ref{affinizedBPSliealgebra}, we gave $\mathcal{A}^{\C^{*},\chi} _{\QTC{K},\WTC{K}}$ a cocommutative coproduct by considering the eigenvalues of $\omega := \sum \omega_i$, and we concluded that there is an isomorphism of algebras \[ \mathcal{A}^{\C^{*},\chi}_{\QTC{K},\WTC{K}} \simeq \bU_{\C[t]}(\widehat{\mathfrak{g}}^{\BPS, \C^{*}}_{\QTC{K},\WTC{K}}), \] where $\widehat{\mathfrak{g}}^{\BPS,\C^{*}}_{\QTC{K},\WTC{K}}$ is a $\C[t]$ linear Lie algebra, called the $\C^{*}$ deformed affinized BPS Lie algebra. Previous results allows us to calculate the Lie algebra $\widehat{\mathfrak{g}}^{\BPS,\C^{*}}_{\QTC{K},\WTC{K}}$ explicitly. Since $\mathfrak{g}^{\BPS,\C^{*}}_{\QTC{K},\WTC{K}} \simeq \mathfrak{g}^{\BPS}_{\QTC{K},\WTC{K}} \otimes \C[t]$ as graded vector spaces, let $\hat{\alpha}_{i}$ and $\hat{\gamma}_{\delta}$ be unique non-canonical lift of $\alpha_i, \gamma_{\delta} \in \mathfrak{g}^{\BPS}_{\QTC{K},\WTC{K}}$. We then have

%However since the morphism $ \mathfrak{M}_{\bd}(\tilde{Q}) \rightarrow \Sym(\A^{1}) $ where $\lambda(\rho)$ are the generalized eigenvalues of $\omega$ is $\C^{*}$ equivariant, it follows that this map infact lifts to \[\mathfrak{M}_{\bd}^{\C^{*}}(\tilde{Q}) \rightarrow \Sym(\A^{1}) \] and so the construction of factorization coproduct also works for $\mathcal{A}^{\C^{*}}_{\QTC{K},\WTC{K}}$.
 
\begin{prop} \label{generationofdeformedaffinized}
The $\C[t]$ linear Lie algebra $\widehat{\mathfrak{g}}^{\BPS,\C^{*}}_{\QTC{K},\WTC{K}}$ is generated by $u^{k} \cdot \hat{\alpha}_{i} $ for all $i \in [0,K], k \geq 0$ and $u^k \cdot \hat{\gamma}_{\delta}$ for all $ k \geq 0$. 
\end{prop}
\begin{proof}
Note that since $\mathcal{A}^{\C^{*},\chi}_{\tilde{Q},\tilde{W}}$ is a free $\C[t]$ module, the isomorphism of algebras  $\mathcal{A}^{\C^{*},\chi}_{\tilde{Q},\tilde{W}}\simeq \bU_{\C[t]} (\widehat{\mathfrak{g}}^{\BPS,\C^{*}}_{\QTC{K},\WTC{K}})$ implies that $\widehat{\mathfrak{g}}^{\BPS,\C^{*}}_{\QTC{K},\WTC{K}}$ is a free $\C[t]$ module. 
We define a map of free $\C[t]$ modules
\[J: \abps{K} \otimes \C[t] \rightarrow  \widehat{\mathfrak{g}}^{\BPS,\C^{*}}_{\QTC{K},\WTC{K}}, \]defined by $J(\alpha^{(k)}_i f(t)): =  u^k \cdot \hat{\alpha}_i$ and $J(\gamma^{(k)}_{\delta}) f(t))= f(t) u^k \cdot \hat{\gamma}_{\delta}$ and extending to $\widehat{\mathfrak{g}}^{\BPS,\C^{*}}_{\QTC{K},\WTC{K}}$ by the Lie algebra relations on $\abps{K}$ (Theorem \ref{Theorem1}). It suffices to show that the morphism $J$ is surjective. Let $m_{\C^{*}}= (t) \subset \C[t]$ be the maximal ideal. Then the morphism $J/(m_{\C^{*}})$ is an isomorphism of $\N^{Q_0} \times \Z$ graded modules, where each graded piece is finite-dimensional and thus $J$ is an isomorphism of vector spaces and so in particular surjective.
\end{proof}

We now calculate $\mathcal{A}^{\C^{*},\chi}_{\QTC{K},\WTC{K}}$. Since there is an isomorphism of algebras 
\[ \mathcal{A}^{\C^{*},\chi}_{\QTC{K},\WTC{K}} \simeq \mathcal{A}^{\T,\chi}_{\QTC{K},\WTC{K}} \otimes_{\C[t_1,t_2]} \C[t_1,t_2]/(t_1+t_2),\] it follows by Theorem \ref{equivcoha}, that there is an isomorphism of algebras 
\[(\mathcal{Y}^{(K+1),+,\CoHA}_{-t_2,t_1+t_2})_{t_1+t_2 =0 } \simeq \mathcal{A}^{\C^{*},\chi}_{\QTC{K},\WTC{K}}.\]
But setting $t_2=\hbar$, yields
\begin{align} 
(\mathcal{Y}^{(K+1),+,\CoHA}_{-t_2,t_1+t_2})_{t_1+t_2 =0 } & \simeq \bU_{\C[\hbar]}(S_{K+1})  \end{align} and an isomorphism of Lie algebras \begin{align}
\widehat{\mathfrak{g}}^{\BPS,\C^{*}}_{\QTC{K},\WTC{K}} & \simeq S_{K+1}, \label{eq1}
\end{align} where $S_{K+1}$ is a $\C[\hbar]$ linear Lie algebra generated by $X^{+}_{i,r}, r \geq 0, i \in [0,K]$ and $\mathfrak{K}^{r}_{K+1}, r \geq 0$ with the defining relations 
\begin{align}
[X_{i,r+1}^{+},X^{+}_{i+1,s}] - [X^{+}_{i,r},X^{+}_{i+1,s+1}] &= -\hbar [X^{+}_{i,r},X^{+}_{i+1,s}] \label{relationscLie1}   \\ 
[X_{i,r+1}^{+},X^{+}_{i-1,s}] - [X^{+}_{i,r},X^{+}_{i-1,s+1}] &= \hbar [X^{+}_{i,r},X^{+}_{i-1,s}]  \label{relationscLie2}\\ 
[X_{i,r+1}^{+},X^{+}_{i,s}] - [X^{+}_{i,r},X^{+}_{i,s+1}] &=  0 \label{relationscLie3} \\ 
[X^{+}_{i,r},X^{+}_{j,s}] &= 0 \ \forall |i-j|> 1 \label{relationscLie4}\\
 \sum_{\sigma \in S_2} [X^{+}_{i,r_{\sigma(1)}},[X^{+}_{i,r_{\sigma(2)}},X^{+}_{i+1,s}]] &= 0 \label{relationscLie5}\\
\TT^r(\mathfrak{Z}_{K+1}) &= \hbar^2 \mathfrak{K}_{K+1}^{(r)} \label{M1} 
\end{align}
when $K>1$ and
\begin{subequations}
\begin{align}
[X^{+}_{i,r+1},X^{+}_{i,s}] - [X^{+}_{i,r},X^{+}_{i,s+1}] &= 0 \label{r222}\\
[X^+_{i,r+2},X^{+}_{i+1,s}] - 2[X^{+}_{i,r+1},X^{+}_{i+1,s+1}]+[X^{+}_{i,r},X^{+}_{i+1,s+2}] &= \hbar^2[X^+_{i,r},X^{+}_{j,s}]  \label{ss}\\ 
\mathrm{Sym}_{r_1,r_2,r_3} [X^{+}_{i,r_1},[X^{+}_{i,r_2},[X^{+}_{i,r_3},X^{+}_{i+1,s}]]] &= 0  \\
\TT^r (\mathfrak{Z}_2) &=  \hbar^2  \mathfrak{K}^{(r)}_{2}  \label{M2}
\end{align} 
\end{subequations}
when $K=2$. 

The isomorphism of the $\C[\hbar]$ linear Lie algebras $\widehat{\mathfrak{g}}^{\BPS,\C^{*}}_{\QTC{K},\WTC{K}} \rightarrow S_{K+1}$ is given by $u^r \cdot \hat{\alpha}_i \mapsto X^{+}_{i,r}$ for all $i \in [0,K], r \geq 0$ and $u^r \cdot \hat{\gamma}_{\delta} \mapsto \mathfrak{K}_{K+1}^{(r)}, r \geq 0$. The Lie algebra $S_{K+1}$ can be described more explicitly. Recall the Lie algebra $(\mathcal{W}_{K+1})^{+}$ from Section \ref{algebras}. We consider a derivation of $(\mathcal{W}_{K+1})^{+}$, defined by $\ad_{P_{K+1}}:= [P_{K+1}, \cdot]$ where 
\[ P_{K+1}= \frac{(K+1)^2}{2} t_{0,2}  + (K+1)T_{0,1}(H^{\prime}_{K+1})+  (K+1) \hbar T_{0,0}(H^{\prime \prime}_{K+1})\]
such that $H^{\prime}_{K+1} = \sum_{i=1}^{K+1} (K+3/2-i)E_{i,i}$ and $H^{\prime \prime}_{K+1} = \sum_{i=1}^{K+1} \mu_i E_{i,i}$ for any $\mu_i$ satisfying  $\mu_i-\mu_{i+1} = (1-i/(K+1))$ and $\sum \mu_i =0$. They are chosen so that $[H^{\prime}_{K+1},E_{i,i+1}] = E_{i,i+1}$ for all $i \in [1,K]$, $[H^{\prime}_{K+1},E_{K+1,1}]=-K$ and $[H^{\prime \prime}_{K+1},E_{i,i+1}] = (1-i/(K+1))E_{i,i+1}$ for all $i \in [1,K].$ This might look adhoc, but important property is that they satisfy equations \ref{impproperty}. 
\begin{prop} \label{positivedegisom}
We have an isomorphism of $\C[\hbar]$ linear Lie algebras \[ G \colon S_{K+1}  \rightarrow (\mathcal{W}_{K+1})^{+} \] given by 
\begin{align*}
    X^{+}_{0,r} & \mapsto \ad_{P_{K+1}}^{r}(T_{1,0}(E_{K+1,1})) \\ 
    X^{+}_{i,r} & \mapsto \ad_{P_{K+1}}^{r}(T_{0,0}(E_{i,i+1}))  \\
    \mathfrak{K}^{(r)}_{K+1} & \mapsto \ad_{P_{K+1}}^{r}(t_{1,0})
\end{align*}
\end{prop}
\begin{proof}
We see that in terms of matrices 
\begin{align} \label{impproperty}
\ad_{P_{K+1}}^{r}(T_{1,0}(E_{K+1,1})) &=  (-1)^r(K+1)^r D^{r}z E_{K+1,1} \\
\ad_{P_{K+1}}^{r}(T_{0,0}(E_{i,i+1})) &= (-1)^r(K+1)^r(D + (1-i/(K+1))\hbar)^{r} E_{i,i+1} 
\end{align}

It can be checked explictly that the defining relations (\ref{relationscLie1})(\ref{relationscLie2})(\ref{relationscLie3}) (\ref{relationscLie4}) (\ref{relationscLie5}) and similarly the ones for $K=1$ of $S_{K+1}$ are satisfied. Recall that we defined 
\[ \mathfrak{Z}_{K+1}  = \sum_{i \in \mathbb{Z}/(K+1)\mathbb{Z}} [X^{+}_{i,0},[X^{+}_{i+1,1}, [X^{+}_{i+2,0},[ \cdots,[X^{+}_{i+K-1,0},X^{+}_{i+K,0}]]]]] \]

Let $R_i = [X^{+}_{i,0},[X^{+}_{i+1,1}, [X^{+}_{i+2,0},[ \cdots,[X^{+}_{i+K-1,0},X^{+}_{i+K,0}]]]]]$. Then
When $i=0,K$, we calculate that 

\[G(R_i) = (-1)(K+1) \left( zD+ (2-\frac{i+1}{K+1}) \hbar z\right) (E_{i,i}-E_{i+1,i+1})\]

When $i=0,K$, we calculate that
\begin{align*}
G(R_0) &= (-1)(K+1) \left( zD(E_{K+1,K+1}-E_{1,1}) + \frac{\hbar z K(E_{K+1,K+1}-E_{1,1}}{K+1}-\hbar z E_{1,1} \right) \\
G(R_K) &= (-1)(K+1) \left( zD(E_{K,K}-E_{K+1,K+1}) + \hbar z (E_{K,K}-E_{K+1,K+1}) \right) \\
\end{align*}
Then summing up them up gives \[ G(\mathfrak{Z}_{K+1}) = \sum_{i=0}^{K} G(R_i) = \hbar z  \mathrm{Id}\]
Thus $G(\mathfrak{Z}_{K+1})= \hbar^2 t_{1,0}$, thus relations \ref{M1} and \ref{M2} are satisfied for $r=0$. But note that $G(T(X^{+}_{i,r})) = G(X^{+}_{i,r+1})= \ad_{P_{K+1}}(G(X_{i,r}))$ thus these relations are also satisfied for $r>1$. It is a surjective map of Lie algebras since $(\mathcal{W}_{K+1})^{+}$ is generated by $t_{1,r}, r \geq 0$ and $T_{0,r}(E_{i,i+1}), T_{1,r}(E_{K+1,1})$, all of which are inductively in the image of $G$. To show the injection, we note that $G$ respects both cohomological and dimension grading on $S_{K+1}$, induced by the isomorphism $ \widehat{\mathfrak{g}}^{\BPS,\C^{*}}_{\QTC{K},\WTC{K}} \simeq S_{K+1}$. It is easy to check by Proposition $\ref{bpsLiealgebracyclic}$ that the graded dimensions of $S_{K+1}$ and $(\mathcal{W}_{K+1})^{+}$ are finite dimensional and are the same. Thus $G$ is, in fact, an isomorphism. Alternatively, it is an injective map after localizing with $\C[\hbar^{-1}]$ since then $G[\hbar^{-1}]$ is the positive half of the isomorphism of Lie algebras considered in \cite{Tsymbaliuk_2017}[Theorem 2.2]. 
\end{proof}
Combining Proposition \ref{positivedegisom} and Equation (\ref{eq1}) implies 

%\begin{proof}

%We have isomorphism of algebras 

%\[ \mathcal{A}^{\C^{*}}_{\tilde{Q},\tilde{W}} \simeq \mathcal{A}^{T}_{\tilde{Q},\tilde{W}} \otimes_{\C[t_1,t_2]} \C[t_1,t_2]/(t_1+t_2) \]

%So by Theorem \ref{equivcoha}, that there is an isomorphism of algebras 

%\[\Psi: (\mathcal{Y}^{(K+1),+,\CoHA}_{-t_1,t_1+t_2})_{t_1+t_2 =0 } \simeq \mathcal{A}^{\C^{*}}_{\tilde{Q},\tilde{W}} \simeq \bU(\hat{\mathfrak{g}^{\C^{*}}_{\mathrm{\Pi}_Q}}) \]
%given by $K^{r} \mapsto \gamma_{\delta}^{(r)}$ and $X^{+}_{i,s} \mapsto \alpha_{i}^{s}$. 
    
%\end{proof}

\begin{thm}\label{Theorem2}
There is an isomorphism of of $\C[\hbar]$ linear Lie algebras 
\[ \widehat{\mathfrak{g}}^{\BPS,\C^{*}}_{\QTC{K},\WTC{K}} \simeq  (\mathcal{W}_{K+1})^{+}   \] 
%so an isomorphism of algebras
%\[ \bU_{\C[\hbar]}(D_{\hbar}^{+}(\C^{*}) \otimes \tilde{\mathfrak{gl}_{K+1}}) \simeq \mathcal{A}^{\C^{*}}_{\QTC{K},\WTC{K}} \]
    
\end{thm}

\subsection{Relation with algebras of Gaiotto-Rapčák-Zhou}\label{physics} \label{speculations}

In \cite{rapcak2023cohomological}, a procedure to do the Drinfeld double of the spherical part of CoHA is explained. Then, a certain quotient of the Drinfeld double gives rise to \textit{reduced} Drinfeld double \cite{yang2017cohomological}[Section 4.1]. Mimicking the same construction of reduced Drinfeld double to $\mathcal{Y}^{(n),+,\CoHA}_{\hbar_1,\hbar_2}$ yields the algebra $\mathcal{Y}^{(n),\CoHA}_{\hbar_1,\hbar_2}$ for $n>2$, defined in Definition \ref{cohayangiandefn}\footnote{More precisely, we apply the construction to the localized algebra (which is spherically generated due to Proposition \ref{sphgeneration} and \ref{Isomorphismloc}) and take the minimal subalgebra containing both $\mathcal{Y}^{(n),+,\CoHA}_{\hbar_1,\hbar_2}$ and $(\mathcal{Y}^{(n),-,\CoHA}_{\hbar_1,\hbar_2}):=(\mathcal{Y}^{(n),+,\CoHA}_{\hbar_1,\hbar_2})^{\textrm{op}}$.}. Then in an analogous way to \cite{GUAY2007436}, we define 

\begin{defn}[CoHA Loop Yangian]
Let $\mathcal{L}^{(n),\CoHA}_{\hbar_1,\hbar_2}$ be the $\C[\hbar_1,\hbar_2]$ algebra, defined to be the quotient  $\mathcal{Y}^{(n),\CoHA}_{\hbar_1,\hbar_2}/(\textbf{c})$ where $\textbf{c} := \sum_{i=0}^{n-1} H_{i,0}$
\end{defn}

We then have a doubled version of the Proposition
\ref{positivedegisom}, which follows from work of Tsymbaliuk \cite{Tsymbaliuk_2017} who calculated the classical limit of affine Yangians. 

\begin{prop} \label{loopcohayangian}
We have an isomorphism of $\C[\hbar]$ linear algebras 
\[ \mathcal{L}^{(n),\CoHA}_{\hbar,\hbar_2}/(\hbar_2)  \simeq \bU_{\C[\hbar]}(\mathcal{W}_{n}). \]
\end{prop}

\begin{proof}
We use the morphism considered in \cite{Tsymbaliuk_2017}[Theorem 2.2]\footnote{The codomain of the morphism considered in \cite{Tsymbaliuk_2017}[Theorem 2.2] is a central extension, however in the proof, they consider the morphism induced after quotienting by the central extension.} whose positive half is the morphism we considered in Theorem \ref{positivedegisom} i.e., we define a morphism of algebras
\[ \Psi: \mathcal{L}^{(n),\CoHA}_{\hbar,\hbar_2}/(\hbar_2) \rightarrow  \bU_{\C[\hbar]}(\mathcal{W}_{n}) \]  which sends
\begin{align*}
X^{+}_{0,r}  & \mapsto n^r D^{r}z \otimes E_{n,1} \\
X^{+}_{i,r} & \mapsto n^r(D + (1-i/n)\hbar)^{r} \otimes E_{i,i+1} \\ 
X^{-}_{0,r} & \mapsto n^{r} (z^{-1}D^{r}) \otimes E_{1,n} \\
X^{-}_{i,r} & \mapsto  n^r(D + (1-i/n)\hbar)^{r} \otimes E_{i+1,i} \\
H_{0,r} & \mapsto n^r(D^r \otimes E_{n,n}  - (D+\hbar)^r \otimes E_{1,1})\\ 
H_{i,r} & \mapsto n^r (D+(1-i/n)\hbar)^r \otimes (E_{i,i}-E_{i+1,i+1}).
\end{align*}
We haven't defined image of $\mathfrak{K}^{(r)}_{\pm}$, but note that we can still calculate image of elements $\mathfrak{Z}_n$, and we can check that this map must send $\mathfrak{K}^{(0)}_{+}$ to $\frac{z}{\hbar} = T_{1,0}(1)/\hbar = t_{1,0}$ while  $\mathfrak{K}^{0}_{-}$ to $-\frac{z^{-1}}{\hbar} = - T_{-1,0}(1)/\hbar =  - t_{-1,0}$ which allows to extend define using the same trick as in proof of Proposition \ref{positivedegisom}. After inverting $\hbar$, this map is an isomorphism by \cite{Tsymbaliuk_2017}[Theorem 2.2] while it is a surjection, since again $D_{\hbar}(\C^{*}) \otimes \tilde{\mathfrak{gl}_{n}}$ is generated by $t_{1,r},t_{-1,r},T_{-1,r}(E_{1,n}),T_{0,r}(E_{i,i+1}),T_{0,r}(E_{i+1,i}),T_{1,r}(E_{n,1})$ all of which can be easily seen to be in the image.  
\end{proof}

We consider the $\C[\hbar_1,\hbar_2]$ linear algebras\footnote{We are switching $\epsilon_1$ to $\hbar_2$ and $\epsilon_2$ to $\hbar_1$.} $Y^K$ and its loop version $L^K$, as defined in \cite{gaiotto2023deformed}[Definition 7.01, 6.02]. 
By \cite{gaiotto2023deformed}[Remark 6.18], the algebra $L^{K}$ satisfies \[ L^{K}/\hbar_2 \simeq \bU_{\C[\hbar]}(\mathcal{W}_K),\] which by Propostion \ref{loopcohayangian} is also satisfied by $\mathcal{L}^{(n),\CoHA}_{\hbar,\hbar_2}$. They also show that the algebras $Y^{K}$ and $L^{K}$ are the integral forms of the affine Yangian $\ddot{\mathcal{Y}}_{\hbar_1,\hbar_2}(\mathfrak{gl}(K))$ and the loop Yangian $\ddot{\mathcal{L}}_{\hbar_1,\hbar_2}(\mathfrak{gl}(K)) := \ddot{\mathcal{Y}}_{\hbar_1,\hbar_2}(\mathfrak{gl}(K))/\textbf{c}$ respectively. For $K=1$, $Y^{K}$ and $L^{K}$ are in fact isomorphic to the affine Yangian $\ddot{\mathcal{Y}}_{\hbar_1,\hbar_2}(\mathfrak{gl}(1))$ of $\widehat{\mathfrak{gl}(1)}$ and its loop version respectively \cite{gaiotto2023deformed}[Theorem 17]. Furthermore, it is  known due to (\cite{davison2022affine}+ \cite{rapcak2023cohomological}) and \cite{schiffmann2012cherednik} that \[\mathcal{A}^{\T}_{\tilde{Q_{\Jor}},\tilde{W}} \simeq \ddot{\mathcal{Y}}^{+}_{\hbar_1,\hbar_2}(\mathfrak{gl}(1)). \]
We thus expect that there is an isomorphism of algebras for all $K>2$, 
\[ Y^K \simeq \mathcal{Y}^{(K),\CoHA}_{\hbar_1,\hbar_2}\]
\printbibliography

\end{document}